\documentclass[11pt]{article}
\usepackage[utf8]{inputenc}

\usepackage{amssymb, amsmath, amsthm, graphicx, bbm, authblk}
\usepackage{mathtools}
\usepackage{fullpage}
\usepackage[dvipsnames]{xcolor}
\RequirePackage[round]{natbib}
\RequirePackage[hyperindex=true,pdftitle={},pagebackref=false,
colorlinks,citecolor=blue,urlcolor=blue,plainpages=false,
pdfpagelabels]{hyperref}
\usepackage{cleveref}
\usepackage{tikz}
\usetikzlibrary{arrows}

\usepackage{comment}

\usepackage{tikz}
\usetikzlibrary{arrows.meta}
\usetikzlibrary{decorations.pathreplacing}

\usepackage{algorithm}% http://ctan.org/pkg/algorithms
\usepackage{algpseudocode}% http://ctan.org/pkg/algorithmicx
\algnewcommand\algorithmicinput{\textbf{INPUT:}}
\algnewcommand\INPUT{\item[\algorithmicinput]}
\algnewcommand\algorithmicoutput{\textbf{OUTPUT:}}
\algnewcommand\OUTPUT{\item[\algorithmicoutput]}

\DeclareMathOperator*{\argmin}{arg\,min}

\DeclareMathOperator*{\sign}{sign}

\DeclareMathOperator*{\cov}{Cov}
\DeclareMathOperator*{\var}{Var}

\newcommand{\I}{{\mathcal I }}

\newcommand{\J}{{  \mathcal J } }
\newcommand{\E}{ { \mathbb{E} } }
\newcommand{\hatp}{  {\widehat  { \mathcal P }  } }
\newcommand{\s}{ { \mathfrak{s}  } }
\newcommand{\p}{\mathbb P}
\DeclareMathOperator*{\supp}{supp}
\DeclareMathOperator*{\CL}{Cl}
\DeclareMathOperator*{\conv}{conv}
\DeclareMathOperator*{\card}{Card}

\theoremstyle{plain}
\newtheorem{theorem}{Theorem}
\newtheorem{lemma}[theorem]{Lemma}

\newtheorem{proposition}[theorem]{Proposition}
\newtheorem{assumption}{Assumption}
\newtheorem{definition}{Definition}

\theoremstyle{definition}
\newtheorem{remark}{Remark}

\allowdisplaybreaks

\title{Change point inference in high-dimensional regression models under temporal dependence}
\author[1]{Haotian Xu}
\author[2]{Daren Wang}
\author[3]{Zifeng Zhao}
\author[1]{Yi Yu}
\affil[1]{Department of Statistics, University of Warwick}
\affil[2]{Department of Statistics, University of Notre Dame}
\affil[3]{Mendoza College of Business, University of Notre Dame}
\date{\today}

\begin{document}

\maketitle

\begin{abstract}
    This paper concerns about the limiting distributions of change point estimators, in a high-dimensional linear regression time series context, where a regression object $(y_t, X_t) \in \mathbb{R} \times \mathbb{R}^p$ is observed at every time point $t \in \{1, \ldots, n\}$. At unknown time points, called change points, the regression coefficients change, with the jump sizes measured in $\ell_2$-norm. We provide limiting distributions of the change point estimators in the regimes where the minimal jump size vanishes and where it remains a constant. We allow for both the covariate and noise sequences to be temporally dependent, in the functional dependence framework, which is the first time seen in the change point inference literature.  We show that a block-type long-run variance estimator is consistent under the functional dependence, which facilitates the practical implementation of our derived limiting distributions. We also present a few important byproducts of our analysis, which are of their own interest. These include a novel variant of the dynamic programming algorithm to boost the computational efficiency, consistent change point localisation rates under temporal dependence and a new Bernstein inequality for data possessing functional dependence.  Extensive numerical results are provided to support our theoretical results.  The proposed methods are implemented in the R package \texttt{changepoints} \citep{changepoints_R}.


\end{abstract}

\section{Introduction}\label{sec:intro}

Given a sequence of data $\{(y_t, X_t)\}_{t = 1}^n \subset \mathbb{R} \times \mathbb{R}^p$, with the dimensionality $p$ being a function of the sample size $n$, assume that 
\begin{equation}\label{eq-model-1}
  y_t = X_t^{\top}\beta_t^* + \epsilon_t, \quad t = 1, \dots, n, 
\end{equation}
where $\{y_t\}_{t = 1}^n$ are the responses, $\{X_t\}_{t = 1}^n$ are the covariates, $\{\epsilon_t\}_{t = 1}^n$ are the error terms and $\{\beta^*_t\}_{t = 1}^n$ are true regression coefficients.  Across time, assume that the coefficients sequence $\{\beta^*_t\}_{t = 1}^n$ possesses a piecewise-constant pattern, i.e.
\begin{equation}\label{eq-model-2}
    \beta^*_t \neq \beta^*_{t-1}, \quad \mbox{if and only if} \quad t \in \{\eta_1, \ldots, \eta_K\},
\end{equation}
where $\{\eta_k\}_{k = 1}^K \subset \{1, \ldots, n\}$ are called change points, satisfying that 
\begin{equation}\label{eq-model-3}
    1 = \eta_0 < \eta_1 < \cdots < \eta_K < \eta_{K+1} = n+1. 
\end{equation}
We refer to the model specified in \eqref{eq-model-1}, \eqref{eq-model-2} and \eqref{eq-model-3} as the high-dimensional linear regression model with change points.  Data of this pattern commonly arise from various application areas, including  biology, climatology, economics, finance and neuroscience.   As a concrete example, \cite{rinaldo2021localizing} studied an air quality data set \citep{data-set}, with the daily index data of Particular Matter 10 in 2015 in Banqiao as responses, other climate and air quality indices of Banqiao and other nearby cities as covariates.  \cite{rinaldo2021localizing} localised two change points of the regression coefficients, corresponding to the first strong typhoon in 2015 in the area and the beginning of the severe air-pollution season starting in winter.

The considered high-dimensional linear regression model with change points falls into the vast body of change point analysis, which has been identified as one of the major challenges for modern data applications~\citep{Council2013}.  The primary interest of change point analysis is to study the possibly piecewise-constant pattern in various data types. The best studied data type is fixed-dimensional time series, see e.g.~\cite{yao1988estimating}, \cite{yao1989least}, \cite{shao2010testing}, \cite{cao2015changepoint}, \cite{wang2020univariate}, \cite{verzelen2020optimal} and \cite{wang2021testing}.  We refer to \cite{Csoergoe1997}, \cite{brodsky2013nonparametric} for book-length treatments, and to \cite{aue:13}, \cite{Casini2019structural} for survey articles.  Due to the increasing availability of big data, high-dimensional data have recently received great attention in change point analysis, including analysis in high-dimensional vector sequences \citep[e.g.][]{wang2021optimal, pilliat2020optimal}, matrices/networks sequences \citep[e.g.][]{bhattacharjee2018change, wang2021optimal}, functional curves sequences \citep[e.g.][]{berkes2009detecting, aue2018detecting, misael2022change}, among many others. 

Depending on specific statistical tasks, work on high-dimensional data change point analysis can be further grouped into several branches: (i) \cite{chen2021inference}, \cite{yuchen2021}, \cite{wangzhushaovolgushev2021AoS}, \cite{zhang2021adaptive} and \cite{wang2022optimal} conducted hypothesis testing on the existence of change points.  This task is referred to as testing.  (ii) \cite{lee2016lasso}, \cite{kaul2019efficient}, \cite{wang2021statistically}, \cite{rinaldo2021localizing}, \cite{chen2021inference}, \cite{yuchen2021} and \cite{kaul2021inference} provided estimation error controls on the estimated change point locations.  This task is referred to as localisation/estimation.  (iii) \cite{bai2010common}, \cite{wangshao2020}, \cite{kaul2021inference} and \cite{chen2021inference} derived limiting distributions of the change point estimators.  This task is referred to as inference.

For all three statistical tasks we summarised above, a handle on the unknown variance is always desirable, and in fact vital for the testing and inference tasks. However, most of the existing change point literature in high-dimension provides only heuristic variance estimators justified through numerical experiments~\citep[e.g.][]{kaul2021inference, wang2022optimal}. For the problem of mean change in a high-dimensional time series, \cite{chen2021inference} proposed a robust block-type estimator for the long-run covariance matrix, with theoretical justification, and \cite{zhang2021adaptive} and \cite{wangzhushaovolgushev2021AoS} cleverly avoided variance estimation via the use of U-statistics and self-normalisation.

Besides different data types and different tasks, different temporal dependence notions have also been adopted in the change point analysis literature.  Temporal dependence assumptions include independence assumptions, linear process assumptions, different types of mixing conditions ($\alpha$, $\beta$, $\tau$, $\rho$, etc.)~and the functional dependence assumption \citep{wu2005nonlinear}. In particular, the functional dependence measure (see \Cref{sec:background} for more details) covers a large class of dependent processes, including linear and nonlinear processes, and enjoys mathematical convenience when dealing with more complex data types.  Incorporating temporal dependence in high-dimensional data is in general challenging but is attracting increasing attention due to its practical importance. In the existing literature, the high-dimensional linear regression model with change points is studied under independence assumptions \citep{lee2016lasso, kaul2019efficient, wang2021statistically, rinaldo2021localizing} for estimation and $\beta$-mixing conditions \citep{wang2022optimal} for testing. To the best of our knowledge, this model is yet to be studied under functional dependence.

In this paper, we consider the inference task for the high-dimensional linear regression model with change points specified in \eqref{eq-model-1}, \eqref{eq-model-2} and \eqref{eq-model-3}.  To our best knowledge, in the existing literature, such model has only been studied in the testing~\citep[e.g.][]{wang2022optimal} or estimation~\citep[e.g.][]{lee2016lasso, kaul2019efficient, wang2021statistically} aspects, while the inference task has yet to be explored.  More discussions on connections with existing literature are deferred to \Cref{sec-rel-lit-result} after we present our main results.  To achieve broad applicability, we allow for both the covariate and noise sequences to possess temporal dependence, as well as dependence between the covariate and noise sequences, detailed in \Cref{sec:background}. This is, arguably, the most general setting considered in the literature. Under temporal dependence and high-dimensionality, we further derive and justify a block-type long-run variance estimator, which facilitates the construction of confidence intervals for the estimated change points.

\subsection{List of contributions}
We summarise the contributions of this paper as follows.
\begin{itemize}
    \item We conduct inference for the high-dimensional linear regression model with change points, allowing for temporal dependence and potentially multiple change points. See \Cref{sec:main_result}. Depending on whether the jump size vanishes as the sample size diverges, the limiting distributions are shown to have two different non-degenerate regimes. This two-regime phenomenon has previously been observed for mean change under the fixed-dimensional time series~\citep[e.g.][]{yao1987approximating,yao1989least,bai1994}, high-dimensional vector time series~\citep[e.g.][]{kaul2021inference} or functional time series setting~\citep[e.g.][]{aue2009estimation}. We believe such result is the first time seen in the high-dimensional regression setting, which requires substantial theoretical effort due to high-dimensionality and temporal dependence.
    
    \item To facilitate the practical implementation of our derived limiting distributions, it is vital to provide a theoretically-justified long-run variance estimator. In \Cref{sec:LRV}, we provide a consistent, block-type estimator for the long-run variance.  Such result is again, the first time seen in the literature for high-dimensional regression under functional dependence.  Based on the long-run variance estimator, we provide a completely data-driven procedure for constructing confidence intervals for change points in \Cref{sec-CI}.
    
    \item In addition, there are three notable byproducts of our analysis that are of independent interest.  
    \begin{itemize}
        \item To improve computational efficiency, we propose a new dynamic programming algorithm to solve the optimal partition problem \citep[e.g.][]{friedrich2008complexity, killick2012optimal} with a Lasso-estimation sub-routine, which lowers the computational complexity from $O(n^3p^2)$ to $O(n^2p^2)$, without sacrificing any memory cost. See \Cref{sec:algorithm}.
        \item In addition to the limiting distributions of the change point estimators, we also provide its localisation properties under functional dependence and more general tail behaviours. See \Cref{sec:theory_preliminary_est}. The final result is jointly characterised by the tail-behaviour and the strength of temporal dependence of covariates and errors.  When the model assumes sub-Gaussian covariates and errors without temporal independence, our results recover the optimal localisation error rate derived in \cite{rinaldo2021localizing}.  
        \item  Last but not least, we provide a novel Bernstein inequality under the functional dependence.  See \Cref{sec:theory_preliminary_est}.  This is an exciting new addition to the literature for solving a wide range of high-dimensional problems under functional dependence.  For instance, in this paper, we derive the estimation error bounds on the Lasso estimator under functional dependence, as an application of this new Bernstein inequality.
    \end{itemize}
\end{itemize}

\subsubsection*{Notation and organisation}

For any vector $v \in \mathbb R^p$ and any $q \in (0, \infty]$, let $|v|_q$ be its $\ell_q$-norm.  For any random variable $Z \in \mathbb{R}$ and any $q > 0$, let $\|Z\|_q= \{ \E[|Z|^q] \}^{1/q}$, if $\E[|Z|^q] < \infty$. For any matrix $A$, let $|A|_{\infty}$ denote its entry-wise max norm. For any set $S$, let $|S|$ denote its cardinality. For two deterministic or random $\mathbb{R}$-valued sequences $a_n, b_n > 0$, write $a_n \gg b_n$ if $a_n/b_n \to \infty$ as $n$ diverges. Write $a_n \lesssim b_n$ if $a_n/b_n \leq C$, for some absolute constant $C>0$ and for all $n$ sufficiently large.  For a deterministic or random $\mathbb{R}$-valued sequence $a_n$, write that a sequence of random variable $X_n = O_p(a_n)$ if $\lim_{M \to \infty}\limsup_{n}\mathbb{P}(|X_n| \geq Ma_n) = 0$.  Write $X_n = o_p(a_n)$ if $\limsup_{n}\mathbb{P}(|X_n| \geq Ma_n) = 0$ for all $M > 0$. The convergences in distribution and probability are respectively denoted by $\overset{\mathcal{D}}{\longrightarrow}$ and $\overset{P.}{\longrightarrow}$. 

The rest of the paper is organised as follows. Functional dependence is introduced in \Cref{sec:background}.  \Cref{sec:method} consists of our estimation procedure, which starts with a new variant of the dynamic programming algorithm tailored for the Lasso estimation sub-routine in \Cref{sec:algorithm}.  The limiting distributions of the change point estimators are collected in \Cref{sec:theory}, accompanied by the new results on Bernstein's inequality, the Lasso estimator and change point localisation rates. In \Cref{sec:LRV}, the long-run variance estimator along with its consistency are presented. Numerical studies including a real data application to macroeconomics are conducted in \Cref{sec:numeric_study}. All the proofs and technical details are relegated to the Appendix. 

\subsection{Background: Functional dependence}\label{sec:background}

Before formally introducing functional dependence, we remark that the functional dependence framework is attractive in two aspects: (1) It is general enough to include a large class of dependence assumptions, including ARMA models, linear processes, GARCH models, iterated random functions \citep{wu2005nonlinear,wu2011asymptotic} and their high-dimensional variants.  (2) It quantifies general nonlinear dependence in terms of moments and is thus easy to be computed explicitly for nonlinear processes.  Many important mathematical tools which are widely used for independent sequences or martingale difference sequences are not available for general nonlinear processes. Functional dependence enables us to import some of the tools from the independence territory to the dependence one, while accommodating a large class of dependence assumptions. In contrast, different mixing type dependence assumptions assert the mixing coefficients to be based on $\sigma$-algebras, thus can be difficult to compute and transfer to useful mathematical tools.  

\medskip
\noindent \textbf{Functional dependence of the covariate sequence.}  For each $t \in \mathbb{Z}$, let
\begin{equation}\label{eq:covariate}
    X_t = (X_{t1}, \dots, X_{tp})^{\top} = G(\mathcal{F}_t^X),  
\end{equation} 
where $G(\cdot) = (g_1(\cdot), \dots, g_p(\cdot))^{\top}$ is an $\mathbb{R}^p$-valued measurable function with input $\mathcal F_t^X  = \{\mathcal{X}_s\}_{s \leq t}$, and $\{\mathcal{X}_t\}_{t \in \mathbb{Z}}$ is a collection of independent and identically distributed~(i.i.d.) random elements.  Note that \eqref{eq:covariate} leads to the strict stationarity of $\{X_t\}_{t \in \mathbb{Z}}$.

For any $t \in \mathbb{Z}$ and any integer pair $s_2 \leq s_1$, let $\mathcal{X}_t^*$ be an independent copy of $\mathcal{X}_t$ and $X_{t, \{s_1, s_2\}} = G(\mathcal{F}^X_{t, \{s_1, s_2\}})$ be a coupled random variable, where 
\begin{equation}\label{eq-def-f-couple}
    \mathcal{F}^X_{t, \{s_1, s_2\}} = \begin{cases}
        \{\dots, \mathcal X_{s_2-1}, \mathcal X_ {s_2}^*, \dots, \mathcal X_{s_1}^*,  \mathcal X _{s_1+1}, \dots, \mathcal X_t\}, & s_1 < t, \\
        \{\dots, \mathcal X_{s_2-1}, \mathcal X_ {s_2}^*, \dots, \mathcal X_t^*\}, & s_2 \leq t \leq s_1, \\
        \{\dots, \mathcal X_t\}, & t < s_2.
    \end{cases}
\end{equation}
For any $s, t \in \mathbb{Z}$, let $\mathcal{F}^X_{t, \{s\}} = \mathcal{F}^X_{t, \{s, s\}}$, satisfying that when $t < s$, $\mathcal{F}^X_{t, \{s\}} = \mathcal{F}^X_t$ and $X_{t, \{s\}} = X_t$.

For any $v\in \mathbb R^p$, any $q > 0$ satisfying that $\|v^\top  X_1\|_q < \infty$ and any $s \in \mathbb{Z}$, let
\begin{equation}\label{eq-functional-dependence-x-def-1}
    \delta_{s, q}^X(v) = \|v^\top X_1 - v^\top X_{1, \{1-s\}}\|_q = \|v^\top X_s - v^\top X _{s, \{0\}}\|_q,
\end{equation}
due to the strict stationarity implied by \eqref{eq:covariate}.  Define the uniform functional dependence measure and its cumulative version as
\begin{equation}\label{eq-functional-dependence-x-def-2}
    \delta_{s, q}^X = \sup_{|v|_2=1} \delta_{s, q}^X (v) \quad \text{and} \quad \Delta_{m, q}^X = \sum_{s = m}^{\infty}\delta ^X _{s, q}, \quad m \in \mathbb{Z},
\end{equation}
respectively. In this paper, functional dependence assumption is imposed on the decay rate of the cumulative uniform function dependence measure $\Delta_{m, q}^X$, see, for example \eqref{eq:X temporal dependence}.

\medskip
\noindent \textbf{Functional dependence of the noise sequence.}  For each $t \in \mathbb{Z}$, let
\begin{equation}\label{eq:error}
    \epsilon_t = g(\mathcal F_t^\epsilon),
\end{equation}
where $g(\cdot)$ is an $\mathbb{R}$-valued measurable function with input $\mathcal F_t^\epsilon = \{\varepsilon_s\}_{s \leq t}$, and $\{\varepsilon_t\}_{t \in \mathbb{Z}}$ is a collection of i.i.d.\ random elements.  Note that \eqref{eq:error} leads to the strict stationarity of $\{\epsilon_t\}_{t \in \mathbb{Z}}$.  

For any $t \in \mathbb{Z}$ and any integer pair $s_2 \leq s_1$, let $\varepsilon_t^*$ be an independent copy of $\varepsilon_t$ and $\epsilon_{t, \{s_1, s_2\}} = g(\mathcal{F}^\epsilon _{t, \{s_1, s_2\}})$ be a coupled random variable, defined in the same way as that in \eqref{eq-def-f-couple}.  For any $q > 0$ satisfying that $\|  \epsilon_1\|_q < \infty$ and any $s \in \mathbb{Z}$, let the functional dependence measure and its cumulative version be
\begin{equation}\label{eq-functional-dependence-eps-def}
    \delta_{s, q}^\epsilon = \|\epsilon_1 - \epsilon_{1, \{1-s\}}\|_q = \|\epsilon_s - \epsilon _{s, \{0\}}\|_q \quad \text{and} \quad \Delta_{m, q}^\epsilon = \sum_{s = m}^{\infty}\delta ^\epsilon_{s, q}, \quad m \in \mathbb{Z},
\end{equation}
respectively. Same as above, functional dependence assumption is imposed on the decay rate of the cumulative function dependence measure $\Delta_{m, q}^\epsilon$, see, for example \eqref{eq:epsilon temporal dependence}.

We remark that although $\{X_t\}_{t \in \mathbb{Z}}$ and $\{\epsilon_t\}_{t \in \mathbb{Z}}$ are assumed to be strictly stationary, which is a direct consequence of the time-invariance of $G(\cdot)$ and $g(\cdot)$, our theoretical analysis in this paper extends to nonstationary cases, which are indeed needed for change point analysis.

\medskip
\noindent \textbf{Functional dependence of nonstationary sequences.}  Despite that $\{X_t\}_{t \in \mathbb{Z}}$ and $\{\epsilon_t\}_{t \in \mathbb{Z}}$ are assumed to be strictly stationary, we still need to deal with nonstationary sequences, for instance~$\{X_{tj}X_t^{\top}\beta^*_t\}_{t \in \mathbb{Z}}$, $j \in \{1, \ldots, p\}$ and $\{\epsilon_tX_t^{\top}\beta^*_t\}_{t \in \mathbb{Z}}$, since the coefficient sequence $\{\beta_t^*\}$ may possess change points.  We therefore introduce the functional dependence measure for a (possibly) nonstationary process $\{Z_t\}_{t \in \mathbb{Z}} \subset \mathbb{R}$.  Suppose that, for each $t \in \mathbb{Z}$,
\begin{equation}\label{eq:nonstationary}
    Z_t = g_t(\mathcal{F}_t^{\zeta}),
\end{equation}
where $g_t(\cdot)$ is a time-dependent $\mathbb{R}$-valued measurable functions with input $\mathcal{F}_t^{\zeta} = \{\zeta_s\}_{s \leq t}$, and $\{\zeta_t\}_{t \in \mathbb{Z}}$ is a collection of i.i.d.\ random elements.  Note that \eqref{eq:nonstationary} covers a large class of locally or piece-wise stationary processes \citep[e.g.][]{dahlhaus2019towards, chen2021inference}.  For any $t \in \mathbb{Z}$ and any integer pair $s_2 \leq s_1$, let $\zeta_t^*$ be an independent copy of $\zeta_t$ and $Z_{t, \{s_1, s_2\}} = g_t(\mathcal{F}^\zeta_{t, \{s_1, s_2\}})$ be a coupled random variable, defined in the same way as that in \eqref{eq-def-f-couple}.  For any $q > 0$ satisfying that $\sup_t\|Z_t\|_q < \infty$ and any $s \in \mathbb{Z}$, let the functional dependence measure and its cumulative version be
\begin{equation}\label{eq:fdm_nonstationary}
    \delta_{s, q} ^Z  = \sup_{t \in \mathbb{Z}}\Vert Z_{t} - Z_{t,\{t-s\}} \|_q \quad \text{and} \quad \Delta^Z_{m, q} = \sum_{s = m}^{\infty}\delta^Z_{s, q}, \quad m \in \mathbb{Z},
\end{equation}
respectively, and we again impose functional dependence assumptions on the decay rate of~$\Delta^Z_{m, q}$.

\section{Methodology}\label{sec:method}

To provide a change point estimator with tractable limiting distributions, we exploit a two-step procedure, where we obtain a preliminary estimator in the first step and further refine it in the second step. This two-step strategy has been summoned regularly in recent literature to handle the multiple change point scenario.  We study the two steps in Sections~\ref{sec:algorithm} and \ref{sec:local_refinement}, respectively.

\subsection{Preliminary estimators}\label{sec:algorithm}

To obtain preliminary estimators, we exploit an $\ell_0$-penalised estimation procedure \citep[e.g.][]{friedrich2008complexity, killick2012optimal, wang2020univariate}.  Let
\begin{align}\label{eq:pre_est}
    \widehat{\mathcal{P}} = \widehat{\mathcal{P}}(\zeta) \in \argmin_{\mathcal{P}} \bigg\{\sum_{\I \in \mathcal{P}}\mathcal{G}(\I) + \zeta|\mathcal{P}| \bigg\},
\end{align}
where the minimisation is over all possible integer partitions of $\{1, \ldots, n\}$, $\mathcal{P}$ denotes an integer partition and $\zeta > 0$ is a penalisation parameter. For any integer interval $\mathcal{I} \subset (0, n]$, the loss function $\mathcal{G}(\cdot)$ is defined as
\begin{equation}\label{eq:interval gof 2}
    \mathcal{G}(\mathcal{I}) = \begin{cases} 
        \sum_{t \in \mathcal I} \{-2y_t X_t^{\top} \widehat{\beta}_{\I} + \widehat{\beta}_{\I}^\top X_t X_t^\top \widehat{\beta}_{\I}\}, & |\I| \geq \zeta, \\
        0, & \text{otherwise},
    \end{cases}    
\end{equation}
with a to-be-specified $\zeta > 0$.  The estimator $\widehat{\beta}_{\mathcal{I}}$ in \eqref{eq:interval gof 2} is defined as 
\begin{align}\label{eq:interval lasso}
\widehat \beta _{ \mathcal I } = \argmin_{\beta \in \mathbb R^p }  \bigg\{\sum_{t \in \mathcal I}   (y_t -X_t ^\top \beta) ^2 +  \lambda |\I |^{1/2}  |\beta |_1 \bigg\},
\end{align}
where $\lambda > 0$ is the Lasso tuning parameter to be specified.  Note that there is a one-to-one relationship between a subset of integers and an integer partition, using the left endpoints of all the intervals.  An estimator $\widehat{\mathcal{P}}$ therefore corresponds to a unique set of change point estimators denoted as $\widehat{\mathcal{B}}$ in \Cref{algorithm:DPDU}. 

\begin{remark}[Tuning parameter $\zeta$ in \eqref{eq:pre_est} and \eqref{eq:interval gof 2}] 
Denote $\zeta_1$ and $\zeta_2$ as the respective tuning parameters in \eqref{eq:pre_est} and \eqref{eq:interval gof 2}, where they play different roles.  In \eqref{eq:pre_est}, $\zeta_1$ penalises over-partitioning and is a high-probability upper bound controlling the loss function $\mathcal{G}(\mathcal{I})$, for desirable $\mathcal{I}$.  While in \eqref{eq:interval gof 2}, $\zeta_2$ restricts the attention to the intervals of lengths at least $\zeta_2$~(see \Cref{algorithm:DPDU}), which are needed for a restricted eigenvalue condition (\Cref{theorem:RES Version II}) to hold.  To see why these two parameters coincide, we discuss from two aspects. Recall that $\kappa$ is the minimum jump size~(see \Cref{assume:regression parameters}\textbf{b} for more details).  Firstly, the success of $\zeta_1$ serving as a high-probability upper bound in \eqref{eq:pre_est} relies on a restricted eigenvalue condition, holding only when the sample size is large enough, i.e.~the sample size is at least larger than the fluctuation.  This means that the tuning parameter in \eqref{eq:pre_est} should be no smaller than $\zeta_2$. Secondly, \eqref{eq:pre_est} implies that the localisation errors, which can be achieved by DPDU, are of order $\zeta_1/\kappa^2$~(see \Cref{prop-1} in the supplementary material) and \eqref{eq:interval gof 2} implies that the localisation errors are at least $\zeta_2$.  Since $\kappa = O(1)$, $\zeta_1/\kappa^2 \gtrsim \zeta_1$, it suggests the choice of $\zeta_1$ (up to a constant factor) should be no larger than $\zeta_2$.  Based on these two, we see that these two tuning parameters can take the same value (up to a constant factor), i.e.~$\zeta_2 = C\zeta_1$ for some absolute constant $C>0$.  Indeed, fine-tuned $\zeta_1$ and $\zeta_2$ can potentially enhance the finite sample performance of \Cref{algorithm:DPDU} by providing more flexibility, we simplify the setting by choosing $\zeta_1 = \zeta_2 = \zeta$ to reduce the number of tuning parameters. Our theoretical results are derived under this restriction.  But we provide this flexibility in function \texttt{DPDU.regression} available in the R package \texttt{changepoints}.  More discussions on $\zeta$ can be found in \Cref{sec:main_result}.
%\zifeng{rewrite better? I think it should be $\zeta_2$ is needed for REC, and then based on $\zeta_2$, we choose the smallest $\zeta_1$ that makes the localisation error $\zeta_1/\kappa^2$ no larger than $\zeta_2$ when $\kappa$ is constant.}
\end{remark}

Despite the different form, we remark that the optimisation problem specified in \eqref{eq:pre_est} and \eqref{eq:interval gof 2} is equivalent to that in \cite{rinaldo2021localizing}.  The formulation in \eqref{eq:interval gof 2} enables us to apply a more computationally-efficient optimisation without affecting the memory cost.
To solve the minimisation problem, we propose a variant of the dynamic programming (DP) algorithm \citep[e.g. Algorithm 1 in][]{friedrich2008complexity} coined as dynamic programming with dynamic update (DPDU) detailed in \Cref{algorithm:DPDU}.

\begin{algorithm}[!ht]
\begin{algorithmic}
    \INPUT Data $\{(y_t, X_t)\}_{t =1}^{n }$, tuning parameters $\lambda, \zeta > 0$.
    \State $\widehat{\mathcal{B}} \leftarrow \varnothing$,\; $\mathfrak{p} \leftarrow (-1,\ldots, -1)^{\top} \in \mathbb{R}^n$,\; $B \leftarrow (-\zeta, \infty, \ldots, \infty)^{\top} \in \overline{\mathbb{R}}^{n+1}$
 	\For{$r \in \{2, \ldots, n+1\}$} 
 	 \State $\mathcal{M}_{\text{temporary}} \leftarrow (0) \in \mathbb R^{p \times p}$,\; $\mathcal{V}_{\text{temporary}} \leftarrow (0) \in \mathbb R^p$
        \For{$l \in \{r-1, \ldots, 1\}$} 
        \begin{equation}\label{eq:updated matrices} 
 	        \mathcal{M}_{\text{temporary}} \leftarrow \mathcal{M}_{\text{temporary}} + X_{l} X_{l}^\top, \; 
            \mathcal{V}_{\text{temporary}}  \leftarrow  \mathcal{V}_{\text{temporary}} + y_{l} X_{l}^\top
            \end{equation}
 	        \State $\I \leftarrow [l, r) \cap \mathbb{Z}$. 
            \If{$|\I| \ge \zeta$}
                \State $\widehat{\beta}_{ \mathcal I} \leftarrow \argmin_{\beta \in \mathbb R^p } |\mathcal I|^{-1} \{-2\mathcal V_{\text{temporary}}\beta + \beta^\top \mathcal M_{\text{temporary}}\beta\} + \lambda/\sqrt{| \I|} |\beta|_1$
                \State $\mathcal G (\I) \leftarrow -2    \mathcal V_{\text{temporary}}  \widehat \beta _{  \I  } + \widehat\beta _{  \I  } ^\top \mathcal M _{\text{temporary}} \widehat\beta _{  \I  }$    
            \Else
                \State$\mathcal G (\I )  \leftarrow 0$  
            \EndIf  
            \State $b \leftarrow B_{l}+ \zeta   + \mathcal G (\I)$
            \If{$b <  B_{r}$}
 	            \State $B_{r} \leftarrow  b$, $\mathfrak {p}_r  \leftarrow l$
            \EndIf  
        \EndFor
        %\State Replace $\{ \mathcal M (r-1,1) , \ldots, \mathcal M (r-1,r-2) \}$ with $\{ \mathcal M (r,1) , \ldots, \mathcal M(r,r-2) \}$ in the memory.
       % \State Replace $\{ \mathcal V  (r-1,1) , \ldots, \mathcal V  (r-1,r-2) \}$ with $\{ \mathcal V (r,1) , \ldots, \mathcal V (r,r-2) \}$ in the memory.
    \EndFor
    \State $k \leftarrow  n$
    \While {$k >1$}
        \State $h \leftarrow  \mathfrak{p}_ k$, $\widehat{\mathcal{B}} \leftarrow \widehat{\mathcal B} \cup  \{h\}$, $k \leftarrow h$
    \EndWhile 
    \OUTPUT $\widehat{\mathcal{B}}$, $\widehat{K} = |\widehat{\mathcal{B}}|$, $\{\widehat{\beta}_k\}_{k = 0}^{\widehat{K}}$
\caption{Dynamic Programming with Dynamic Update. DPDU $(\{(y_t, X_t)\}_{t =1}^{n}, \lambda, \zeta)$} \label{algorithm:DPDU}
\end{algorithmic}
\end{algorithm}

DPDU involves two nested for-loops that update the endpoints of an interval $\I = [l, r)$. For a given right endpoint $r$, DPDU iterates backward through all possible left endpoints $l$ and updates the current minimal loss $B_r$ along with the associated pointer~$\mathfrak{p}_r$.  Once the two for-loops are completed, the locations of estimated change points can be extracted from the vector of pointers $\mathfrak{p}$.  We highlight the properties of DPDU over the standard DP.
\begin{itemize}
    \item \textbf{Smaller computational cost.} The formulation \eqref{eq:interval gof 2} suggests that the loss function $\mathcal{G}(\mathcal{I})$ contains the partial sums of $\sum_{t \in \mathcal I} X_tX_t^{\top}$ and $\sum_{t \in \mathcal I} y_t X_t$, which can be updated iteratively. We can save computational costs by storing the partial sums in the memory as in \eqref{eq:updated matrices}.  These partial sums are directly used to obtain the Lasso estimator $\widehat \beta_{\I}$.
    
    To be specific, solving \eqref{eq:pre_est} using the standard DP requires $O(n^2{\bf Cost})$ operations, where ${\bf Cost}$ represents the number of operations needed to compute $\mathcal G(\mathcal I)$ for an $\mathcal I \subset (0,n]$. Without storing and updating the partial sums iteratively as in \Cref{algorithm:DPDU}, one has to compute the partials sums $\sum_{t \in \mathcal I} y_t X_t \in \mathbb{R}^p$ and $\sum_{t \in \mathcal I} X_t X_t^\top \in \mathbb R^{p\times p}$ for each $\mathcal{I}$ incurring a cost up to $O(np^2)$. Denote ${\bf{Lasso}}(p)$ as the computational cost of the lasso procedure given quantities $\sum_{t \in \mathcal I} y_t X_t$ and $\sum_{t \in \mathcal I} X_t X_t^\top$.  The exact value of ${\bf{Lasso}}(p)$ depends on the choice of algorithms.  For example, for the coordinate descent algorithm, ${\bf{Lasso}} (p) =O(p^2)$, see \Cref{lemma:lasso one iteration} in the supplementary material. For the LARS algorithm, ${\bf{Lasso}} (p)  = O(p^3)$, see \cite{efron2004least}.  We thus have ${\bf Cost} = np^2 + {\bf{Lasso}}(p)$, and the total cost of the standard DP is $O(n^3p^2 + n^2{\bf Lasso}(p))$. DPDU improves the computational efficiency by decreasing the term $n^3$ to $n^2$, with the help of storing the partial sums in memory~(see \Cref{lemma:memory_computation_DPDU}).  

    \item \textbf{Same memory cost.}  Unlike the standard DP \citep[e.g.][]{rinaldo2021localizing}, DPDU updates the left endpoint $l$ in a \textit{backward} manner from $l=r-1$ to $l=1$. Importantly, this backward updating scheme enables us to store the reusable partial sums without paying a higher memory cost than that of the standard DP, i.e.~$O(n+p^2)$.
\end{itemize}

The memory and computational costs of the DPDU are formally collected below.

\begin{lemma}\label{lemma:memory_computation_DPDU}
For $\mathrm{DPDU}(\{y_t, X_t\}_{t = 1}^n$, $\lambda$, $\zeta$) with $\{X_t\}_{t = 1}^n \subset \mathbb{R}^p$ and any $\lambda, \zeta > 0$, the memory cost is $O(n+p^2)$ and the computational  cost is $O(n^2p^2 + n^2 {\bf Lasso}(p))$.
\end{lemma}

%\zifeng{I feel a forward updating can achieve this as well?}

\subsection{Final estimators}\label{sec:local_refinement}

Denote the preliminary estimators obtained from the optimisation problem specified in \eqref{eq:pre_est}, \eqref{eq:interval gof 2} and \eqref{eq:interval lasso}, or equivalently from \Cref{algorithm:DPDU} as 
\begin{equation}\label{eq:preliminary_cpt}
\{\widehat  {\eta}_k\}_{k = 1}^{\widehat{K}} = \widehat{\mathcal{B}}.
\end{equation}
Denote the corresponding regression coefficient estimators \eqref{eq:interval lasso} as $\{ \widehat  \beta_k \}_{k=0}^{\widehat{K}}$.  Our final estimators are obtained via a local refinement step, which is considered in \cite{rinaldo2021localizing}, \cite{yu2022localising} and others.  The final estimators are defined as
\begin{align}\label{eq:refinement}
    \widetilde{\eta}_k = \argmin_{s_k < \eta < e_k} Q_k(\eta) = \argmin_{s_k < \eta < e_k} \left\{\sum_{t = s_k }^{\eta-1}(y_t - X_t^{\top}\widehat{\beta}_{k-1})^2 + \sum_{t = \eta}^{e_k-1}(y_t - X_t^{\top}\widehat{\beta}_{k})^2\right\}, 
\end{align}
where $k \in \{1, \ldots \widehat{K}\}$ and $(s_k, e_k)$ is defined as
\begin{align}\label{eq:corrected boundaries}
   s_k = 9\widehat{\eta}_{k-1}/10 + \widehat{\eta}_{k}/10 \quad \text{and} \quad e_k = \widehat{\eta}_{k}/10 + 9\widehat{\eta}_{k+1}/10.
\end{align}

Provided the preliminary estimators $\{\widehat {\eta}_k\}_{k = 1}^{\widehat{K}}$ are good enough, each interval $(s_k, e_k)$ contains one and only one true change point $\eta_k$.  This fact grants us to refine the initial estimators in $(s_k, e_k)$'s.  The choice of the constants $9/10$ and $1/10$ in \eqref{eq:corrected boundaries} is arbitrary.  To be specific, \Cref{theorem:DUDP} shows that, under certain conditions, with high probability, the output of DPDU satisfies that 
\[
    \widehat{K} = K \quad \mbox{and} \quad \max_{k = 1, \ldots, K} \frac{|\widehat{\eta}_k - \eta_k|}{\min\{\eta_k - \eta_{k-1}, \, \eta_{k+1} - \eta_k\}} \to 0, \quad n \to \infty,
\]
which guarantees the success of the $\{(s_k, e_k)\}_{k = 1}^K$ construction above.

\section{Theory}\label{sec:theory}

We kick off by a list of assumptions in \Cref{sec:assumptions}, followed by the main results on the limiting distributions in \Cref{sec:main_result}.  The connection of our main results with relevant literature is discussed in \Cref{sec-rel-lit-result}.  We conclude this section by discussing a few important byproducts of our analysis in \Cref{sec:theory_preliminary_est}.

\subsection{Assumptions}\label{sec:assumptions}

\Cref{assume:regression parameters} formalises some of the model assumptions we have mentioned so far. 

\begin{assumption}[Model] \label{assume:regression parameters} 
Consider the high-dimensional linear regression model with change points specified in \eqref{eq-model-1}, \eqref{eq-model-2} and \eqref{eq-model-3}, with the covariate sequence $\{X_t\}_{t = 1}^n \subset \mathbb{R}^p$ and noise sequence $\{\epsilon_t\}_{t = 1}^n \subset \mathbb{R}$, such that $\mathbb E[\epsilon_t|X_t] = 0$ for any $t \in \{1, \dots, n\}$.  For any $k \in \{1, \dots, K\}$, where $K \geq 1$ is an absolute constant integer, let the $k$th jump vector and its normalised version be $\Psi_k = \beta_{\eta_k}^* - \beta_{\eta_k-1}^*$ and $v_k = \Psi_k/|\Psi_k|_2 =  \Psi_k/\kappa_k$, respectively.

\noindent \textbf{a.}~(Sparsity) For any $t \in \{1, \ldots, n\}$, assume that there exists a support $S_t \subset \{1, \ldots, p\}$, with $|S_t|\le \mathfrak{s}$, such that $\beta^*_{t, j} = 0$, if $j \not \in S_t$.

\noindent \textbf{b.}~(Change points) Let the minimal jump size be
\begin{equation}\label{eq-kappa-def}
    \kappa = \min_{k = 1, \ldots, K} \kappa_k = \min_{k = 1, \ldots, K} |\Psi_k|_2. 
\end{equation}
Assume that there exists an absolute constant $C_\kappa > 0$ such that for any $k\in \{1,\ldots, K \}$, $\kappa_k \le C_\kappa$.  Let the minimal spacing be $\Delta = \min_{k = 1, \ldots, K+1}(\eta_{k} - \eta_{k-1})$.
\end{assumption} 

In \Cref{assume:regression parameters}, we impose the $\ell_0$-sparsity on the regression coefficients, following the suit of standard high-dimensional linear regression literature.  We also upper bound the jump size $\kappa_k$, i.e.\ the $\ell_2$-norm of the difference between the true regression coefficients before and after the $k$th change point for $k \in \{1, \ldots, K\}$. Since the larger the jump size, the more prominent the change points, \Cref{assume:regression parameters}\textbf{b} focuses our analysis in the more challenging territory.   In \Cref{assume:regression parameters}, although the number of change points $K$ is assumed to be of order $O(1)$, we do allow the minimum spacing $\Delta$ to be a function of $n$ and allow $\Delta/n$ to vanish as $n$ diverges, as specified later in \Cref{assump-snr}.  We assume that $K$ to be independent of $n$ for simplicity of presentation, as our main result \Cref{thm:asymptotics} focuses on the limiting distribution of every change point estimator.  

We highlight that \Cref{assume:regression parameters} only requires the noise to be conditional mean zero given the covariate.  This allows for dependence between the covariate and noise sequences. In Assumptions~\ref{assume:X} and \ref{assume:epsilon} below, we will further exploit the flexibility of our model by specifying the assumptions on the covariate and noise sequences, respectively.

\begin{assumption}[Covariate sequence] \label{assume:X}
Assume that the covariate sequence $\{X_t\}_{t = 1}^n \subset \mathbb R^p$ is a consecutive subsequence of an infinite sequence $\{X_t\}_{t \in \mathbb{Z}} \subset \mathbb R^p$, which has marginal mean zero and is of the form~\eqref{eq:covariate}, with covariance matrix $\Sigma = \cov(X_1)$.

\noindent \textbf{a.}~(Functional dependence) For the cumulative uniform functional dependence measure $\Delta_{\cdot, \cdot}^X$ defined in \eqref{eq-functional-dependence-x-def-1} and \eqref{eq-functional-dependence-x-def-2}, for some absolute constants $\gamma_1, c > 0$, assume that there exists an absolute constant $D_X > 0$ such that
\begin{align}\label{eq:X temporal dependence}
    \sup_{m \geq 0}\exp(cm^{\gamma_1})\Delta_{m,4}^X \leq D_X.
\end{align}  

\noindent \textbf{b.}~(Tail behaviour) Assume that there exist absolute constants $\gamma_2 \in (0, 2]$ and $C_X > 0$, such that for any $\tau > 0$, 
\begin{align}\label{eq:X subGaussian}
    \sup_{|v|_2=1 }\mathbb{P}(|v^\top X_1| > \tau) \leq 2\exp\{-(\tau/C_X)^{\gamma_2}\}.
\end{align}

\noindent \textbf{c.}~(Marginal covariance matrix) Assume that the smallest eigenvalue of $\Sigma$ satisfies $0 < c_{\min} \leq \Lambda_{\min}(\Sigma)$, where $c_{\min} > 0$ is an absolute constant.  For any $k \in \{1, \ldots, K\}$, assume that there exists an absolute constant $\varpi_k > 0$ such that $v_k^{\top}\Sigma v_k \to \varpi_k$, as $n \to \infty$.
\end{assumption}

As mentioned in \Cref{sec:background}, the functional dependence is measured through the quantity~$\Delta^X_{\cdot, \cdot}$ and their decay rate.  \Cref{assume:X}\textbf{a}~specifies that $\Delta_{m, 4}^X$ decays exponentially in $m$ at the rate of at least $\exp(-cm^{\gamma_1})$.  Recall the definition of $\Delta^X_{\cdot, \cdot}$ in \eqref{eq-functional-dependence-x-def-1} and \eqref{eq-functional-dependence-x-def-2}, the term~$4$ means that the functional dependence is measured in the $4$th moment.  This moment condition is required to unlock a functional central limit theorem under functional dependence.  The term~$\gamma_1$ characterises the exponential decay rate.  In the extreme case that $\gamma_1 = \infty$, \Cref{assume:X}\textbf{a} essentially reduces to the i.i.d.~case.

\Cref{assume:X}\textbf{b}~regulates the tail behaviour of the covariate random vectors' marginal distribution.  The specific rate is characterised by the parameter $\gamma_2 \in (0, 2]$, indicating that the marginal distribution falls in the sub-Weibull family, with sub-Exponential ($\gamma_2 = 1$) and sub-Gaussian ($\gamma_2 = 2$) as its special cases.  \Cref{assume:X}\textbf{b}~not only implies the boundedness of the largest eigenvalue, i.e.~$\Lambda_{\max}(\Sigma) < \infty$, but it also ensures the finiteness of the asymptotic variances and the existence of the limiting distribution, which are respectively specified in \Cref{lemma:long-run_UB} and \Cref{thm:asymptotics}.

\Cref{assume:X}\textbf{c}~regulates the smallest eigenvalue, being bounded away from zero. It is a necessity to lead to a form of restricted eigenvalue condition \citep[e.g.][]{van2018tight}, which is again a necessary condition in providing estimators of reliable performance. %\zifeng{wording, see my comment in the response.}

Temporal dependence and moment assumptions are two indispensable assumptions when studying the theoretical properties of high-dimensional time series.  In the existing literature \citep[e.g.][]{wu2016performance, zhang2017gaussian, kuchibhotla2021uniform}, these two assumptions are often entangled via a single dependence-adjusted moment assumption. Instead, \Cref{assume:X} provides separate assumptions on the temporal dependence and tail behaviour, which is more intuitive. We specify the temporal dependence and tail behaviour of the noise sequence $\{\epsilon_t\}_{t = 1}^n$ in \Cref{assume:epsilon}.

\begin{assumption}[Noise sequence] \label{assume:epsilon}	
Assume that the noise sequence $\{\epsilon_t\}_{t=1}^n \subset \mathbb{R}$ is a consecutive subsequence of an infinite sequence $\{\epsilon_t\}_{t \in \mathbb{Z}} \subset \mathbb{R}$, which has marginal mean zero and is of the form \eqref{eq:error}, with variance $\sigma_{\epsilon}^2 = \mathrm{Var}(\epsilon_1) > 0$ being an absolute constant.

\noindent \textbf{a.}~(Functional dependence)  For the cumulative functional dependence measure $\Delta_{\cdot, \cdot}^{\epsilon}$, defined in \eqref{eq-functional-dependence-eps-def}, for some absolute constants $\gamma_1, c > 0$, assume there exists an absolute constant $D_{\epsilon} > 0$ such that   
\begin{align}\label{eq:epsilon temporal dependence}
    \sup_{m \geq 0}\exp(cm^{\gamma_1})\Delta_{m, 4}^\epsilon \leq D_\epsilon.
\end{align}  

\noindent \textbf{b.}~(Tail behaviour) Assume that there exist absolute constants $\gamma_2 \in (0, 2]$ and $C_{\epsilon} > 0$, such that for any $\tau >0$, $\mathbb{P}(|\epsilon_0| >  \tau) \leq 2\exp\{-(\tau/C_\epsilon)^{\gamma_2}\}$.
\end{assumption}

In spite of the same notation $\gamma_1$ and $\gamma_2$ used in Assumptions~\ref{assume:X} and \ref{assume:epsilon}, there is no reason to force these two sequences to have the same functional dependence measure decay rates or the same tail behaviour. We adopt the same notation for simplicity and they can be seen as worse cases between two.  To be specific, let $\gamma_1(X)$ and $\gamma_1(\epsilon)$ be the functional dependence measure decay rate indicator for the covariate and noise sequences, respectively.  In Assumptions~\ref{assume:X} and~\ref{assume:epsilon}, we in fact let $\gamma_1 = \min\{\gamma_1(X), \gamma_1(\epsilon)\}$ and $\gamma_2 = \min\{\gamma_2(X), \gamma_2(\epsilon)\}$.  

With the temporal dependence and tail behaviour specified for both the covariate and noise sequences, we introduce the parameter
\begin{equation}\label{eq-gamma-def}
    \gamma = \left(\gamma_1^{-1} + 2 \gamma_2^{-1}\right)^{-1}.
\end{equation}
The coefficient 2 in the term $\gamma_2^{-1}$ accounts for the tail behaviour of the cross-term sequence $\{\epsilon_t X_t\}_{t = 1}^n$.  In Assumptions~\ref{assume:X}\textbf{b}~and \ref{assume:epsilon}\textbf{b}, we allow $\gamma_1 \in (0, \infty)$ provided that $\gamma < 1$. We restrict our analysis to $\gamma_2 \leq 2$ in this paper, as $\gamma_2>2$ corresponds to tails lighter than sub-Gaussian, which is less interesting for practical purposes. % and also to guarantee that~$\gamma < 1$.

We state the required signal-to-noise ratio condition below.

\begin{assumption}[Signal-to-noise ratio] \label{assump-snr}
\
\\
\noindent \textbf{a.}
Assume that there exists a sufficiently large absolute constant $C_{\mathrm{snr}} > 0$ such that 
\begin{equation}\label{eq:SNR_cond}
    \Delta \kappa^2 \geq C_{\mathrm{snr}} \big\{\s\log(pn)\big\}^{2/\gamma - 1}\alpha_n,
\end{equation}
where $\alpha_n > 0$ is any diverging sequence.

{\color{black}
\noindent \textbf{b.} 
In addition to \eqref{eq:SNR_cond}, assume that $\alpha_n$ diverges faster than $\s\{\log(pn)\}^{2/\gamma}$, i.e.
\[
    \alpha_n \gg \s\{\log(pn)\}^{2/\gamma}.
\]}
\end{assumption}

\Cref{assump-snr}\textbf{a}~and \textbf{b}~impose lower bounds on the quantity $\Delta \kappa^2$ to ensure theoretical guarantees of the change point estimation (\Cref{theorem:DUDP}) and inference (\Cref{thm:asymptotics}) procedures, respectively.  Note that \Cref{assump-snr}\textbf{b}~is strictly stronger than \Cref{assump-snr}\textbf{a}.  We conjecture that Assumption~\ref{assump-snr}\textbf{b} is necessary for the inference tasks, to be specific, to establish the uniform tightness of the refined change point estimators in the proof of \Cref{thm:asymptotics}.  A sketch of the proof of \Cref{thm:asymptotics} is provided in \Cref{sec:main_result}.  Such signal-to-noise ratio conditions are widely seen in different change point analysis problems. In \cite{rinaldo2021localizing}, it is shown that for change point estimation of the model specified in \Cref{assume:regression parameters}, when the data considered are temporally independent and sub-Gaussian ($\gamma = 1$), the minimax optimal condition is that $\Delta \kappa^2 \asymp \s$, saving a logarithmic factor, matching our \Cref{assump-snr}\textbf{a}. We extend \cite{rinaldo2021localizing} by allowing for temporal dependence, indexed by $\gamma_1$, and potentially heavier tails, indexed by $\gamma_2$.  Such effects are captured by $\gamma$, which not only inflates the logarithmic term but also affects the term regarding the sparsity parameter $\s$.

Similar effects of $\gamma$ on the sparsity level are seen in the existing literature, on the Lasso estimation in high-dimensional linear regression problems with temporal dependence, but without the presence of change points \citep[e.g.][]{wu2016performance, wong2020lasso, kuchibhotla2021uniform, han2020high}. In the presence of change points, \cite{chen2021inference} characterise the jump in terms of vector $\ell_{\infty}$-norm, which turns a high-dimensional problem to a univariate one.  The parameter $\gamma$'s counterpart therein hence does not affect the sparsity parameter therein. 

\subsection{Limiting distributions} \label{sec:main_result}

For the true change points $\{\eta_k\}_{k = 1}^K$, in \Cref{sec:method} we proposed a two-step estimator $\{\widetilde{\eta}_k\}_{k = 1, \ldots, \widehat{K}}$.  Under the assumptions listed in \Cref{sec:assumptions}, we present the limiting distributions of $\{\widetilde{\eta}_k\}_{k = 1, \ldots, \widehat{K}}$ in two regimes, regarding the jump sizes $\{\kappa_k\}_{k = 1}^K$, defined in \eqref{eq-kappa-def}: (i) the non-vanishing regime where $\kappa_k \to \varrho_k$, as $n \to \infty$, with $\varrho_k > 0$ being an absolute constant; and (ii) the vanishing regime where $\kappa_k \to 0$ as $n \to \infty$.

Since the limiting distribution in the vanishing regime involves an unknown long-run variance, we first define the long-run variance and prove its existence.

\begin{lemma}\label{lemma:long-run_UB}
Suppose that Assumptions~\ref{assume:regression parameters}, \ref{assume:X} and \ref{assume:epsilon} hold.  For $k \in \{1, \ldots, K\}$, under the vanishing regime, the long-run variance
\begin{align}\label{eq:long-run_var}
    \sigma_{\infty}^2(k)
    = 4\lim_{n \to \infty}\var\bigg(n^{-1/2}\sum_{t = 1}^n \epsilon_tv_k^{\top}X_t  \bigg)
\end{align}
exists and satisfies that $\sigma^2_{\infty}(k) < \infty$. 
\end{lemma}

We are now ready to present the main results.     

\begin{theorem}\label{thm:asymptotics}
Given data $\{(y_t, X_t)\}_{t = 1}^n$, suppose that Assumptions~\ref{assume:regression parameters}, \ref{assume:X}, \ref{assume:epsilon} and \ref{assump-snr}\textbf{b}~hold.  Let $\{\widetilde{\eta}_k\}_{k = 1}^{\widehat{K}}$ be the change point estimators defined in \eqref{eq:refinement}, with 
\begin{itemize}
    \item the intervals $\{(s_k, e_k)\}_{k = 1}^{\widehat{K}}$ defined in \eqref{eq:corrected boundaries};
    \item the preliminary estimators $\{\widehat{\eta}_k\}_{k = 1}^{\widehat{K}}$ from $\mathrm{DPDU}(\{(y_t, X_t)\}_{t = 1}^n, \lambda, \zeta)$ detailed in \Cref{algorithm:DPDU};
    \item the $\mathrm{DPDU}$ tuning parameters $\lambda = C_\lambda \sqrt{\log(pn)}$ and $\zeta = C_\zeta\{\s \log(pn)\}^{2/\gamma - 1}$, with $C_{\lambda}, C_{\zeta} > 0$ being absolute constants; and
    \item the Lasso estimators $\{\widehat{\beta}_k\}_{k = 1}^{\widehat{K}}$ defined in \eqref{eq:interval lasso} with the intervals constructed based on $\{\widehat{\eta}_k\}_{k = 1}^{\widehat{K}}$.
\end{itemize}

\noindent \textbf{a.}~(Non-vanishing regime)  For $k \in \{1, \ldots, K\}$, if $\kappa_k \to \varrho_k$, as $n \to \infty$, with $\varrho_k > 0$ being an absolute constant, then the following results hold.
    \begin{itemize}
        \item [\textbf{a.1.}] The estimation error satisfies that $|\widetilde{\eta}_k - \eta_k| = O_p(1)$, as $n \to \infty$.
        %There exist absolute constants $C, c > 0$, such that $\mathbb{P}\{|\widetilde{\eta}_k - \eta_k| \leq C\} > 1 - (n \vee p)^{-c}$.
        \item [\textbf{a.2.}] When $n \to \infty$, 
        \[
            \widetilde{\eta}_k - \eta_k \overset{\mathcal{D}}{\longrightarrow} \argmin_{r \in \mathbb{Z}} P_k(r),
        \]
        where, for $r \in \mathbb{Z}$, $P_k(r)$ is a two-sided random walk defined as 
        \[
            P_k(r) = \begin{cases}
                \sum_{t = r}^{-1} \{-2\varrho_k\epsilon_t\xi_t(k) + \varrho_k^2\xi^2_t(k)\}, & r < 0,\\
                0, & r = 0,\\
                \sum_{t = 1}^r \{2\varrho_k\epsilon_t\xi_t(k) + \varrho_k^2\xi^2_t(k)\}, & r > 0,
                \end{cases}
        \]
        where we assume additionally $\{(\epsilon_t, v_k^{\top}X_t)\}_{t \in \mathbb Z} \overset{\mathcal{D}}{\longrightarrow} \{(\epsilon_t, \xi_t(k))\}_{t \in \mathbb Z}$, for $k \in \{1, \dots, K\}$, and $v_k$ is the normalised jump vector defined in \Cref{assume:regression parameters}.
    \end{itemize}

\noindent \textbf{b.}~(Vanishing regime)  For $k \in \{1, \ldots, K\}$, if $\kappa_k \to 0$, as $n \to \infty$, then the following results hold.   
    \begin{itemize}
        \item [\textbf{b.1.}] The estimation error satisfies that $|\widetilde{\eta}_k - \eta_k| = O_p(\kappa_k^{-2})$, as $n \to \infty$.
        \item [\textbf{b.2.}] When $n \to \infty$, 
        \[
            \kappa_k^{2}( \widetilde{\eta}_k - \eta_k) \overset{\mathcal{D}}{\longrightarrow} \argmin_{r \in \mathbb{R}} \left\{\frac{\varpi_k}{\sigma_{\infty}(k)}|r| + \mathbb{W}(r)\right\},
        \]
        where $\varpi_k$ is the limiting drift coefficient defined in \Cref{assume:X}\textbf{c}, and $\mathbb{W}(r)$ is a two-sided standard Brownian motion defined as 
        \begin{align*}
            \mathbb{W}(r) = \begin{cases}
                \mathbb{B}_1(-r), & r < 0,\\
                0, & r = 0,\\
                \mathbb{B}_2(r), & r > 0,
                \end{cases}
        \end{align*}
        with $\mathbb{B}_1(r)$ and $\mathbb{B}_2(r)$ being two independent standard Brownian motions.
    \end{itemize}
\end{theorem}

\Cref{thm:asymptotics} provides not only estimation error controls of the change point estimators $\{\widetilde{\eta}_k\}_{k = 1}^{\widehat{K}}$ under two scenarios of jump sizes, but more importantly their limiting distributions. We remark that the localisation rate in \Cref{thm:asymptotics} is an exact match of the minimax lower bounds derived under the temporal independence assumption \citep[e.g.][]{rinaldo2021localizing} for the high-dimensional regression model with change points, while all previous works are off by a logarithmic factor. As far as we are aware, this exact match is the first time seen in the change point literature in high-dimensional regression setting. 

In the following, we further discuss \Cref{thm:asymptotics} in more details from the aspects of tuning parameters, the sketch of proof and the intuition of the two different forms of limiting distributions.

\medskip
\noindent \textbf{Choices of the tuning parameters.}  Recall that the estimators $\{\widetilde{\eta}_k\}_{k = 1}^{\widehat{K}}$ are obtained based on the preliminary estimators $\{\widehat{\eta}_k\}_{k = 1}^{\widehat{K}}$, which are outputs of the DPDU algorithm detailed in \Cref{algorithm:DPDU}.  There are two tuning parameters involved in the DPDU step: the Lasso penalisation parameter $\lambda$ and the interval length cutoff $\zeta$, see \eqref{eq:pre_est}, \eqref{eq:interval gof 2} and \eqref{eq:interval lasso}.  
\begin{itemize}
    \item Lasso parameter $\lambda$.  It is interesting to see that despite the temporal dependence and more general tail behaviour, the Lasso parameter $\lambda$ is chosen to be of order $\log^{1/2}(np)$, which is the same as that in the independence and sub-Gaussian tail case \citep[e.g.][]{buhlmann2011statistics}.  This is because, the Lasso procedures are only conducted on the intervals of length at least $\zeta$.  To elaborate, we investigate a large-probability upper bound on the term $|\mathcal{I}|^{-1}| \sum_{t \in \mathcal{I}} X_t\epsilon_t|_{\infty}$.  Due to our new Bernstein's inequality in \Cref{thm:bernstein_exp_subExp_nonlinear}, it is shown in \Cref{lemma:lasso deviation bound 1} in the supplementary material that, with large probability,
    \[
        |\mathcal{I}|^{-1}| \sum_{t \in \mathcal{I}} X_t\epsilon_t|_{\infty} \lesssim \sqrt{\frac{\log(np)}{|\mathcal{I}|}} + \frac{\log^{1/\gamma}(np)}{|\mathcal{I}|}.
    \]
    Provided that $|\mathcal{I}| \geq \zeta \gtrsim \{\log(pn)\}^{2/\gamma - 1}$, the term $\sqrt{\log(np)/|\mathcal{I}|}$ dominates $\log^{1/\gamma}(np)/|\mathcal{I}|$.  This explains the choice of $\lambda$.
    \item Interval length cutoff $\zeta$.  In addition to facilitate the choice of $\lambda$, the choice of $\zeta$ is also to enable the application of a restricted eigenvalue condition we provide for data with functional dependence, see \Cref{theorem:RES Version II} in the supplementary material.  We would also like to point out, since this paper focuses on a more challenging case where $\max_{k = 1}^K \kappa_k \lesssim 1$, \Cref{assump-snr}\textbf{a} guarantees that 
    \begin{equation}\label{eq-delta-zeta}
        \Delta \gtrsim \kappa^{-2} \{\s\log(np)\}^{2/\gamma-1} \alpha_n \gtrsim \zeta. 
    \end{equation}
    Our Lasso procedure only considers intervals of length at least $\zeta$, implying that the change point localisation error is at least of order $\zeta$. This price is acceptable by noticing that $\Delta \gtrsim \zeta$, as demonstrated in \eqref{eq-delta-zeta}.
\end{itemize}

\medskip
\noindent \textbf{Sketch of the proof.}  The proof can be roughly modularised into three steps.  
\begin{itemize}
    \item We first show that with large probability, the preliminary estimators are good enough such that for each $k \in \{1, \ldots, K\}$, the following holds. (i) The interval $(s_k, e_k)$ defined in \eqref{eq:corrected boundaries} contains one and only one true change point $\eta_k$, which is sufficiently far away from both $s_k$ and $e_k$; and (ii) the Lasso estimators $\widehat{\beta}_k$ is a good enough estimator of the true coefficient $\beta_k$.  These two statements have been shown previously in \cite{rinaldo2021localizing} for the temporal independence and sub-Gaussian case.  In Lemmas~\ref{theorem:DUDP} and \ref{lemmma:refinement_lasso}, we provide the functional dependence and general tail behaviour counterparts for this set of results, based on our new Bernstein inequality in \Cref{thm:bernstein_exp_subExp_nonlinear}.
    
    \item Based on the preliminary estimators, we then show that $\{\kappa_k^2|\widetilde{\eta}_k - \eta_k|\}_{k = 1}^{\widehat{K}}$ are uniformly tight.  This is a stepping stone towards our final limiting distributions, but this result is interesting \textit{per se}.  The preliminary estimators satisfy that, with large probability, 
    \[
        \max_{k = 1}^K \kappa_k^2|\widehat{\eta}_k - \eta_k| \lesssim \{\s \log(np)\}^{2/\gamma - 1},
    \]
    the upper bound of which is diverging.  The final estimators, however, refine the preliminary results by constraining the localisation error to $O_p(\kappa_k^{-2})$.
    \item The limiting distributions are thus derived, by exploiting the argmax continuous mapping theorem \citep[e.g.~Theorem~3.2.2~in][]{Vaart1996weak} and the functional central limit theorem under temporal dependence \citep[e.g.~Theorem~3~in][]{wu2011asymptotic}.
\end{itemize}

\medskip
\noindent \textbf{Limiting distributions.}  The limiting distributions differ based on the limits of the jump sizes $\{\kappa_k\}_{k = 1}^K$.  To explain this, we provide more insights in the proofs.  For $k \in \{1, \ldots, K\}$, recall that $\widetilde{\eta}_k$ is the minimiser to
\[
    Q_k(\eta) = \sum_{t = s_k }^{\eta-1}(y_t - X_t^{\top}\widehat{\beta}_{k-1})^2 + \sum_{t = \eta}^{e_k-1}(y_t - X_t^{\top}\widehat{\beta}_{k})^2.
\]
This is equivalent to finding the minimiser of 
\[
    \widetilde{Q}_k(r) = Q_k(\eta_k + r) - Q_k(\eta_k) = \begin{cases}
        \sum_{t = \eta_k}^{\eta_k + r - 1} \big\{(y_t - X_t^{\top}\widehat{\beta}_{k-1})^2 - (y_t - X_t^{\top}\widehat{\beta}_{k})^2\big\}, & r \in \mathbb{Z}_+,\\
        0, & r = 0, \\
        \sum_{t = \eta_k + r + 1}^{\eta_k} \big\{(y_t - X_t^{\top}\widehat{\beta}_{k})^2 - (y_t - X_t^{\top}\widehat{\beta}_{k-1})^2\big\}, & r \in \mathbb{Z}_-.
    \end{cases}
\]
Considering only the $r > 0$ case, provided that $\widehat{\beta}_{k-1}$ and $\widehat{\beta}_{k}$ are good enough, $\widetilde{Q}_k(r)$ can be thus written as the sum of
\begin{align}
    Q^*_k(\eta_k + r) - Q^*_k(\eta_k) & = \sum_{t = \eta_k}^{\eta_k + r - 1} \big\{(y_t - X_t^{\top}\beta_{\eta_{k-1}}^*)^2 - (y_t - X_t^{\top}\beta_{\eta_{k}}^*)^2\big\} \nonumber \\
    & = \kappa_k \sum_{t = \eta_k}^{\eta_k + r - 1}(2\epsilon_tv_k^{\top}X_t) + \kappa_k^2 \sum_{t = \eta_k}^{\eta_k + r - 1} (v_k^{\top}X_t)^2 \label{eq-Q-Q-main}
\end{align} 
and the remainder term, which is of order $o_p(r^{1/2}\kappa_k \vee r\kappa_k^2 \vee 1)$. The main reason for imposing the stronger signal-to-noise ratio condition in \Cref{assump-snr}\textbf{b} is to ensure that remainder term is indeed of such order.  This consequently establishes the uniform tightness of $|\widetilde{\eta}_k - \eta_k|$.

\begin{itemize}
    \item In the non-vanishing case, all the choices of $r$ considered are finite due to the uniform tightness of $|\widetilde{\eta}_k - \eta_k|$ and $\kappa_k$ is also finite due to \Cref{assume:regression parameters}. Term \eqref{eq-Q-Q-main} thus can be approximated by a two-sided random walk distribution.  
    \item In the vanishing case, the remainder term vanishes.  Regarding the terms in \eqref{eq-Q-Q-main}, all $r$, such that $r\kappa_k^2$ are finite, are considered. Since $\kappa_k$ vanishes, the functional central limit theorem leads to the two-sided Brownian motion distribution in the limit.  
\end{itemize}

\subsection{Connections with relevant literature}\label{sec-rel-lit-result}

\Cref{thm:asymptotics} presents limiting distributions of change point estimators, under both the non-vanishing and vanishing jump size regimes, in a high-dimensional linear regression model, with nonlinear temporal dependence in both covariate and noise sequences.  Such results are novel at a few fronts, especially that this is a new addition to high-dimensional linear regression settings.  Nevertheless, limiting distributions of change point estimators have been studied previously in other models.  At a high level, all are built on different versions of argmax continuous mapping theorem, but specific work faces specific challenges brought in by specific models and temporal dependence assumptions.  

The pioneer work by \cite{yao1989least} studies a mean change problem in an independent univariate sequence.  They present a two-sided random walk-type limiting distribution in the non-vanishing regime, as a univariate mean counterpart of \Cref{thm:asymptotics}\textbf{a}.  Despite being univariate and independent, when the distribution of noise is unknown, there is no practical guidance on utilising the limiting distribution.

Fully-observed functional mean change point problems are studied in \cite{aue2009estimation} and \cite{aue2018detecting}.  Both works study one change point scenario with the minimal spacing $\Delta = cn$, where $c \in (0, 1)$ being a constant.  In both these two papers, limiting distributions of change point estimators are derived in non-vanishing and vanishing regimes.  Although this is a different model, the two limiting distributions are also in the form of two-sided random walks and Brownian motions, respectively.  \cite{aue2009estimation} consider a temporally-independent case and the noise sequence in \cite{aue2018detecting} is assumed to be a linear process.

High-dimensional mean change point problems have been recently studied in a stream of papers, including \cite{bai2010common}, \cite{wangshao2020}, \cite{chen2021inference}, \cite{kaul2021} and \cite{kaul2021inference}.  This is again a fundamentally different model, where at every time point, a high-dimensional vector is observed and the means of the vectors change at certain points. For high-dimensional mean change point problems, in terms of the jump size limits regimes, \cite{bai2010common}, \cite{wangshao2020} and \cite{chen2021inference} consider the vanishing jump size regime, providing two-sided Brownian motion type limiting distributions; while \cite{kaul2021} and \cite{kaul2021inference} consider both regimes, providing two-sided random walks and Brownian motions type limiting distributions.  In terms of the number of change points, \cite{bai2010common}, \cite{wangshao2020} and \cite{kaul2021} study one change point scenario with the minimal spacing $\Delta = cn$, where $c \in (0, 1)$ being a constant; while \cite{chen2021inference} and \cite{kaul2021inference} allow for multiple change points.  In terms of the temporal dependence assumption, \cite{wangshao2020}, \cite{kaul2021} and \cite{kaul2021inference} assume temporal independence; while \cite{bai2010common} and \cite{chen2021inference} assume the noise sequence is a linear process. 

In view of these aforementioned works, we see that our framework is by far the most general setting.  Since we allow both the high-dimensional covariate and univariate noise sequences to be temporally dependent in a nonlinear way, new techniques handling the large deviation behaviours of stationary and non-stationary sequences are developed.  This is unseen in the existing work.  Moreover, compared to the mean change problems, we need to conduct high-dimensional regression based on non-stationary and temporally-dependent sequences.  This relies on a set of new techniques, including a new restricted eigenvalue-type result, which is, among other key technical ingredients, not dealt with in the existing change point inference literature.

\subsection{Intermediate results} \label{sec:theory_preliminary_est}

We collect some important intermediate results of our analysis, which are interesting byproducts \textit{per se}.  The one at the core of the analysis is a new Bernstein inequality, controlling the large deviation bound under temporal dependence and more general tail behaviour. 

\begin{theorem}\label{thm:bernstein_exp_subExp_nonlinear}
Let $\{Z_t\}_{t \in \mathbb{Z}}$ be an $\mathbb{R}$-valued, mean-zero, possibly nonstationary process of the form~\eqref{eq:nonstationary}.  Let the cumulative functional dependence measure $\Delta^Z_{m,2}$ be defined in \eqref{eq:fdm_nonstationary}.  Assume that there exist absolute constants $\gamma_1(Z), \gamma_2(Z), c, C_{\mathrm{FDM}} > 0$ such that 
\begin{equation}\label{eq-z-bern-gamma-1}
    \sup_{m \geq 0}\exp(cm^{\gamma_1(Z)})\Delta^Z_{m,2} \leq C_{\mathrm{FDM}}    
\end{equation}
and $\sup_{t \in \mathbb{Z}}\mathbb{P}(|Z_t| > x ) \leq \exp(1 - x ^{\gamma_2(Z)})$, for any $x > 0$.  Let $\gamma(Z) = \{1/\gamma_1(Z) + 1/\gamma_2(Z)\}^{-1} < 1$. There exist absolute constants $c_1, c_2 > 0$ such that for any $x \ge 1$ and integer $m \geq 3$,
\begin{equation}\label{eq-z-bern-result}
    \mathbb{P}\left\{m^{-1/2}\left|\sum_{t = 1}^m Z_t\right| \geq x\right\} \leq m\exp\{- c_1 x^{\gamma(Z)} m^{\gamma(Z)/2}\} + 2\exp\{-c_2 x^2\}.
\end{equation}
\end{theorem} 

\Cref{thm:bernstein_exp_subExp_nonlinear} is repeatedly used in the proof of \Cref{thm:asymptotics}.  Specific forms of $\{Z_t\}_{t \in \mathbb{Z}}$ used include strictly stationary series $\{\epsilon_tX_{tj}\}$ and nonstationary series $\{\epsilon_tX_t^{\top} \beta^*_t\}_{t \in \mathbb{Z}}$.  Note that the moment condition in the cumulative functional dependence measure $\Delta^Z_{m,2}$ used in \Cref{thm:bernstein_exp_subExp_nonlinear} is 2.  This further explains why the fourth moment condition is assumed in Assumptions~\ref{assume:X} and \ref{assume:epsilon}.  In \eqref{eq-z-bern-result}, we see that the decay rate for a sum of $m$ observations is governed by the interplay of two terms.  When the sample size is large enough, the decay rate is governed by the same rate as that in the temporal independence and sub-Gaussian case \citep[e.g.][]{vershynin2018high}.  This coincides with the conventional wisdom in understanding standard Bernstein's inequality.

With the help of \Cref{thm:bernstein_exp_subExp_nonlinear}, we are able to show the properties of preliminary estimators under temporal dependence and the general tail behaviour.  The localisation error, the error bounds on Lasso estimators and the error bounds on the jump size estimators are presented in Lemmas~\ref{theorem:DUDP}, \ref{lemmma:refinement_lasso} and~\ref{lemma:consistency_jump_size_est}, respectively.

\begin{lemma}[Preliminary change point estimators] \label{theorem:DUDP}
Given data $\{(y_t, X_t)\}_{t = 1}^n$, suppose that Assumptions~\ref{assume:regression parameters}, \ref{assume:X}, \ref{assume:epsilon} and \ref{assump-snr}\textbf{a} hold.  Let $\{\widehat{\eta}_k\}_{k = 1}^{\widehat{K}}$ be the change point estimators output by $\mathrm{DPDU}(\{y_t, X_t\}_{t = 1}^n$, $\lambda$, $\zeta)$, with $\lambda = C_\lambda \sqrt{\log(np)}$ and $\zeta = C_\zeta\{\s \log(pn)\}^{2/\gamma - 1}$, $C_{\lambda}, C_{\zeta} > 0$ being absolute constants.  It then holds with probability at least $1 - (n \vee p)^{-3}$ that, with an absolute constant $C_{\mathfrak{e}} > 0$,
\[
    \widehat{K} = K \quad \mbox{and} \quad \max_{k = 1}^K \kappa_k^2|\widehat{\eta}_k - \eta_k| \leq C_{\mathfrak{e}} \{\s \log(pn)\}^{2/\gamma - 1}.
\]
\end{lemma}

\begin{lemma}[Lasso estimators] \label{lemmma:refinement_lasso}
Under the same conditions and with the same notation as that in \Cref{theorem:DUDP}, let $\{\widehat{\beta}_k\}_{k = 1}^{\widehat{K}}$ be the Lasso estimators defined in \eqref{eq:interval lasso} with the intervals constructed based on $\{\widehat{\eta}_k\}_{k = 1}^{\widehat{K}}$.  It then holds with probability at least $1 - (n \vee p)^{-3}$ that, with an absolute constant $C > 0$,
\begin{align*}
    & \max\bigg\{\max_{k = 1}^K \kappa_k^{-1}|\widehat{\beta}_k - \beta_{\eta_k}^*|_2, \quad \max_{k = 0}^{K - 1} \kappa_{k+1}^{-1}|\widehat{\beta}_k - \beta_{\eta_k}^*|_2, \\
    & \hspace{2cm} \s^{-1/2} \max_{k = 1}^K \kappa_k^{-1}|\widehat{\beta}_k - \beta_{\eta_k}^*|_1, \quad \s^{-1/2} \max_{k = 0}^{K-1} \kappa_{k+1}^{-1}|\widehat{\beta}_k - \beta_{\eta_k}^*|_1\bigg\} \leq C \alpha^{-1/2}_{n},
\end{align*}
where $\alpha_n > 0$ is any diverging sequence specified in Assumption~\ref{assump-snr}\textbf{a}.
\end{lemma}

\begin{lemma}[Jump size estimators] \label{lemma:consistency_jump_size_est}
Under the same conditions and with the same notation as that in \Cref{theorem:DUDP}, for each $k \in \{1, \ldots, \widehat{K}\}$, let 
\begin{equation}\label{eq-jump-size-estimator}
    \widehat{\kappa}_k = |\widehat{\beta}_k - \widehat{\beta}_{k-1}|_2
\end{equation} 
be the $k$th jump size estimator.  It then holds with probability at least $1 - (n \vee p)^{-3}$ that, with an absolute constant $C > 0$,  
\[
    \max_{k = 1}^K |\widehat{\kappa}_k/\kappa_k - 1| \leq C\alpha_n^{-1/2},
\]
where $\alpha_n > 0$ is any diverging sequence specified in Assumption~\ref{assump-snr}\textbf{a}.
\end{lemma}

\section{Consistent long-run variance estimation}\label{sec:LRV}

As pointed out in existing change point inference literature, e.g.\ \cite{bai2010common}, in the non-vanishing jump size regime, change points can be estimated within a constant error rate; while in the vanishing jump size regime, the estimators have diverging localisation error rates.  In practice, it is therefore more interesting to study the limiting distributions in the vanishing regime.

To practically perform inference on change point locations based on change point estimators $\{\widetilde{\eta}_k\}_{k = 1}^{\widehat{K}}$ in the vanishing jump size regime, it is crucial to have access to consistent estimators of the long-run variances $\{\sigma^2_{\infty}(k)\}_{k = 1}^K$ and the drift coefficients $\{\varpi_k\}_{k = 1}^K$, involved in the limiting distribution in \Cref{thm:asymptotics}\textbf{b.2}.  In this section, we first propose a block-type long-run variance estimator and derive its consistency. With the help of this long-run variance estimator, we are able to provide a complete inference procedure of the vanishing regime in \Cref{sec-CI}.   
  
As we have mentioned in \Cref{sec:intro}, most of the existing literature for change point inference in high dimension lacks theoretical discussions on the long-run variance estimator.  To the best of our knowledge, \cite{chen2021inference} is one of the few existing works showing the consistency of the long-run variance estimators.  \cite{chen2021inference} studied the consistency of a block-type estimator, for linear processes.  In the following, we are to show that a block-type estimator, detailed in \Cref{alg-lrv-est}, is also consistent for the general, nonlinear, temporal dependence sequences.

\begin{algorithm}[ht]
\begin{algorithmic}
    \INPUT $\{(y_t, X_t)\}_{t = 1}^n$, $\{\widehat{\eta}_k\}_{k = 1}^{\widehat{K}}$, $\{\widehat{\beta}_k\}_{k = 1}^{\widehat{K}}$, $\{\widehat{\kappa}_k\}_{k = 1}^{\widehat{K}}$,  $\{(s_k, e_k)\}_{k = 1}^{\widehat{K}}$ and tuning parameter $R \in \mathbb{N}$
    \For{$k \in \{1, \ldots, \widehat{K}\}$}
        \For{$t \in \{s_k, \ldots, e_k-1\}$}
            \State $Z_t^{(k)} \leftarrow (y_t - X_t^{\top}\widehat{\beta}_k)X_t^{\top}(\widehat{\beta}_{k+1} - \widehat{\beta}_k)$, $Z_t^{(k+1)} \leftarrow (y_t - X_t^{\top}\widehat{\beta}_{k+1})X_t^{\top}(\widehat{\beta}_{k+1} - \widehat{\beta}_k)$
            \State $Z_t \leftarrow Z_t^{(k)} + Z_t^{(k+1)}$
        \EndFor
        \State $S \leftarrow \lfloor (e_k - s_k)/(2R)\rfloor$
        \For{$r \in \{1, \ldots, 2R\}$}
            \State $\mathcal{S}_r \leftarrow \{s_k + (r-1)S, \ldots, s_k + rS - 1\}$
        \EndFor
        \For{$r \in \{1, \ldots, R\}$}
            \State $D_r \leftarrow (2S)^{-1/2} (\sum_{t \in \mathcal{S}_{2r-1}} Z_t - \sum_{t \in \mathcal{S}_{2r}} Z_t)$
        \EndFor
        \State $\widehat{\sigma}_{\infty}^2(k) \leftarrow R^{-1} \widehat{\kappa}_k^{-2} \sum_{r = 1}^R D_r^2$
    \EndFor
    \OUTPUT $\{\widehat{\sigma}_{\infty}^2(k)\}_{k = 1}^{\widehat{K}}$
\caption{Long-run variance estimators} \label{alg-lrv-est}
\end{algorithmic}
\end{algorithm} 

\begin{theorem}\label{thm:lrv_consist}
Under all the assumptions and with all the notation in \Cref{thm:asymptotics}, let $\{\widehat{\sigma}_{\infty}^2(k)\}_{k = 1}^{\widehat{K}}$ be the output of \Cref{alg-lrv-est} with $\{\widehat{\eta}_k\}_{k = 1}^{\widehat{K}}$ output by $\mathrm{DPDU}$ in \Cref{algorithm:DPDU}, $\{\widehat{\beta}_k\}_{k = 1}^{\widehat{K}}$ constructed from \eqref{eq:interval lasso} based on the intervals with endpoints $\{\widehat{\eta}_k\}_{k = 1}^{\widehat{K}}$, $\{\widehat{\kappa}_k\}_{k = 1}^{\widehat{K}}$ defined in \eqref{eq-jump-size-estimator}, $\{(s_k, e_k)\}_{k = 1}^{\widehat{K}}$ defined in \eqref{eq:corrected boundaries} and $R \in \mathbb{N}$ satisfying $(e_k - s_k) \gg R \gg \sqrt{e_k - s_k}$, as $n$ diverges.  We have that 
\[
    \max_{k = 1}^K |\widehat{\sigma}_{\infty}^2(k) - \sigma_{\infty}^2(k)| \overset{P.}{\longrightarrow} 0, \quad n \to \infty.
\]
\end{theorem}

\Cref{thm:lrv_consist} demonstrates the consistency of the long-run variance estimators $\{\widehat{\sigma}_{\infty}^2(k)\}_{k = 1}^{\widehat{K}}$, which are functions of $\{Z_t\}_{t = 1}^n$ defined in \Cref{alg-lrv-est}.  The sequence $\{Z_t\}_{t = 1}^n$ is a sample version of the unobservable sequence $\{Z_t^*\}_{t = 1}^n$, defined as
\[
    Z_t^* = \begin{cases}
        2\epsilon_t\Psi_k^{\top}X_t - (\Psi_k^{\top}X_t)^2, & t = s_k, \ldots, \eta_k-1,\\
        2\epsilon_t\Psi_k^{\top}X_t + (\Psi_k^{\top}X_t)^2, & t = \eta_k, \ldots, e_k-1.
    \end{cases}
\]
As shown from \Cref{lemma:long-run_equiv} in the supplementary material and the proof of \Cref{thm:asymptotics}, $\sigma_{\infty}^2(k)$ is, in fact, the long-run variance of the process $\{\kappa_k^{-1}Z_t^*\}_{t \in \mathbb{Z}}$. Note that $\var(Z_t^*)$ maintains the same for all $t \in (s_k,e_k]$, but $\mathbb{E}[Z_t^*]$ changes at $\eta_k$ from $-\Psi_k^{\top}\Sigma \Psi_k$ to $\Psi_k^{\top}\Sigma \Psi_k$.  This requires a robust estimation method against the mean change at the unknown change point $\eta_k$ and is materialised via the block-type estimator.  It is also possible to construct a long-run variance estimator based on the stationary process $\{2\epsilon_tv_k^{\top}X_t\}_{t \in \mathbb{Z}}$. We consider the estimator based on $\{\kappa_k^{-1}Z_t^*\}_{t \in \mathbb{Z}}$ due to consideration of finite sample performances.

For any $k \in \{1, \ldots, \widehat{K}\}$, the estimator $\widehat{\sigma}_{\infty}^2(k)$ is constructed based on data from the interval $(s_k, e_k)$.  When the data are assumed to be independent, the sample variance usually serves as a good estimator of the population variance.  To capture the temporal dependence, we proceed by constructing sample variances of means of blocks of data, with block width $S \to \infty$, as $n$ diverges.  This is done by dividing the interval $(s_k, e_k)$ into $2R$ width-$S$ intervals, with $2R = \lfloor (e_k-s_k)/S \rfloor$.  To subtract the mean when estimating the variance, we define $D_r$, $r \in \{1, \ldots, R\}$, by taking the difference of the sample means in two consecutive blocks.  As long as there is no change point in $\mathcal{S}_{2r-1} \cup \mathcal{S}_{2r}$, it roughly holds that $\mathbb{E}\{D_r\} = 0$.  With large probability, there is one and only one true change point in $(s_k, e_k)$.  This means, only one $D_r$, $r \in \{1, \ldots, R\}$, does not have mean zero, namely $D_{r^*}$.  The long-run variance can thus be estimated by the sample mean of $\{D_r^2\}_{r = 1}^R$, hedging out the effect of $D_{r^*}$, provided that $R \to \infty$, as $n$ diverges.

\subsection{Confidence interval construction}\label{sec-CI}

An important application of the consistent long-run variance estimator derived above is to construct confidence intervals for the change points.  In view of the limiting distributions
    \[
        \kappa_k^{2}( \widetilde{\eta}_k - \eta_k) \overset{\mathcal{D}}{\longrightarrow} \argmin_{r \in \mathbb{R}} \big\{\varpi_k|r| + \sigma_{\infty}(k)\mathbb{W}(r)\big\}, \quad n \to \infty, \, \forall k \in \{1, \ldots, K\},
    \]
    derived in \Cref{thm:asymptotics}\textbf{b}, aside from the long-run variance, the drift coefficients $\varpi_k = v_k^{\top}\Sigma v_k$, $k \in \{1, \ldots, K\}$, are also unknown.  In the rest of this subsection, we first propose consistent estimators of $\{\varpi_k\}_{k = 1}^K$ in \Cref{thm:quadratic_consist} and then provide a confidence interval construction strategy. 
    
\begin{proposition}[Consistent estimators of $\{\varpi_k\}_{k = 1}^K$] \label{thm:quadratic_consist}
Given data $\{(y_t, X_t)\}_{t = 1}^n$, suppose that Assumptions~\ref{assume:regression parameters}, \ref{assume:X}, \ref{assume:epsilon} and \ref{assump-snr}\textbf{a} hold.  Let $\{\widehat{\eta}_k\}_{k = 1}^{\widehat{K}}$ be the change point estimators output by $\mathrm{DPDU}(\{y_t, X_t\}_{t = 1}^n$, $\lambda$, $\zeta)$, with $\lambda = C_\lambda \sqrt{\log(np)}$ and $\zeta = C_\zeta\{\s \log(pn)\}^{2/\gamma - 1}$, $C_{\lambda}, C_{\zeta} > 0$ being absolute constants. Let $\{\widehat{\beta}_k\}_{k = 1}^{\widehat{K}}$ be the Lasso estimators defined in \eqref{eq:interval lasso} with the intervals constructed based on $\{\widehat{\eta}_k\}_{k = 1}^{\widehat{K}}$.  Let $\{\widehat{\kappa}_k\}_{k = 1}^{\widehat{K}}$ be the jump size estimators defined in \eqref{eq-jump-size-estimator}.  For $k \in \{1, \ldots, \widehat{K}\}$, define the drift estimators
\begin{align}\label{eq-drift-est}
    \widehat{\varpi}_k =  \frac{1}{n\widehat{\kappa}_k^2}\sum_{t = 1}^n(\widehat{\beta}_{k} - \widehat{\beta}_{k-1})^{\top}X_tX_t^{\top}(\widehat{\beta}_{k} - \widehat{\beta}_{k-1}).
\end{align}
It then holds that 
\[
    \max_{k = 1}^K |\widehat{\varpi}_k -\varpi_k| \overset{P.}{\longrightarrow} 0, \quad n \to \infty.
\]
\end{proposition}

Comparing the conditions required in \Cref{thm:lrv_consist} and \Cref{thm:quadratic_consist}, we see that the slightly stronger signal-to-noise ratio condition \Cref{assump-snr}\textbf{b} is in fact needed when estimating the long-run variance.  The weaker condition \Cref{assump-snr}\textbf{a} is sufficient for \Cref{thm:quadratic_consist}, which only concerns the properties of the initial change point estimators.

For $k \in \{1, \ldots, \widehat{K}\}$, with the ingredients: change point estimator $\widetilde{\eta}_k$ defined in \eqref{eq:refinement}, local refinement interval $(s_k, e_k)$ defined in \eqref{eq:corrected boundaries}, long-run variance estimator $\widehat{\sigma}_{\infty}(k)$ output by \Cref{alg-lrv-est} and drift estimator $\widehat{\varpi}_k$ defined in \eqref{eq-drift-est}, we streamline the construction of a numerical $(1-\alpha)$-confidence interval of $\eta_k$, with a user-specified $\alpha \in (0, 1)$. 

\begin{itemize}
    \item [Step 1.] Let $B\in \mathbb{Z}_+$ and $M \in \mathbb{R}_+$.  For $b \in \{1, \ldots, B\}$, let
        \begin{align}\label{eq:simu_est}
            \widehat{u}^{(b)} = \argmin_{r \in (-M, M)} \big\{\widehat{\varpi}_k|r| + \widehat{\sigma}_{\infty}(k)\mathbb{W}^{(b)}(r)\big\},
        \end{align}
        where
        \begin{align*}
            \mathbb{W}^{(b)}(r) = \begin{cases}
                \frac{1}{\sqrt{n}}\sum_{i = \lceil nr \rceil}^{-1}z_i^{(b)}, & r < 0,\\
                0, & r = 0,\\
                \frac{1}{\sqrt{n}}\sum_{i = 1}^{\lfloor nr \rfloor}z_i^{(b)}, & r > 0, 
            \end{cases}
        \end{align*}
        the quantity $n$ is the sample size and $\{z_i^{(b)}\}_{i = -\lfloor nM \rfloor}^{\lceil nM \rceil}$ are independent standard Gaussian random variables. 
    \item [Step 2.] Let $\widehat{q}_u(1-\alpha/2)$ and $\widehat{q}_u(\alpha/2)$ be the $(1-\alpha/2)$- and $\alpha/2$-quantiles of the empirical distribution of $\{\widehat{u}^{(b)}\}_{b = 1}^B$.  Provided that $\widehat{\kappa}_k \neq 0$, a numerical $(1-\alpha)$ confidence interval of $\eta_k$ is 
    \begin{align}\label{eq-CI-cons}
        \bigg[\widetilde{\eta}_k + \frac{\widehat{q}_u(\alpha/2)}{\widehat{\kappa}_k^2}\mathbbm{1}\{\widehat{\kappa}_k \neq 0\}, \, \widetilde{\eta}_k + \frac{\widehat{q}_u(1-\alpha/2)}{\widehat{\kappa}_k^2}\mathbbm{1}\{\widehat{\kappa}_k \neq 0\} \bigg].
    \end{align}
\end{itemize}

The rationale of this procedure is that, following \Cref{thm:asymptotics}, each $\widehat{u}^{(b)}$ is a simulated $\widehat{\kappa}^2_k (\widetilde{\eta}_k - \eta_k)$.  Due to the uniform tightness, it is sufficient to restrict the optimisation \eqref{eq:simu_est} in the interval $(-M, M)$, with a sufficiently large $M > 0$.  When $B$ is large enough, the empirical distribution of $\{\widehat{u}^{(b)}\}_{b = 1}^B$ serves as a reasonable approximation of the limiting distribution of $\widehat{\kappa}^2_k (\widetilde{\eta}_k - \eta_k)$, which justifies that \eqref{eq-CI-cons} is a numerical $(1-\alpha)$ confidence interval of $\eta_k$.  This is shown in \Cref{thm:sampling_dist} below, that the set $\{\widehat{u}^{(b)}\}_{b = 1}^B$ approximates the sampling distribution of $\widehat \kappa_k^{2}(\widetilde{\eta}_k - \eta_k)$.

\begin{theorem}\label{thm:sampling_dist}
Under all the assumptions in \Cref{thm:lrv_consist}, for any $k \in \{1, \ldots, K\}$, let $\widehat{u}^{(1)}$ be defined in \eqref{eq:simu_est} with $M = \infty$ and $\widehat{\kappa}^2_k$ be the $k$th jump size estimator defined in \eqref{eq-jump-size-estimator}.  It holds that
    \begin{align*}
        \frac{\kappa_k^2}{\widehat{\kappa}_k^2}\widehat{u}^{(1)} \mathbbm{1}\{\widehat{\kappa}_k \neq 0\} \overset{\mathcal{D}}{\longrightarrow} \argmin_{r \in \mathbb{R}} \big\{\varpi_k|r| + \sigma_{\infty}(k)\mathbb{W}(r)\big\}, \quad n \to \infty.
    \end{align*}
\end{theorem}

\section{Numerical experiments}\label{sec:numeric_study}
In this section, we conduct numerical studies to investigate the computational efficiency of the proposed DPDU algorithm (\Cref{algorithm:DPDU}) in \Cref{sec:comp_DPDU} and the numerical performance of change point confidence intervals \eqref{eq-CI-cons} in \Cref{sec:numeric_perform}. This section is concluded with a real data application in \Cref{sec:realdata} to showcase the practicability of our proposed methodology.

\subsection{Computational efficiency}\label{sec:comp_DPDU}

Recall that we have shown in \Cref{lemma:memory_computation_DPDU} that DPDU improves the computational efficiency of the standard dynamic programming (DP).  We refer readers to \cite{friedrich2008complexity} for a thorough discussion on a generic dynamic programming algorithm (a.k.a.~optimal partition algorithm, Algorithm~1 therein) and \cite{rinaldo2021localizing} for the deployment of DP in the high-dimensional linear regression change point localisation context.  The implementation of DP and DPDU with the same coordinate descent for the Lasso is available in the R package \texttt{changepoints}\footnote{The development version of \texttt{changepoints} is available at \url{https://github.com/HaotianXu/changepoints}.} \citep{changepoints_R} as \texttt{DP.regression} and \texttt{DPDU.regression} functions, respectively. 

In this subsection, we reinforce the computational efficiency of DPDU over DP through simulations. 
Given the data $\{(y_t, X_t)\}_{t =1}^{n} \subset \mathbb{R}\times\mathbb{R}^p$ and fixed tuning parameters, we compare DUDP and DP in the settings where $n \in \{200, 300, 400\}$ and $p \in \{30, 50, 70\}$. All the simulations are run on a machine equipped with an Apple M2 chip (8-core CPU). Figure \ref{fig:time_n} depicts the boxplots of computational cost (in seconds) over $100$ repetitions with the $x$-axis being $p$.  For each fixed $p$, as~$n$ linearly increases, it shows that the computational cost of DP increases much faster compared to that of DPDU, justifying the implication of \Cref{lemma:memory_computation_DPDU}. For each fixed $n$, as $p$ linearly increases, it shows that the increasing rates of the computational cost of DP are similar to that of DPDU.  In all considered settings, DPDU is nearly twice as fast as DP.  All findings align with \Cref{lemma:memory_computation_DPDU}.

\begin{figure}[t]
\centering
\includegraphics[width = 0.8\textwidth]{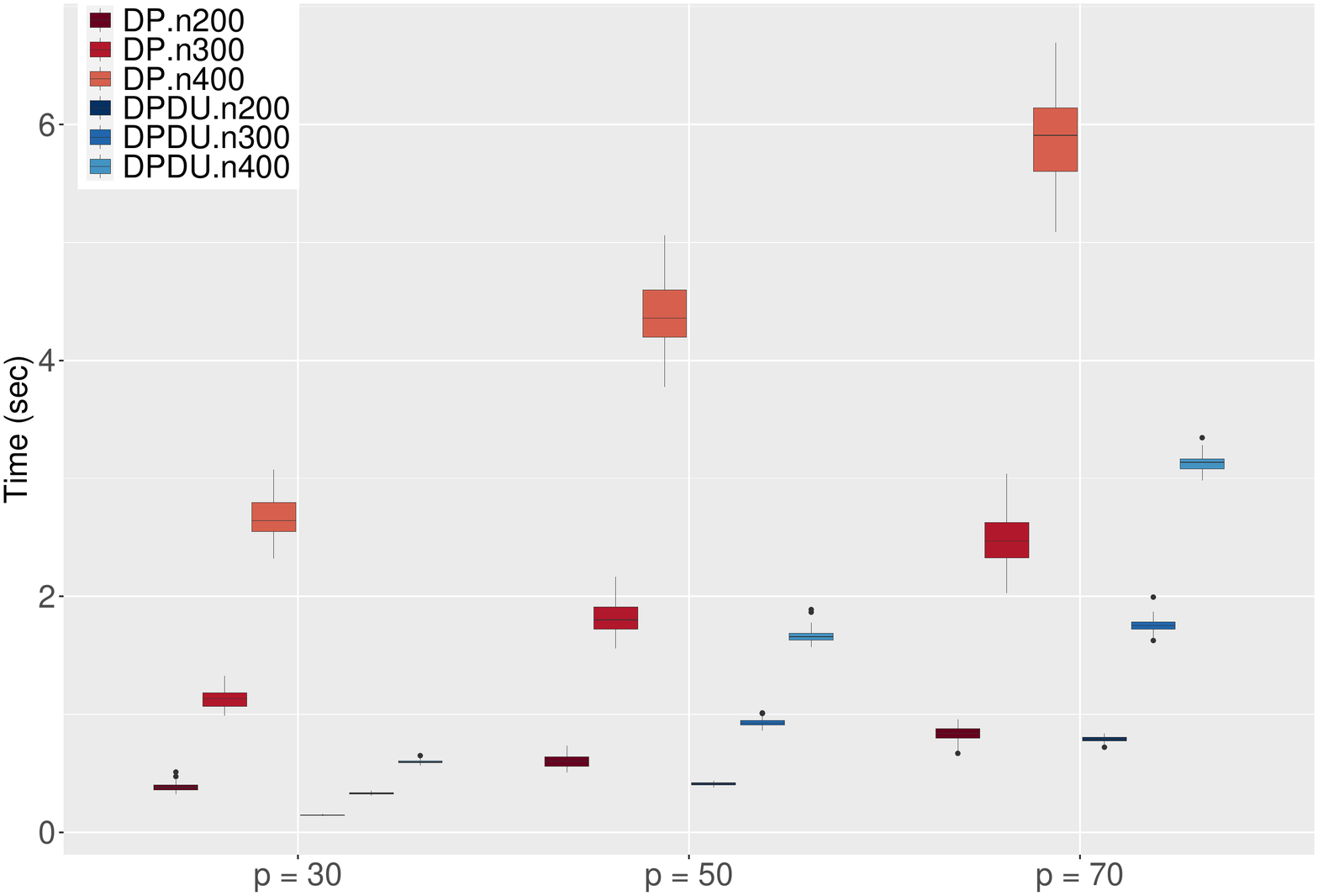}
    \caption{Boxplots of the computational cost (in seconds) of DP and DPDU over $100$ repetitions.} 
\label{fig:time_n}
\end{figure}

\subsection{Simulation studies}\label{sec:numeric_perform}

Our change point inference procedure contains three steps: (i) the preliminary estimation (see~\Cref{sec:algorithm}), (ii) the local refinement (see~\Cref{sec:local_refinement}) and (iii) the confidence interval construction (see~\Cref{sec-CI}).  To the best of our knowledge, no competitor exists for change point inference in high-dimensional regression settings. We therefore focus on investigating the performance of our proposed procedure for change point inference in different scenarios. As an important byproduct, the localisation results are also reported.

\medskip
\noindent \textbf{Settings.} 
We simulate data from the model described in \eqref{eq-model-1}, i.e.
\begin{align*}
  y_t = X_t^{\top}\beta_t^* + \epsilon_t, \quad t = 1, \dots, n.
\end{align*}
We consider temporal dependence in both $\{X_t\}_{t = 1}^n$ and $\{\epsilon_t\}_{t = 1}^n$. The covariate process $\{X_t\}_{t = 1}^n$ is generated from an autoregressive model, i.e.~$X_t = 0.3X_{t-1}+\sqrt{1-0.3^2}e_t$ with $e_t \overset{\mbox{i.i.d.}}{\sim} \mathcal{N}(0, I_p)$ and~$I_p$ being the $p$-dimensional identity matrix. The noise process $\{\epsilon_t\}_{t = 1}^n$ is generated by a moving average model, i.e.~$\epsilon_t = (e^{\prime}_{t}+0.3e^{\prime}_{t-1})/(2\sqrt{1+0.3^2})$ with $e^{\prime}_t \overset{\mbox{i.i.d.}}{\sim} \mathcal{N}(0, 1)$. Let $\beta_0 = (\beta_{0i}, i = 1, \dots, p)^{\top}$ with $\beta_{0i} = 2^{-1}\s^{-1/2}\kappa$ for $i = 1, \dots, \s$ and zero otherwise. Let $\{\eta_k\}_{k = 1}^{K}$ be the true change points.  We consider the regression coefficients
\begin{align*}
    \beta_t^* = \begin{cases}
        (-1)^0\beta_0, &t \in \{1, \dots, \eta_1-1\},\\
        (-1)^1\beta_0, &t \in \{\eta_1, \dots, \eta_2-1\},\\
        \dots \\
        (-1)^{K}\beta_0, &t \in \{\eta_{K}, \dots, n\}.
    \end{cases}
\end{align*}
In each scenario below, we simulate $500$ repetitions and fix $\s = 5$.

\medskip
\noindent \textbf{Evaluation measurements.} Let $\{\widehat{\eta}_k\}_{k = 1}^{\widehat{K}}$ and $\{\widetilde{\eta}_k\}_{k = 1}^{\widehat{K}}$ be the preliminary and final change point estimators defined in \eqref{eq:preliminary_cpt} and \eqref{eq:refinement}, respectively. To assess the performance of localisation, in Tables~\ref{tab:sce1_1}, \ref{tab:sce2_1} and \ref{tab:sce3_1}, we report (i) the proportions of the repetitions with $\widehat{K} < K$ (i.e.~underestimating the number of change points) and $\widehat{K} >K$ (i.e.~overestimating the number of change points), and (ii) means and standard deviations (in parentheses) of scaled Hausdorff distances of the preliminary and final estimators. The scaled Hausdorff distance of $\{\widehat{\eta}_k\}_{k = 1}^{K}$ is defined as
\begin{align*}
    d_{\mathrm{H}}^{\mathrm{pre}} = \frac{1}{n}\max\Big\{\max_{j = 0, \dots \widehat{K}+1}\min_{k = 0, \dots, K+1}|\widehat{\eta}_{j} - \eta_k|, \max_{k = 0, \dots , K+1} \min_{j = 0, \dots, \widehat{K}+1} |\widehat{\eta}_{j} - \eta_k| \Big\},
\end{align*}
where we set $\widehat{\eta}_0 = 1$ and $\widehat{\eta}_{\widehat{K}+1} = n+1$. We define similarly $d_{\mathrm{H}}^{\mathrm{fin}}$ for $\{\widetilde{\eta}_k\}_{k = 1}^{\widehat{K}}$.

The performance of change point inference is measured by the coverage rate of the constructed confidence interval, i.e.
\begin{equation*}
    \mathrm{cover}_k(1-\alpha)  = \mathbbm{1}\bigg\{ \eta_k \in \bigg[\Big\lfloor\widetilde{\eta}_k + \frac{\widehat{q}_u(\alpha/2)}{\widehat{\kappa}_k^2} \mathrm{1}\{\widehat{\kappa}_k \neq 0\}\Big\rfloor, \, \Big\lceil\widetilde{\eta}_k + \frac{\widehat{q}_u(1-\alpha/2)}{\widehat{\kappa}_k^2}\mathrm{1}\{\widehat{\kappa}_k \neq 0\}\Big\rceil \bigg]\bigg\},
\end{equation*}
where for $k = 1, \dots, K$, $\widehat{\kappa}_k$ is defined in \eqref{eq-jump-size-estimator}, $\widehat{q}_u(\alpha/2)$ and $\widehat{q}_u(1-\alpha/2)$ are defined in \eqref{eq-CI-cons} with $M = n$ and $B = 1000$. We remark that the floor and ceiling functions are applied respectively to the lower and upper endpoint of the confidence interval given in \eqref{eq-CI-cons}, considering that change points are integers.  We report the average coverage among all the repetitions with $\widehat{K} = K$ (out of $500$ repetitions), along with the widths of numerical ($1-\alpha$) confidence intervals, i.e.
\begin{equation*}
    \mathrm{width}_k(1-\alpha)  = \Big\lceil\widetilde{\eta}_k + \frac{\widehat{q}_u(1-\alpha/2)}{\widehat{\kappa}_k^2} \mathrm{1}\{\widehat{\kappa}_k \neq 0\} \Big\rceil - \Big\lfloor \widetilde{\eta}_k +\frac{\widehat{q}_u(\alpha/2)}{\widehat{\kappa}_k^2} \mathrm{1}\{\widehat{\kappa}_k \neq 0\} \Big\rfloor.
\end{equation*}
In Tables \ref{tab:sce1_2}, \ref{tab:sce2_2}, \ref{tab:sce3_2} and \ref{tab:sce3_3}, we report the average $\mathrm{cover}_k(1-\alpha)$ and means and standard deviations (in parentheses) of $\mathrm{width}_k(1-\alpha)$ with $\alpha \in \{1\%, 5\%\}$. %\zifeng{How can the width reported in the tables be less than 1?} \hx{all widths are large than 1}

\medskip
\noindent \textbf{Tuning parameters selection and other unknown quantities estimation.}  Three tuning parameters are involved in the proposed change point inference procedure: $\lambda$ and $\zeta$ for the DPDU algorithm (see~\Cref{algorithm:DPDU}) in the preliminary estimation and the number of pairs $R$ for long-run variance estimation (see~\Cref{alg-lrv-est}) in the confidence interval construction.  We adopt the cross-validation method proposed by \cite{rinaldo2021localizing} to select $\lambda$ and $\zeta$. Specifically, we first divide $\{(y_t, X_t)\}_{t =1}^{n}$ into training and validation sets according to odd and even indices. For each possible combination of $\lambda \in \{0.1, 0.5, 1, 2, 3\}$ and $\zeta \in \{10, 15, 20, 25\}$, we obtain the DPDU outputs ($\widehat{\mathcal{B}}$, $\widehat{K}$ and $\{\widehat{\beta}_k\}_{k = 0}^{\widehat{K}}$) based on the training set.  Using the DPDU outputs and the validation set, we compute the least squared prediction error as the validation loss.  We select the values of $\lambda$ and $\zeta$, which minimise the validation loss. The selection of $R$ is guided by \Cref{thm:lrv_consist}, i.e.~$R = \Big\lfloor  \Big(\max_{k = 1}^{\widehat{K}}\{e_k - s_k\}\Big)^{3/5} \Big\rfloor$ with $\{(s_k, e_k)\}_{k = 1}^{\widehat{K}}$ given in \eqref{eq:corrected boundaries}.  In addition, we use $\{\widehat{\kappa}_k\}_{k = 1}^{\widehat{K}}$, $\{\widehat{\sigma}^2_{\infty}(k)\}_{k = 1}^{\widehat{K}}$ and $\{\widehat{\varpi}_k\}_{k = 1}^{\widehat{K}}$, defined respectively in \eqref{eq-jump-size-estimator}, \Cref{alg-lrv-est} and \eqref{eq-drift-est}, to estimate unknown quantities in constructing confidence intervals.

\medskip
\noindent \textbf{Scenario 1: single change point.} Let $K = 1$, $\eta = n/2$ and the jump size $\kappa = 2$.  Vary $n \in \{100, 200, 300, 400\}$ and $p \in \{100, 200, 300\}$. Tables \ref{tab:sce1_1} and \ref{tab:sce1_2} respectively summarise the localisation and inference performances of $\eta$. We can see that both the preliminary and final estimators perform better as $n$ increases or $p$ decreases. Our proposed inference procedure produces good coverage in the considered settings.  As $n$ increases or $p$ decreases, the average length of our constructed confidence intervals decreases.
\begin{table}[t]
\centering
\caption{Localisation results of Scenario 1. }
\begin{tabular}{ccccc}
\hline
\multicolumn{5}{c}{$K = 1$, $\kappa = 2$} \\ 
$n$   & $\widehat{K} < K$ & $\widehat{K} > K$ & $d_{\mathrm{H}}^{\mathrm{pre}}$ & $d_{\mathrm{H}}^{\mathrm{fin}}$ \\ \hline
\multicolumn{5}{c}{$p = 100$}                                                                             \\
100 & 0.188   & 0         & 0.119 (0.183)  & 0.111 (0.186)  \\
200 & 0       & 0     & 0.007 (0.013)  & 0.003 (0.005)   \\
300 & 0       & 0     & 0.003 (0.006)  & 0.001 (0.002)   \\
400 & 0       & 0     & 0.002 (0.005)  & 0.001 (0.002)   \\
\multicolumn{5}{c}{$p = 200$}                                                                             \\
100 & 0.428   & 0         & 0.234 (0.225)  & 0.230 (0.228)   \\
200 & 0       & 0     & 0.009 (0.017)  & 0.004 (0.007)   \\
300 & 0       & 0     & 0.004 (0.006)  & 0.002 (0.003)   \\
400 & 0       & 0     & 0.002 (0.004)  & 0.001 (0.002)   \\
\multicolumn{5}{c}{$p = 300$}                                                                             \\
100 & 0.610     & 0       & 0.319 (0.218)   & 0.315 (0.221)   \\
200 & 0.004     & 0       & 0.013 (0.035)  & 0.007 (0.032)   \\
300 & 0         & 0   & 0.004 (0.008)  & 0.002 (0.003)   \\
400 & 0         & 0   & 0.003 (0.005)  & 0.001 (0.003)   \\
\hline
\end{tabular}
\label{tab:sce1_1}
\end{table}

\begin{table}[t]
\centering
\caption{Inference results of Scenario 1. }
\begin{tabular}{ccccc}
\hline
\multicolumn{5}{c}{$K = 1$, $\kappa = 2$} \\
 & \multicolumn{2}{c}{$\alpha = 0.01$} & \multicolumn{2}{c}{$\alpha = 0.05$} \\
$n$   & $\mathrm{cover}(1-\alpha)$ & $\mathrm{width}(1-\alpha)$ & $\mathrm{cover}(1-\alpha)$ & $\mathrm{width}(1-\alpha)$ \\ \hline
\multicolumn{5}{c}{$p = 100$}                                                                             \\
100 & 0.938            & 16.721 (12.351)  & 0.897            & 11.867 (8.193) \\
200 & 0.980            & 7.152 (1.913)  & 0.956    & 5.254 (1.416)  \\
300 & 0.988            & 6.072 (1.255)  & 0.980    & 4.420 (0.879) \\
400 & 0.974            & 5.270 (1.060)  & 0.958    & 4.022 (0.553) \\
\multicolumn{5}{c}{$p = 200$}                                                                             \\
100 & 0.920            & 28.972 (23.469) & 0.867 & 20.549 (16.514) \\
200 & 0.978            & 8.222 (2.655)  & 0.962  & 5.926 (1.894)  \\
300 & 0.994            & 6.284 (1.201)  & 0.968  & 4.556 (0.921)  \\
400 & 0.992            & 5.366 (1.065)  & 0.982    & 4.080 (0.599) \\
\multicolumn{5}{c}{$p = 300$}                                                                             \\
100 & 0.913            & 33.431 (22.318)   & 0.877    & 23.754 (16.469) \\
200 & 0.984            & 10.490 (5.637)  & 0.946    & 7.616 (4.165) \\
300 & 0.984            & 6.562 (1.406)  & 0.966   & 4.794 (1.055) \\
400 & 0.966            & 5.562 (1.053)  & 0.944    & 4.158 (0.680) \\ \hline
\end{tabular}
\label{tab:sce1_2}
\end{table}

\medskip
\noindent \textbf{Scenario 2: varying jump size $\kappa$.}  Consider the same model as that in Scenario~1.  Fix $(n, p) = (200, 300)$ and vary $\kappa \in \{1, 2, 3, 4\}$. The localisation performances are summarised in \Cref{tab:sce2_1}, which shows that all the estimators achieve better localisation performances as $\kappa$ increases. \Cref{tab:sce2_2} summarises the inference performances. It shows that the best performance is achieved when $\kappa$ is between $2$ and $3$. When $\kappa$ is too small -- it leads to a bad localisation performance, or too big -- it may fall in the non-vanishing jump size regime, the inference performance may deteriorate.
\begin{table}[htbp]
\centering
\caption{Localisation results of Scenario 2. }
\begin{tabular}{ccccc}
\hline
\multicolumn{5}{c}{$K = 1$, $n = 200$ $p = 300$} \\ 
$\kappa$   & $\widehat{K} < K$ & $\widehat{K} > K$ & $d_{\mathrm{H}}^{\mathrm{pre}}$ & $d_{\mathrm{H}}^{\mathrm{fin}}$ \\ \hline                        
1 & 0.532    & 0        & 0.279 (0.232)  & 0.274 (0.237)  \\
2 & 0.004    & 0        & 0.013 (0.035)  & 0.007 (0.032)   \\
3 & 0        & 0    & 0.007 (0.011)  & 0.002 (0.004)   \\
4 & 0        & 0    & 0.005 (0.008)  & 0.001 (0.003)   \\
\hline
\end{tabular}
\label{tab:sce2_1}
\end{table}

\begin{table}[htbp]
\centering
\caption{Inference results of Scenario 2. }
\begin{tabular}{ccccc}
\hline
\multicolumn{5}{c}{$K = 1$, $n = 200$, $p = 300$} \\
 & \multicolumn{2}{c}{$\alpha = 0.01$} & \multicolumn{2}{c}{$\alpha = 0.05$} \\
$\kappa$   & $\mathrm{cover}(1-\alpha)$ & $\mathrm{width}(1-\alpha)$ & $\mathrm{cover}(1-\alpha)$ & $\mathrm{width}(1-\alpha)$ \\ \hline
1 & 0.953            & 37.427 (16.649)  & 0.902            & 27.487 (11.563) \\
2 & 0.984            & 10.490 (5.637)  & 0.946    & 7.616 (4.165)  \\
3 & 0.978            & 4.546 (1.367)  & 0.960    & 3.422 (1.183)  \\
4 & 0.976            & 2.938 (1.064)  & 0.966    & 2.188 (0.594)  \\ \hline
\end{tabular}
\label{tab:sce2_2}
\end{table}

\medskip  
\noindent \textbf{Scenario 3: unequally-spaced multiple change points.} Let $K = 2$ and the unequally-spaced change points $\{\eta_1, \eta_2\} = \{n/4, 5n/8\}$. We fix $\kappa_1 = \kappa_2 = 2$, varying $n \in \{200, 400, 800\}$ and $p \in \{100, 200, 300\}$. \Cref{tab:sce3_1} shows that the localisation performances of both preliminary and final estimators improve as $n$ increases or $p$ decreases. For change point inference, the comparison between Tables \ref{tab:sce3_2} and \ref{tab:sce3_3} suggests that our inference procedure applied on $\eta_2$, the one with a larger sample size, performs better.

\begin{table}[htbp]
\centering
\caption{Localisation results of Scenario 3. }
\begin{tabular}{ccccc}
\hline
\multicolumn{5}{c}{$K = 1$, $\kappa_1 = \kappa_2 = 2$} \\ 
$n$   & $\widehat{K} < K$ & $\widehat{K} > K$ & $d_{\mathrm{H}}^{\mathrm{pre}}$ & $d_{\mathrm{H}}^{\mathrm{fin}}$ \\ \hline
\multicolumn{5}{c}{$p = 100$}                                                                             \\
200 & 0.036    & 0        & 0.034 (0.050)  & 0.026 (0.070)  \\
400 & 0        & 0    & 0.005 (0.005)  & 0.003 (0.003)   \\
800 & 0        & 0    & 0.002 (0.002)  & 0.001 (0.001)   \\
\multicolumn{5}{c}{$p = 200$}                                                                             \\
200 & 0.208    & 0.004        & 0.084 (0.110)  & 0.090 (0.142)   \\
400 & 0        & 0    & 0.007 (0.009)  & 0.002 (0.003)   \\
800 & 0        & 0    & 0.002 (0.002)  & 0.001 (0.001)   \\
\multicolumn{5}{c}{$p = 300$}                                                                             \\
200 & 0.392    & 0.004        & 0.136 (0.133)   & 0.158 (0.170)   \\
400 & 0      & 0      & 0.007 (0.009)  & 0.002 (0.003)   \\
800 & 0      & 0      & 0.002 (0.003)  & 0.001 (0.001)   \\
\hline
\end{tabular}
\label{tab:sce3_1}
\end{table}

\begin{table}[htbp]
\centering
\caption{Inference results for the first change point ($\eta_1$) in Scenario 3. }
\begin{tabular}{ccccc}
\hline
\multicolumn{5}{c}{$K = 1$, $\kappa_1 = \kappa_2 = 2$} \\
 & \multicolumn{2}{c}{$\alpha = 0.01$} & \multicolumn{2}{c}{$\alpha = 0.05$} \\
$n$   & $\mathrm{cover}_1(1-\alpha)$ & $\mathrm{width}_1(1-\alpha)$ & $\mathrm{cover}_1(1-\alpha)$ & $\mathrm{width}_1(1-\alpha)$ \\ \hline
\multicolumn{5}{c}{$p = 100$}                                                                             \\
200 & 0.927            & 12.110 (7.063)  & 0.876            & 8.654 (4.886) \\
400 & 0.984            & 6.844 (1.641)  & 0.960    & 4.954 (1.258)  \\
800 & 0.974            & 5.290 (0.923)  & 0.962    & 4.026 (0.312) \\
\multicolumn{5}{c}{$p = 200$}                                                                             \\
200 & 0.914            & 19.043 (13.639) & 0.865 & 13.668 (9.382) \\
400 & 0.978            & 7.278 (1.798)  & 0.964  & 5.286 (1.356)  \\
800 & 0.984            & 5.348 (0.908)  & 0.972  & 4.062 (0.377)  \\
\multicolumn{5}{c}{$p = 300$}                                                                             \\
200 & 0.924            & 27.298 (22.191)   & 0.884    & 19.308 (15.590) \\
400 & 0.994            & 7.792 (2.090)  & 0.982    & 5.690 (1.531) \\
800 & 0.972            & 5.352 (0.983)  & 0.966   & 4.044 (0.388) \\ \hline
\end{tabular}
\label{tab:sce3_2}
\end{table}

\begin{table}[htbp]
\centering
\caption{Inference results for the second change point ($\eta_2$) in Scenario 3. }
\begin{tabular}{ccccc}
\hline
\multicolumn{5}{c}{$K = 1$, $\kappa_1 = \kappa_2 = 2$} \\
 & \multicolumn{2}{c}{$\alpha = 0.01$} & \multicolumn{2}{c}{$\alpha = 0.05$} \\
$n$   & $\mathrm{cover}_2(1-\alpha)$ & $\mathrm{width}_2(1-\alpha)$ & $\mathrm{cover}_2(1-\alpha)$ & $\mathrm{width}_2(1-\alpha)$ \\ \hline
\multicolumn{5}{c}{$p = 100$}                                                                             \\
200 & 0.967            & 10.697 (5.545)  & 0.948            & 7.701 (3.830) \\
400 & 0.972            & 6.112 (1.181)  & 0.950    & 4.400 (0.816)  \\
800 & 0.980            & 4.926 (0.931)  & 0.976    & 4.034 (0.247) \\
\multicolumn{5}{c}{$p = 200$}                                                                             \\
200 & 0.982            & 17.244 (14.101) &    0.954 & 12.345 (9.833) \\
400 & 0.980            & 6.254 (1.202)  & 0.966  & 4.572 (0.922)  \\
800 & 0.986            & 4.994 (0.916)  & 0.976  & 3.992 (0.237)  \\
\multicolumn{5}{c}{$p = 300$}                                                                             \\
200 & 0.954            & 24.891 (17.430)   & 0.917    & 17.709 (12.593) \\
400 & 0.984            & 6.640 (1.295)  & 0.976    & 4.818 (1.011) \\
800 & 0.982            & 5.066 (0.914)  & 0.972   & 4.006 (0.224) \\ \hline
\end{tabular}
\label{tab:sce3_3}
\end{table}

\subsection{Real data application}\label{sec:realdata}
We apply the proposed procedure to the Federal Reserve Economic Database~(FRED)\footnote{The data set is publicly available at \url{https://research.stlouisfed.org/econ/mccracken/fred-databases}.} \citep{mccracken2016fred}. \cite{wang2022optimal} recently conducted a hypothesis testing on the existence of change points on this dataset. Their results suggest that there exist change points in the relationship between the monthly growth rate of the US industrial production (IP) index -- an important indicator of macroeconomic activity, and other macroeconomic variables. We consider the response variable (monthly growth rate of IP) $y_t = \log(\mathrm{IP}_t / \mathrm{IP}_{t-1})$ and the covariates with $116$ macroeconomic variables recorded for the month $t-1$ (after removing variables containing missing values).

Following the recommendation given on the FRED website, we transform the raw data into stationary time series and remove outliers using the R package \texttt{fbi}~\citep{fbi}. We consider the period from January 2000 to December 2022 with sample size $n = 275$. Tuning parameters selection and unknown quantities estimation follow the same method specified in \Cref{sec:numeric_perform}. We first perform DPDU, which outputs July 2008 and January 2020 as the initial change point estimator. We then refine the initial estimator and obtain June 2008 and January 2020 as the final estimator. The first estimated change point coincides with the beginning of the 2008 financial crisis. The $99\%$ confidence interval is [November 2007, January 2009], which covers major milestone events of the crisis, including the bankruptcy of Lehman Brothers, the acquisition of Merrill Lynch by Bank of America, the bailout and sell of Bear Stearns, among many others. The second estimated change point occurs around the beginning of the COVID-19 pandemic, with a narrow $99\%$ confidence interval of [December 2019, February 2020]. In particular, China issued the lockdown of Wuhan in January 2020, the U.S.\ imposed coronavirus‐related travel restrictions in February, 2020 and declared a nationwide emergency in March 2020.

\Cref{fig:realdata} plots the estimated coefficients of the high-dimensional linear regression on each segment. As can be seen, the estimated coefficients are indeed sparse with only a handful of non-zero entries on each segment. We note that the estimated coefficients vary in both scale and support across the three segments, which suggests the U.S. economy underwent notable changes after the financial crisis and the COVID-19 pandemic.

\begin{figure}[t]
\centering
\makebox[\textwidth][c]{\includegraphics[width=0.8\textwidth]{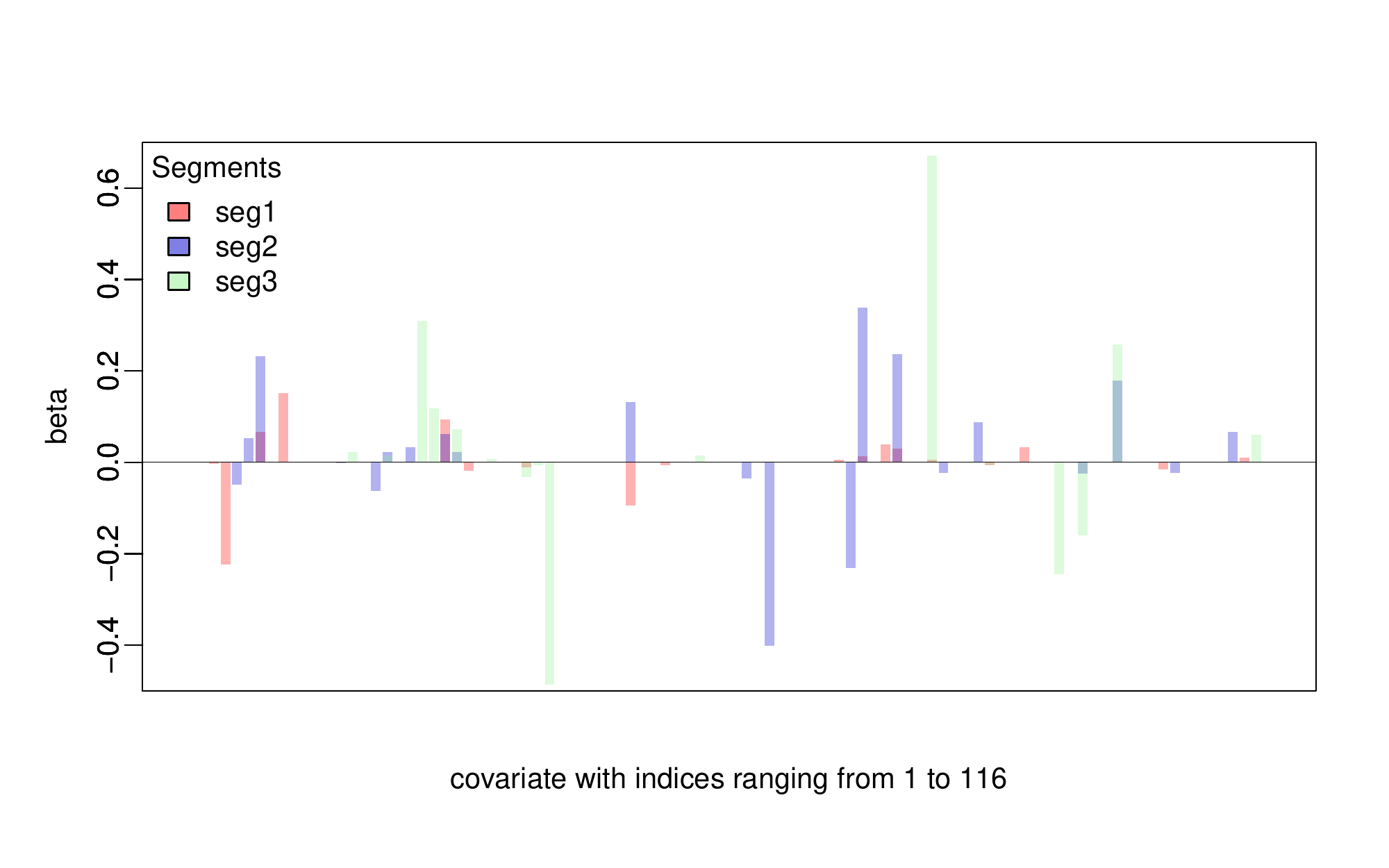}}%
\vspace{-1cm}
\caption{Estimated regression coefficients $\hat\beta$ on segments 1, 2, 3.} 
\label{fig:realdata}
\end{figure}

Note that, we have also implemented three other change point localisation methods, BSA \citep{leonardi2016computationally}, SGL \citep{zhang2015multiple} and VPC \citep{wang2021statistically}, on the same data set. BSA and SGL report no change point. \textcolor{blue}{VPC estimates two change points at June 2008 and December 2019}, which lie in our $99\%$ confidence intervals.

\section{Discussion}

In this paper, we derived limiting distributions of change point estimators, in the context of high-dimensional linear regression model, with temporal dependence in both covariate and noise sequences.  We believe such limiting distributions are first time seen in the high-dimensional linear regression change point literature, including those temporally-independent scenarios.  In addition to the limiting distributions, this paper is also a host of various interesting byproducts.

One potential extension is the implementation of the limiting distribution when the jump size is non-vanishing.  Although from a change point detection point of view, the vanishing jump size scenario is much more challenging, the limiting distribution of change point estimators when the jump size is non-vanishing lacks universality.  This is to say, as detailed in \Cref{thm:asymptotics}\textbf{a}, the limiting distribution is a function of the data generating mechanism.  This is understandable, since in this case the localisation error of order $O(1)$, and such phenomenon is observed and exists in simpler settings \citep[e.g.][]{yao1988estimating, aue2009estimation, mukherjee2022robust}.  It would be challenging and interesting to devise a practical way to distinguish the vanishing and non-vanishing jump size regimes, and further utilise the derived limiting distributions in the non-vanishing regime.

%\zifeng{Do we need to highlight changes in the supplementary material?}

\clearpage
\bibliographystyle{apalike}
\bibliography{citation}

\clearpage
\appendix

\section*{Appendices}
We collect all the technical details in the supplementary material.  \Cref{sec-app-method} collects all the proofs of the results in \Cref{sec:method}.  \Cref{sec-app-main} collects the proofs of the main results, which are in \Cref{sec:main_result}.  \Cref{sec-app-pre} collects the proofs of the intermediate results, which are in \Cref{sec:theory_preliminary_est}, except the proof of  \Cref{thm:bernstein_exp_subExp_nonlinear}.  \Cref{sec-app-lrv} collects the proofs of the results regarding the long-run variance estimator and its application, which are in \Cref{sec:LRV}.  Some additional technical details are collected in \Cref{sec-app-add}.  \Cref{sec-proof-th4} is of its own interest, providing the proof of \Cref{thm:bernstein_exp_subExp_nonlinear}.  \Cref{sec-proof-th4} is a self-contained section and the notation therein only applies within the section.

\section{Proofs for Section 2}\label{sec-app-method}

Given data $ \{y_t, X_t\}_{t=l}^{r-1} \subset \mathbb R \times \mathbb R^p$ with any integer pair $1 \leq l < r \leq n+1$, and a tuning parameter $\lambda > 0$, let 
\[
    \widehat \beta_{[l,r)} = \argmin_{\beta \in \mathbb{R}^p } \left\{\frac{1}{ r-l }\sum_{t =l}^{r-1}  (y_t  -X_t^\top \beta) ^2 +  \frac{\lambda}{\sqrt{r-l} }  |\beta |_1\right\}.
\]
Assume for simplicity that there is no intercept.  Consider the coordinate descent algorithm.  Given $\widetilde  \beta  \in \mathbb R ^p $, the coordinate-wise update of $\widetilde \beta $ has the form 
\[
    \widetilde \beta _j  \leftarrow   S \bigg(\frac{1}{r-l}\sum_{t=l}^{r-1} X_{tj} (y_t  -\widetilde y _t^{ (j) } ) , \, \lambda \bigg), \quad j \in \{1, \ldots, p\},
\]
where  $\widetilde y_t^{(j) } = \sum_{k\not =j  } X_{tk} \widetilde \beta_k$ and   $S(z,\lambda) $ is the soft-thresholding operator $S(z, \lambda) = \mathrm{sign} (z) ( |z| - \lambda)_+$.

\begin{lemma} \label{lemma:lasso one iteration}
Given the quantities $\sum_{t=l}^{r-1} y_t X_{t} \in \mathbb R^p$ and 
$\sum_{t=l}^{r-1} X_{t} X_{t}^\top   \in \mathbb R^{p\times p}$.  Each complete iteration of the coordinate descent has computational cost of order $O(p^2)$.
\end{lemma}

\begin{proof}
Let $\widetilde \beta  $ be given. Then for each fixed $j\in \{ 1,\ldots, p \}$,  note that 
$$   \sum_{t=l}^{r-1} X_{tj} (y_t  -\widetilde y _t^{ (j) } )  = \sum_{ t=l}^{r-1} X_{tj} y_t  -  \sum_{k\not = j }  \bigg( \sum_{t=l}^{r-1} X_{tj}X_{tk}\bigg) \widetilde \beta_{k} .$$
Therefore, with  $ \sum_{t=l}^{r-1} y_t X_{t}   \in \mathbb R^p$ and 
$ \sum_{t=l}^{r-1} X_{t} X_{t}^\top   \in \mathbb R^{p\times p} $ stored in the memory,
it  takes $O(p)$ operations to compute $    \sum_{t=l}^{r-1} X_{tj} (y_t  -\widetilde y _t^{ (j) } )$ and  update $\widetilde \beta_j  $. So it takes $O(p^2) $ to update each of  the $p$ coordinates once, which is   a  complete iteration of the coordinate descent.
\end{proof}

\begin{proof}[Proof of \Cref{lemma:memory_computation_DPDU}] 
For given  $r\in\{2, \ldots, n+1 \}$ and $l \in\{r-1,\ldots,1\}$,  \eqref{eq:updated matrices} indicates that $\mathcal M_{\text{temporary}} =\sum_{t = l}^{r-1} X_tX_t^\top$ and $\mathcal V_{\text{temporary}} =\sum_{t = l}^{r-1} y_tX_t ^\top$.  The claim on the memory cost follows from the observation that only $  \mathcal M_{\text{temporary}}  \in \mathbb R^{p\times p }  $,   $\mathcal V_{\text{temporary}} \in \mathbb R^{p  } $, $ \widehat B \in \mathbb R^n$  and $B \in \overline {R}^{n+1}$ are temporally stored in the memories.

As for the computational cost, from \Cref{algorithm:DPDU},  the major computational cost comes from the three sources: updating  $\mathcal{M}_{\text{temporary}} $, $\mathcal V _{\text{temporary}}$ and  $\mathcal G(\I) $. 

\noindent \textbf{Step 1.} For any fixed $r\in\{2, \ldots, n+1 \}$ and $l \in\{r-1, \ldots,1\}$,  in \eqref{eq:updated matrices}, updating $\mathcal{M}_{\text{temporary}}$   requires $O(p^2)$ operations, and
updating $\mathcal{V}_{\text{temporary}}$   requires $O(p )$ operations. Since there are at most $n^2$ many difference choices of $l$ and $r$,   it therefore requires   $O(n^2p^2)$ operations 
to update $\mathcal{M}_{\text{temporary}}$ and  $\mathcal{V}_{\text{temporary}}$ throughout the {\it entire} DUDP procedure.
 %computing $ \{ \mathcal M (r,1) , \ldots, \mathcal M(r,r-1) \} $  based on  $ \{ \mathcal M (r-1,1) , \ldots, \mathcal M (r-1,r-2) \}$  requires at most $O(np^2)$ operations.  Computing $ \{ \mathcal V (r,1) , \ldots, \mathcal V(r,r-1) \} $  based on  $ \{ \mathcal V (r-1,1) , \ldots, \mathcal V  (r-1,r-2) \}$  requires at most $O(np )$ operations.   to compute $\{ \mathcal M (r,l)\}_{2\le r \le n+1, 1\le l\le r-1}$ and $\{ \mathcal V (r,l)\}_{2\le r \le n+1, 1\le l\le r-1}$.

\noindent \textbf{Step 2.} For any $\I =[l, r)$, given $\mathcal M_{\text{temporary}} =\sum_{t = l}^{r-1} X_tX_t^\top$ and $\mathcal V_{\text{temporary}} =\sum_{t = l}^{r-1} y_tX_t ^\top$, it requires ${\bf Lasso}(p)$ to compute $ \widehat \beta _\I $. Given $ \widehat \beta _\I $, it requires $O(p^2)$ to compute $ \mathcal G(\I)$. There are at most $n^2$ many intervals in $(0, n]$. 
So the cost of computing $ \{ \mathcal G(\I)\} _{\I \subset  (0, n] }$ via  DPDU is $O(n^2  p^2 +  n^2 {\bf Lasso}(p))$.
\end{proof}

\section{Proofs for Section 3.2} \label{sec-app-main}
\subsection[]{Proof of \Cref{lemma:long-run_UB}}

\begin{proof}[Proof of \Cref{lemma:long-run_UB}]
For any fixed $n \in \mathbb{Z}_+$, we consider the following quantity
\begin{align*}
    \var\bigg(n^{-1/2}\sum_{t = 1}^n2\epsilon_t v_k^{\top}X_t \bigg).
\end{align*}
Note that, by \Cref{assume:regression parameters}, $\mathbb E[\epsilon_t v_k^{\top}X_t] = \mathbb E\big[v_k^{\top}X_t\mathbb E[\epsilon_t|X_t ]\big] = 0$ for any $t \in \{1, \dots, n\}$. 
\\
\\
{\bf Upper bound.}
We first consider the upper bounds using the martingale decomposition and Burkholder's inequality (\Cref{lemma:Burkholder}).  
Define the projection operator $\mathcal{P}_j\cdot = \mathbb{E}[\cdot|\mathcal{F}_j] - \mathbb{E}[\cdot| \mathcal{F}_{j-1}]$ with $j \in \mathbb{Z}$.
\
\\
For any $t$, we decompose the random variable $\epsilon_t v_k^{\top}X_t$ as
    \begin{align*}
        \epsilon_tv_k^{\top}X_t = \epsilon_tv_k^{\top}X_t - \mathbb{E}[\epsilon_tv_k^{\top}X_t] =  \sum_{i = 0}^{\infty}\mathcal{P}_{t-i}(\epsilon_tv_k^{\top}X_t),
    \end{align*}
    where $\{\mathcal{P}_{t-i}(\epsilon_tv_k^{\top}X_t)\}_{t \in \mathbb{Z}}$ is a martingale difference sequence.
    It holds for any fixed $n \in \mathbb{Z}_+$ that 
    \begin{align}\label{eq:lrv_ub_part1}
        &\sqrt{\var\bigg(n^{-1/2}\sum_{t = 1}^n \epsilon_tv_k^{\top}X_t\bigg)} = \bigg\Vert n^{-1/2}\sum_{t=1}^n\big\{\epsilon_tv_k^{\top}X_t - \mathbb{E}[\epsilon_tv_k^{\top}X_t]\big\} \bigg\Vert_2 \nonumber \\
        =& \bigg\Vert n^{-1/2}\sum_{t=1}^n\sum_{i = 0}^{\infty}\mathcal{P}_{t-i}(\epsilon_tv_k^{\top}X_t) \bigg\Vert_2 \leq n^{-1/2}\sum_{i = 0}^{\infty}\bigg\Vert \sum_{t=1}^n\mathcal{P}_{t-i}(\epsilon_tv_k^{\top}X_t) \bigg\Vert_2 \nonumber \\
        \leq& n^{-1/2}\sum_{i = 0}^{\infty}\bigg( \sum_{t=1}^n\big\Vert \mathcal{P}_{t-i}(\epsilon_tv_k^{\top}X_t) \big\Vert_2^2\bigg)^{1/2} \leq \sum_{i = 0}^{\infty}\big\Vert \epsilon_iv_k^{\top}X_i - \epsilon_{i,\{0\}}v_k^{\top}X_{i,\{0\}} \big\Vert_{2} \nonumber \\
        =& \sum_{i = 0}^{\infty}\big\Vert \epsilon_iv_k^{\top}X_i - \epsilon_{i,\{0\}}v_k^{\top}X_i + \epsilon_{i,\{0\}}v_k^{\top}X_{i} -\epsilon_{i,\{0\}}v_k^{\top}X_{i,\{0\}} \big\Vert_{2} \nonumber\\
        \leq& \sum_{i = 0}^{\infty}\big\Vert (\epsilon_i-\epsilon_{i,\{0\}})v_k^{\top}X_i\big\Vert_2 + \sum_{i = 0}^{\infty}\big\Vert \epsilon_{i,\{0\}}v_k^{\top}(X_i - X_{i,\{0\}}) \big\Vert_{2} \nonumber\\
        \leq& \big\Vert v_k^{\top} X_0 \big\Vert_4\sum_{i = 0}^{\infty}\delta_{i,4}^{\epsilon} + \Vert \epsilon_0 \Vert_4\sum_{i = 0}^{\infty}\big\Vert v_k^{\top} (X_i - X_{i,\{0\}}) \big\Vert_{4} \nonumber\\
        \leq& D_{\epsilon}\sup_{|v|_2 = 1}\big\Vert v^{\top} X_0 \big\Vert_4 + D_{X}\Vert \epsilon_0 \Vert_4,
    \end{align}
    where the first inequality follows from Fubini's theorem and the triangle inequality, the second inequality follows from Burkholder's inequality, the fourth inequality follows from the triangle inequality, the fifth inequality follows from H\"older's inequality, and the sixth inequality follows from Assumptions \ref{assume:X} and \ref{assume:epsilon}. The third inequality follows from the stationarity and the fact that for any $i \geq 0$ and any random variable $Z_t$ of the form \eqref{eq:nonstationary} that
    \begin{equation*}
    \begin{aligned}
        &\Vert \mathcal{P}_{t-i}Z_t \Vert_2 = \Vert \mathbb{E}[Z_t | \mathcal{F}_{t-i}] -\mathbb{E}[Z_t | \mathcal{F}_{t-i-1}] \Vert_2\\
        =& \Vert\mathbb{E}[Z_t| \mathcal{F}_{t-i}] - \mathbb{E}[Z_{t,\{t-i\}}| \mathcal{F}_{t-i-1}]\Vert_2\\
        =& \Vert\mathbb{E}[Z_t-Z_{t,\{t-i\}}| \mathcal{F}_{t-i}]\Vert_2 \leq \Vert Z_{t}-Z_{t,\{t-i\}}\Vert_2,
    \end{aligned}
    \end{equation*}
    where the second and the third equalities follow from the definition of the coupled random variables $Z_{t,\{t-i\}}$ in \Cref{sec:background}, and the first inequality follows from Jensen's inequality.
\
\\
Letting $n \to \infty$, in the vanishing regime, it satisfies that
    \begin{align*}
       \sigma_{\infty}^2 
       \leq 4\bigg(D_{\epsilon}\sup_{|v|_2 = 1}\big\Vert v^{\top} X_0 \big\Vert_4 + D_{X}\Vert \epsilon_0 \Vert_4 \bigg)^2 < \infty.
    \end{align*}
\
\\
{\bf Convergence.} Next, we show that $\var\big(n^{-1/2}\sum_{t = 1}^n 2\epsilon_t v_k^{\top}X_t \big)$ has a limit in the vanishing regime. It is suffice to show that $\var\big(n^{-1/2}\sum_{t = 1}^n 2\epsilon_t v_k^{\top}X_t\big)$ is a Cauchy sequence.
Note that $\{\epsilon_t v_k^{\top}X_t\}_{t \in \mathbb{Z}}$ is a stationary process. Denote the lag-$\ell$ covariance as 
\begin{align*}
    \gamma(\ell) = \cov\big(\epsilon_{t-\ell} v_k^{\top}X_{t-\ell}, \epsilon_t v_k^{\top}X_{t}\big) = \mathbb E\big[\epsilon_{t-\ell} v_k^{\top}X_{t-\ell} \epsilon_t v_k^{\top}X_{t}\big].
\end{align*}
We have that for $\ell \geq 0$
\begin{align*}
    &|\gamma(\ell)| = \bigg|\mathbb{E}\Big[\Big(\sum_{k = 0}^{\infty}\mathcal{P}_{-k}\{\epsilon_tv_k^{\top}X_{t}\}\Big)\Big(\sum_{k = 0}^{\infty}\mathcal{P}_{\ell-k}\{\epsilon_{t+l}v_k^{\top}X_{t+\ell}\}\Big)\Big]\bigg|\\
    =& \bigg|\sum_{k = 0}^{\infty}\mathbb{E}\big[(\mathcal{P}_{-k}\{\epsilon_tv_k^{\top}X_{t}\})(\mathcal{P}_{-k}\{\epsilon_{t+\ell}v_k^{\top}X_{t+\ell}\})\big]\bigg|\\
    \leq& \sum_{k =     0}^{\infty}\Big|\mathbb{E}\big[(\mathcal{P}_{-k}\{\epsilon_tv_k^{\top}X_{t}\})(\mathcal{P}_{-k}\{\epsilon_{t+\ell}v_k^{\top}X_{t+\ell}\})\big]\Big| \leq \sum_{k = 0}^{\infty}\Vert \mathcal{P}_{-k}\{\epsilon_tv_k^{\top}X_{t}\} \Vert_2 \Vert \mathcal{P}_{-k}\{\epsilon_{t+\ell}v_k^{\top}X_{t+\ell}\}\Vert_2\\
    \leq& \Big(\sup_{|v|_2 = 1}\big\Vert v^{\top} X_0 \big\Vert_4\Big)^2\sum_{k = 0}^{\infty}\delta^{\epsilon}_{k,4}\delta^{\epsilon}_{k+\ell,4} + \Vert \epsilon_0 \Vert_4\sup_{|v|_2 = 1}\big\Vert v^{\top} X_0 \big\Vert_4\sum_{k = 0}^{\infty}\delta^{X}_{k,4}\delta^{\epsilon}_{k+\ell,4}\\
    &+\Vert \epsilon_0 \Vert_4\sup_{|v|_2 = 1}\big\Vert v^{\top} X_0 \big\Vert_4\sum_{k = 0}^{\infty}\delta^{\epsilon}_{k,4}\delta^{X}_{k+\ell,4} + \Vert \epsilon_0 \Vert_4^2\sum_{k = 0}^{\infty}\delta^{X}_{k,4}\delta^{X}_{k+\ell,4}\\
    \leq& \Big(\sup_{|v|_2 = 1}\big\Vert v^{\top} X_0 \big\Vert_4\Big)^2\Delta^{\epsilon}_{0,4}\Delta^{\epsilon}_{\ell,4} + \Vert \epsilon_0 \Vert_4\sup_{|v|_2 = 1}\big\Vert v^{\top} X_0 \big\Vert_4\big(\Delta^{X}_{0,4}\Delta^{\epsilon}_{\ell,4} + \Delta^{\epsilon}_{0,4}\Delta^{X}_{\ell,4}\big) + \Vert \epsilon_0 \Vert_4^2\Delta^{X}_{0,4}\Delta^{X}_{\ell,4},
\end{align*}
    where the first inequality follows from the triangle inequality and the second and fourth inequalities follow from H\"older's inequality. The second equality follows from the orthogonality of $\mathcal{P}_j\cdot$, i.e.~for $i < j$
    \begin{align*}
        \mathbb{E}[(\mathcal{P}_i\epsilon_r)(\mathcal{P}_j\epsilon_s)] = \mathbb{E}[\mathbb{E}[(\mathcal{P}_i\epsilon_r)(\mathcal{P}_j\epsilon_s)|\mathcal{F}_{i}]] = \mathbb{E}[(\mathcal{P}_i\epsilon_r)\mathbb{E}[\epsilon_s-\epsilon_s|\mathcal{F}_i]] = 0,
    \end{align*}
    and the orthogonality also holds for $i > j$ by symmetry. The third inequality follows from
    \begin{align*}
       &\Vert \mathcal{P}_{j}\{\epsilon_iv_k^{\top}X_{i}\} \Vert_2 = \Vert \mathbb{E}[\epsilon_iv_k^{\top}X_{i}| \mathcal{F}_{j}] - \mathbb{E}[\epsilon_iv_k^{\top}X_{i}| \mathcal{F}_{j-1}]\Vert_2\\
       =& \Vert\mathbb{E}[\epsilon_iv_k^{\top}X_{i}| \mathcal{F}_{j}] - \mathbb{E}[\epsilon_{i,\{j\}}v_k^{\top}X_{i,\{j\}}| \mathcal{F}_{j-1}]\Vert_2 = \Vert\mathbb{E}[\epsilon_iv_k^{\top}X_{i}-\epsilon_{i,\{j\}}v_k^{\top}X_{i,\{j\}}| \mathcal{F}_{j}]\Vert_2\\
       \leq& \Vert \epsilon_iv_k^{\top}X_{i}-\epsilon_{i,\{j\}}v_k^{\top}X_{i,\{j\}}\Vert_2 \leq \Vert (\epsilon_i - \epsilon_{i,\{j\}})v_k^{\top}X_{i}\Vert_2 + \Vert \epsilon_{i,\{j\}}v_k^{\top}(X_{i} - X_{i,\{j\}})\Vert_2\\
       \leq& \sup_{|v|_2 = 1}\big\Vert v^{\top} X_0 \big\Vert_4\delta^{\epsilon}_{i-j,4} + \Vert \epsilon_0 \Vert_4\delta^{X}_{i-j,4},
    \end{align*}
    where the second and the third equality follows the definition of the coupled random variables, the first inequality follows Jensen's inequality, and the second inequality follows the definition of the functional dependence measure. By the same arguments, we have for any $\ell < 0$ that
    \begin{align*}
    &|\gamma(\ell)|\\
    \leq& \Big(\sup_{|v|_2 = 1}\big\Vert v^{\top} X_0 \big\Vert_4\Big)^2\Delta^{\epsilon}_{0,4}\Delta^{\epsilon}_{-\ell,4} + \Vert \epsilon_0 \Vert_4\sup_{|v|_2 = 1}\big\Vert v^{\top} X_0 \big\Vert_4\big(\Delta^{X}_{0,4}\Delta^{\epsilon}_{-\ell,4} + \Delta^{\epsilon}_{0,4}\Delta^{X}_{-\ell,4}\big) + \Vert \epsilon_0 \Vert_4^2\Delta^{X}_{0,4}\Delta^{X}_{-\ell,4}.
    \end{align*}
Therefore, it follows from Assumptions \ref{assume:X} and \ref{assume:epsilon} that
\begin{align}\label{eq:autocov_abs_ub}
    \big|\gamma(\ell)\big| \leq \Big(\sup_{|v|_2 = 1}\big\Vert v^{\top} X_0 \big\Vert_4D_{\epsilon} + \Vert \epsilon_0 \Vert_4D_{X}\Big)^2\exp(-2c|\ell|^{\gamma_1}).
\end{align}
\
\\
By some calculation, we have that
\begin{align*}
    \var\bigg(n^{-1/2}\sum_{t = 1}^n 2\epsilon_t v_k^{\top}X_t\bigg) = 4\sum_{-n < \ell < n}\Big(1- \frac{|\ell|}{n}\Big)\gamma(\ell).
\end{align*}
It follows from \eqref{eq:autocov_abs_ub} that
\begin{align*}
    &\bigg|\var\bigg(n^{-1/2}\sum_{t = 1}^n 2\epsilon_t v_k^{\top}X_t\bigg) - \var\bigg(m^{-1/2}\sum_{t = 1}^m 2\epsilon_t v_k^{\top}X_t\bigg)\bigg|\\
    \leq& \frac{8}{\min\{n,m\}}\sum_{\ell = 0}^{\min\{n,m\}-1}\ell|\gamma(\ell)| + 4\sum_{\ell = \min\{n,m\}}^{\max\{n,m\}}|\gamma(\ell)|\\
    \leq& \frac{8\Big(\sup_{|v|_2 = 1}\big\Vert v^{\top} X_0 \big\Vert_4D_{\epsilon} + \Vert \epsilon_0 \Vert_4D_{X}\Big)^2}{\min\{n,m\}}\sum_{\ell = 0}^{\min\{n,m\}-1}\ell\exp(-2c|\ell|^{\gamma_1})\\
    &+ 4\Big(\sup_{|v|_2 = 1}\big\Vert v^{\top} X_0 \big\Vert_4D_{\epsilon} + \Vert \epsilon_0 \Vert_4D_{X}\Big)^2\sum_{\ell = \min\{n,m\}}^{\max\{n,m\}}\exp(-2c|\ell|^{\gamma_1}),
\end{align*}
which shows that $\var\big(n^{-1/2}\sum_{t = 1}^n 2\epsilon_t v_k^{\top}X_t\big)$ is Cauchy sequence since the sum converges. It further leads to that $\var\big(n^{-1/2}\sum_{t = 1}^n2 \epsilon_t v_k^{\top}X_t\big)$ has a limit.
Combining the upper bounds and the existence of a limit completes the proof.
\end{proof}

\begin{lemma}\label{lemma:long-run_equiv}
Suppose the same assumptions in \Cref{lemma:long-run_UB} hold. Under the vanishing regime, we have the equivalence that
\begin{align}\label{eq:long-run_var2}
    \sigma_{\infty}^2(k)
    =& 4\lim_{n \to \infty}\var\bigg(n^{-1/2}\sum_{t = 1}^n \epsilon_tv_k^{\top}X_t  \bigg) \nonumber\\
    =& \lim_{n \to \infty}\var\bigg(n^{-1/2}\sum_{t = 1}^n \kappa_k^{-1}\big\{a\epsilon_t\Psi_k^{\top}X_t + (\Psi_k^{\top}X_t)^2\big\} \bigg)
\end{align}
for any $a \in \{-2, 2\}$.
\end{lemma}
\begin{proof}
For any fixed $n \in \mathbb{Z}_+$, we have the following decomposition that
\begin{align*}
    &\var\bigg(n^{-1/2}\sum_{t = 1}^n\kappa_k^{-1}\big\{a\epsilon_t \Psi_k^{\top}X_t + (\Psi_k^{\top}X_t)^2\big\} \bigg)\\
    =& 4\var\bigg(n^{-1/2}\sum_{t = 1}^n \epsilon_t v_k^{\top}X_t\bigg) + \kappa_k^{-2}\var\bigg(n^{-1/2}\sum_{t = 1}^n (\Psi_k^{\top}X_t)^2 \bigg),
\end{align*}
which follows from \Cref{assume:regression parameters}, i.e. $\mathbb E[\epsilon_t|X_t] = 0$ for any $t \in \{1, \dots, n\}$.
\
\\
Note that the upper bound and the convergence of $4\var\bigg(n^{-1/2}\sum_{t = 1}^n \epsilon_t v_k^{\top}X_t\bigg)$ have been studied in the proof of \Cref{lemma:long-run_UB}. It is suffice to show that
\begin{align}\label{eq:lrv_ub_part2}
    \lim_{n \to \infty}\kappa_k^{-2}\var\bigg(n^{-1/2}\sum_{t = 1}^n (\Psi_k^{\top}X_t)^2 \bigg) = 0.
\end{align}
Consider the process $\{\kappa_k^{-1}(\Psi_k^{\top}X_t)^2\}_{t \in \mathbb{Z}}$. Using the similar arguments in \eqref{eq:lrv_ub_part1} of the proof of \Cref{lemma:long-run_UB}, it holds for any fixed $n \in \mathbb{Z}_+$ that 
    \begin{align*}
        \sqrt{\var\bigg(n^{-1/2}\sum_{t = 1}^n (\Psi_k^{\top}X_t)^2 \bigg)} \leq& 2\sum_{i = 0}^{\infty}\big\Vert \Psi_k^{\top}(X_i + X_{i,\{0\}})\big\Vert_4\big\Vert \Psi_k^{\top}(X_i - X_{i,\{0\}})\big\Vert_{4} \nonumber\\
        \leq& 4\kappa_k^2 \sup_{|v|_2 = 1}\big\Vert v^{\top} X_0\big \Vert_4D_X,
    \end{align*}
which further leads to \eqref{eq:lrv_ub_part2}, since $\kappa_k \to 0$.
\end{proof}

\subsection[]{Proof of \Cref{thm:asymptotics}}
Define the event
\begin{align}\label{eq:event_DP}
    \mathcal{E}_{\mathrm{DP}} = \Big\{ \widehat K = K \;\; \text{and} \;\; \max_{ 1\le k \le K } \kappa_k^2|\eta_k-\widehat \eta_k|  \le  C_{\mathfrak e}   \big(\s \log(np) + \zeta      \big) \Big\}.
\end{align}
By \Cref{theorem:DUDP}, under Assumptions~\ref{assume:regression parameters}, \ref{assume:X} and \ref{assume:epsilon}, we have that $\mathcal{E}_{\mathrm{DP}}$ holds with probability at least $1 - cn^{-3}$.

\begin{proof}[Proof of \Cref{thm:asymptotics}]

\
\\
\textbf{Preliminary.}
Since $\mathcal{E}_{\mathrm{DP}}$ holds with probability tending to $1$, we condition on the $\mathcal{E}_{\mathrm{DP}}$ in the following proof. For $k \in \{0, 1, \dots, K\}$, let $\mathcal{I}_k = [l_k, r_k)$, where $l_k = \widehat{\eta}_{k}$, $r_k = \widehat{\eta}_{k+1}$, $\widehat{\eta}_{0} = 1$ and $\widehat{\eta}_{K+1} = n+1$. We have that
\begin{align*}
    |l_k - \eta_{k}| \leq C_{\mathfrak{e}}\frac{\s \log(p) + \zeta}{\kappa_{k}^2} \;\; \text{and} \;\; |r_k - \eta_{k+1}| \leq C_{\mathfrak{e}}\frac{\s \log(p) + \zeta }{\kappa_{k+1}^2}.
\end{align*}
It follows from \Cref{assump-snr} that
\begin{align*}
    |\mathcal{I}_k| = r_k-l_k \geq \eta_{k+1} - \eta_{k} - C_{\mathfrak{e}}\frac{\s \log(np) + \zeta }{\kappa_{k}^2} - C_{\mathfrak{e}}\frac{\s \log(np) + \zeta }{\kappa_{k+1}^2} \geq \frac{\Delta}{2} \geq \zeta,
\end{align*}
and
\begin{align*}
    |\mathcal{I}_k| \leq \eta_{k+1} - \eta_{k} + C_{\mathfrak{e}}\frac{\s \log(np) + \zeta }{\kappa_{k}^2} + C_{\mathfrak{e}}\frac{\s \log(np) + \zeta }{\kappa_{k+1}^2} \leq C_{I}(\eta_{k+1} - \eta_{k}),
\end{align*}
where $C_{I} > 1$ is some absolute constant.
Therefore, by \Cref{lemma:interval lasso}, we have that for any $k \in \{0, 1, \dots, K\}$
\begin{align}
&    | \widehat \beta_k -\beta^*_{\I_k}  | _2^2 \le \frac{C_1 \s\log(p)}{ \Delta  },
\label{eq:interval_lasso_k_1} \\
 &     | \widehat \beta_k  -\beta^*_{\I_k}  | _1  \le     C_2  \s\sqrt {  \frac{\log(p )} {\Delta } },
 \label{eq:interval_lasso_k_2}\\
 &     | (\widehat \beta_k  )_{S^c} | _1 \le  3  | (\widehat \beta _k  -\beta^*_{\I_k} )_{S }  | _1,\label{eq:interval_lasso_k_3}
\end{align}
where $\beta^*_{\I_k} = |\I_k|^{-1}\sum_{t \in \I_k}\beta_t^*$.
Moreover, by \Cref{assump-snr}\textbf{a}, we have that
\begin{align}\label{eq:min_spacing_lb}
    \Delta \ge C_{snr}\alpha_n(\s\log( np))^{2/\gamma -1} \max\big\{\kappa_k^{-2}, \kappa_{k-1}^{-1}\kappa_{k}^{-1}, \kappa_{k}^{-1}\kappa_{k+1}^{-1}\big\}.
\end{align}
\\
\\
By \Cref{prop-1}, $[l_{k-1}, r_{k}) = [l_{k-1}, r_{k-1}) \cup [l_{k}, r_{k})$ can contain $1$, $2$ or $3$ change points. One of the cases that $[l_{k-1}, r_{k})$ contains $3$ change points is illustrated in the following figure, and the biases will be analysed. The analyses of the other cases are similar and simpler, and therefore omitted.
\
\\
\begin{center}
    \begin{tikzpicture}[scale=7,decoration=brace]
\draw[-, thick] (-1.1,0) -- (1.1,0);
\foreach \x/\xtext in {-1/$l_{k-1}$,-0.9/$\eta_{k-1}$,0.05/$\eta_{k}$,0.2/$r_{k-1}(l_{k})$,0.87/$\eta_{k+1}$,1/$r_{k}$}
\draw[thick] (\x,0.5pt) -- (\x,-0.5pt) node[font = {\footnotesize}, below] {\xtext};
%\draw[[-), ultra thick, blue] (0,0) -- (0.2,0);
%\draw (-0.25,0) node {$x=1$};
\end{tikzpicture}
\end{center}
In the following, we analyse three types of biases. It follows that
\begin{align*}
    &\beta_{\eta_{k-1}}^* - \beta_{\mathcal{I}_{k-1}}^* = \beta_{\eta_{k-1}}^* - \frac{\eta_{k-1}-l_{k-1}}{|\I_{k-1}|}\beta_{\eta_{k-2}}^* - \frac{\eta_{k} - \eta_{k-1}}{|\I_{k-1}|}\beta_{\eta_{k-1}}^* - \frac{r_{k-1} - \eta_{k}}{|\I_{k-1}|}\beta_{\eta_{k}}^*\\
    \leq& \frac{\eta_{k-1}-l_{k-1}}{|\I_{k-1}|}(\beta_{\eta_{k-1}}^* - \beta_{\eta_{k-2}}^*) + \frac{r_{k-1} - \eta_{k}}{|\I_{k-1}|}(\beta_{\eta_{k-1}}^* - \beta_{\eta_{k}}^*).
\end{align*}
Similarly,
\begin{align*}
    \beta_{\eta_{k}}^* - \beta_{\mathcal{I}_{k}}^* = \beta_{\eta_{k}}^* - \frac{\eta_{k+1}-l_{k}}{|\I_{k}|}\beta_{\eta_{k}}^* - \frac{r_{k} - \eta_{k+1}}{|\I_{k}|}\beta_{\eta_{k+1}}^* \leq \frac{r_{k} - \eta_{k+1}}{|\I_{k}|}(\beta_{\eta_{k}}^* - \beta_{\eta_{k+1}}^*),
\end{align*}
and
\begin{align*}
    &\beta^*_{\mathcal{I}_{k-1}} - \beta^*_{\eta_{k}}\\
    =& \frac{\eta_{k-1}-l_{k-1}}{|\I_{k-1}|}\beta_{\eta_{k-2}}^* + \frac{\eta_{k} - \eta_{k-1}}{|\I_{k-1}|}\beta_{\eta_{k-1}}^* + \frac{r_{k-1} - \eta_{k}}{|\I_{k-1}|}\beta_{\eta_{k}}^* -\beta_{\eta_{k}}^*\\
    =& \frac{\eta_{k-1}-l_{k-1}}{|\I_{k-1}|}\big(\beta_{\eta_{k-2}}^* - \beta_{\eta_{k-1}}^* + \beta_{\eta_{k-1}}^* - \beta_{\eta_{k}}^*\big) + \frac{\eta_{k} - \eta_{k-1}}{|\I_{k-1}|}\big(\beta_{\eta_{k-1}}^* - \beta_{\eta_{k}}^*\big).
\end{align*}
Therefore, we have that
\begin{align}\label{eq:bias_1}
    \big| \beta_{\eta_{k-1}}^* - \beta_{\mathcal{I}_{k-1}}^* \big|_2 \leq 2C_{\mathfrak{e}}\frac{\s \log(p) + \zeta}{\Delta\kappa_{k-1}^2}\kappa_{k-1} + 2C_{\mathfrak{e}}\frac{\s \log(p) + \zeta}{\Delta\kappa_{k}^2}\kappa_{k} \leq C_3\kappa_k\alpha_n^{-1},
\end{align}
\begin{align}\label{eq:bias_2}
    \big| \beta_{\eta_{k}}^* - \beta_{\mathcal{I}_{k}}^* \big|_2 \leq  2C_{\mathfrak{e}}\frac{\s \log(p) + \zeta}{\Delta\kappa_{k+1}^2}\kappa_{k+1} \leq C_4\kappa_{k}\alpha_n^{-1},
\end{align}
and
\begin{align}\label{eq:bias_3}
    \big| \beta^*_{\mathcal{I}_{k-1}} - \beta^*_{\eta_{k}} \big|_2 \leq  2C_{\mathfrak{e}}\frac{\s \log(p) + \zeta}{\Delta\kappa_{k-1}^2}(\kappa_{k-1} + \kappa_{k}) + \kappa_{k} \leq C_5\kappa_{k}.
\end{align}
\\
\\
\textbf{Uniform tightness of $\kappa_k^2|\widetilde{\eta}_k - \eta_k|$.}
Denote $r = \widetilde{\eta}_k - \eta_k$. Without loss of generality, suppose $r \geq 0$. Since $\widetilde{\eta}_k = \eta_k + r$, defined in \eqref{eq:refinement}, is the minimizer of $Q_k(\eta)$, it follows that
$$Q_k(\eta_k+ r) - Q_k(\eta_k) \leq 0.$$
Observe that
\begin{align*}
    &Q_k(\eta_k+ r) - Q_k(\eta_k) = \sum_{t = \eta_k}^{\eta_k + r - 1} \big\{(y_t - X_t^{\top}\widehat{\beta}_{k-1})^2 - (y_t - X_t^{\top}\widehat{\beta}_{k})^2\big\}\\ 
    =&\sum_{t = \eta_k}^{\eta_k + r - 1} \big\{(y_t - X_t^{\top}\widehat{\beta}_{k-1})^2 - (y_t - X_t^{\top}\beta^*_{\mathcal{I}_{k-1}})^2\big\} - \sum_{t = \eta_k}^{\eta_k + r - 1} \big\{(y_t - X_t^{\top}\widehat{\beta}_{k})^2 - (y_t - X_t^{\top}\beta^*_{\mathcal{I}_{k}})^2\big\}\\
    &+ \sum_{t = \eta_k}^{\eta_k + r - 1} \big\{(y_t - X_t^{\top}\beta^*_{\mathcal{I}_{k-1}})^2 - (y_t - X_t^{\top}\beta^*_{\eta_{k-1}})^2\big\} - \sum_{t = \eta_k}^{\eta_k + r - 1} \big\{(y_t - X_t^{\top}\beta^*_{\mathcal{I}_{k}})^2 - (y_t - X_t^{\top}\beta^*_{\eta_{k}})^2\big\}\\
    &+ \sum_{t = \eta_k}^{\eta_k + r - 1} \big\{(y_t - X_t^{\top}\beta^*_{\eta_{k-1}})^2 - (y_t - X_t^{\top}\beta^*_{\eta_{k}})^2\big\}\\
    =& I - II + III - IV + V.
\end{align*}
Therefore, we have that
\begin{align}\label{eq:refine_signal_noise}
    V \leq -I + II - III + IV \leq |I| + |II| + |III| + |IV|.
\end{align}
If $r \leq 1/\kappa_k^2$ or $r \leq 2$, then there is nothing to show. So for the rest of the argument, for contradiction, let us assume that
$$r \geq \frac{1}{\kappa_k^2} \;\; \text{and} \;\; r \geq 3.$$
\
\\
\textbf{Step 1: the order of magnitude of $I$.} We have that
\begin{align*}
    I =& \sum_{t = \eta_k}^{\eta_k + r - 1}(X_t^{\top}\widehat{\beta}_{k-1} - X_t^{\top}\beta^*_{\mathcal{I}_{k-1}})^2 - 2(\widehat{\beta}_{k-1} - \beta^*_{\mathcal{I}_{k-1}})^{\top}\sum_{t = \eta_k}^{\eta_k + r - 1}X_t(y_t - X_t^{\top}\beta_{\mathcal{I}_{k-1}}^*)\\
    =& \sum_{t = \eta_k}^{\eta_k + r - 1}(X_t^{\top}\widehat{\beta}_{k-1} - X_t^{\top}\beta^*_{\mathcal{I}_{k-1}})^2 - 2(\widehat{\beta}_{k-1} - \beta^*_{\mathcal{I}_{k-1}})^{\top}\sum_{t = \eta_k}^{\eta_k + r - 1}X_tX_t^{\top}(\beta_{\eta_{k}}^* - \beta_{\mathcal{I}_{k-1}}^*)\\
    &- 2(\widehat{\beta}_{k-1} - \beta^*_{\mathcal{I}_{k-1}})^{\top}\sum_{t = \eta_k}^{\eta_k + r - 1}X_t\epsilon_t\\
    =&I_1 - 2I_2 - 2I_3. 
\end{align*}
Therefore
\begin{align*}
    &|I_1| = \sum_{t = \eta_k}^{\eta_k + r - 1}(X_t^{\top}\widehat{\beta}_{k-1} - X_t^{\top}\beta^*_{\mathcal{I}_{k-1}})^2\\
    =& (\widehat{\beta}_{k-1} - \beta^*_{\mathcal{I}_{k-1}})^{\top}\bigg\{\sum_{t = \eta_k}^{\eta_k + r - 1}(X_tX_t^{\top} - \Sigma)\bigg\}(\widehat{\beta}_{k-1} - \beta^*_{\mathcal{I}_{k-1}}) + r(\widehat{\beta}_{k-1} - \beta^*_{\mathcal{I}_{k-1}})^{\top}\Sigma(\widehat{\beta}_{k-1} - \beta^*_{\mathcal{I}_{k-1}})\\
    =& O_{p}\bigg(\frac{\s\log(p)}{ \Delta  }\Big( \sqrt {  r \s\log(p) }  +  \{  \s \log(p) \} ^{1/\gamma } + r \Big)\bigg)\\
    =& O_{p}\bigg(\alpha_n^{-1}\Big\{\sqrt{r\kappa_k^{2}}(\s\log(p))^{-2/\gamma +5/2} + (\s\log(p))^{-1/\gamma +2} + r\kappa_k^{2}(\s\log(p))^{-2/\gamma +2}\Big\}\bigg),
\end{align*}
where the third equality follows from \Cref{lemma:RES Version II}, \eqref{eq:interval_lasso_k_3} and \eqref{eq:interval_lasso_k_1}, and the fourth equality follows from \eqref{eq:min_spacing_lb} and $\kappa_k < C_{\kappa}$.
Similarly
\begin{align*}
    &|I_2| = \bigg|(\widehat{\beta}_{k-1} - \beta^*_{\mathcal{I}_{k-1}})^{\top}\sum_{t = \eta_k}^{\eta_k + r-1}X_tX_t^{\top}(\beta_{\eta_{k}}^* - \beta_{\mathcal{I}_{k-1}}^*)\bigg|\\
    \leq& \big|\widehat{\beta}_{k-1} - \beta^*_{\mathcal{I}_{k-1}}\big|_1\bigg|(\beta_{\eta_{k}}^* - \beta_{\mathcal{I}_{k-1}}^*)^{\top}\sum_{t = \eta_k}^{\eta_k + r - 1}(X_tX_t^{\top} - \Sigma)\bigg|_{\infty} + r\Lambda_{\max}(\Sigma)\big|\widehat{\beta}_{k-1} - \beta^*_{\mathcal{I}_{k-1}}\big|_2\big|\beta_{\eta_{k}}^* - \beta_{\mathcal{I}_{k-1}}^*\big|_2\\
    =& O_p\bigg(\big|\beta_{\eta_{k}}^* - \beta_{\mathcal{I}_{k-1}}^*\big|_2\s\sqrt {  \frac{\log(p )} {\Delta } }\Big( \sqrt {  r\log(p) }  +  \log^{1/\gamma }(p)  \Big) + r\Lambda_{\max}(\Sigma)\sqrt{\frac{\s\log(p)}{ \Delta  }}\big|\beta_{\eta_{k}}^* - \beta_{\mathcal{I}_{k-1}}^*\big|_2\bigg)\\
    =& O_p\bigg(\alpha_n^{-1/2}\Big\{\sqrt{r\kappa_k^{2}}(\s\log(p))^{-1/\gamma +3/2} + \s^{-1/\gamma+3/2}\log(p) + r\kappa_k^{2}(\s\log(p))^{-1/\gamma +1} \Big\}\bigg),
\end{align*}
where  the second equality follows from \Cref{lemma:lasso l infty deviations} and the last equality follows from \eqref{eq:bias_1} and $\kappa_k < C_{\kappa}$. Moreover,
\begin{align*}
    &|I_3| = \bigg|(\widehat{\beta}_{k-1} - \beta^*_{\I_{k-1}})^{\top}\sum_{t = \eta_k}^{\eta_k + r - 1}X_t\epsilon_t\bigg| \leq \big| \widehat{\beta}_{k-1} - \beta^*_{\I_{k-1}} \big|_1\bigg| \sum_{t = \eta_k}^{\eta_k + r - 1}X_t\epsilon_t \bigg|_{\infty}\\
    =& O_p\bigg(\s\sqrt {  \frac{\log(p )} {\Delta } }\Big( \sqrt {  r\log(p) }  +  \log^{1/\gamma }(p)  \Big)\bigg)\\
    =& O_p\bigg(\alpha_n^{-1/2}\Big\{\sqrt{r\kappa_k^2}(\s\log(p))^{-1/\gamma +3/2} + \s^{-1/\gamma +3/2}\log(p)\Big\}\bigg),
\end{align*}
where the second equality follows from \Cref{lemma:lasso l infty deviations} and $\kappa_k < C_{\kappa}$.
Therefore,
\begin{align}\label{eq:refine_noise1}
    |I| =& O_p\bigg(\alpha_n^{-1}\Big\{\sqrt{r\kappa_k^{2}}(\s\log(p))^{-2/\gamma +5/2} + (\s\log(p))^{-1/\gamma +2} + r\kappa_k^{2}(\s\log(p))^{-2/\gamma +2}\Big\}\nonumber\\
    &\quad\quad+ \alpha_n^{-1/2}\Big\{\sqrt{r\kappa_k^2}(\s\log(p))^{-1/\gamma +3/2} + \s^{-1/\gamma +3/2}\log(p) + r\kappa_k^{2}(\s\log(p))^{-1/\gamma +1}\Big\}\bigg)\nonumber\\
    =& o_p\Big(\sqrt{r\kappa_k^2}\Big) + o_p(1) + o_p(r\kappa_k^2),
\end{align}
where the second inequality follows from \Cref{assump-snr}\textbf{b}, i.e.~$\alpha_n \gg \s\{\log(pn)\}^2$ and $\gamma < 1$.

\
\\
\textbf{Step 2: the order of magnitude of $II$.} The same arguments for the term $I$ lead to
\begin{align}\label{eq:refine_noise2}
    |II| = o_p\Big(\sqrt{r\kappa_k^2}\Big) + o_p(1) + o_p(r\kappa_k^2).
\end{align}
\
\\
\textbf{Step 3: the order of magnitude of $III$.} We have that
\begin{align*}
    III =& \sum_{t = \eta_k}^{\eta_k + r - 1}(X_t^{\top}\beta^*_{\mathcal{I}_{k-1}} - X_t^{\top}\beta^*_{\eta_{k-1}})^2 - 2(\beta^*_{\mathcal{I}_{k-1}} - \beta^*_{\eta_{k-1}})^{\top}\sum_{t = \eta_k}^{\eta_k + r - 1}X_t(y_t - X_t^{\top}\beta_{\eta_{k-1}}^*)\\
    =& \sum_{t = \eta_k}^{\eta_k + r - 1}(X_t^{\top}\beta^*_{\mathcal{I}_{k-1}} - X_t^{\top}\beta^*_{\eta_{k-1}})^2 - 2(\beta^*_{\mathcal{I}_{k-1}} - \beta^*_{\eta_{k-1}})^{\top}\sum_{t = \eta_k}^{\eta_k + r - 1}X_tX_t^{\top}(\beta_{\eta_{k}}^* - \beta_{\eta_{k-1}}^*)\\
    &- 2(\beta^*_{\mathcal{I}_{k-1}} - \beta^*_{\eta_{k-1}})^{\top}\sum_{t = \eta_k}^{\eta_k + r - 1}X_t\epsilon_t\\
    =& r(\beta^*_{\mathcal{I}_{k-1}} - \beta^*_{\eta_{k-1}})^{\top}\Sigma(\beta^*_{\mathcal{I}_{k-1}} - \beta^*_{\eta_{k-1}}) + (\beta^*_{\mathcal{I}_{k-1}} - \beta^*_{\eta_{k-1}})^{\top}\sum_{t = \eta_k}^{\eta_k + r - 1}(X_tX_t^{\top}-\Sigma)(\beta^*_{\mathcal{I}_{k-1}} - \beta^*_{\eta_{k-1}})\\
    &- 2(\beta^*_{\mathcal{I}_{k-1}} - \beta^*_{\eta_{k-1}})^{\top}\sum_{t = \eta_k}^{\eta_k + r - 1}X_t\epsilon_t - 2(\beta^*_{\mathcal{I}_{k-1}} - \beta^*_{\eta_{k-1}})^{\top}\sum_{t = \eta_k}^{\eta_k + r - 1}(X_tX_t^{\top}-\Sigma)(\beta_{\eta_{k}}^* - \beta_{\eta_{k-1}}^*)\\
    &- 2r(\beta^*_{\mathcal{I}_{k-1}} - \beta^*_{\eta_{k-1}})^{\top}\Sigma(\beta_{\eta_{k}}^* - \beta_{\eta_{k-1}}^*)\\
    =&III_1 + III_2 - 2III_3 - 2III_4 - 2III_5. 
\end{align*}
For the term $III_1$, \eqref{eq:bias_1} leads to
\begin{align*}
    |III_1| \leq r\Lambda_{\max}(\Sigma)\big|\beta^*_{\mathcal{I}_{k-1}} - \beta^*_{\eta_{k-1}} \big|_2^2 \leq C_1r\kappa_k^2\alpha_n^{-2}.
\end{align*}
For the term $III_2$, let $$z_t = \frac{1}{|\beta^*_{\mathcal{I}_{k-1}} - \beta^*_{\eta_{k-1}}|_2^2}(\beta^*_{\mathcal{I}_{k-1}} - \beta^*_{\eta_{k-1}})^{\top}\big(X_tX_t^{\top} - \Sigma\big)(\beta^*_{\mathcal{I}_{k-1}} - \beta^*_{\eta_{k-1}}).$$
Note that for some $q > 2$, we have that
\begin{align*}
    &\big\|z_t - z_{t,\{0\}}\big\|_q\\
    \leq& \bigg\| \frac{1}{|\beta^*_{\mathcal{I}_{k-1}} - \beta^*_{\eta_{k-1}}|_2^2}(\beta^*_{\mathcal{I}_{k-1}} - \beta^*_{\eta_{k-1}})^{\top}\Big\{X_t(X_t - X_{t,\{0\}})^{\top} + (X_t - X_{t,\{0\}})X_{t,\{0\}}^{\top}\Big\}(\beta^*_{\mathcal{I}_{k-1}} - \beta^*_{\eta_{k-1}}) \bigg\|_q\\
    \leq& 2\sup_{|v|_2 = 1}\big\|v^{\top}X_1\big\|_{2q}\delta^X_{t,2q},
\end{align*}
and $$\sum_{t = 1}^{\infty}\|z_t - z_{t,\{0\}}\|_q \leq 2\sup_{|v|_2 = 1}\big\|v^{\top}X_1\big\|_{2q}\Delta_{0,2q}^X < \infty.$$
It holds that
\begin{align*}
    |III_2| = (\beta^*_{\mathcal{I}_{k-1}} - \beta^*_{\eta_{k-1}})^{\top}\sum_{t = \eta_k}^{\eta_k + r - 1}\big(X_tX_t^{\top} - \Sigma\big)(\beta^*_{\mathcal{I}_{k-1}} - \beta^*_{\eta_{k-1}}) = O_p\big(\alpha_n^{-2}\sqrt{r\kappa_k^2}\{\log(r\kappa_k^2) + 1\}\big),
\end{align*}
where the second equality follows from \Cref{lemma:iterated log under dependence} with $\nu = \kappa_k^2$, \eqref{eq:bias_1} and $\kappa_k \leq C_{\kappa}$.
\
\\
Next, we consider the term $III_3$.
Let $$z_t = \frac{1}{|\beta^*_{\mathcal{I}_{k-1}} - \beta^*_{\eta_{k-1}}|_2}(\beta^*_{\mathcal{I}_{k-1}} - \beta^*_{\eta_{k-1}})^{\top}X_t\epsilon_t.$$
Note that for some $q > 2$, we have that
\begin{align*}
    &\big\|z_t - z_{t,\{0\}}\big\|_q\\
    \leq& \bigg\| \frac{1}{|\beta^*_{\mathcal{I}_{k-1}} - \beta^*_{\eta_{k-1}}|_2}(\beta^*_{\mathcal{I}_{k-1}} - \beta^*_{\eta_{k-1}})^{\top}\Big\{X_t(\epsilon_t -\epsilon_{t,\{0\}}) + (X_t - X_{t,\{0\}})\epsilon_{t,\{0\}}\Big\}\bigg\|_q\\
    \leq& \sup_{|v|_2 = 1}\big\|v^{\top}X_1\big\|_{2q}\delta^{\epsilon}_{t,2q} + \sup_{|v|_2 = 1}\big\|\epsilon_1\big\|_{2q}\delta^{X}_{t,2q},
\end{align*}
and $$\sum_{t = 1}^{\infty}\|z_t - z_{t,\{0\}}\|_q \leq \sup_{|v|_2 = 1}\big\|v^{\top}X_1\big\|_{2q}\Delta^{\epsilon}_{0,2q} + \sup_{|v|_2 = 1}\big\|\epsilon_1\big\|_{2q}\Delta^{X}_{0,2q} < \infty.$$
By \eqref{eq:bias_1} and \Cref{lemma:iterated log under dependence} and letting $\nu = \kappa_k^2$, we have that
\begin{align*}
    |III_3| = O_p\big(\alpha_n^{-1}\sqrt{r\kappa_k^2}\{\log(r\kappa_k^2) + 1\}\big).
\end{align*}
\
\\
For the term $III_4$, similar arguments for the term $III_2$ lead to
\begin{align*}
    |III_4| = O_p\big(\alpha_n^{-1}\sqrt{r\kappa_k^2}\{\log(r\kappa_k^2) + 1\}\big).
\end{align*}
\
\\
For the term $III_5$, we have that
\begin{align*}
    |III_5| \leq r\Lambda_{\max}(\Sigma)\big|\beta_{\eta_{k}}^* - \beta^*_{\eta_{k-1}}\big|_2\big|\beta^*_{\mathcal{I}_{k-1}} - \beta^*_{\eta_{k-1}}\big|_2 \leq C\alpha_n^{-1}r\kappa_k^2,
\end{align*}
where the second inequality follows from \eqref{eq:bias_1}.
Therefore, we have that
\begin{align}\label{eq:refine_noise3}
    |III| = o_p\big(\sqrt{r\kappa_k^2}\{\log(r\kappa_k^2) + 1\}\big) + o_p(r\kappa_k^2).
\end{align}
\\
\\
\textbf{Step 4: the order of magnitude of $IV$.} The same arguments for the term $III$ lead to
\begin{align}\label{eq:refine_noise4}
    |IV| = o_p\big(\sqrt{r\kappa_k^2}\{\log(r\kappa_k^2) + 1\}\big) + o_p(r\kappa_k^2).
\end{align}
\textbf{Step 5: lower bound of $V$.} Observe that
\begin{align*}
    &V = \sum_{t = \eta_k}^{\eta_k + r - 1} \big\{(y_t - X_t^{\top}\beta^*_{\eta_{k-1}})^2 - (y_t - X_t^{\top}\beta^*_{\eta_{k}})^2\big\}\\
    =& \sum_{t = \eta_k}^{\eta_k + r - 1} (X_t^{\top}\beta^*_{\eta_{k-1}} - X_t^{\top}\beta^*_{\eta_{k}})^2 - 2\sum_{t = \eta_k}^{\eta_k + r - 1}(y_t - X_t^{\top}\beta^*_{\eta_{k}})(X_t^{\top}\beta^*_{\eta_{k-1}} - X_t^{\top}\beta^*_{\eta_{k}})\\
    =& r(\beta^*_{\eta_{k-1}} - \beta^*_{\eta_{k}})^{\top}\Sigma(\beta^*_{\eta_{k-1}} - \beta^*_{\eta_{k}}) + (\beta^*_{\eta_{k-1}} - \beta^*_{\eta_{k}})^{\top}\sum_{t = \eta_k}^{\eta_k + r - 1}\big(X_tX_t^{\top} - \Sigma\big)(\beta^*_{\eta_{k-1}} - \beta^*_{\eta_{k}})\\
    &-2(\beta^*_{\eta_{k-1}} - \beta^*_{\eta_{k}})^{\top}\sum_{t = \eta_k}^{\eta_k + r - 1}\epsilon_tX_t\\
    =& V_1 + V_2 - 2V_3.
\end{align*}
Thus, we have that
\begin{align*}
    V \geq V_1 - |V_2| - 2|V_3|.
\end{align*}
\
\\
For the term $V_1$, we have that
\begin{align*}
    V_1 \geq r\Lambda_{\min}(\Sigma)\big|\beta^*_{\eta_{k-1}} - \beta^*_{\eta_{k}}\big|_2^2 \geq c_{\min}r\kappa_k^2,
\end{align*}
where the second inequality follows from \Cref{assume:X}.
\
\\
The analysis of the term $V_2$ and $V_3$ are similar to those of $III_2$ and $III_3$, and thus omitted.
It holds that
\begin{align*}
    |V_2| = O_p\big(\sqrt{r\kappa_k^2}\{\log(r\kappa_k^2) + 1\}\big) \;\; \text{and} \;\;
    |V_3| = O_p\big(\sqrt{r\kappa_k^2}\{\log(r\kappa_k^2) + 1\}\big),
\end{align*}
\
\\
Therefore, we have that
\begin{align}\label{eq:refine_signal}
    V \geq Cr\kappa_k^2 - O_p\big(\sqrt{r\kappa_k^2}\{\log(r\kappa_k^2) + 1\}\big).
\end{align}
\\
\\
Combining \eqref{eq:refine_signal_noise}, \eqref{eq:refine_noise1}, \eqref{eq:refine_noise2}, \eqref{eq:refine_noise3}, \eqref{eq:refine_noise4} and \eqref{eq:refine_signal}, we have uniformly for all $r \geq 1/\kappa_k^2$ that
\begin{align*}
    Cr\kappa_k^2 \leq O_p\big(\sqrt{r\kappa_k^2}\{\log(r\kappa_k^2) + 1\}\big) + o_p\big(\sqrt{r\kappa_k^2}\log(r\kappa_k^2) \big) + o_p\Big(\sqrt{r\kappa_k^2}\Big) + o_p(1) + o_p(r\kappa_k^2),
\end{align*}
which implies that
\begin{align*}
    r\kappa_k^2 = O_p(1),
\end{align*}
and completes the proofs of \textbf{a.1} and \textbf{b.1}.
\\
\\
\textbf{Limiting distributions.}
For any $k \in \{1, \dots, K\}$, given the end points $s_k$ and $e_k$ and the true coefficients $\beta_{\eta_{k-1}}^*$ and $\beta_{\eta_k}^*$ before and after the change point $\eta_k$, we define the function on $\eta \in (s_k, e_k)$ as  
\begin{align}\label{eq:loss_unknown}
    Q_k^*(\eta) = \sum_{t = s_k}^{\eta-1} (y_t - X_t^{\top}\beta_{\eta_{k-1}}^*)^2 + \sum_{t = \eta}^{e_k-1} (y_t - X_t^{\top}\beta_{\eta_{k}}^*)^2.
\end{align}
Note that $$V = \sum_{t = \eta_k}^{\eta_k + r - 1} \big\{(y_t - X_t^{\top}\beta^*_{\eta_{k-1}})^2 - (y_t - X_t^{\top}\beta^*_{\eta_{k}})^2\big\} = Q^*(\eta_k + r) - Q^*(\eta_k).$$
Due to the uniform tightness of $r\kappa_k^2$, we have that as $n \to \infty$
$$\big|Q(\eta_k + r) - Q(\eta_k) - (Q^*(\eta_k + r) - Q^*(\eta_k))\big| \leq |I| + |II| + |III| + |IV| \overset{p}{\to} 0.$$
Therefore, it is sufficient to show the limiting distributions of $Q^*(\eta_k + r) - Q^*(\eta_k)$ when $n \to \infty$.  Next, we consider the two regimes for $\kappa_k$ separately.
\\
\\
\textbf{Non-vanishing regime.}
Observe that for $r > 0$, we have that when $n \to \infty$
\begin{align*}
    &Q_k^*(\eta_k+ r) - Q_k^*(\eta_k) = \sum_{t = \eta_k}^{\eta_k + r - 1} \big\{(y_t - X_t^{\top}\beta_{\eta_{k-1}}^*)^2 - (y_t - X_t^{\top}\beta_{\eta_{k}}^*)^2\big\}\\
    =& \sum_{t = \eta_k}^{\eta_k + r - 1}\big\{(y_t - X_t^{\top}\beta_{\eta_{k}}^* + \Psi_k^{\top}X_t)^2 - (y_t - X_t^{\top}\beta_{\eta_{k}}^*)^2\big\}\\
    =& \sum_{t = \eta_k}^{\eta_k + r - 1}\big\{2\epsilon_t\Psi_k^{\top}X_t +  (\Psi_k^{\top}X_t)^2\big\} \overset{\mathcal{D}}{\to} \sum_{t = \eta_k}^{\eta_k + r - 1}\big\{2\varrho_k\epsilon_t\xi_t(k) +  \varrho_k^2\xi_t^2(k)\big\}.
\end{align*}
For $r < 0$, we have that  when $n \to \infty$
\begin{align*}
    & Q_k^*(\eta_k+ r) - Q_k^*(\eta_k) = \sum_{t = \eta_k + r}^{\eta_k - 1} \big\{(y_t - X_t^{\top}\beta_{\eta_{k}}^*)^2 - (y_t - X_t^{\top}\beta_{\eta_{k-1}}^*)^2\big\} \\
    = & \sum_{t = \eta_k + r}^{\eta_k - 1} \{(y_t - X_t^{\top}\beta_{\eta_{k-1}}^* - \Psi_k^{\top}X_t)^2 - (y_t - X_t^{\top}\beta_{\eta_{k-1}}^*)^2\} \\ 
    = & \sum_{t = \eta_k + r}^{\eta_k-1} \big\{-2\epsilon_t\Psi_k^{\top}X_t +  (\Psi_k^{\top}X_t)^2\big\} \overset{\mathcal{D}}{\to} \sum_{t = \eta_k + r}^{\eta_k-1}\big\{-2\varrho_k\epsilon_t\xi_t(k) +  \varrho_k^2\xi_t^2(k)\big\}.
\end{align*}
So Slutsky's theorem and the Argmax (or Argmin) continuous mapping theorem \cite[see 3.2.2 Theorem][]{Vaart1996weak} lead to
    \begin{align*}
       \widetilde{\eta}_k - \eta_k
       \overset{\mathcal{D}}{\longrightarrow} \argmin_rP_k(r),
    \end{align*}
which completes the proof of \textbf{a.2}.
\\
\\
\textbf{Vanishing regime.}
Let $m = \kappa_k^{-2}$, and we have that $m \to \infty$ as $n \to \infty$. Observe that for $r > 0$, we have that
\begin{align*}
    &Q_k^*(\eta_k+ rm) - Q_k^*(\eta_k) = \sum_{t = \eta_k}^{\eta_k + rm - 1} \big\{(y_t - X_t^{\top}\beta_{\eta_{k-1}}^*)^2 - (y_t - X_t^{\top}\beta_{\eta_{k}}^*)^2\big\}\\
    =& \sum_{t = \eta_k}^{\eta_k + rm - 1}\big\{(y_t - X_t^{\top}\beta_{\eta_{k}}^* + \Psi_k^{\top}X_t)^2 - (y_t - X_t^{\top}\beta_{\eta_{k}}^*)^2\big\}\\
    =& \frac{1}{\sqrt{m}}\sum_{t = \eta_k}^{\eta_k + rm - 1}\big\{2\epsilon_tv_k^{\top}X_t\big\} +  \frac{1}{m}\sum_{t = \eta_k}^{\eta_k + rm - 1}\big\{(v_k^{\top}X_t)^2 - v_k^{\top}\Sigma v_k\big\} + rv_k^{\top}\Sigma v_k.
\end{align*}
Note that condition of the functional CLT under functional dependence \cite[Theorem 3 in][]{wu2011asymptotic} can be verified easily under Assumptions \ref{assume:X} and \ref{assume:epsilon}, By the functional CLT, we have that when $n \to \infty$
\begin{align*}
    \frac{1}{\sqrt{m}}\sum_{t = \eta_k}^{\eta_k + rm - 1}\big\{2\epsilon_tv_k^{\top}X_t\big\} \overset{\mathcal{D}}{\to}  \sigma_{\infty}(k)\mathbb{B}(r) \;\; \text{and} \;\; \frac{1}{m}\sum_{t = \eta_k}^{\eta_k + rm - 1}\big\{(v_k^{\top}X_t)^2 - v_k^{\top}\Sigma v_k\big\} \overset{p}{\to} 0,
\end{align*}
where $\mathbb{B}(r)$ is a standard Brownian motion and $\sigma^2_{\infty}(k)$ is the long-run variance given in \eqref{eq:long-run_var}.
Therefore, it holds that when $n \to \infty$
\begin{align*}
    Q_k^*(\eta_k+ rm) - Q_k^*(\eta_k) 
    \overset{\mathcal{D}}{\to} r\varpi_k + \sigma_{\infty}(k)\mathbb{B}_1(r).
\end{align*}
Similarly, for $r < 0$, we have that  when $n \to \infty$
\begin{align*}
    Q_k^*(\eta_k+ rm) - Q_k^*(\eta_k) 
    \overset{\mathcal{D}}{\to} -r\varpi_k + \sigma_{\infty}(k)\mathbb{B}_2(-r).
\end{align*}
So Slutsky's theorem and the Argmax (or Argmin) continuous mapping theorem \cite[see 3.2.2 Theorem][]{Vaart1996weak} lead to
    \begin{align*}
       \kappa_k^2(\widetilde{\eta}_k - \eta_k)
       \overset{\mathcal{D}}{\longrightarrow} \argmin_r \big\{\varpi_k|r| + \sigma_{\infty}(k)\mathbb{W}(r)\big\},
    \end{align*}
which completes the proof of \textbf{b.2}.
\end{proof}

\section{Proofs for Section 3.3}\label{sec-app-pre}  
\subsection[]{Proof of \Cref{theorem:DUDP}}
\begin{proof}[Proof of \Cref{theorem:DUDP}]
	It follows from \Cref{prop-1} that, $K \leq |\widehat{\mathcal{P}}| \leq 3K$.  This combined with \Cref{prop-2} completes the proof.
\end{proof} 

\begin{proposition}\label{prop-1}
Under the same conditions in \Cref{theorem:DUDP} and letting $\widehat{\mathcal{P}}$ being the solution to \Cref{algorithm:DPDU}, the following events uniformly hold with probability at least $1 - cn^{-3}$. 
	\begin{itemize}
		\item [(i)] For all intervals  $\widehat \I = [s, e) \in \widehat{\mathcal{P}}$    containing one and only one true change point $\eta_k$, it must be the case that for sufficiently large constant $ C_{\mathfrak e} $, 
			\[
				\min\{e - \eta_k, \eta_k - s\} \leq C_{\mathfrak e} \left(\frac{ \s \log(pn)  + \zeta }{\kappa_k^2}\right);
			\]
		\item [(ii)] for all intervals  $\widehat \I = [s, e) \in \widehat{\mathcal{P}}$ containing exactly two true change points, say  $\eta_k < \eta_{k+1} $, it must be the case that for sufficiently large constant $ C_{\mathfrak e} $, 
			\[
				    e - \eta_{k+1}\leq  C_{\mathfrak e}  \left(\frac{\s \log(pn)  + \zeta }{\kappa_{k+1}^2}\right)    \quad \text{and} \quad 
				   \eta_k - s  \leq  C_{\mathfrak e}  \left(\frac{\s \log(pn)  + \zeta}{\kappa_{k }^2}\right);	
			\]  
		\item [(iii)] for all consecutive intervals $\widehat \I$ and $\widehat {\J}$ in $\hatp $, the interval $\widehat \I \cup \widehat{\J}$ contains at least one true change point; and
		\item [(iv)] no interval $\widehat \I \in \widehat{\mathcal{P}}$ contains strictly more than two true change points.
	 \end{itemize}
\end{proposition}

\begin{proposition}\label{prop-2}
Under the same conditions in \Cref{theorem:DUDP}, with $\widehat{\mathcal{P}}$ being the solution to \eqref{algorithm:DPDU}, satisfying $K \leq |\widehat{\mathcal{P}}| \leq 3K$, then with probability at least $1 -   cn  ^{-3}$, it holds that $|\widehat{\mathcal{P}}| = K$.	
\end{proposition}

Throughout the rest of this section, denote 
$$S  = \bigcup_{t=1}^n S_t . $$
Since there are $K$ change points in $\{\beta_t^*\}_{t=1}^n$, it follows that $| S| \le \sum_{k=0}^K | S_{\eta_k+1}|\le (K+1) \s $.

\subsubsection[]{Proof of \Cref{prop-1}}
In this subsection, we present the four cases of 
\Cref{prop-1}. Throughout this  subsection, we assume that all the conditions in \Cref{theorem:DUDP} hold.

 \begin{lemma}
 Let   $ \I=[s,e) \in \mathcal {\widehat P}  $  be such that $\I$ contains exactly  one    change point $ \eta_k $. 
 Then  with probability  at least $1- cn^{-5}$,   it holds that  for sufficiently large constant $ C_{\mathfrak e} $, 
 $$\min\{ \eta_k -s ,e-\eta_k \}  \leq  C_{\mathfrak e}  \left(\frac{ \s \log(pn)  + \zeta }{\kappa_k^2}\right)  .$$
  \end{lemma}
\begin{proof}   
If either $ \eta_k -s \le  C_{\mathfrak e}  \left(\frac{ \s \log(pn)  + \zeta }{\kappa_k^2}\right)   $ or $e-\eta_k\le  C_{\mathfrak e}  \left(\frac{ \s \log(pn)  + \zeta }{\kappa_k^2}\right) $, then there is nothing to show. So assume that 
\begin{align}\label{eq:localization_error_case1_assume}
\eta_k -s  >  C_{\mathfrak e}  \left(\frac{ \s \log(pn)  + \zeta }{\kappa_k^2}\right)  \quad \text{and} \quad   e-\eta_k  >  C_{\mathfrak e}  \left(\frac{ \s \log(pn)  +  \zeta }{\kappa_k^2}\right)  .
\end{align}
For sufficiently large $ C_{\mathfrak e} $, this implies that 
$$ \eta_k-s > \zeta  \quad \text{and} \quad e - \eta_k  >  \zeta.$$
{\bf Step 1.}  Denote 
$$ \J_ 1 = [s,\eta_k) \quad \text{and} \quad \J_2 = [\eta_k, e) .$$
So $|\J_1| \ge \zeta$ and $|\J _2| \ge \zeta$. Thus 
for any $ \ell\in \{ 1, 2 \}$ 
$$ \mathcal H(\J_\ell) =\sum_{t  \in \J_\ell }(y_t  -X_t ^\top  \widehat  \beta _{\J_\ell}  ) ^2 
\le \sum_{t  \in \J_\ell  }(y_t  -X_t ^\top     \beta _{\J _\ell }^*  ) ^2 +C_1 \s\log(pn), $$
where the inequality follows from \Cref{eq:regression one change deviation bound}. 

  Since 
$ \I \in \hatp   $,   it follows that 
\begin{align}\nonumber
\sum_{t  \in \I }(y_t  - X_t ^\top \widehat \beta _\I  ) ^2 \le &  \sum_{t  \in \J_1 }(y_t  -X_t ^\top  \widehat  \beta _{\J _1}  ) ^2  + \sum_{t  \in \J _2 }(y_t  -X_t ^\top  \widehat \beta _{\J _2}  ) ^2    + \zeta  
\\
\le &  \sum_{t  \in \J _1 }(y_t  -X_t ^\top     \beta ^* _{\J _1}  ) ^2  + \sum_{t  \in \J _2 }(y_t  -X_t ^\top   \beta^*  _{\J _2}  ) ^2  +2C_1  \s \log(pn)    +  \zeta   
 \label{eq:one change point step 1}
\end{align}
\Cref{eq:one change point step 1}  gives 
\begin{align}  \nonumber 
& \sum_{t  \in \J_1 }(y_t   -X_t ^\top  \widehat \beta _\I ) ^2 +\sum_{t  \in \J_2 }(y_t  - X_t ^\top  \widehat \beta _\I  ) ^2  
\\ 
 \le \label{eq:regression step 2 first term} & \sum_{t  \in \J _1  }(y_t -  X_t ^\top    \beta _{\J  _1}  ^*    ) ^2   +\sum_{t  \in \J _2  }(y_t - X_t ^\top    \beta _{\J  _2}  ^*  ) ^2  +2 C_1      \s \log(pn)  +  \zeta  .
\end{align}
{\bf Step 2.}  \Cref{eq:regression step 2 first term} leads to 
\begin{align} \nonumber
&\sum_{t  \in \J  _1 } \big\{ X_t ^\top  (  \widehat \beta _\I  -  \beta^* _{\J_1 }) \big\}  ^2 +\sum_{t   \in \J _2  } \big\{ X_t  ^\top  (  \widehat \beta _\I  -  \beta^* _{\J_2 }) \big\}  ^2  
\\  \label{eq:regression one change step 2 term 2}
 \le  & 2 \sum_{ t  \in \J _1  } \epsilon_t   X_t ^\top (  \widehat \beta _\I  -  \beta^* _{\J _1 }) +  2 \sum_{ t  \in \J  _2  } \epsilon_t   X_t ^\top (  \widehat \beta _\I  -  \beta^* _{\J _2  })    + 2 C_1      \s \log(pn)  +   \zeta  .
\end{align} 

Note that for $\ell\in\{1,2\}$, 
\begin{align*} 
 |  ( \widehat \beta _\I  -  \beta _{\J _\ell}^* )_{S^c }  |_1 =  |  ( \widehat \beta _\I )_{S^c}  |_{1} 
   \le 3  |  ( \widehat \beta _\I -\beta _\I ^* )_{S }  |_{1}   \le C_3 \s \sqrt { \frac{ \log(np)}{| \I|} }    , 
\end{align*} 
where the last two inequalities follows from \Cref{lemma:interval lasso}.  So
\begin{align}  \label{eq:regression one change step 2 term 1}
 |   \widehat \beta _\I  -  \beta _{\J_\ell  }^*    |_1 =  |  ( \widehat \beta _\I  -  \beta _{\J_\ell }^* )_{S  }  |_1 +  |  ( \widehat \beta _\I  -  \beta _{\J_\ell }^* )_{S^c }  |_1  \le \sqrt {\s }   |  \widehat \beta _\I  -  \beta _{\J_\ell  }^*    |_2 + C_3 \s \sqrt { \frac{ \log(np)}{| \I|} } .
\end{align} 
Therefore, for any $\ell=1, 2$,
  it holds with probability at least $1 -n^{-5}$ that
\begin{align}  \nonumber  \sum_{ t  \in \J _\ell   } \epsilon_t   X_t ^\top (  \widehat \beta _\I  -  \beta^* _{\J _\ell  }) 
\le 
&
       \bigg| \sum_{ t  \in \J _\ell   } \epsilon_t   X_t ^\top \bigg| _\infty  |   \widehat \beta _\I  -  \beta^* _{\J _\ell  } | _1  
     \le C_2 \sqrt { \log(np)    |\J_\ell | }  |   \widehat \beta _\I  -  \beta _{\J_\ell  }^*    |_1 
     \\\nonumber  
      \le & C_2   \sqrt { \log(np)    |\J_\ell | }   \bigg(   \sqrt {\s }   |  \widehat \beta _\I  -  \beta _{\J_\ell }^*    |_2  +C_3 \s \sqrt { \frac{ \log(np)}{| \I|} }  \bigg) 
      \\
\le  &  \frac{ \Lambda_{\min}(\Sigma)  | \J_ \ell |}{64}  |  \widehat \beta _\I  -  \beta _{\J_\ell }^*    |_2  ^2  + C_4    \s \log(pn),
 \label{eq:regression one change step 2 term 3}
\end{align}
where the  second  inequality follows from \Cref{lemma:lasso deviation bound 1} and that  $|\J  _\ell|   \ge \zeta   $, the third inequality follows from \Cref{eq:regression one change step 2 term 1}, and  the last inequality follows from the observation     that   $ |\J_\ell| \le |\I|   $ and the inequality $2ab \leq a^2 + b^2$.

{\bf Step 3.} For any $\ell=1, 2$,
  it holds with probability at least $1 -2n^{-5}$ that
 \begin{align}\nonumber 
		&  \sum_{ t  \in \J_\ell  } \big\{ X_t ^\top  (  \widehat \beta _\I  -  \beta^* _{\J_ \ell  }) \big\}  ^2 
		 \\  \nonumber  
		  \ge&  \frac{\Lambda_{\min} (\Sigma) |\J_\ell | }{2}  |\widehat \beta _\I  -  \beta^* _{\J_ \ell  }  |_2^2 - C_5  |\J_\ell|^{1-\gamma }  \log(pn)  |\widehat \beta _\I  -  \beta^* _{\J_ \ell  }  |_1  ^2 
		 \\\nonumber  
		 \ge &  \frac{\Lambda_{\min} (\Sigma) |\J_\ell | }{2}  |\widehat \beta _\I  -  \beta^* _{\J_ \ell  }  |_2^2  -   C_6   \s  |\J_\ell|^{1-\gamma }   \log(pn)      |\widehat \beta _\I -  \beta^* _{\J_ \ell  }  |_2^2   -  C_6  |\J_\ell|^{1-\gamma } \log(np) \s^2     \frac{ \log(np)}{| \I|}   
		 \\ \nonumber 
		 \ge &  \frac{\Lambda_{\min} (\Sigma) |\J_\ell | }{4}  |\widehat \beta _\I  -  \beta^* _{\J_ \ell  }  |_2^2  -   C_6  |\J_\ell|^{1-\gamma } \log(np) \s^2     \frac{ \log(np)}{| \I|}      
		 \\ \label{eq:regression one change step 3 term 1}
		 \ge &   \frac{\Lambda_{\min} (\Sigma) |\J_\ell | }{4}  |\widehat \beta _\I  -  \beta^* _{\J_ \ell  }  |_2^2  -   C_7  \s  \log(np)      
	\end{align} 
	where the first inequality follows from   
 \Cref{theorem:RES},   the second inequality follows from \Cref{eq:regression one change step 2 term 1},     the third  inequality follows from the observation that 
 $$   |\J_\ell|   \ge \zeta   =  C_\zeta (\s\log(pn))^{2/\gamma -1 }\ge C_\zeta (\s\log(pn))^{1/\gamma   }, $$  
 and the last inequality follows from the observation that 
 $$  \frac{|\J_\ell  |^{1-\gamma}}{ |\I | }\le |\J_\ell| ^{-\gamma } \le C_\zeta^{-\gamma} (\s\log(pn))^{-1} .
 $$
{\bf Step 4.} \Cref{eq:regression one change step 2 term 2},  \Cref{eq:regression one change step 2 term 3} and \Cref{eq:regression one change step 3 term 1} together  imply that 
 $$   |\J_1  |    |\widehat \beta _\I  -  \beta^* _{\J_ 1  }  |_2^2  + |\J_2  |    |\widehat \beta _\I  -  \beta^* _{\J_ 2  }  |_2^2  \le C_8 \s \log(np)  +  \zeta  .$$
Observe  that 
$$ \inf_{  \beta  \in \mathbb R ^p   }  |\J_1 |    | \beta    -  \beta^* _{\J_ 1  }  |_2^2  + |\J_2 |  |  \beta    -  \beta^* _{\J_ 2  }  |_2^2    = \kappa_k ^2  \frac{|\J_1| |\J_2|}{| \I| }   \ge \frac{ \kappa_k^2 }{2} \min\{ |\J_1| ,|\J_2| \}    . $$ 
This leads to 
$$  \frac{ \kappa_k^2 }{2}\min\{ |\J_1| ,|\J_2| \}    \le  C_8 \s \log(np)   + \zeta   .$$
Therefore
$$  \min\{ |\J_1| ,|\J_2| \}    \le  C_{\mathfrak e}  \left(\frac{ \s \log(pn)  + \zeta }{\kappa_k^2}\right),$$
which is a contradiction with \Cref{eq:localization_error_case1_assume}. This immediate gives  the desired result.
\end{proof}

 \begin{lemma}  \label{lemma:regression two change points}
Let   $ \I=[s,e) \in \mathcal {\widehat P}  $  be such that $\I$ contains  exactly two  change points $ \eta_k,\eta_{k+1}  $.
 Then  with probability  at least $1- cn^{-5}$,   it holds that for sufficiently large constant $ C_{\mathfrak e} $, 
\begin{align*} 
 \eta_k -s     \le     C_{\mathfrak{e} }  \bigg( \frac{\s \log(np) +\zeta   }{\kappa _k^2 }  \bigg)   \quad \text{and} \quad 
 e-\eta_{k+1}   \le   C_{\mathfrak{e} }\bigg( \frac{\s \log(np) +\zeta  }{\kappa_{k+1 } ^2 }   \bigg).
 \end{align*} 
 \end{lemma} 
\begin{proof} 
Denote 
 $$ \mathcal J_ 1= [ s, \eta_k ), \quad \J_2  =[\eta_k, \eta_{k+1}) \quad \text{and} \quad \J_3  = [\eta_{k+1 },e). $$
 By symmetry, it suffices to show that 
 $$|\J_1| \le  C_{\mathfrak{e} }    \bigg( \frac{\s \log(np) +  \zeta   }{\kappa _k^2 }  \bigg)   . $$
 If $ |\J_1| \le  \zeta $, then   for sufficiently large $ C_{\mathfrak{e} } $, it holds that
 $$  |\J_1| \le  \zeta \le   C_{\mathfrak{e} }    \bigg( \frac{\s \log(np) +  \zeta   }{\kappa _k^2 } \bigg)   .$$
So it suffices to assume 
$$|\J_1| \ge \zeta .  $$
{\bf Step 1.} 
 Since $ \I $ contains two change points, it follows that 
 $ |\I|  \ge \eta_{k+1}-\eta _k  \ge \zeta. $
 So 
 $$\mathcal H (\I) = \sum_{t\in \I } (y_t- X_t ^\top \widehat \beta_\I )^2.  $$

Since for any $\ell\in \{1,2\}$, $ |\J_ \ell | \ge \zeta   $, it follows with probability at least $1 - n^{-5}$ that
$$ \mathcal H(\J_\ell ) = \sum_{ t  \in \J_\ell  }(y_t  - X_t ^\top \widehat \beta _{\J_\ell }  ) ^2  \le \sum_{t  \in \J_ \ell   }(y_t  -X_t ^\top     \beta _{\J _ \ell }^*  ) ^2 +C_1 \s\log(pn)   $$
where the inequality follows from \Cref{eq:regression one change deviation bound}.  
 If   $ |\J_3 |  \ge   \zeta   $, then 
 $$ \mathcal H(\J_ 3) =\sum_{t  \in \J_3 }(y_t  -X_t ^\top  \widehat  \beta _{\J_3 }  ) ^2 
\le \sum_{t  \in \J_3  }(y_t  -X_t ^\top     \beta _{\J _3 }^*  ) ^2 +C_1 \s\log(pn), $$
where the inequality follows from \Cref{eq:regression one change deviation bound}.  
If $ |\J_3  |  \le  \zeta   $, then 
 $$ \mathcal H(\J_3 ) = 0 
\le \sum_{t  \in \J_ 3  }(y_t  -X_t ^\top     \beta _{\J _3  }^*  ) ^2 +C_1 \s\log(pn) . $$
So 
$$ \mathcal H(\J_3 )   
\le \sum_{t  \in \J_3   }(y_t  -X_t ^\top     \beta _{\J _3  }^*  ) ^2 +C_1 \s\log(pn) . $$

 {\bf Step 2.} Since $\I \in \widehat{\mathcal P} $, we have
 \begin{align}\label{eq:regression two change points local min} \mathcal H (\I ) \le \mathcal H (\J_1  ) +\mathcal H (\J_2 )+\mathcal H (\J_3 ) + 2 \zeta. 
 \end{align} The above display and the calculations in  {\bf Step 1}   imply that
\begin{align} 
 \sum_{ t  \in \I }(y_t  -  X_t ^\top \widehat \beta _\I ) ^2  
\le   \sum_{\ell=1}^3 \sum_{t  \in \J_ \ell  }(y_t  -  X_t ^\top    \beta^*  _{\J_\ell }  ) ^2   +3C_1 \s\log(pn)   +2 \zeta .
  \label{eq:regression two change points step 2 term 1}
\end{align}

\Cref{eq:regression two change points step 2 term 1}  gives 
\begin{align} 
 \sum_{\ell=1}^3 \sum_{t  \in \J_{\ell}  }(y_t  - X_t ^\top  \widehat \beta _\I ) ^2 
 \le& \sum_{\ell=1}^3 \sum_{t  \in \J_ \ell  }(y_t  -  X_t ^\top    \beta^*  _{\J_\ell }  ) ^2    
   +3C_1 \s\log(pn)   +2\zeta       \label{eq:regression two change points first}.
\end{align}

{\bf Step 3.} Note that  for any $\ell\in \{1,2\}$, $ |\J_\ell| \ge \zeta   $.  So it holds with probability at least $1 -3n^{-5}$ that
\begin{align}  \nonumber 
 &  \sum_{t  \in \J_{\ell }  }(y_t  - X_t ^\top  \widehat \beta _\I ) ^2 
 -   \sum_{t  \in \J_ \ell  }(y_t  -  X_t ^\top    \beta^*  _{\J_\ell   }  ) ^2   
 \\
 =&  \nonumber   \sum_{t  \in \J_{\ell  }  }  \{ X_t^\top (\widehat \beta _\I  -  \beta^*  _{\J_\ell  }  ) \} ^2 - 2  \sum_{t  \in \J_{\ell }  } \epsilon_t X_t^\top (\widehat \beta _\I  -  \beta^*  _{\J_\ell  }  ) 
 \\
 \ge&  \nonumber   \frac{\Lambda_{\min} (\Sigma) |\J_\ell  | }{4}  |\widehat \beta _\I  -  \beta^* _{\J_ \ell  }  |_2^2  -   C_2 \s  \log(np)       - \frac{ \Lambda_{\min}(\Sigma)  | \J_ \ell |}{64}  |  \widehat \beta _\I  -  \beta _{\J_\ell  }^*    |_2  ^2  - C_2'    \s \log(pn)
 \\ \label{eq:regression two change points step 3 term 1}
 \ge &  \frac{\Lambda_{\min} (\Sigma) |\J_\ell  | }{8}  |\widehat \beta _\I  -  \beta^* _{\J_ \ell  }  |_2^2  -   C_3 \s  \log(np)       ,
\end{align}
where the first inequality follows from the same arguments that lead to 
   \Cref{eq:regression one change step 3 term 1} and  \Cref{eq:regression one change step 2 term 3}.

{\bf Step 4.} 
 If $|\J_3| \ge \zeta$, then the same calculations in {\bf Step 3} give
\begin{align} \label{eq:regression two change points step 4 term 1}
   \sum_{t  \in \J_{3   }  }(y_t  - X_t ^\top  \widehat \beta _\I ) ^2 
 -   \sum_{t  \in \J_ 3    }(y_t  -  X_t ^\top    \beta^*  _{\J_3   }  ) ^2   
 \ge    \frac{\Lambda_{\min} (\Sigma) |\J_3 | }{8}  |\widehat \beta _\I  -  \beta^* _{\J_ 3  }  |_2^2  -   C_3 \s  \log(np)       .
\end{align}
Putting \Cref{eq:regression two change points step 3 term 1} and  \Cref{eq:regression two change points step 4 term 1} into \Cref{eq:regression two change points first},
it follows that 
 $$  \sum_{\ell =1}^3 \frac{ \Lambda_{\min}(\Sigma) |\J_\ell | }{8}  |\widehat \beta _\I  -  \beta^* _{\J_ \ell  }  |_2^2      \le C_4 \s \log(np)  +2\zeta  .$$
 This leads to 
  \begin{align} \label{eq:step 4 colusion two changes}    |\J_1 |    |\widehat \beta _\I  -  \beta^* _{\J_ 1  }  |_2^2     + |\J_2 |   |\widehat \beta _\I  -  \beta^* _{\J_ 2  }  |_2^2       \le  C_4 \s \log(np)  +2\zeta  .\end{align}
Observe  that 
$$ \inf_{  \beta  \in \mathbb R ^p   }  |\J_1 |   | \beta    -  \beta^* _{\J_ 1  }  |_2^2  + |\J_2 |  |  \beta    -  \beta^* _{\J_ 2  }  |_2^2    = \kappa_k   ^2  \frac{|\J_1| |\J_2|}{| \J_1| +|\J_2|  }   \ge \frac{ \kappa_k  ^2 }{2} \min\{ |\J_1| ,|\J_2| \}     . $$ 
Thus
$$    \kappa_k ^2 \min\{ |\J_1| ,|\J_2| \}    \le  C_{5}  \s \log(np)   + 2  \zeta    ,$$
which is 
$$ \min\{ |\J_1| ,|\J_2| \}   \le  \frac{ C_{5}  \s \log(np)   + 2  \zeta }{\kappa_k ^2}.$$ 
Since  $ |\J_2  | \ge \Delta \ge  \frac{ C   \s \log(np)   + 2 \alpha_n \zeta }{\kappa_k ^2}     $
 for sufficiently large constant $C$, 
 it follows that
 $$   |\J_1|    \le\frac{ C_{5}  \s \log(np)   + 2  \zeta }{\kappa_k ^2}  .$$

{\bf Step 5.}  If $|\J_{3}| \le \zeta $, note that  a similar and simpler argument of \Cref{lemma:lasso deviation bound 1}    gives 
$$ \p\bigg( \bigg| \sum_{t\in \J_3} \epsilon_t ^2 -|\J_3| \E(\epsilon_t^2) \bigg| \le  C _6  \bigg\{  \sqrt {|\J_3|   \log( n)   }  +    \log^{1/\gamma }(  n )   \bigg \}  \bigg) \le n^{-5} .   $$
So with probability at least $1-n^{-5} $, it follows that 
\begin{align}\label{eq:two change points small intervals epsilon}   \sum_{t\in \J_3} \epsilon_t ^2   \le |\J_3|\sigma_{\epsilon}^2  +
 C _6  \bigg\{  \sqrt {|\J_3|   \log( n)   }  +    \log^{1/\gamma }(  n )   \bigg \} 
 \le \frac{\zeta}{4}, \end{align} 
 where the last inequality follows from that $  |\J_3| \le  \zeta$ and $\zeta \ge  C_\zeta(\s\log(pn) )^{1/\gamma}$ for sufficient large $C_{\zeta}$. 
So
\begin{align} & \sum_{t  \in \J_3    }(y_t  - X_t ^\top  \widehat \beta _\I ) ^2  - \sum_{ t  \in \J_ 3    }(y_t  - X_t ^\top    \beta^*  _{\J_ 3 }  ) ^2 \nonumber 
=   \sum_{t  \in \J_ 3    }\big\{ X_t ^\top (\widehat \beta _\I -\beta^*  _{\J_ 3  } )\big\}^2 -2 \sum_{ t  \in \J_ 3    } \epsilon_t  X_t ^\top (\widehat \beta _\I -\beta^*  _{\J_ 3   }  )
\\  \nonumber 
\ge  & \sum_{i \in \J_ 3     }\big\{ X_ t ^\top (\widehat \beta _\I -\beta^*  _{\J_ 3   } )\big\}^2  -\frac{1}{2 }\sum_{t  \in \J_ 3     }\big\{ X_t ^\top (\widehat \beta _\I -\beta^*  _{\J_ 3   } )\big\}^2  - 2   \sum_{t  \in \J_ 3     }\epsilon_t ^2 
\\   
\ge &  \frac{1}{2 }\sum_{t  \in \J_ 3     }\big\{ X_t ^\top (\widehat \beta _\I -\beta^*  _{\J_ 3   } )\big\}^2  -\frac{\zeta}{2}
\ge -\frac{\zeta}{2},  \label{eq:regression change point step 4 last item}
\end{align} 
where the second inequality follows from \Cref{eq:two change points small intervals epsilon}. Putting \Cref{eq:regression two change points step 3 term 1} and  \Cref{eq:regression change point step 4 last item} into \Cref{eq:regression two change points first},
it follows that 
 $$  \sum_{\ell =1}^2 \frac{ \Lambda_{\min}(\Sigma) |\J_\ell | }{8}  |\widehat \beta _\I  -  \beta^* _{\J_ \ell  }  |_2^2      \le C_7 \s \log(np)  +C_8\zeta  .$$
  Similar argument for \Cref{eq:step 4 colusion two changes} leads to 
 $$   |\J_1|    \le\frac{ C_{9} \big(  \s \log(np)   +  \zeta \big)  }{\kappa_k ^2}  .$$  

 \end{proof} 

 \begin{lemma} 
With probability  at least $1- cn^{-3}$,    there is no intervals in   $  \widehat { \mathcal P}  $  containing three or more true change points.   
 \end{lemma}

\begin{proof} 
For  contradiction,  suppose    $ \I=[s,e) \in \mathcal {\widehat P}  $  be such that $ \{ \eta_1, \ldots, \eta_M\} \subset \I $ with $M\ge 3$. 

{\bf Step 1.}  Since $\I$ contains three change points, 
 $|\I| \ge \zeta   $ and therefore 
 $$ \mathcal H(\I) = \sum_{t \in \I} (y_t -X_t^\top \widehat \beta_\I )^2 .$$
Denote 
   $$  \J_1 =[s,  \eta_1), \ \J_m = [\eta_{m-1}, \eta_m) \;\; \text{for} \;\;  2 \le m \le M , \ \J_{M+1}  =[\eta_M, e). $$ 
Since for $m\in \{ 2,\ldots, M \}$, $ |\J_ m | \ge \zeta   $, it follows that
$$ \mathcal H(\J_m ) = \sum_{ t  \in \J_m   }(y_t  - X_t ^\top \widehat \beta _{\J_m  }  ) ^2  \le \sum_{t  \in \J_ m   }(y_t  -X_t ^\top     \beta _{\J _ m  }^*  ) ^2 +C_1 \s\log(pn),   $$
where the inequality follows from \Cref{eq:regression one change deviation bound}.  
For $\ell \in \{ 1, M+1  \}$, if  $ |\J_ \ell | \ge \zeta   $, then 
 $$ \mathcal H(\J_ \ell) =\sum_{t  \in \J_\ell }(y_t  -X_t ^\top  \widehat  \beta _{\J_\ell  }  ) ^2 
\le \sum_{t  \in \J_\ell    }(y_t  -X_t ^\top     \beta _{\J _\ell    }^*  ) ^2 +C_1 \s\log(pn), $$
where the inequality follows from \Cref{eq:regression one change deviation bound}.  
If $ |\J_\ell   |  \le  \zeta   $, then 
 $$ \mathcal H(\J_\ell  ) = 0 
\le \sum_{t  \in \J_ \ell   }(y_t  -X_t ^\top     \beta _{\J _\ell   }^*  ) ^2 +C_1 \s\log(pn) . $$
So 
$$ \mathcal H(\J_\ell   )   
\le \sum_{t  \in \J_\ell    }(y_t  -X_t ^\top     \beta _{\J _\ell   }^*  ) ^2 +C_1 \s\log(pn) . $$

 {\bf Step 2.} Since $\I \in \widehat{\mathcal P} $, we have
 \begin{align}\label{eq:regression three change points local min} \mathcal H(\I ) \le \sum_{m=1}^{M+1}\mathcal H (\J_m  )  +  M \zeta. 
 \end{align} The above display and the calculations in  {\bf Step 1} and  {\bf Step 2} implies that
\begin{align} 
 \sum_{t  \in \I }(y_t  -  X_t ^\top \widehat \beta _\I ) ^2  
\le   \sum_{m=1}^{M+1 } \sum_{t  \in \J_m   }(y_t  -   X_t ^\top \beta_t ^*   ) ^2  
 +   (M +1) C_1   \s\log(pn)  
+M  \zeta  .
  \label{eq:regression three change points step 2 term 1}
\end{align}
\Cref{eq:regression three change points step 2 term 1}  gives 
\begin{align} 
 \sum_{m=1}^{M+1}  \sum_{t  \in \J_m   }(y_t  - X_t ^\top  \widehat \beta _\I ) ^2 
 \le& \sum_{m=1}^{M+1} \sum_{t  \in \J_  m   }(y_t  -  X_t ^\top    \beta^*  _{\J_m  }  ) ^2    
 +  (M +1) C_1   \s\log(pn)  
+M  \zeta     \label{eq:regression three change points first}.   
\end{align}

 {\bf Step 3.} Observe that  $|\J_m |\ge \zeta  $ for $m\in \{ 2,\ldots, M\}$.  Using the same argument that leads to \Cref{eq:regression two change points step 3 term 1}, it follows that  
\begin{align} \label{eq:regression three change point step 3 last term}\sum_{t  \in \J_m   }(y_t  - X_t ^\top  \widehat \beta _\I ) ^2  - \sum_{t  \in \J_m   }(y_t  - X_t ^\top    \beta^*  _{\J_m }  ) ^2 
\ge \frac{\Lambda_{\min} (\Sigma) |\J_m   | }{8}  |\widehat \beta _\I  -  \beta^* _{\J_ m  }  |_2^2  -   C_3 \s  \log(np)   . 
\end{align} 
{\bf Step 4.} For  $\ell\in \{1,M+1\}$, if  $ |\J_\ell| \ge \zeta   $, then similar to \Cref{eq:regression three change point step 3 last term}, 
\begin{align*}  \sum_{t  \in \J_\ell    }(y_t  - X_t ^\top  \widehat \beta _\I ) ^2  - \sum_{t  \in \J_\ell    }(y_t  - X_t ^\top    \beta^*  _{\J_ \ell }  ) ^2 
\ge \frac{\Lambda_{\min} (\Sigma) |\J_\ell  | }{8}  |\widehat \beta _\I  -  \beta^* _{\J_ \ell  }  |_2^2  -   C_3 \s  \log(np)  
\ge   -   C_3 \s  \log(np)   . 
\end{align*} 
where the first inequality follows from the same arguments that lead to 
   \Cref{eq:regression one change step 3 term 1} and  \Cref{eq:regression one change step 2 term 3}.
   \\
    If  $ |\J_\ell| \le \zeta   $, then using the same argument that leads to \Cref{eq:regression change point step 4 last item}, it follows that 
    \begin{align*} & \sum_{t  \in \J_\ell    }(y_t  - X_t ^\top  \widehat \beta _\I ) ^2  - \sum_{ t  \in \J_ \ell    }(y_t  - X_t ^\top    \beta^*  _{\J_ \ell }  ) ^2 \nonumber
\ge -2C_4 \zeta.
\end{align*} 
Putting the two cases together, it follows that for $\ell\in \{1,M+1\}$,
\begin{align}   \sum_{t  \in \J_\ell    }(y_t  - X_t ^\top  \widehat \beta _\I ) ^2  - \sum_{ t  \in \J_ \ell    }(y_t  - X_t ^\top    \beta^*  _{\J_ \ell }  ) ^2 
\ge -C_3\s\log(pn) -2C_4 \zeta. \label{eq:regression three change point step 3 last term 2}
\end{align}  

{\bf Step 5.} 
\Cref{eq:regression three change points first}, \Cref{eq:regression three change point step 3 last term} and \Cref{eq:regression three change point step 3 last term 2} together imply  that 
\begin{align} \label{eq:regression three change points step six} \sum_{ m  =2}^M \frac{\Lambda_{\min} (\Sigma) |\J_m   | }{8}  |\widehat \beta _\I  -  \beta^* _{\J_ m  }  |_2^2      \le C_5 M (\s \log(np)   + \zeta) .
\end{align} 
 For any $  m \in\{2, \ldots, M\}$, it holds that
\begin{align}  \label{eq:regression three change points signal lower bound} \inf_{  \beta  \in \mathbb R^p   }  |\J_{m-1} |   |   \beta - \beta^* _{\J_ {m-1}   }   |  ^2 + |\J_{m} |  |   \beta - \beta^* _{\J_ m  }  |  ^2   =&     \frac{|\J_{m-1}| |\J_m|}{ |\J_{m-1}| +  |\J_m| } \kappa_m ^2  \ge \frac{1}{2} \Delta \kappa^2,
\end{align}  
where the last inequality follows from the assumptions that $\eta_k - \eta_{k-1}\ge \Delta  $ and $ \kappa_k \ge \kappa$  for all $1\le k \le K$. So
\begin{align}    \nonumber  2 \sum_{ m=2}^{M   }   |\J_m | |\widehat \beta _\I  -  \beta^* _{\J_ m  }  |_2^2   
\ge    &   \sum_{m=3}^{M}   \bigg(  |\J_{m-1}   |\widehat \beta _\I  -  \beta^* _{\J_ {m-1}   }  |_2^2      + |\J_m | |\widehat \beta _\I  -  \beta^* _{\J_ m  }  |_2^2    \bigg) 
\\
 \label{eq:regression three change points signal lower bound two} 
 \ge&   (M-2)  \frac{ 1}{2} \Delta \kappa^2  \ge \frac{M}{6} \Delta \kappa^2 ,
\end{align} 
where the second inequality follows from  \Cref{eq:regression three change points signal lower bound} and the last inequality follows from $M\ge 3$. \Cref{eq:regression three change points step six} and  \Cref{eq:regression three change points signal lower bound two} together imply that 
\begin{align}\label{eq:regression three change points signal lower bound three} 
 \frac{M}{12} \Delta \kappa^2 \le    C_5  M\big(  \s   \log(np )  
+  \zeta  \big) .
\end{align}
Note that by Assumption \ref{assume:regression parameters}, we have 
\begin{align} \label{eq:regression three change points snr}
\Delta \kappa^2 \ge C _{snr}     \big( \s \log(np) +  \zeta ).
\end{align} 
\Cref{eq:regression three change points signal lower bound three} contradicts \Cref{eq:regression three change points snr} for sufficiently large $C_{snr}$.

 \end{proof} 
 
 \begin{lemma} 
With probability  at least $1- cn^{-3}$,    there is no  two consecutive intervals $\I_1= [s,t) \in \widehat {\mathcal  P}  $,   $ \I_2=[t, e)  \in   \widehat  {\mathcal  P}  $    such that $\I_1 \cup \I_2$ contains no change points.  
 \end{lemma} 
 \begin{proof} 
 For contradiction, suppose that 
 $$ \I =\I_1\cup  \I_2 $$
 contains no change points. If $|\I |\le \zeta  $, then 
 $ \mathcal H(\I) =\mathcal H(\I_1)=\mathcal H(\I_2) =0$. 
 So $$ \mathcal H(\I_1) +\mathcal H(\I_2) +  \zeta  \ge \mathcal H(\I).$$
 This is a contradiction to the assumption that $\I_1 \in \hatp$ and  $\I_2 \in \hatp  $.
For the rest of the proof, assume that $|\I|\ge \zeta  $.  
Therefore,  by \Cref{eq:regression one change deviation bound}, it follows that 
 \begin{align}  \label{eq:no change points I}
 \bigg|   \mathcal H(\I  )   -  \sum_{t \in \I  } (y_t  - X_t ^\top  \beta^*  _{t } )^2 \bigg|  = \bigg|   \sum_{t \in \I  } (y_t  - X_t ^\top \widehat \beta _{\I  } )^2- \sum_{ t \in \I   } (y_t  - X_t ^\top \beta_{\I }^*   )^2    \bigg| \le  C_1 \s \log(np)  .
 \end{align}

 {\bf Case a.} Suppose $|\I_1| \ge  \zeta   $ and $|\I_2| \ge  \zeta   $. Let $\ell \in \{ 1,2\} $. 
Then 
 by \Cref{eq:regression one change deviation bound}, it follows that 
 \begin{align} \label{eq:no change points I 1}
 \bigg|   \mathcal H(\I_\ell )   -  \sum_{t \in \I_\ell } (y_t  - X_t ^\top  \beta^*  _{t } )^2 \bigg|  = \bigg|   \sum_{t \in \I_\ell } (y_t  - X_t ^\top \widehat \beta _{\I_\ell } )^2- \sum_{ t \in \I_\ell } (y_t  - X_t ^\top \beta_{\I_\ell}^*   )^2    \bigg| \le  C_1 \s \log(np)  .
 \end{align}
 So by \Cref{eq:no change points I} and \eqref{eq:no change points I 1} 
 \begin{align*}  
  \mathcal H(\I) \le \sum_{t \in \I  } (y_t  - X_t ^\top  \beta^*  _{t } )^2 + C_1 \s\log(pn)  \le \mathcal H(\I_1) + \mathcal H(\I_2) +3C_1\s \log(pn)\le \mathcal H(\I_1) + \mathcal H(\I_2) + \zeta .
 \end{align*} 
  This is a contradiction to the assumption that $\I_1 \in \hatp$ and  $\I_2 \in \hatp  $.
  
   {\bf Case b.} Suppose $|\I_1| <  \zeta   $ and $|\I_2| <  \zeta   $. 
Then 
 by \Cref{eq:regression one change deviation bound}, it follows that 
 \begin{align} \label{eq:no change points I 2}
 \bigg|   \mathcal H(\I_\ell )   -  \sum_{t \in \I_\ell } (y_t  - X_t ^\top  \beta^*  _{t } )^2 \bigg| =\bigg|   \mathcal H(\I_\ell )   -   0  \bigg|   = \bigg|     \sum_{t \in \I_\ell }  \epsilon_t ^2 \bigg|  \le \frac{\zeta}{4},
 \end{align}
where the last inequality follows from  \Cref{eq:two change points small intervals epsilon}.
So by \Cref{eq:no change points I} and \eqref{eq:no change points I 2} 
 \begin{align*}  
  \mathcal H(\I) \le \sum_{t \in \I  } (y_t  - X_t ^\top  \beta^*  _{t } )^2 + C_1 \s\log(pn)  \le \mathcal H(\I_1) + \mathcal H(\I_2) + C_1\s \log(pn) + \frac{\zeta}{2} \le \mathcal H(\I_1) + \mathcal H(\I_2) +  \zeta ,
 \end{align*} 
 where $\zeta \ge  C_\zeta ( \s \log(pn) )$ is used in the last inequality.  
  This is a contradiction to the assumption that $\I_1 \in \hatp$ and  $\I_2 \in \hatp  $. 
 
  {\bf Case c.} Suppose $|\I_1| <  \zeta   $ and $|\I_2| \ge  \zeta   $. 
Then  \Cref{eq:no change points I 1} holds for $\ell=2$ and  \Cref{eq:no change points I 2} holds for $\ell=1$. So
 \begin{align*}  
  \mathcal H(\I) \le \sum_{t \in \I  } (y_t  - X_t ^\top  \beta^*  _{t } )^2 + C_1 \s\log(pn)  \le \mathcal H(\I_1) + \mathcal H(\I_2) + 2 C_1\s \log(pn) +  \frac{\zeta}{4} \le \mathcal H(\I_1) + \mathcal H(\I_2) + \zeta.
 \end{align*} 
  This is a contradiction to the assumption that $\I_1 \in \hatp$ and  $\I_2 \in \hatp  $. 
  
   {\bf Case d.} Suppose $|\I_1| \ge  \zeta   $ and $|\I_2| <  \zeta   $.  This is the same as {\bf Case c}.
 \end{proof}

   \subsubsection[]{Proof of \Cref{prop-2}}

\begin{proof}[Proof of \Cref{prop-2}] 
Denote $S^*_n = \sum_{t=1}^n (y_t - X_t^{\top} \beta^*_t)^2$.  Given any collection $\{t_1, \ldots, t_m\}$, where $t_1 < \cdots < t_m$, and $t_1 = 0$, $t_{m+1} = n$, let 
	\begin{equation}\label{eq-sn-def}
		S_n(t_1, \ldots, t_{m}) = \sum_{k=1}^{m}  \mathcal H([t_k, t_{k+1})) . 
	\end{equation}
	For any collection of time points, when defining \eqref{eq-sn-def}, we will assume that  the time points are sorted in an increasing order.

Let $\{ \widehat \eta_{k}\}_{k=1}^{\widehat K}$ denote the change points induced by $\widehat {\mathcal P}$.  It suffices to   justify that 
	\begin{align}
		S^*_n + K \zeta   \ge  &S_n(\eta_1,\ldots,\eta_K)   + K \zeta   - C_a ( K+1)  \s \log(pn) \label{eq:K consistency step 1} \\ 
		\ge & S_n (\widehat \eta_{1},\ldots, \widehat \eta_{\widehat K } ) +\widehat  K  \zeta   - C_a(K+1) \s \log(pn)  \label{eq:K consistency step 2} \\ 
		\ge   &   S_n ( \widehat \eta_{1},\ldots, \widehat \eta_{\widehat K } , \eta_1,\ldots,\eta_K ) - C_a(K+1) \s \log(pn) + \widehat K   \zeta     - C_a(K+1) \s \log(pn)   \label{eq:K consistency step 3}
	\end{align}
	and that 
	\begin{align}\label{eq:K consistency step 4}
		S^*_n  -S_n ( \widehat \eta_{1},\ldots, \widehat \eta_{\widehat K } , \eta_1,\ldots,\eta_K ) \le   C _b(K + \widehat{K} + 2)\s\log(pn) . 
	\end{align}
	This is because if $\widehat K \ge K+1 $, then
	\begin{align*}
		C_b (4K +2 ) \s\log(pn) \ge&  C _b(K + \widehat{K} + 2) \s\log(pn)  
		\\ \geq  & S^*_n  -S_n ( \widehat \eta_{1},\ldots, \widehat \eta_{\widehat K } , \eta_1,\ldots,\eta_K ) \\
		 \geq &   - 2  C_a  (K+1) \s \log(pn) +(\widehat K -K) \zeta  
		\\
		\ge & - 2  C_a  (K+1) \s \log(pn) +   \zeta,
	\end{align*} 
where the first inequality holds because $| \hatp|  =\widehat K\le 3K $,the second inequality follows from  \Cref{eq:K consistency step 4}, the third inequality follows from  \Cref{eq:K consistency step 3}, and the last inequality holds because $\widehat K \ge K+1 $.    This would imply 
$$\zeta\le C_b(4K+2)\s\log(pn) + 2 C_a(K+1)\s\log(pn) ,$$
which is a contradiction since $\zeta   =  C_\zeta (\s\log(pn))^{2/\gamma -1 }\ge C_\zeta \s\log(pn)$ for sufficient large $C_{\zeta}$. 

{\bf Step 1.}
Note that \eqref{eq:K consistency step 1}  is implied by 
	\begin{align}\label{eq:step 1 K consistency}  
		\left| 	S^*_n   -   S_n(\eta_1,\ldots,\eta_K)    \right| \le  C_a (K  +1) \s \log(pn)  ,
	\end{align}
	which is an immediate consequence  of \Cref{eq:regression one change deviation bound} and the assumption that 
	$ \Delta \ge C_{snr}(\s\log(pn))^{2/\gamma -1 }$.  

	{\bf Step 2.}
	Since $\{ \widehat \eta_{k}\}_{k=1}^{\widehat K}$ are the change points induced by $\widehat {\mathcal P}$, \eqref{eq:K consistency step 2} holds because $\hatp$ is a minimiser.

{\bf Step 3.}
For any $\widehat \I =(s,e]\in \hatp$, denote
	\[
		\widehat \I  =  [s ,\eta_{p+1})\cup \ldots \cup  [\eta_{p+q} ,e)  = \J_1 \cup \ldots  \cup \J_{q+1},
	\]
	where $\{ \eta_{p+l}\}_{l=1}^{q+1}  =\widehat \I \ \cap \ \{\eta_k\}_{k=1}^K$. 
	Then \eqref{eq:K consistency step 3} is an immediate consequence of the following inequality
	\begin{align} \label{eq:one change point step 3 1}
		\mathcal H(\widehat \I )   \ge \sum_{l=1}^{q+1}   \mathcal  H(\J_l ) - C_2(q+1) \s\log(pn) .
	\end{align}
	Note that since $ \J_l \subset \widehat \I$, it follows that if $ |\widehat \I| \le \zeta  $, then $\mathcal H(\widehat \I ) = \mathcal H(\widehat \J_l)  =0 $ and \Cref{eq:one change point step 3 1} trivially holds. 
	Therefore it suffice to assume that 
	$$   |\widehat \I| \ge \zeta  = C_\zeta(\s\log(pn))^{2/\gamma -1}. $$
	In this case, \Cref{eq:one change point step 3 1} is implied by \begin{align} \label{eq:one change point step 3}
		\sum_{t\in \widehat \I }(y_t -X_t^{\top} \widehat \beta _{\widehat \I} )^2  \ge \sum_{l=1}^{q+1}   H(\J_l )   - C_2'(q+1) \s\log(pn) .
	\end{align}
	For any $ l\in \{ 1,\ldots, q+1\}$, there are two cases.
	\\
	{\bf Case 1.} $|\J_l|\le \zeta   $. Then 
	$$  \mathcal H(\J_l ) = 0 \le  \sum_{ t \in \widehat \I }(y_t- X_t^\top \widehat  \beta_{\widehat \I })^2 + C_3 \s\log(pn).  $$
	
	{\bf Case 2.} $|\J_l|> \zeta   $.
	Then
	\begin{align*} 
		\mathcal H(\J_l )  =   \sum_{t \in \J_l}(y_t -X_t^{\top}  \widehat \beta _{\J_l})^2 & \le   \sum_{t \in \J_l}(y_t -X_t^{\top}    \beta^*_t  )^2  +  C_4\s\log(pn)  \nonumber \\
		&    \le   \sum_{t \in \J_l}(y_t -X_t^{\top}    \widehat  \beta_{\widehat \I }   )^2  +  C_5\s\log(pn) ,
	\end{align*} 
	where  the first inequality follows from \Cref{eq:regression one change deviation bound}, and the second inequality follows from  \Cref{lem-case-5-prop-2-needed}.
	Therefore  for any $ l\in \{ 1,\ldots, q+1\}$ 
	$$  \mathcal H(\J_l )  \le  \sum_{ t \in \widehat \I }(y_t- X_t^\top \widehat  \beta_{\widehat \I })^2 + C_6\s\log(pn),  $$
	and this immediately implies  \Cref{eq:one change point step 3}.

{\bf Step 4.}
Finally, to show \eqref{eq:K consistency step 4}, observe that from \eqref{eq:step 1 K consistency}, it suffices to show that 
	\[
		S_n(\eta_1,\ldots,\eta_K)  -  S_n ( \widehat \eta_{1},\ldots, \widehat \eta_{\widehat K } , \eta_1,\ldots,\eta_K )  \le C_7 (K+\widehat K+1) \s \log(pn).
	\]
	This follows from a similar but simpler argument to {\bf Step 3}.
\end{proof}  

\subsubsection{Additional results related to DUDP}

Given regression data $\{ y_t, X_t\}_{t=1}^n$, we study the nonasymptotic properties of the lasso estimator $\widehat \beta_ \I$ defined in \eqref{eq:interval lasso} for a generic interval  $\mathcal I \subset (0, n] $.
 
\begin{lemma} \label{lemma:interval lasso} 
 Let $\widehat \beta_ \I $ be defined as in  \eqref{eq:interval lasso} with 
    $\lambda = C_\lambda \sqrt { \log(pn) }$ for some sufficiently large constant $C _ \lambda > 0$. Under Assumptions \ref{assume:regression parameters}, \ref{assume:X} and \ref{assume:epsilon}
for any $ \mathcal I  \subset (0, n] $   such that 
$ | \mathcal I | \ge   \zeta$, it holds uniformly with probability at least $1 - cn^{-5}$ that 
\begin{align} \label{eq:lemma:interval lasso term 1}
&    | \widehat \beta_\I -\beta^*_\I  | _2^2 \le \frac{C_1 \s\log(pn)}{ |\I|  },
 \\\label{eq:lemma:interval lasso term 2}
 &     | \widehat \beta_\I  -\beta^*_\I  | _1  \le     C_2  \s\sqrt {  \frac{\log(p n )} {|\I|  } },
 \\\label{eq:lemma:interval lasso term 3}
 &     | (\widehat \beta_\I  )_{S^c} | _1 \le  3  | (\widehat \beta _\I  -\beta^*_\I )_{S }  | _1,
\end{align} 
where $ \beta^*_\I  = \frac{1}{|\I|} \sum_{  t \in \I } \beta^*_t  $, $S $ is union of the  supports of $\beta_t^* $ for $t\in \I$, i.e.~$S = \bigcup_{t \in \I}S_t$, and $C_1, C_2, c > 0$ are absolute constants depending on $C_\epsilon$, $D_{\epsilon}$, $C_X$, $D_{X}$, $C_{\kappa}$ and $C_{\lambda}$.
 \end{lemma}

\begin{lemma} \label{eq:regression one change deviation bound}
Let $\mathcal I =[s,e) $ be any generic interval.  
  If $\I$ contains no change points and that 
$|\I| \ge \zeta =  C_\zeta    ( \s \log(pn) )^{2/\gamma -1 }$, then it holds that 
$$\p \bigg( \bigg|   \sum_{ t   \in \I } (y_ t   - X_t  ^\top \widehat  \beta _\I   )^2  - \sum_{ t    \in \I } (y_t       -  X_t ^\top \beta^*_\I    )^2  \bigg|  \ge C    \s  \log(np)  \bigg)  \le  n^{-5}, $$
where $  \beta _\I^*  = \frac{1}{|\I | } \sum_{t \in \I } \beta ^* _t  $ and $C > 0$ is an absolute constant.
\end{lemma}
\begin{proof} 
   Note that 
\begin{align}  \nonumber 
   \bigg|  \sum_{ t     \in \I } (y_ t      - X_t   ^\top \widehat \beta_\I   )^2  - \sum_{ t     \in \I } (y_ t      - X_t  ^\top \beta^* _t      )^2  \bigg|   
 = &\bigg| \sum_{t \in \I } \big\{  X_t  ^\top  ( \widehat \beta_\I  -   \beta^* _t  )   \big\} ^2 -  2 \sum_{i\in \I }  \epsilon_t  X_t ^\top ( \widehat \beta_\I   -  \beta^*_t  ) \bigg| 
 \\ \label{eq:regression one change point deviation bound term 1}
 \le &   \sum_{t \in \I } \big\{  X_t^\top  ( \widehat \beta_\I  -    \beta^*_\I  )   \big\}   ^2 
 \\   \label{eq:regression one change point deviation bound term 3}
 &+ 2 \bigg|    \sum_{t \in \I }  \epsilon_t X_t^\top ( \widehat \beta_\I   -  \beta^*_\I ) \bigg|  .
\end{align}

Suppose all the good events in \Cref{lemma:interval lasso} holds. 

{\bf Step 1.} By \Cref{lemma:interval lasso},   $\widehat \beta_\I  -    \beta^*_\I  $ satisfies the cone condition that 
$$  | (\widehat \beta_\I  -    \beta^*_\I  )_{S^c} |_1 \le 3  | (\widehat \beta_\I  -    \beta^*_\I )_S   |_1   .$$
It follows from \Cref{theorem:RES Version II} that with probability at least $ 1-n^{-5}$,
\begin{align*} 
 \bigg| \frac{1}{|\I| } \sum_{ t \in \I } \big\{  X_t^\top  ( \widehat \beta_\I  -    \beta^*_\I  )   \big\} ^2    - ( \widehat \beta_\I  -    \beta^*_\I    )^\top \Sigma  ( \widehat \beta_\I  -    \beta^*_\I  )  \bigg|  \le C_1  \sqrt { \frac{\s \log(pn) }{|\I| }} | \widehat \beta_\I  -    \beta^*_\I     |_2 ^2  .
\end{align*}
The above display   gives
\begin{align*}    \bigg| \frac{1}{|\I| } \sum_{t \in \I } \big\{  X_t^\top  ( \widehat \beta_\I  -    \beta^*_\I  )    \big\}^2 \bigg| \le &  \Lambda_{\max}(\Sigma)   | \widehat \beta_\I  -    \beta^*_\I     |_2 ^2   + C_1  \sqrt { \frac{\s \log(pn) }{|\I| }} | \widehat \beta_\I  -    \beta^*_\I     |_2 ^2 
\\
\le & \Lambda_{\max}(\Sigma)  | \widehat \beta_\I  -    \beta^*_\I     |_2 ^2 + C_1  \sqrt { \frac{\s \log(pn) }{C_\zeta\s \log(pn)  }} | \widehat \beta_\I  -    \beta^*_\I     |_2 ^2   
\\
\le &  \frac{ C_2\s\log(pn)}{|\I| }  , \end{align*} 
where the second inequality follows from the assumption that  $|\I| \ge  C_\zeta ( \s \log(pn) )^{2/\gamma -1 } \ge C_\zeta\s \log(pn) $ when $ \gamma < 1$, and the last inequality follows from  \Cref{eq:lemma:interval lasso term 1}  in \Cref{lemma:interval lasso} and $\Lambda_{\max}(\Sigma) < \infty$.
This gives 
$$    \bigg|   \sum_{t \in \I } \big\{  X_t^\top  ( \widehat \beta_\I  -    \beta^*_\I  )    \big\}^2 \bigg| \le  2 C_2\s\log(pn) . $$

{\bf Step 2.} It holds with probability at least $ 1-n^{-5}$ that
\begin{align*}  \nonumber  \bigg|\sum_{ t  \in \I   } \epsilon_t   X_t ^\top (  \widehat \beta _\I  -  \beta^* _{\I  }) \bigg|
\le 
&
       \bigg| \sum_{ t  \in \I   } \epsilon_t   X_t ^\top \bigg| _\infty  |   \widehat \beta _\I  -  \beta^* _{\I  } | _1 \le C_3   \sqrt { \log(np)    |\I | }  \s\sqrt {  \frac{\log(p n )} {|\I|  } } 
      \\
=  &  C_3    \s \log(pn),
\end{align*}
where the second inequality follows from \Cref{lemma:lasso deviation bound 1} with $|\I| \ge  C_\zeta ( \s \log(pn) )^{2/\gamma -1 }$ and  \Cref{eq:lemma:interval lasso term 2}  in \Cref{lemma:interval lasso}.

 The desired result follows by putting {\bf Step 1} into \Cref{eq:regression one change point deviation bound term 1} and {\bf Step 2} into \Cref{eq:regression one change point deviation bound term 3}.
\end{proof}

\begin{lemma} \label{lem-case-5-prop-2-needed}
Under the assumptions and notation in \Cref{prop-2}, suppose there exists no true change point in the interval $\I$.   For any interval $\mathcal J \supset \I$ with 
$$ | \I | \ge \zeta  = C_\zeta ( \s \log(pn) )^{2/\gamma -1 } , $$
	  it holds that with probability at least $1 -  2n^{-5}$,
	$$
		\sum_{t \in \I} (y_t - X _t^{\top} \beta^*_\I)^2 - \sum_{t \in \I} (y_t - X_t^{\top} \widehat{\beta} _\J)^2 \leq C \s \log(pn).
	$$
\end{lemma}
 
\begin{proof} 
	Let $\Theta  = \beta^*_\I  -\widehat \beta _\J$. Then with probability at least $ 1-n^{-5}$, we have
	\begin{align}
		   \sum_{t \in I} (X_t^{\top} \Theta  )^2  \geq & \frac{\Lambda_{\min} (\Sigma) }{2}  |\I|   |\Theta    |_2 ^2  - C_1  \log(pn)  |\I|^{1-\gamma } |\Theta  |_1^2  \nonumber \\
		\geq &  \frac{\Lambda_{\min} (\Sigma) }{2}   |\I|  |\Theta  |_2^2  -  2C_1  \log(pn)  |\I|^{1-\gamma } |\Theta _S   |_1 ^2 - 2C_1  \log(pn)  |\I|^{1-\gamma } |\Theta _{S^c } |_1^2  \nonumber \\
		\geq &  \frac{\Lambda_{\min} (\Sigma) }{2}   |\I|  |\Theta  |_2^2  -  2C_1\s    \log(pn)  |\I|^{1-\gamma } |\Theta _S   |_2 ^2 - 2C_1   \log(pn)  |\I|^{1-\gamma } | ( \widehat \beta _{\J} )   _{S^c } |_1^2 \nonumber \\ 
		\geq &  \frac{\Lambda_{\min} (\Sigma) }{4}   |\I|  |\Theta  |_2 ^2 - 2C_1   \log(pn)  |\I|^{1-\gamma } | ( \widehat \beta _{\J} )   _{S^c } |_1^2
		\label{eq-lem16-pf-1-1}
	\end{align} 
	where the first inequality follows from \Cref{theorem:RES}, the  second inequality follows from the inequality that $(|a| + |b|)^2 \leq 2|a|^2+2|b|^2$, the third inequality follows from $(\beta_\I^* )_{S^c} = 0 $  and the  last inequality follows from $| \I | \ge  C_\zeta ( \s \log(pn) )^{2/\gamma -1 } \ge  C_\zeta(\s\log(pn) )^{1/\gamma}$ because $ \gamma<1$ and that $|\Theta_S|_2 \le |\Theta|_2$. 
Since $$|\J| \ge |\I| \ge  C_\zeta ( \s \log(pn) )^{2/\gamma -1 }, $$	
	it follows from  \Cref{lemma:interval lasso} that,
	\begin{align} | ( \widehat \beta _{\J} )   _{S^c } |_1  \le C_2 \s \sqrt { \frac{\log(pn) }{|\I| }}.\label{eq-lem16-pf-01},
	\end{align}  
So  \Cref{eq-lem16-pf-1-1} and \Cref{eq-lem16-pf-01} give  
\begin{align} \nonumber 
		   \sum_{t \in I} (X_t^{\top} \Theta  )^2  
		\geq  &  \frac{\Lambda_{\min} (\Sigma) }{4}   |\I|  |\Theta  |_2 ^2 - 2C_1   \log(pn)  |\I|^{1-\gamma }C_2^2   \s^2 \frac{\log(pn) }{|\I| }  
		\\\nonumber  
		\ge &  \frac{\Lambda_{\min} (\Sigma) }{4}   |\I|  |\Theta  |_2 ^2   -C_3  \frac{ \s^2 \log^2 (pn) }{|\I|^\gamma  }  
		\\
		\ge &   \frac{\Lambda_{\min} (\Sigma) }{4}   |\I|  |\Theta  |_2 ^2   -C_3    \s  \log  (pn)  \label{eq-lem16-pf-1},
	\end{align}  
	where the third inequality follows from $| \I | \ge  C_\zeta ( \s \log(pn) )^{2/\gamma -1 } \ge  C_\zeta(\s\log(pn) )^{1/\gamma}$. Similarly 
	\begin{align}\nonumber 
	  \sum_{t \in I} \epsilon_t X_t^{\top}\Theta  
	\leq & 2  \bigg |\sum_{t \in I }  X_t \epsilon_t   \bigg |_{\infty }  \bigg (   \sqrt {\s } | \Theta _S |_2  + |(\widehat \beta_\J)_{S^c} |_1  \bigg)   
	\\\nonumber 
	\le & C_4 \sqrt {|\I|   \log(pn)  }  \bigg (   \sqrt {\s } | \Theta   |_2  + | (\widehat \beta_\J)_{S^c} |_1  \bigg)   
	\\\nonumber 
	\le 
	&  \frac{\Lambda_{\min}(\Sigma) }{8} |\I| | \Theta|_2^2  +C_5  \s\log(pn) + C_4 \sqrt {|\I|   \log(pn)  }   |(\widehat \beta_\J)_{S^c} |_1  
	\\
	\le & \frac{\Lambda_{\min}(\Sigma) }{16} |\I| | \Theta|_2^2  +C_5  \s\log(pn) + C_6\s\log(pn),\label{eq-lem16-pf-2} 
	\end{align} 
	where the second inequality follows from \Cref{lemma:lasso deviation bound 1} and that $|\Theta_S|_2 \le |\Theta|_2$, the third inequality follows from the inequality $2ab \leq a^2 + b^2$, and the last inequality follows from  \Cref{eq-lem16-pf-01}. 
	 We then have from \Cref{eq-lem16-pf-1} and \Cref{eq-lem16-pf-2} that 
	\begin{align*}
	& \sum_{t \in \I} (y_t - X _t^{\top} \beta^*_\I)^2 - \sum_{t \in \I} (y_t - X_t^{\top} \widehat{\beta} _\J)^2 
	\\  =&  2 \sum_{t \in \I} \epsilon_t X_t^{\top}\Theta  - \sum_{t \in \I} (X_t^{\top}\Theta )^2 \\
	\leq & \frac{\Lambda_{\min}(\Sigma) }{8} |\I| | \Theta|_2^2  + 2 C_5  \s\log(pn) + 2 C_6\s\log(pn) - \frac{\Lambda_{\min} (\Sigma) }{4}   |\I|  |\Theta  |_2 ^2   +C_3    \s  \log  (pn) 
	\\
	\le &  C_7 \s\log(pn),
	\end{align*}
	 as desired. 
\end{proof}

\subsection[]{Proof of \Cref{lemmma:refinement_lasso}}
\begin{proof}[Proof of \Cref{lemmma:refinement_lasso}]
Since with probability at least $1 - cn^{-3}$, $\mathcal{E}_{\mathrm{DP}}$ defined in \eqref{eq:event_DP} holds, we condition on the $\mathcal{E}_{\mathrm{DP}}$ in the following proof. 
Let $l_k = \widehat{\eta}_{k}$ and $r_k = \widehat{\eta}_{k+1}$. We have that
\begin{align}\label{eq:local_err_preliminary}
    |l_k - \eta_{k}| \leq C_{\mathfrak{e}}\frac{\s \log(np) + \zeta}{\kappa_{k}^2} \;\; \text{and} \;\; |r_k - \eta_{k+1}| \leq C_{\mathfrak{e}}\frac{\s \log(np) + \zeta }{\kappa_{k+1}^2}.
\end{align}
By the definitions of the preliminary estimators given in \Cref{sec:algorithm}, we have
\begin{align*}
   \widehat  \beta_k  = \arg\min_{\beta \in \mathbb R^p }  \frac{1}{  l_k - r_k }\sum_{t =l_k }^ { r_k-1 }   (y_t  -X_t^\top \beta) ^2 +  \frac{\lambda}{\sqrt {l_k - r_k}  }  |\beta |_1.
\end{align*}
%It suffices to show that
%\begin{align*}
%    |\widetilde  \beta_k - \beta_{\eta_k}^*|_2^2 \leq C_2\alpha_n^{-1}\kappa_{k}^2.
%\end{align*}
By \Cref{assume:regression parameters} and \eqref{eq:local_err_preliminary}, it follows that
\begin{align}\label{eq:interv_leng_preliminary}
    l_k-r_k \geq \eta_{k+1} - \eta_{k} - C_{\mathfrak{e}}\frac{\s \log(np) + \zeta }{\kappa_{k}^2} - C_{\mathfrak{e}}\frac{\s \log(np) + \zeta }{\kappa_{k+1}^2} \geq \frac{\Delta}{2} \geq \zeta.
\end{align}
Therefore, by \eqref{eq:lemma:interval lasso term 1} in \Cref{lemma:interval lasso},
\begin{align}\label{eq:lasso_err_preliminary_1}
    \bigg| \widehat \beta_k - \frac{1}{l_k-r_k}\sum_{t = l_k}^{r_k-1}\beta_t^* \bigg|_2^2 \leq C_3\frac{\s\log(pn)}{l_k-r_k} \leq 2C_3\frac{\s\log(pn)}{\Delta} \leq C_4\alpha_n^{-1}\kappa_{k+1}^2,
\end{align}
and also
\begin{align}\label{eq:lasso_err_preliminary_3}
    \bigg| \widehat \beta_k - \frac{1}{r_k-l_k}\sum_{t = l_k}^{r_k-1}\beta_t^* \bigg|_2^2 \leq C_5\alpha_n^{-1}\kappa_{k}^2.
\end{align}
Moreover, by \eqref{eq:lemma:interval lasso term 2} in \Cref{lemma:interval lasso},
\begin{align}\label{eq:lasso_err_preliminary_l1_1}
    \bigg| \widehat \beta_k - \frac{1}{r_k-l_k}\sum_{t = l_k}^{r_k-1}\beta_t^* \bigg|_1 \leq C_6\s\sqrt{\frac{\log(pn)}{r_k-l_k}} \leq \sqrt{2}C_6\s\sqrt{\frac{\log(pn)}{\Delta}} \leq C_7\s^{1/2}\alpha_n^{-1/2}\kappa_{k+1},
\end{align}
and also
\begin{align}\label{eq:lasso_err_preliminary_l1_3}
    \bigg| \widehat \beta_k - \frac{1}{r_k-l_k}\sum_{t = l_k}^{r_k-1}\beta_t^* \bigg|_1 \leq C_8\s^{1/2}\alpha_n^{-1/2}\kappa_{k}.
\end{align}
Note that $[l_k,r_k)$ can contain $0$, $1$ or $2$ change points. The case that $[l_k, r_k)$ contains $2$ change points is analysed. The other cases are similar and simpler, and therefore omitted. Let $\eta_{k}$ and $\eta_{k+1}$ be the two change points contained in $[l_k,r_k)$, it follows that
\begin{align*}
    \beta_{\eta_k}^* - \frac{1}{r_k-l_k}\sum_{t = l_k}^{r_k-1}\beta_t^* =&  \beta_{\eta_k}^* - \frac{1}{r_k-l_k}\bigg(\sum_{t = l_k}^{\eta_{k}-1}\beta_{\eta_{k-1}}^* + \sum_{t = \eta_{k}}^{\eta_{k+1}-1}\beta_{\eta_{k}}^* + \sum_{t = \eta_{k+1}}^{r_k-1}\beta_{\eta_{k+1}}^*\bigg)\\
    =& \frac{\eta_{k-1}-l_k}{r_k-l_k}(\beta_{\eta_{k}}^* - \beta_{\eta_{k-1}}^*) + \frac{r_k-\eta_k}{r_k-l_k}(\beta_{\eta_{k}}^* - \beta_{\eta_{k+1}}^*).
\end{align*}
Therefore
\begin{align}\label{eq:lasso_err_preliminary_2}
    \bigg|\beta_{\eta_k}^* - \frac{1}{r_k-l_k}\sum_{t = l_k}^{r_k-1}\beta_t^*\bigg|_2^2 
    \leq& 2\frac{(\eta_{k-1}-l_k)^2}{(r_k-l_k)^2}(\beta_{\eta_{k}}^* - \beta_{\eta_{k-1}}^*)^2 + 2\frac{(r_k - \eta_{k})^2}{(r_k-l_k)^2}(\beta_{\eta_{k}}^* - \beta_{\eta_{k+1}}^*)^2 \nonumber\\
    \leq& 8\frac{(\s \log(np) + \zeta)^2}{\kappa_{k}^4\Delta^2}\kappa_{k}^2 + 8\frac{(\s \log(np) + \zeta)^2}{\kappa_{k+1}^4\Delta^2}\kappa_{k+1}^2 \nonumber\\
    \leq& 16\frac{(\s \log(np) + \zeta)^2}{C_{snr}^2\alpha_n^2(\s\log( np))^{4/\gamma -2}}\kappa_{k+1}^2\nonumber\\
    =& C_9\alpha_n^{-2}\kappa_{k+1}^2,
\end{align}
where the second inequality follows from \eqref{eq:local_err_preliminary} and \eqref{eq:interv_leng_preliminary}, and the third inequality follows from \Cref{assume:regression parameters}. More specifically, by \Cref{assume:regression parameters},
$$\Delta \ge C_{snr} \max_{k \in \{1, \dots, K\}}\frac{\alpha_n(\s\log( np))^{2/\gamma -1}}{\kappa_k^2},$$
which further leads to that
$$\Delta^2 \ge C_{snr}^2\frac{\alpha_n^2(\s\log( np))^{4/\gamma -2}}{\kappa_{k+1}^2\kappa_{k}^2} \;\; \text{and} \;\; \Delta^2 \ge C_{snr}^2\frac{\alpha_n^2(\s\log( np))^{4/\gamma -2}}{\kappa_{k+1}^4}.$$
By symmetry, we also have 
\begin{align}\label{eq:lasso_err_preliminary_4}
    \bigg|\beta_{\eta_k}^* - \frac{1}{r_k-l_k}\sum_{t = l_k}^{r_k-1}\beta_t^*\bigg|_2^2 
    \leq C_{10}\alpha_n^{-2}\kappa_{k}^2,
\end{align}
Note that combining \eqref{eq:lasso_err_preliminary_1} and \eqref{eq:lasso_err_preliminary_2} and combining \eqref{eq:lasso_err_preliminary_3} and \eqref{eq:lasso_err_preliminary_4} immediately lead to (i).

Similarly, for the bias, we have that
\begin{align}\label{eq:lasso_err_preliminary_l1_2}
    \bigg|\beta_{\eta_k}^* - \frac{1}{r_k-l_k}\sum_{t = l_k}^{r_k-1}\beta_t^*\bigg|_1 
    \leq& \frac{\eta_{k-1}-l_k}{r_k-l_k}\big|\beta_{\eta_{k}}^* - \beta_{\eta_{k-1}}^*\big|_1 + 2\frac{r_k - \eta_{k}}{r_k-l_k}\big|\beta_{\eta_{k}}^* - \beta_{\eta_{k+1}}^*\big|_1 \nonumber\\
    \leq& 2\sqrt{2}\frac{\s \log(np) + \zeta}{\kappa_{k}^2\Delta}\s^{1/2}\kappa_{k} + 2\sqrt{2}\frac{\s \log(np) + \zeta}{\kappa_{k+1}^2\Delta}\s^{1/2}\kappa_{k+1} \nonumber\\
    \leq& 4\sqrt{2}\frac{\s \log(np) + \zeta}{C_{snr}\alpha_n(\s\log( np))^{2/\gamma -1}}\s^{1/2}\kappa_{k+1}\nonumber\\
    =& C_{11}\alpha_n^{-1}\s^{1/2}\kappa_{k+1},
\end{align}
By symmetry, we also have 
\begin{align}\label{eq:lasso_err_preliminary_l1_4}
    \bigg|\beta_{\eta_k}^* - \frac{1}{r_k-l_k}\sum_{t = l_k}^{r_k-1}\beta_t^*\bigg|_1 
    \leq C_{12}\alpha_n^{-1}\s^{1/2}\kappa_{k},
\end{align}
Note that combining \eqref{eq:lasso_err_preliminary_l1_1} and \eqref{eq:lasso_err_preliminary_l1_2} and combining \eqref{eq:lasso_err_preliminary_l1_3} and \eqref{eq:lasso_err_preliminary_l1_4} immediately lead to (ii).
\end{proof}

\subsection[]{Proof of \Cref{lemma:consistency_jump_size_est}}
\begin{proof}[Proof of \Cref{lemma:consistency_jump_size_est}]
By definition,
\begin{align*}
    \frac{\widehat{\kappa}_k}{\kappa_k} - 1 =& \kappa_k^{-1} \big|\widehat{\beta}_{k} - \widehat{\beta}_{k-1}\big|_2 - \kappa_k^{-1}\big|\beta^*_{\eta_{k}} - \beta^*_{\eta_{k-1}}\big|_2 \leq \kappa_k^{-1} \big|\widehat{\beta}_{k} - \beta^*_{\eta_{k}}\big|_2 + \kappa_k^{-1}\big|\widehat{\beta}_{k-1} - \beta^*_{\eta_{k-1}}\big|_2,
\end{align*}
where the inequality follows from the triangle inequality. Applying \Cref{lemmma:refinement_lasso} immediately lead to the desired result.
\end{proof}

\subsection[]{Proof of \Cref{lemma:interval lasso}}

\begin{lemma} \label{lemma:beta bounded 1}
Denote $ \beta^*_\I  = \frac{1}{|\I|} \sum_{  i \in \I } \beta^*_i  $. Suppose   \Cref{assume:regression parameters} holds. Let $\I \subset (0,n] $.  Then 
   $$ |\beta^*_\I - \beta^*_ i |_2    \le CC_\kappa,$$
  for some absolute constant $C$ independent of $n$.
\end{lemma} 
 \begin{proof}
 It suffices to consider $\I =[1,n]$  and  $\beta^*_ i= \beta^*_1 $ as the general case is similar. Let $\Delta_k = \eta_{k}-\eta_{k-1}$. Observe that  
 \begin{align*}
   | \beta^*_{[1,n]}  - \beta_ 1^*   |_2  =& 
  \bigg |  \frac{1}{n} \sum_{t =1}^n \beta_t  ^*  - \beta_1^*  
   \bigg |_2 
 = \bigg |  \frac{1}{n} \sum_{k=0}^{K} \Delta_k \beta_{\eta_k+ 1} ^*   -    \frac{1}{n}\sum_{k=0}^{K} \Delta_k  \beta_1 ^*
  \bigg |_2  
  \\
 \le 
 &    \frac{1}{n} \sum_{k=0}^{K}  | \Delta_k  (\beta_{\eta_k+ 1} ^* -\beta_1 ^* )  |_2    \le 
 \frac{1}{n}\sum_{k=0}^{K} \Delta_k  (K+1)  C_\kappa   \le  (K+1)C_\kappa . 
  \end{align*}
  By \Cref{assume:regression parameters}, both $C_\kappa  $ and $K   $ are finite.
 \end{proof}
\begin{proof}[Proof of \Cref{lemma:interval lasso}]   It follows from the definition of $\widehat{\beta}_\I$ that 
	\[
	 \frac{1}{|\I | }	\sum_{  t  \in I} (y_ t  - X_t ^{\top}\widehat{\beta}_\I )^2 + \frac{ \lambda}{ \sqrt{  |\I|  }}     |\widehat{\beta} _\I |_1 \leq  
	 \frac{1}{|\I | } \sum_{t \in \I} (y_t  - X_t ^{\top}\beta^*_\I )^2 + \frac{ \lambda}{ \sqrt{  |\I|  }}   |\beta^*_\I   |_1.
	\]
	This gives 
	\begin{align*}
		 \frac{1}{|\I | } \sum_{t  \in \I} \bigl\{X_t  ^{\top}(\widehat{\beta}_\I - \beta^*_\I )\bigr\}^2 +    \frac{ 2 }{|\I | }   \sum_{t  \in \I}(y_t  - X_t ^{\top}\beta^*_\I)X_t ^{\top}(\beta^*_\I  - \widehat{\beta}_\I )  +  \frac{ \lambda}{ \sqrt{  |\I|  }}   \bigl |\widehat{\beta}_\I \bigr |_1 
		\leq  \frac{ \lambda}{ \sqrt{  |\I|  }}    \bigl |\beta^*_\I \bigr |_1,
	\end{align*}
	and therefore
	\begin{align}
		& \frac{1}{|\I | }	  \sum_{ t   \in \I} \bigl\{X_t  ^{\top}(\widehat{\beta} _\I  - \beta^*_\I )\bigr\}^2 + \frac{ \lambda}{ \sqrt{  |\I|  }} \bigl |\widehat{\beta}_\I \bigr |_1  \nonumber \\
		\nonumber
		 \leq  &  \frac{ 2}{|\I | }	 \sum_{ t  \in  \I } \epsilon_t   X_t  ^{\top}(\widehat{\beta}_\I  - \beta^*_\I ) 
+ 2   (\widehat{\beta}_\I  - \beta^*_\I )^{\top}\frac{  1 }{|\I | }	 \sum_{t \in   \I} X_ t   X_t ^{\top}(    \beta^*_t  -\beta^*_\I   )		 
		+  \frac{ \lambda}{ \sqrt{  |\I|  }}    \bigl |\beta^*_\I \bigr |_1  
		\\
		\le &  2 \bigg  | \frac{1}{|\I| }\sum_{t\in \I }  \epsilon_t X_t \bigg |_\infty  | \widehat \beta_\I -\beta^*_\I | _1  + 2 \bigg|  \frac{1}{|\I| }\sum_{ t   \in \I} X_ t  X_t  ^\top (\beta^*  _\I  -\beta^*_t  )   \bigg|_\infty | \widehat \beta_\I -\beta^*_\I | _1 +  \frac{ \lambda}{ \sqrt{  |\I|  }}    \bigl |\beta^*_\I \bigr |_1   
		 .	 \label{eq-lem10-pf-2}
	\end{align}
	By \Cref{lemma:lasso deviation bound 1}, with probability at least  $1-n^{-5}$, it holds that 
	$$ \bigg | \frac{1}{|\I | }\sum_{t\in |\I | }  X_{t } \epsilon_t  \bigg |_\infty  \le  C _3 \bigg (   \sqrt { \frac{\log(pn) }{|\I| }  }  +   \frac{ \log^{1/\gamma }(np) }{|\I| }  \bigg)  \le C_3' \sqrt { \frac{\log(pn) }{|\I| }  }  \le \frac{\lambda }{ 8 \sqrt {|\I| }} ,$$
	where the  second   inequality follows from 
	$| \mathcal I | \ge  C_{\zeta}   (\s \log(pn)) ^{2/ \gamma  - 1}  $, and the third  inequality follows from the choice of $\lambda$. 
Observe that $\sum_{t\in \I} (\beta^*_\I  - \beta^*_t )=0 $  and that by \Cref{lemma:beta bounded 1},
$ | \beta^*_\I -\beta^* _t |_2 \le C_4 C_\kappa$.
	Therefore by \Cref{lemma: mis-specified deviation}, 
	with probability at east $1-n^{-5}$,
	$$\bigg|  \frac{1}{|\I| }\sum_{ t   \in \I} X_ t  X_t  ^\top (\beta^*  _\I  -\beta^*_t  )   \bigg|_\infty \le  C _4 \bigg (   \sqrt { \frac{\log(pn) }{|\I| }  }  +   \frac{ \log^{1/\gamma }(np) }{|\I| }  \bigg)  \le C_4' \sqrt { \frac{\log(pn) }{|\I| }  }   \le \frac{\lambda }{ 8 \sqrt {|\I| }}, $$ where the  second   inequality follows from 
	$| \mathcal I | \ge  C_{\zeta}   (\s \log(pn)) ^{2/ \gamma  - 1}  $, and the third  inequality follows from the choice of $\lambda$.  
 \\
	\\
	So   \eqref{eq-lem10-pf-2}  gives 
	\begin{align*}
		  \frac{1}{|\I | }	  \sum_{ t   \in \I} \bigl\{X_t  ^{\top}(\widehat{\beta} _\I  - \beta^*_\I )\bigr\}^2 + \frac{ \lambda}{ \sqrt{  |\I|  }} \bigl |\widehat{\beta}_\I \bigr |_1  
		 \leq  \frac{\lambda}{2\sqrt { |\I| } }    |\widehat{\beta}_\I  - \beta^*_\I  |_1  
		+  \frac{ \lambda}{ \sqrt{  |\I|  }}    \bigl |\beta^*_\I \bigr |_1   .	 
	\end{align*} 
 The above inequality and the fact that $|\beta | _1 =|\beta_S|_1 + |\beta_{S^c} |_1 $ for any $\beta \in \mathbb R^p$  imply that 
\begin{align}
\label{eq:two sample lasso deviation 1} \frac{1}{|\I|} \sum_{i \in \I } \big\{  X_t^\top   (\widehat  \beta _\I   - \beta^*_\I  )  \big\} ^2 + \frac{ \lambda}{2\sqrt{  |\I|  } } |  (\widehat \beta _\I)   _{ S ^c} |_1  
 \le & \frac{3\lambda}{2\sqrt{  |\I|  }  }  | ( \widehat  \beta _\I  -  \beta^*_\I  )    _{S}  | _1  .
 \end{align} 
 Let $\Theta = \widehat  \beta _\I   - \beta^*_\I   $.   \Cref{eq:two sample lasso deviation 1} and the observation $(\beta_t^*)_{S^c}  = 0  $
 also imply that 
$$ \frac{\lambda }{2} | \Theta _{S^c}   |_1 = \frac{ \lambda}{2 } |   (\widehat \beta_\I)  _{ S ^c} |_1    \le  
\frac{3\lambda}{2 }  |  ( \widehat  \beta _\I   - \beta^* _\I  ) _{S}  | _1  = \frac{3\lambda}{2 }  |  \Theta  _{S}  | _1 . $$
The above inequality gives 
$$ |\Theta_{S^c} |_1 \le 3 | \Theta _{S}|_1 .$$ Thus \Cref{theorem:RES Version II} gives 
$$ \frac{1}{|\I| } \sum_{i \in \I }  \left( X_t^\top  \Theta \right)^2   
\ge 
 \frac{\Lambda_{\min} (\Sigma) } {2 }  | \Theta|_2^2     .$$
Therefore
\begin{align}
\label{eq:two sample lasso deviation 2} \frac{\Lambda_{\min} (\Sigma) } {2 }  | \Theta|_2^2    + \frac{ \lambda}{2\sqrt{  |\I|  }  } | \Theta  _{ S ^c} |_1    =   \frac{\Lambda_{\min} (\Sigma) } {2 }  | \Theta|_2^2    + \frac{ \lambda}{2\sqrt{  |\I|  }  } | ( \widehat \beta_\I    ) _{ S ^c} |_1    
 \le    \frac{3\lambda}{2\sqrt{  |\I|  }  }  | \Theta_{S}  | _1  \le \frac{3\lambda \sqrt \s }{2 \sqrt{  |\I|  } }  | \Theta  | _2 ,
 \end{align}  
 where the first  equality follows from $ \Theta  _{ S ^c} = ( \widehat \beta_\I    ) _{ S ^c}  $, the first inequality follows from  \Cref{eq:two sample lasso deviation 1}, and the second inequality follows from the observation that $|S|\le K \s $ and $K<\infty$.  Note  that  \Cref{eq:two sample lasso deviation 2}  leads to 
\begin{align}
\label{eq:two sample lasso deviation 3} |\Theta |_2  \le   \frac{C _5 \lambda \sqrt  \s}{\sqrt{| \I |} } . \end{align} 
 This  and the fact that  $ | \Theta_{S}  | _1 \le \sqrt {|S|}   |\Theta|_2 \le \sqrt{K  \s } |\Theta|_2   $   also implies that 
 $   | \Theta_{S}  | _1 \le  \frac{C_6\lambda \s}{\sqrt{|\I| }}. $ 
 Since 
  $  | \Theta_{S^c }  | _1 \le 3  | \Theta_{S}  | _1 ,$ 
it also holds that 
$$ | \Theta   | _1 =  | \Theta_{S  }  | _1 +  | \Theta_{S^c }  | _1  \le 4  | \Theta_{S}  | _1 \le  \frac{4C_6\lambda \s}{\sqrt{|\I|} } .$$
  \end{proof}

\section{Proof for Section 4}\label{sec-app-lrv}
For $t \in [s_k, e_k)$, the corrected interval given in \eqref{eq:corrected boundaries}, recall that
\begin{align*}
    Z_t^*
    = \begin{cases}
        2\epsilon_t\Psi_k^{\top}X_t - (\Psi_k^{\top}X_t)^2, \quad& s_k \leq t \leq \eta_k-1,\\
        2\epsilon_t\Psi_k^{\top}X_t + (\Psi_k^{\top}X_t)^2, \quad& \eta_k \leq t \leq e_k-1.
    \end{cases}
\end{align*}
Given the unobserved process $\{Z_t^*\}_{t = s_k}^{e_k-1}$, we define the auxiliary long-run variance estimator as follows.
\begin{enumerate}
    \item[\textbf{Step 1.}] Divide interval $[s_k, e_k)$ into $2R$ integer intervals with a size $S$, and denote the $2R$ intervals as $\mathcal{S}_1, \dots, \mathcal{S}_{2R}$.
    \item[\textbf{Step 2.}] Compute the rescaled mean differences of $Z_t^*$'s from the odd interval $\mathcal{S}_{2i-1}$ and the even interval $\mathcal{S}_{2i}$, defined by
\begin{align*}
    D_i^* = \frac{1}{\sqrt{2S}}\bigg\{\sum_{t \in \mathcal{S}_{2i-1}}Z_t^* - \sum_{t \in \mathcal{S}_{2i}}Z_t^*\bigg\}.
\end{align*}
    \item[\textbf{Step 3.}] Compute the auxiliary long-run variance estimator based on the unobservable process $\{Z_t^*\}_{s_k+1}^{e_k}$, which is defined as
 \begin{align}\label{eq:lrv_auxiliary}
    \check{\sigma}_{\infty}^2 = \frac{1}{R}\sum_{i = 1}^R(\kappa_k^{-1}D_i^*)^2.
 \end{align}
\end{enumerate}

\begin{lemma}\label{lemma:auxiliary_lrv_consist}
    Let $\check{\sigma}_{\infty}^2$ be the auxiliary long-run variance estimator defined in \eqref{eq:lrv_auxiliary}. Suppose assumptions \ref{assume:regression parameters}, \ref{assume:X}, \ref{assume:epsilon} hold, and assume $S, R \to \infty$ and $R \gg S$ as $n \to \infty$. It holds that
\begin{align*}
    \check{\sigma}_{\infty}^2 \overset{P.}{\to} \sigma_{\infty}^2,
\end{align*}
where $\sigma_{\infty}^2$ is given in \eqref{eq:long-run_var} of \Cref{lemma:long-run_UB}.
\end{lemma}

\begin{proof}[Proof of \Cref{lemma:auxiliary_lrv_consist}]

{\bf Step 1: Bias.}
We consider the rescaled process $\{\kappa_k^{-1}D_i^*\}_{i \in \mathbb Z}$.
For notational convenience, denote
\begin{align*}
    F_t = v_k^{\top}X_tX_t^{\top}v_k - v_k^{\top}\Sigma v_k \;\; \text{and} \;\; G_t = \epsilon_tv_k^{\top}X_t.
\end{align*}
Note that $\{F_t\}_{t \in \mathbb{Z}}$ and $\{G_t\}_{t \in \mathbb{Z}}$ are stationary, and $\mathbb{E}[F_t] = 0$, $\mathbb{E}[G_t] = 0$ and $\mathbb{E}[F_{t_1}G_{t_2}] = 0$ for any $t_1, t_2 \in \mathbb{Z}$. For $l \in \mathbb{Z}$, denote the lag-$l$ autocovariances of $\{F_t\}_{t \in \mathbb{Z}}$ and $\{G_t\}_{t \in \mathbb{Z}}$ respectively as
\begin{align*}
    \gamma_{l}^F = \mathbb{E}[F_{t-l}F_t] \;\; \text{and} \;\; \gamma_{l}^G = \mathbb{E}[G_{t-l}G_t].
\end{align*}
Note that $\gamma_{l}^F = \gamma_{-l}^F$ and $\gamma_{l}^G = \gamma_{-l}^G$.
\
\\
\textbf{Case 1}.
If there is no change point in $\mathcal{S}_{2i-1}$ and $\mathcal{S}_{2i}$, the mean can be fully cancelled out. More specifically, for any $i$ such that $s_k + 2iS < \eta_k$, it follows that
\begin{align*}
   D_i^*
   =& -\frac{\kappa_k^2}{\sqrt{2S}}\bigg\{\sum_{t \in \mathcal{S}_{2i-1}}F_t - \sum_{t \in \mathcal{S}_{2i}}F_t\bigg\}+\sqrt{\frac{2}{S}}\kappa_k\bigg\{\sum_{t \in \mathcal{S}_{2i-1}}G_t - \sum_{t \in \mathcal{S}_{2i}}G_t\bigg\},
\end{align*}
and for any $i$ such that $s_k + (2i-1)S \geq \eta_k$ that
\begin{align*}
   D_i^* =& \frac{\kappa_k^2}{\sqrt{2S}}\bigg\{\sum_{t \in \mathcal{S}_{2i-1}}F_t - \sum_{t \in \mathcal{S}_{2i}}F_t\bigg\}+\sqrt{\frac{2}{S}}\kappa_k\bigg\{\sum_{t \in \mathcal{S}_{2i-1}}G_t - \sum_{t \in \mathcal{S}_{2i}}G_t\bigg\}.
\end{align*}
For both of the two above settings, we have that
\begin{align*}
    \mathbb{E}[(D_i^*)^2] =& \kappa_k^4\gamma^F_0 + \kappa_k^4\sum_{l = 1}^{S}\Big(2-\frac{3l}{S}\Big)\gamma_l^F - \kappa_k^4\sum_{l = S+1}^{2S-1}\Big(2-\frac{l}{S}\Big)\gamma_l^F\\
    &+ 4\kappa_k^2\gamma^G_0 + 4\kappa_k^2\sum_{l = 1}^{S}\Big(2-\frac{3l}{S}\Big)\gamma_l^G - 4\kappa_k^2\sum_{l = S+1}^{2S-1}\Big(2-\frac{l}{S}\Big)\gamma_l^G.
\end{align*}
\
\\
\textbf{Case 2}.
If there is a change point $\eta_k \in \mathcal{S}_{2i-1}$ or $\eta_k \in \mathcal{S}_{2i}$, the mean cannot be fully cancelled out, but as we required $R \gg S$ as $n \to \infty$, the bias due to change point is negligible. If $\eta_k \in \mathcal{S}_{2i-1}$, we have that
\begin{align*}
   D_i^*
   =& -\frac{\kappa_k^2}{\sqrt{2S}}\bigg\{\sum_{t \in [s_k+\{2i-2\}S, \eta_k)}F_t - \sum_{t \in [\eta_k, s_k+\{2i-1\}S)}F_t + \sum_{t \in \mathcal{S}_{2i}}F_t\bigg\}\\
   & +\sqrt{\frac{2}{S}}\kappa_k\bigg\{\sum_{t \in \mathcal{S}_{2i-1}}G_t - \sum_{t \in \mathcal{S}_{2i}}G_t\bigg\} - \frac{\sqrt{2}\{\eta_k - s_k-(2i-2)S\}}{\sqrt{S}}\kappa_k^2v_k^{\top}\Sigma v_k.
\end{align*}
If $\eta_k \in \mathcal{S}_{2i}$, we have that
\begin{align*}
   D_i^* 
   =& -\frac{\kappa_k^2}{\sqrt{2S}}\bigg\{\sum_{t \in \mathcal{S}_{2i-1}}F_t - \sum_{t \in [s_k+\{2i-1\}S, \eta_k)}F_t + \sum_{t \in [\eta_k, s_k+2iS)}F_t\bigg\}\\
   & +\sqrt{\frac{2}{S}}\kappa_k\bigg\{\sum_{t \in \mathcal{S}_{2i-1}}G_t - \sum_{t \in \mathcal{S}_{2i}}G_t\bigg\} - \frac{\sqrt{2}(\eta_k - s_k - 2iS)}{\sqrt{S}}\kappa_k^2v_k^{\top}\Sigma v_k.
\end{align*}

Without loss of generality, we focus on the \textbf{Case 2}, i.e.~the case that $\eta_k \in \mathcal{S}_{2j}$ for some $j \in \{1, \dots, R\}$. Let $h_1 = \eta_k - s_k - (2j-1)S$ and $h_2 = s_k+2jS - \eta_k$. Assume $h_2 \geq h_1$. The following figure depicts the definitions of $h_1$ and $h_2$.
\begin{center}
    \begin{tikzpicture}[scale=7,decoration=brace]
\draw[-, thick] (-0.1,0) -- (1.1,0);
\foreach \x/\xtext in {0/$s_k+2(j-1)S$,0.5/$s_k+(2j-1)S$,0.7/$\eta_k$,1/$s_k+2jS$}
    \draw[thick] (\x,0.5pt) -- (\x,-0.5pt) node[font = {\footnotesize}, below] {\xtext};
\draw[decorate, thick, yshift=0.6ex]  (0,0) -- node[font = {\footnotesize}, above=0.2ex] {$\mathcal{S}_{2j-1}$}  (0.5,0);
\draw[decorate, thick, yshift=0.6ex]  (0.5,0) -- node[font = {\footnotesize}, above=0.2ex] {$\mathcal{S}_{2j}$}  (1,0);
\draw (0.6,0.5pt) node[font = {\footnotesize}, above] {$h_1$};
\draw (0.85,0.5pt) node[font = {\footnotesize}, above] {$h_2$};
%\draw[[-), ultra thick, blue] (0,0) -- (0.2,0);
%\draw (-0.25,0) node {$x=1$};
\end{tikzpicture}
\end{center}
It follows that
\begin{align*}
    \mathbb{E}[(D_j^*)^2] =& \kappa_k^4\gamma^F_0 + \kappa_k^4\sum_{l = 1}^{h_1}\Big(2-\frac{5l}{S}\Big)\gamma_l^F + \kappa_k^4\sum_{l = h_1+1}^{h_2}\Big(2-\frac{4h_1+l}{S}\Big)\gamma_l^F + \kappa_k^4\sum_{l = h_2+1}^{S+h_1}\frac{l - 2h_1}{S}\gamma_l^F\\
    &+ \kappa_k^4\sum_{l = S+h_1+1}^{2S-1}\Big(2-\frac{l}{S}\Big)\gamma_l^F + 4\kappa_k^2\gamma^G_0 + 4\kappa_k^2\sum_{l = 1}^{S}\Big(2-\frac{3l}{S}\Big)\gamma_l^G - 4\kappa_k^2\sum_{l = S+1}^{2S-1}\Big(2-\frac{l}{S}\Big)\gamma_l^G\\
    &+ \frac{2h_2^2}{S}\kappa_k^2\big(v_k^{\top}\Sigma v_k\big)^2\\
    =& \kappa_k^4\gamma^F_0 + \kappa_k^4\sum_{l = 1}^{S}\Big(2-\frac{3l}{S}\Big)\gamma_l^F - \kappa_k^4\sum_{l = S+1}^{2S-1}\Big(2-\frac{l}{S}\Big)\gamma_l^F + 4\kappa_k^2\gamma^G_0 + 4\kappa_k^2\sum_{l = 1}^{S}\Big(2-\frac{3l}{S}\Big)\gamma_l^G\\
    &- 4\kappa_k^2\sum_{l = S+1}^{2S-1}\Big(2-\frac{l}{S}\Big)\gamma_l^G - \kappa_k^4\sum_{l = 1}^{h_1}\frac{2l}{S}\gamma_l^F - \kappa_k^4\sum_{l = h_1+1}^{h_2}\frac{4h_1-2l}{S}\gamma_l^F + \kappa_k^4\sum_{l = h_2+1}^{S}\Big(\frac{4l - 2h_1}{S} - 2\Big)\gamma_l^F\\
    &+ \kappa_k^4\sum_{l = S+1}^{S+h_1}\Big(2-\frac{ 2h_1}{S}\Big)\gamma_l^F + \kappa_k^4\sum_{l = S+h_1+1}^{2S-1}\Big(4-\frac{2l}{S}\Big)\gamma_l^F  + \frac{2h_2^2}{S}\kappa_k^2\big(v_k^{\top}\Sigma v_k\big)^2.
\end{align*}
For any $l > 0$, we construct the coupled random variables, which are independence of $F_{t-l}$ and $G_{t-l}$, as 
\begin{align*}
    F_{t, \{t-l, - \infty\}} = v_k^{\top}X_{t, \{t-l, - \infty\}}X_{t, \{t-l, - \infty\}}^{\top}v_k - v_k^{\top}\Sigma v_k
\end{align*}
and
\begin{align*}
    G_{t, \{t-l, - \infty\}} = \epsilon_{t, \{t-l, - \infty\}}v_k^{\top}X_{t, \{t-l, - \infty\}}.
\end{align*}
Hence, it follows that
\begin{align*}
    |\gamma_l^F| =& \big| \mathbb{E}[F_{t-l}(F_{t} - F_{t,\{t-l,-\infty\}})] \big| \leq \big\|(v_k^{\top}X_{t-l})^2\big\|_2\big\|(v_k^{\top}X_{t})^2 - (v_k^{\top}X_{t,\{t-l,-\infty\}})^2\big\|_2\\
    \leq& 2\big\|v_k^{\top}X_{t}\big\|_4^3\big\|v_k^{\top}(X_{t} - X_{t,\{t-l,-\infty\}})\big\|_4 \leq 2\sup_{|v|_2 = 1}\|v^{\top}X_{t}\|_4^3D_X\exp(-cl^{\gamma_1})\\
    =& K_F\exp(-cl^{\gamma_1}),
\end{align*}
where the first two inequalities follow from the triangle and H\"older's inequalities and the third inequality follows from \Cref{assume:X}. Similarly, by Assumptions \ref{assume:X} and \ref{assume:epsilon}, we have that
\begin{align*}
    |\gamma_l^G| =& \big| \mathbb{E}[G_{t-l}(G_{t} - G_{t,\{t-l,-\infty\}})] \big| \leq \big\|\epsilon_{t}v_k^{\top}X_{t}\big\|_2\big\|\epsilon_{t}v_k^{\top}X_{t} - \epsilon_{t, \{t-l, - \infty\}}v_k^{\top}X_{t, \{t-l, - \infty\}}\big\|_2\\
    \leq& \|\epsilon_t\|_4\big\|v_k^{\top}X_{t}\big\|_4\big(\|\epsilon_t - \epsilon_{t, \{t-l,-\infty\}}\|_4\big\|v_k^{\top}X_{t}\big\|_4 +\|\epsilon_t\|_4\big\|v_k^{\top}(X_{t} - X_{t,\{t-l,-\infty\}})\big\|_4\big)\\
    \leq& \|\epsilon_t\|_4\sup_{|v|_2 = 1}\|v^{\top}X_{t}\|_4\Big(\|\epsilon_t\|_4D_{X}+\sup_{|v|_2 = 1}\|v^{\top}X_{t}\|_4D_{\epsilon}\Big)\exp(-cl^{\gamma_1})\\
    =& K_G\exp(-cl^{\gamma_1}).
\end{align*}
Thus,
 \begin{align}\label{eq:LRV_L1}
    &\kappa_k^2\big|\mathbb{E}[\check{\sigma}_{\infty}^2 - \sigma_{\infty}^2]\big| = \bigg|\frac{1}{R}\sum_{i \neq j}\mathbb{E}[(D_i^*)^2] + \frac{\mathbb{E}[(D_j^*)^2]}{R} - \kappa_k^4\sum_{l = -\infty}^{\infty}\gamma_l^F - 4\kappa_k^2\sum_{l = -\infty}^{\infty}\gamma_l^G\bigg| \nonumber\\
    =&\bigg|-\kappa_k^4\sum_{l = 1}^{S}\frac{3l}{S}\gamma_l^F - \kappa_k^4\sum_{l = S+1}^{2S-1}\Big(4-\frac{l}{S}\Big)\gamma_l^F - 4\kappa_k^2\sum_{l = 1}^{S}\frac{3l}{S}\gamma_l^G - 4\kappa_k^2\sum_{l = S+1}^{2S-1}\Big(4-\frac{l}{S}\Big)\gamma_l^G\nonumber \\
    &- \kappa_k^4\sum_{l = 1}^{h_1}\frac{2l}{RS}\gamma_l^F - \kappa_k^4\sum_{l = h_1+1}^{h_2}\frac{4h_1-2l}{RS}\gamma_l^F + \kappa_k^4\sum_{l = h_2+1}^{S}\Big(\frac{4l - 2h_1}{RS} - \frac{2}{R}\Big)\gamma_l^F + \kappa_k^4\sum_{l = S+1}^{S+h_1}\Big(\frac{2}{R}-\frac{2h_1}{RS}\Big)\gamma_l^F \nonumber\\
    &+ \kappa_k^4\sum_{l = S+h_1+1}^{2S-1}\Big(\frac{4}{R}-\frac{2l}{RS}\Big)\gamma_l^F  + \frac{2h_2^2}{RS}\kappa_k^2\big(v_k^{\top}\Sigma v_k\big)^2 - 2\kappa_k^4\sum_{l \geq 2S}\gamma_l^F - 8\kappa_k^2\sum_{l \geq 2S}\gamma_l^G \bigg| \nonumber\\
    \leq& \frac{3\kappa_k^4}{S}\sum_{l = 1}^{S}l|\gamma_l^F| + 3\kappa_k^4\sum_{l = S+1}^{2S-1}|\gamma_l^F| + \frac{12\kappa_k^2}{S}\sum_{l = 1}^{S}l|\gamma_l^G| + 12\kappa_k^2\sum_{l = S+1}^{2S-1}|\gamma_l^G| + \frac{2\kappa_k^4}{RS}\sum_{l = 1}^{h_1}l|\gamma_l^F| + \frac{2\kappa_k^4}{R}\sum_{l = h_1+1}^{2S-1}|\gamma_l^F| \nonumber\\
    &+ \frac{2S}{R}\kappa_k^2\big(v_k^{\top}\Sigma v_k\big)^2 + 2\kappa_k^4\sum_{l \geq 2S}|\gamma_l^F| + 8\kappa_k^2\sum_{l \geq 2S}|\gamma_l^G| \nonumber\\
    \leq& \frac{3}{S}\kappa_k^4\sum_{l = 1}^{S}l|\gamma_l^F| + 3\kappa_k^4\sum_{l = S+1}^{\infty}|\gamma_l^F| + \frac{2\kappa_k^4}{R}\sum_{l = 1}^{2S-1}|\gamma_l^F| + \frac{12\kappa_k^2}{S}\sum_{l = 1}^{S}l|\gamma_l^G| + 12\kappa_k^2\sum_{l = S+1}^{\infty}|\gamma_l^G| + \frac{2S}{R}\kappa_k^2(\Lambda_{\max}(\Sigma))^2 \nonumber\\
    \leq& \frac{3K_F\kappa_k^4 + 12K_G\kappa_k^2}{S}\sum_{l = 1}^{S}l\exp(-cl^{\gamma_1}) + (3K_F\kappa_k^4+12K_G\kappa_k^2)\sum_{l = S+1}^{\infty}\exp(-cl^{\gamma_1}) + \frac{2K_F\kappa_k^4}{R}\sum_{l = 1}^{2S-1}\exp(-cl^{\gamma_1}) \nonumber\\
    &+ \frac{2S}{R}\kappa_k^2(\Lambda_{\max}(\Sigma))^2 \nonumber\\
    \leq& C_1\frac{K_F\kappa_k^4 + 4K_G\kappa_k^2}{S} + C_2\frac{K_F\kappa_k^4}{R} + \frac{2S}{R}\kappa_k^2(\Lambda_{\max}(\Sigma))^2,
 \end{align}
where the last inequality follows from the fact that, for any $\nu > 0$,
\begin{align*}
    \sup_{S \geq 1}S^{\nu}\sum_{l = S}^{\infty}\exp(-cl^{\gamma_1}) \leq \sum_{l = 0}^{\infty}l^{\nu}\exp(-cl^{\gamma_1}) < \infty.
\end{align*}
Therefore, let $S, R \to \infty$ and $R \gg S$ as $n \to \infty$, it holds that
\begin{align}\label{eq:LRV_L1_conv}
    \big|\mathbb{E}[\check{\sigma}_{\infty}^2 - \sigma_{\infty}^2]\big| \to 0.
\end{align}

{\bf Step 2: Variance.}
    We start with the calculation of the functional dependence measure of the process $\{\kappa_k^{-1}D_i^*\}_{i \in \mathbb{Z}}$. Recall that for $i \in \mathbb{Z}$
    \[
    \kappa_k^{-1}D_i^* = \frac{1}{\kappa_k\sqrt{2S}}\bigg\{\sum_{t \in \mathcal{S}_{2i-1}}Z_t^* - \sum_{t \in \mathcal{S}_{2i}}Z_t^*\bigg\}.
    \]
    For any $l \geq 0$, define the coupled version of $D_i^*$ as
    \[
    \kappa_k^{-1}D_{i,\{i-l\}}^* = \frac{1}{\kappa_k\sqrt{2S}}\bigg\{\sum_{t \in \mathcal{S}_{2i-1}}Z_{t, \{\mathcal{S}_{2(i-l)-1} \cup \mathcal{S}_{2(i-l)}\}}^* - \sum_{t \in \mathcal{S}_{2i}}Z_{t, \{\mathcal{S}_{2(i-l)-1} \cup \mathcal{S}_{2(i-l)}\}}^*\bigg\}.
    \]
    \\
    \textbf{Case 1}: $l = 0$. We have that
    \begin{align}\label{eq:D_i^*_Lq}
        &\left\|\kappa_k^{-1}D_{i}^* - \kappa_k^{-1}D_{i, \{i\}}^*\right\|_{4} \leq 2\left\|\kappa_k^{-1}(D_{i}^* - \mathbb{E}[D_i^*])\right\|_{4} = \frac{\sqrt{2}}{\kappa_k\sqrt{S}}\left\|\sum_{t \in \mathcal{S}_{2i-1} \cup \mathcal{S}_{2i}}\sum_{r = 0}^{\infty}\mathcal{P}_{t-r}\{a_tZ_{t}^*\}\right\|_{4} \nonumber\\
        \leq& \frac{\sqrt{2}}{\kappa_k\sqrt{S}}\sum_{r = 0}^{\infty}\left\|\sum_{t \in \mathcal{S}_{2i-1} \cup \mathcal{S}_{2i}}\mathcal{P}_{t-r}\{a_tZ_{t}^*\}\right\|_{4} \leq 2\kappa_k^{-1}\sum_{r = 0}^{\infty}\left\|a_tZ_{t}^* - a_tZ_{t, \{t-r\}}^*\right\|_{4} \nonumber\\
        \leq& 2\sum_{r = 0}^{\infty}\left\{\|v_k^{\top}X_1\|_{8}\delta_{r, 8}^{\epsilon} + \|\epsilon_1\|_{8}\delta_{r, 8}^{X}\right\} + 2\kappa_k\|v^{\top}_kX_1\|_{8}\sum_{r = 0}^{\infty}\delta_{r, 8}^X \nonumber\\
        \leq& 2\|v_k^{\top}X_1\|_{8}D_{\epsilon}^{\prime} + 2\|\epsilon_1\|_{8}D_{X}^{\prime} + 2\kappa_k\|v^{\top}_kX_1\|_{8}D_X^{\prime} \equiv C_1,
    \end{align}
    where the first inequality follows from the triangle inequality and the stationarity. Then we construct a martingale difference sequence using the projection operator $\mathcal{P}_t\cdot$. The third inequality follows from Burkholder's inequality, the fourth inequality follows from the upper bound of the functional dependence measure of $\{Z_t^*\}_{t \in \mathbb{Z}}$ (see the proof of \Cref{thm:asymptotics}), and the last inequality follows from Assumptions \ref{assume:X} and \ref{assume:epsilon}. In fact, for the covariate sequence, \Cref{assume:X}\textbf{a} and \textbf{b} lead to 
    $\sup_{m \geq 0}\exp(cm^{\gamma_1})\Delta_{m,8}^X \leq D_X^{\prime} < \infty$.  This fact is due to the equivalence property of the exponentially decay functional dependence measure (see Lemma 1 in \cite{wu2004limit} or Lemma 2 in \cite{wu2005linear}). Specifically, for $\{X_{t}\}_{t \in \mathbb{Z}}$ with $q$th moment $\Vert X_t \Vert_q < \infty$, if \eqref{eq:X temporal dependence} holds for an $\alpha \in (0, q)$, then \eqref{eq:X temporal dependence} holds for all $\beta \in (0, q)$.  Similarly, \Cref{assume:epsilon}\textbf{a} and \textbf{b} lead to 
    $\sup_{m \geq 0}\exp(cm^{\gamma_1})\Delta_{m, 8}^\epsilon \leq D_\epsilon^{\prime} < \infty$.
    \\
    \textbf{Case 2}: $l \geq 1$. Note that for any $t \in \mathcal{S}_{2i-1} \cup \mathcal{S}_{2i}$, any $a \in \{1,-1\}$ and any $S \in \mathbb{N}_+$, it follows similarly that
    \begin{align*}
        &\left\|aZ_{t}^* - aZ_{t, \{\mathcal{S}_{2(i-l)-1} \cup \mathcal{S}_{2(i-l)}\}}^*\right\|_{4} \leq \sum_{m = 0}^{2S}\delta^{Z^*}_{t-\max\{\mathcal{S}_{2(i-l)}\}+m, 4}\\
        \leq& 2\kappa_k\sum_{m = 0}^{2S}\left\{\|v_k^{\top}X_1\|_{8}\delta_{t-\max\{\mathcal{S}_{2(i-l)}\}+m, 8}^{\epsilon} + \|\epsilon_1\|_{8}\delta_{t-\max\{\mathcal{S}_{2(i-l)}\}+m, 8}^{X}\right\}\\
        &+ 2\kappa_k^2\sum_{m = 0}^{2S}\|v^{\top}_kX_1\|_{8}\delta_{t-\max\{\mathcal{S}_{2(i-l)}\}+m, 8}^X\\
        \leq& 2\kappa_k\big\{\|v_k^{\top}X_1\|_{8}D_{\epsilon}^{\prime} + \|\epsilon_1\|_{8}D_X^{\prime} + \kappa_k\|v_k^{\top}X_1\|_{8}D_X^{\prime}\big\}\exp\big(-c(t- \max\{\mathcal{S}_{2(i-l)}\})^{\gamma_1}\big).
    \end{align*}
    By the triangle inequality, we have that for any $l \geq 1$
    \begin{align*}
        &\left\|\kappa_k^{-1}D_{i}^* - \kappa_k^{-1}D_{i, \{i-l\}}^*\right\|_{4} \leq \frac{1}{\kappa_k\sqrt{2S}}\sum_{t \in \mathcal{S}_{2i-1} \cup \mathcal{S}_{2i}}\left\|aZ_{t}^* - aZ_{t, \{\mathcal{S}_{2(i-l)-1} \cup \mathcal{S}_{2(i-l)}\}}^*\right\|_{4}\\
        \leq& \sqrt{\frac{2}{S}}\big\{\|v_k^{\top}X_1\|_{8}D_{\epsilon}^{\prime} + \|\epsilon_1\|_{8}D_X^{\prime} + \kappa_k\|v_k^{\top}X_1\|_{8}D_X^{\prime}\big\}\sum_{t \in \mathcal{S}_{2i-1} \cup \mathcal{S}_{2i}}\exp\big(-c(t- \max\{\mathcal{S}_{2(i-l)}\})^{\gamma_1}\big)\\
        =& \sqrt{\frac{2}{S}}\big\{\|v_k^{\top}X_1\|_{8}D_{\epsilon}^{\prime} + \|\epsilon_1\|_{8}D_X^{\prime} + \kappa_k\|v_k^{\top}X_1\|_{8}D_X^{\prime}\big\}\sum_{m = 1}^{2S}\exp\big(-c(2Sl-2S+m)^{\gamma_1}\big).
    \end{align*}
    Thus,
    \begin{align}\label{eq:sum_fdm_D_i^*}
        &\sum_{l = 0}^{\infty}\left\|\kappa_k^{-1}D_{i}^* - \kappa_k^{-1}D_{i, \{i-l\}}^*\right\|_{4} \nonumber\\
        \leq& 2\|v_k^{\top}X_1\|_{8}D_{\epsilon}^{\prime} + 2\|\epsilon_1\|_{8}D_{X}^{\prime} + 2\kappa_k\|v^{\top}_kX_1\|_{8}D_X^{\prime} \nonumber\\
        &+ \sqrt{\frac{2}{S}}\big\{\|v_k^{\top}X_1\|_{8}D_{\epsilon}^{\prime} + \|\epsilon_1\|_{8}D_X^{\prime} + \kappa_k\|v_k^{\top}X_1\|_{8}D_X^{\prime}\big\}\sum_{l = 1}^{\infty}\sum_{m = 1}^{2S}\exp\big(-c(2Sl-2S+m)^{\gamma_1}\big) \nonumber\\
        =& 2\|v_k^{\top}X_1\|_{8}D_{\epsilon}^{\prime} + 2\|\epsilon_1\|_{8}D_{X}^{\prime} + 2\kappa_k\|v^{\top}_kX_1\|_{8}D_X^{\prime} \nonumber\\
        &+ \sqrt{\frac{2}{S}}\big\{\|v_k^{\top}X_1\|_{8}D_{\epsilon}^{\prime} + \|\epsilon_1\|_{8}D_X^{\prime} + \kappa_k\|v_k^{\top}X_1\|_{8}D_X^{\prime}\big\}\sum_{r = 1}^{\infty}\exp\big(-cr^{\gamma_1}\big) \nonumber\\
        =& 2\|v_k^{\top}X_1\|_{8}D_{\epsilon}^{\prime} + 2\|\epsilon_1\|_{8}D_{X}^{\prime} + 2\kappa_k\|v^{\top}_kX_1\|_{8}D_X^{\prime} \nonumber\\
        &+ \sqrt{\frac{2}{S}}\big\{\|v_k^{\top}X_1\|_{8}D_{\epsilon}^{\prime} + \|\epsilon_1\|_{8}D_X^{\prime} + \kappa_k\|v_k^{\top}X_1\|_{8}D_X^{\prime}\big\}D \nonumber\\
        \equiv& C_2,
    \end{align}
    where $D = \sum_{r = 1}^{\infty}\exp\big(-cr^{\gamma_1}\big) < \infty$.
    \\
    \\
    Recall that
    $\check{\sigma}_{\infty}^2 = \frac{1}{R}\sum_{i = 1}^R(\kappa_k^{-1}D_i^*)^2$.
    It follows that
    \begin{align*}
        &\sqrt{\var(\check{\sigma}_{\infty}^2)} = \frac{1}{R}\left\|\sum_{i = 1}^R\left\{(\kappa_k^{-1}D_i^*)^2 - \mathbb{E}[(\kappa_k^{-1}D_i^*)^2]\right\}\right\|_2 = \frac{1}{R}\left\|\sum_{i = 1}^R\sum_{l = 0}^{\infty}\mathcal{P}_{i-l}(\kappa_k^{-1}D_i^*)^2\right\|_2\\
        \leq& \frac{1}{R}\sum_{l = 0}^{\infty}\left\|\sum_{i = 1}^R\mathcal{P}_{i-l}(\kappa_k^{-1}D_i^*)^2\right\|_2 \leq \frac{1}{\sqrt{R}}\sum_{l = 0}^{\infty}\left\|(\kappa_k^{-1}D_i^*)^2 - (\kappa_k^{-1}D_{i, \{i-l\}}^*)^2\right\|_2\\
        \leq& \frac{1}{\sqrt{R}}\left\|\kappa_k^{-1}D_i^* + \kappa_k^{-1}D_{i, \{i-l\}}^*\right\|_4\sum_{l = 0}^{\infty}\left\|\kappa_k^{-1}D_i^* - \kappa_k^{-1}D_{i, \{i-l\}}^*\right\|_4\\
        \leq& \frac{2}{\sqrt{R}}\left\|\kappa_k^{-1}(D_i^* - \mathbb{E}[D_i^*])\right\|_4\sum_{l = 0}^{\infty}\left\|\kappa_k^{-1}D_i^* - \kappa_k^{-1}D_{i, \{i-l\}}^*\right\|_4 + \frac{2}{\sqrt{R}}\kappa_k^{-1}|\mathbb{E}[D_i^*]|\sum_{l = 0}^{\infty}\left\|\kappa_k^{-1}D_i^* - \kappa_k^{-1}D_{i, \{i-l\}}^*\right\|_4\\
        \leq& \frac{1}{\sqrt{R}}C_1C_2 + 2\sqrt{\frac{S}{R}}\kappa_pC_2v^{\top}\Sigma v \overset{n \to \infty}{\longrightarrow} 0
    \end{align*}
    \\
    Combining {\bf Step 1.} and {\bf Step 2.} concludes the proof.
\end{proof}

\begin{proof}[Proof of \Cref{thm:lrv_consist}]
Since with probability at least $1 - cn^{-3}$, the event $\mathcal{E}_{\mathrm{DP}}$ defined in \eqref{eq:event_DP} holds, we condition on $\mathcal{E}_{\mathrm{DP}}$ in the following proof. 
It follows by definition that
\begin{align*}
    &Z_t - Z_t^*\\
    =& \begin{cases}
         -2\widehat{ \nu}_{k+1}^{\top}X_tX_t^{\top}\Psi_k + (\widehat{ \nu}_k -\widehat{ \nu}_{k+1})^{\top}X_tX_t^{\top}(\widehat{ \nu}_k +\widehat{ \nu}_{k+1}) + 2\epsilon_t(\widehat{ \nu}_{k+1} - \widehat{ \nu}_k)^{\top}X_t, \quad& s_k \leq t \leq \eta_k-1,\\
        -2\widehat{ \nu}_k^{\top}X_tX_t^{\top}\Psi_k  + (\widehat{ \nu}_k -\widehat{ \nu}_{k+1})^{\top}X_tX_t^{\top}(\widehat{ \nu}_k +\widehat{ \nu}_{k+1}) + 2\epsilon_t(\widehat{ \nu}_{k+1} - \widehat{ \nu}_k)^{\top}X_t, \quad& \eta_k \leq t \leq e_k-1,
    \end{cases}
\end{align*}
where we denote $\widehat{ \nu}_{k} = \widehat{\beta}_k - \beta_{\eta_k}^*$ and $\widehat{ \nu}_{k+1} = \widehat{\beta}_{k+1} - \beta_{\eta_{k+1}}^*$.
\\
\\
Recall the estimated jump size $\widehat{\kappa}_k$ defined in \eqref{eq-jump-size-estimator}. Some algebra shows that 
\begin{align}\label{eq:weighted_lrv_diff}
    |\widehat{\kappa}_k^2\widehat{\sigma}_{\infty}^2 - \kappa_k^2\check{\sigma}_{\infty}^2| \leq& \bigg|\frac{1}{R}\sum_{i = 1}^R\big\{D_i^2 - (D_i^*)^2\big\}\bigg| \leq \frac{1}{R}\sum_{i = 1}^R(D_i - D_i^*)^2 + \bigg|\frac{2}{R}\sum_{i = 1}^R D_i^*(D_i - D_i^*)\bigg|\nonumber\\
    \leq& \frac{1}{R}\sum_{i = 1}^R(D_i - D_i^*)^2 + 2\kappa_k\check{\sigma}_{\infty}\sqrt{\frac{1}{R}\sum_{i = 1}^R(D_i - D_i^*)^2}.
\end{align}
By \Cref{lemma:auxiliary_lrv_consist}, we have that $\check{\sigma}_{\infty}^2 \overset{P.}{\to} \sigma_{\infty}^2 \in [0, \infty)$. Next, we will focus on upper bounding $|D_i - D_i^*|$.
Without loss of generality, we consider the case that $\eta_k \in \mathcal{S}_{2j}$ for some $j \in \{1, \dots, R\}$. Let $h_1 = \eta_k - s_k - (2j-1)S$ and $h_2 = s_k+2jS - \eta_k$. Assume $h_2 \geq h_1$. The following figure depicts the definitions of $h_1$ and $h_2$.
\begin{center}
    \begin{tikzpicture}[scale=7,decoration=brace]
\draw[-, thick] (-0.1,0) -- (1.1,0);
\foreach \x/\xtext in {0/$s_k+2(j-1)S$,0.5/$s_k+(2j-1)S$,0.7/$\eta_k$,1/$s_k+2jS$}
    \draw[thick] (\x,0.5pt) -- (\x,-0.5pt) node[font = {\footnotesize}, below] {\xtext};
\draw[decorate, thick, yshift=0.6ex]  (0,0) -- node[font = {\footnotesize}, above=0.2ex] {$\mathcal{S}_{2j-1}$}  (0.5,0);
\draw[decorate, thick, yshift=0.6ex]  (0.5,0) -- node[font = {\footnotesize}, above=0.2ex] {$\mathcal{S}_{2j}$}  (1,0);
\draw (0.6,0.5pt) node[font = {\footnotesize}, above] {$h_1$};
\draw (0.85,0.5pt) node[font = {\footnotesize}, above] {$h_2$};
%\draw[[-), ultra thick, blue] (0,0) -- (0.2,0);
%\draw (-0.25,0) node {$x=1$};
\end{tikzpicture}
\end{center}
\textbf{Case 1}: $i < j$.
It follows from \Cref{lemmma:refinement_lasso} that
\begin{align*}
    &\sqrt{2S}|D_i - D_i^*| = \bigg|\sum_{t \in \mathcal{S}_{2i-1}}(Z_t - Z_t^*) - \sum_{t \in \mathcal{S}_{2i}}(Z_t - Z_t^*)\bigg|\\
    =& \bigg|\sum_{t \in \mathcal{S}_{2i-1}}\big(-2\widehat{ \nu}_{k+1}^{\top}X_tX_t^{\top}\Psi_k + (\widehat{ \nu}_k -\widehat{ \nu}_{k+1})^{\top}X_tX_t^{\top}(\widehat{ \nu}_k +\widehat{ \nu}_{k+1}) + 2\epsilon_t(\widehat{ \nu}_{k+1} - \widehat{ \nu}_k)^{\top}X_t\big)\\
    &- \sum_{t \in \mathcal{S}_{2i}}\big(-2\widehat{ \nu}_{k+1}^{\top}X_tX_t^{\top}\Psi_k + (\widehat{ \nu}_k -\widehat{ \nu}_{k+1})^{\top}X_tX_t^{\top}(\widehat{ \nu}_k +\widehat{ \nu}_{k+1}) + 2\epsilon_t(\widehat{ \nu}_{k+1} - \widehat{ \nu}_k)^{\top}X_t\big)\bigg|\\
    \leq& 2\kappa_k|\widehat{ \nu}_{k+1}|_1\bigg|\sum_{t \in \mathcal{S}_{2i-1}}(X_tX_t^{\top} - \Sigma)v_k - \sum_{t \in \mathcal{S}_{2i}}(X_tX_t^{\top} - \Sigma)v_k\bigg|_{\infty}\\
    &+ 2\big|\widehat{ \nu}_k\big|_1^2\bigg|\sum_{t \in \mathcal{S}_{2i-1}}(X_tX_t^{\top} - \Sigma) - \sum_{t \in \mathcal{S}_{2i}}(X_tX_t^{\top} - \Sigma)\bigg|_{\infty} + 2|\widehat{ \nu}_{k}|_1\bigg|\sum_{t \in \mathcal{S}_{2i-1}}\epsilon_tX_t - \sum_{t \in \mathcal{S}_{2i}}\epsilon_tX_t\bigg|_{\infty}\\
    & + 2\big|\widehat{ \nu}_{k+1}\big|_1^2\bigg|\sum_{t \in \mathcal{S}_{2i-1}}(X_tX_t^{\top} - \Sigma) - \sum_{t \in \mathcal{S}_{2i}}(X_tX_t^{\top} - \Sigma)\bigg|_{\infty} + 2|\widehat{ \nu}_{k+1}|_1\bigg|\sum_{t \in \mathcal{S}_{2i-1}}\epsilon_tX_t - \sum_{t \in \mathcal{S}_{2i}}\epsilon_tX_t\bigg|_{\infty}.
\end{align*}
By \Cref{theorem:DUDP} and \Cref{lemmma:refinement_lasso}, it follows with probability at least $1 - c(n^{-3} + S^{-5})$ that
\begin{align*}
    |D_i - D_i^*| \leq C_1\big(\kappa_k^2\s^{1/2}\alpha_n^{-1/2} + \kappa_k^2\s\alpha_n^{-1} + \kappa_k\s^{1/2}\alpha_n^{-1/2}\big)\bigg( \sqrt{\log(p)} + \frac{\log^{1/\gamma}(p)}{\sqrt{S}} \bigg).
\end{align*}
\\
\\
\textbf{Case 2}: $i > j$. Similarly, it follows with probability at least $1 - c(n^{-3} + S^{-5})$ that
\begin{align*}
    |D_i - D_i^*| \leq C_2\big(\kappa_k^2\s^{1/2}\alpha_n^{-1/2} + \kappa_k^2\s\alpha_n^{-1} + \kappa_k\s^{1/2}\alpha_n^{-1/2}\big)\bigg( \sqrt{\log(p)} + \frac{\log^{1/\gamma}(p)}{\sqrt{S}} \bigg).
\end{align*}
\\
\\
\textbf{Case 3}: $i = j$. We have that
\begin{align*}
    &\sqrt{2S}|D_j - D_j^*| = \bigg|\sum_{t \in \mathcal{S}_{2j-1}}(Z_t - Z_t^*) - \sum_{t \in \mathcal{S}_{2j}}(Z_t - Z_t^*)\bigg|\\
    =& \bigg|\sum_{t \in \mathcal{S}_{2j-1}}\big(-2\widehat{ \nu}_{k+1}^{\top}X_tX_t^{\top}\Psi_k + (\widehat{ \nu}_k -\widehat{ \nu}_{k+1})^{\top}X_tX_t^{\top}(\widehat{ \nu}_k +\widehat{ \nu}_{k+1}) + 2\epsilon_t(\widehat{ \nu}_{k+1} - \widehat{ \nu}_k)^{\top}X_t\big)\\
    &- \sum_{t \in [s_k+\{2j-1\}S, \eta_k)}\big(-2\widehat{ \nu}_{k+1}^{\top}X_tX_t^{\top}\Psi_k + (\widehat{ \nu}_k -\widehat{ \nu}_{k+1})^{\top}X_tX_t^{\top}(\widehat{ \nu}_k +\widehat{ \nu}_{k+1}) + 2\epsilon_t(\widehat{ \nu}_{k+1} - \widehat{ \nu}_k)^{\top}X_t\big)\\
    &- \sum_{t \in [\eta_k, s_k+2jS)}\big(-2\widehat{ \nu}_{k}^{\top}X_tX_t^{\top}\Psi_k + (\widehat{ \nu}_k -\widehat{ \nu}_{k+1})^{\top}X_tX_t^{\top}(\widehat{ \nu}_k +\widehat{ \nu}_{k+1}) + 2\epsilon_t(\widehat{ \nu}_{k+1} - \widehat{ \nu}_k)^{\top}X_t\big)\bigg|\\
    \leq& 2\kappa_k|\widehat{ \nu}_{k+1}|_1\bigg|\sum_{t \in \mathcal{S}_{2i-1}}(X_tX_t^{\top} - \Sigma)v_k - \sum_{t \in (s_k+\{2j-1\}S, \eta_k]}(X_tX_t^{\top} - \Sigma)v_k\bigg|_{\infty}\\
    &+ 2\kappa_k|\widehat{ \nu}_{k}|_1\bigg|\sum_{t \in (\eta_k, s_k+2jS]}(X_tX_t^{\top} - \Sigma)v_k\bigg|_{\infty} + 2(S-h_1)\kappa_k| \widehat{\nu}_{k+1}|_2 \Lambda_{\max} + 2h_2\kappa_k| \widehat{\nu}_{k}|_2 \Lambda_{\max}\\
    &+ 2\big|\widehat{ \nu}_k\big|_1^2\bigg|\sum_{t \in \mathcal{S}_{2i-1}}(X_tX_t^{\top} - \Sigma) - \sum_{t \in \mathcal{S}_{2i}}(X_tX_t^{\top} - \Sigma)\bigg|_{\infty} + 2|\widehat{ \nu}_{k}|_1\bigg|\sum_{t \in \mathcal{S}_{2i-1}}\epsilon_tX_t - \sum_{t \in \mathcal{S}_{2i}}\epsilon_tX_t\bigg|_{\infty}\\
    & + 2\big|\widehat{ \nu}_{k+1}\big|_1^2\bigg|\sum_{t \in \mathcal{S}_{2i-1}}(X_tX_t^{\top} - \Sigma) - \sum_{t \in \mathcal{S}_{2i}}(X_tX_t^{\top} - \Sigma)\bigg|_{\infty} + 2|\widehat{ \nu}_{k+1}|_1\bigg|\sum_{t \in \mathcal{S}_{2i-1}}\epsilon_tX_t - \sum_{t \in \mathcal{S}_{2i}}\epsilon_tX_t\bigg|_{\infty}.
\end{align*}
By \Cref{theorem:DUDP} and \Cref{lemmma:refinement_lasso}, it follows with probability at least $1 - c(n^{-3} + S^{-5})$ that
\begin{align*}
    |D_j - D_j^*| \leq C_3\big(\kappa_k^2\s^{1/2}\alpha_n^{-1/2} + \kappa_k^2\s\alpha_n^{-1} + \kappa_k\s^{1/2}\alpha_n^{-1/2}\big)\bigg( \sqrt{\log(p)} + \frac{\log^{1/\gamma}(p)}{\sqrt{S}} \bigg) + C_4\sqrt{S}\kappa_k^2\alpha_n^{-1/2}.
\end{align*}
\\
\\
Combining all the above cases, it follows with probability at least $1 - c(n^{-2} + RS^{-5})$ that
\begin{align}\label{eq:conditional_exp_UB}
    &\frac{1}{R\kappa_k^2}\sum_{i = 1}^R(D_i - D_i^*)^2 = \frac{1}{R\kappa_k^2}\sum_{i \neq j}(D_i - D_i^*)^2 + \frac{1}{R\kappa_k^2}(D_j - D_j^*)^2\nonumber\\
    \leq& C_5\big(\kappa_k\s^{1/2}\alpha_n^{-1/2} + \kappa_k\s\alpha_n^{-1} + \s^{1/2}\alpha_n^{-1/2}\big)^2\bigg( \sqrt{\log(p)} + \frac{\log^{1/\gamma}(p)}{\sqrt{S}} \bigg)^2 + C_6\frac{S\kappa_k^2\alpha_n^{-1}}{R}.
\end{align}
Combining \eqref{eq:weighted_lrv_diff} and \eqref{eq:conditional_exp_UB}, we have that with probability at least $1 - cn^{-3}$
\begin{align*}
    \bigg|\frac{\widehat{\kappa}_k^2}{\kappa_k^2}\widehat{\sigma}_{\infty}^2 - \check{\sigma}_{\infty}^2\bigg| \leq C_7\frac{\s}{\alpha_n}\bigg( \sqrt{\log(p)} + \frac{\log^{1/\gamma}(p)}{\sqrt{S}} \bigg)^2.
\end{align*}
Under \Cref{assump-snr}\textbf{b}, applying \Cref{lemma:consistency_jump_size_est}, \Cref{lemma:auxiliary_lrv_consist} and Slutsky's theorem, we have that as $n \to \infty$,
\begin{align*}
    \big|\widehat{\sigma}_{\infty}^2 - \sigma_{\infty}^2\big| \overset{P.}{\longrightarrow} 0.
\end{align*}
\end{proof}

\begin{proof}[Proof of \Cref{thm:quadratic_consist}]
This proof uses the results in the \textbf{Preliminary.} of the proof of \Cref{thm:asymptotics}, which are omitted here for brevity.  Note that
\begin{align*}
    &\widehat{\varpi}_k - \varpi_k = \frac{1}{n\widehat{\kappa}_k^2}\sum_{t = 1}^n \widehat{\Psi}_k^{\top}X_tX_t^{\top}\widehat{\Psi}_k - v_k^{\top}\Sigma v_k\\
    =& \frac{\kappa_k^2}{\widehat{\kappa}_k^2}\frac{1}{n\kappa_k^2}\sum_{t = 1}^n \big\{\widehat{\Psi}_k^{\top}X_tX_t^{\top}\widehat{\Psi}_k - \Psi_k^{\top}X_tX_t^{\top}\Psi_k\big\} + \bigg(\frac{\kappa_k^2}{\widehat{\kappa}_k^2}-1\bigg)\frac{1}{n}\sum_{t = 1}^n  v_k^{\top}X_tX_t^{\top}v_k\\
    &+ \frac{1}{n}\sum_{t = 1}^n\big\{ v_k^{\top}X_tX_t^{\top}v_k - v_k^{\top}\Sigma v_k\big\}\\
    =& I + II + III.
\end{align*}
\\
\\
\textbf{Step 1.} We consider the term $I$. It follows that
\begin{align*}
    &\frac{1}{n\kappa_k^2}\sum_{t = 1}^n\big\{\widehat{\Psi}_k^{\top}X_tX_t^{\top}\widehat{\Psi}_k - \Psi_k^{\top}X_tX_t^{\top}\Psi_k\big\}\\
    =& \frac{1}{n\kappa_k^2}\sum_{t = 1}^n(\widehat{\beta}_{k} - \widehat{\beta}_{k-1} - \beta_{\mathcal{I}_k}^* + \beta_{\mathcal{I}_{k-1}}^*)^{\top}(X_tX_t^{\top} - \Sigma)(\widehat{\beta}_{k} - \widehat{\beta}_{k-1}- \beta_{\mathcal{I}_k}^* + \beta_{\mathcal{I}_{k-1}}^*)\\
    &- \frac{1}{n\kappa_k^2}\sum_{t = 1}^n(\beta_{\eta_k}^* - \beta_{\eta_{k-1}}^* - \beta_{\mathcal{I}_k}^* + \beta_{\mathcal{I}_{k-1}}^*)^{\top}(X_tX_t^{\top} - \Sigma)(\beta_{\eta_k}^* - \beta_{\eta_{k-1}}^*- \beta_{\mathcal{I}_k}^* + \beta_{\mathcal{I}_{k-1}}^*)\\
    &+ \frac{2}{n\kappa_k^2}\sum_{t = 1}^n(\widehat{\beta}_{k} - \widehat{\beta}_{k-1} - \beta_{\eta_k}^* + \beta_{\eta_{k-1}}^*)^{\top}(X_tX_t^{\top} - \Sigma)( \beta_{\mathcal{I}_k}^* - \beta_{\mathcal{I}_{k-1}}^*)\\
    &+ \frac{1}{\kappa_k^2}(\widehat{\beta}_{k} - \widehat{\beta}_{k-1} - \beta_{\mathcal{I}_k}^* + \beta_{\mathcal{I}_{k-1}}^*)^{\top}\Sigma(\widehat{\beta}_{k} - \widehat{\beta}_{k-1}- \beta_{\mathcal{I}_k}^* + \beta_{\mathcal{I}_{k-1}}^*)\\
    &- \frac{1}{\kappa_k^2}(\beta_{\eta_k}^* - \beta_{\eta_{k-1}}^* - \beta_{\mathcal{I}_k}^* + \beta_{\mathcal{I}_{k-1}}^*)^{\top}\Sigma(\beta_{\eta_k}^* - \beta_{\eta_{k-1}}^*- \beta_{\mathcal{I}_k}^* + \beta_{\mathcal{I}_{k-1}}^*)\\
    &+ \frac{2}{\kappa_k^2}(\widehat{\beta}_{k} - \widehat{\beta}_{k-1} - \beta_{\eta_k}^* + \beta_{\eta_{k-1}}^*)^{\top}\Sigma( \beta_{\mathcal{I}_k}^* - \beta_{\mathcal{I}_{k-1}}^*)\\
    =& I_1 - I_2 + 2I_3 + I_4 - I_5 + 2I_6.
\end{align*}
\
\\
\textbf{Term $I_1$.} It holds that
\begin{align*}
    |I_1| \leq& \frac{2}{n\kappa_k^2}\sum_{t = 1}^n(\widehat{\beta}_{k} - \beta_{\mathcal{I}_k}^*)^{\top}(X_tX_t^{\top} - \Sigma)(\widehat{\beta}_{k} - \beta_{\mathcal{I}_k}^*) + \frac{2}{n\kappa_k^2}\sum_{t = 1}^n( \widehat{\beta}_{k-1} -  \beta_{\mathcal{I}_{k-1}}^*)^{\top}(X_tX_t^{\top} - \Sigma)( \widehat{\beta}_{k-1}- \beta_{\mathcal{I}_{k-1}}^*),
\end{align*}
where the inequality follows from the Cauchy–Schwarz inequality and the symmetry of $X_tX_t^{\top} - \Sigma$.
It follows from \Cref{theorem:RES Version II}\textbf{b}, \eqref{eq:interval_lasso_k_1} and \eqref{eq:interval_lasso_k_3} that with probability at least $ 1-n^{-5}$,
\begin{align*} 
 |I_1|  \le& C_1 \sqrt { \frac{\s \log(pn) }{\Delta }} \frac{\s\log(pn)}{ \Delta \kappa_k^2 } \leq C_2 \alpha_n^{-3/2}.
\end{align*}
\
\\
\textbf{Term $I_2$.} It holds that 
\begin{align*}
    |I_2| \leq& \frac{2}{n\kappa_k^2}\sum_{t = 1}^n(\beta_{\eta_k}^* - \beta_{\mathcal{I}_k}^*)^{\top}(X_tX_t^{\top} - \Sigma)(\beta_{\eta_k}^* - \beta_{\mathcal{I}_k}^*) + \frac{2}{n\kappa_k^2}\sum_{t = 1}^n( \beta_{\eta_{k-1}^*} -  \beta_{\mathcal{I}_{k-1}}^*)^{\top}(X_tX_t^{\top} - \Sigma)(\beta_{\eta_{k-1}}^*- \beta_{\mathcal{I}_{k-1}}^*).
\end{align*}
Let $$z_t = \frac{1}{|\beta^*_{\mathcal{I}_{k-1}} - \beta^*_{\eta_{k-1}}|_2^2}(\beta^*_{\mathcal{I}_{k-1}} - \beta^*_{\eta_{k-1}})^{\top}\big(X_tX_t^{\top} - \Sigma\big)(\beta^*_{\mathcal{I}_{k-1}} - \beta^*_{\eta_{k-1}}).$$
Note that
\begin{align*}
    &\big\|z_t - z_{t,\{0\}}\big\|_2\\
    \leq& \bigg\| \frac{1}{|\beta^*_{\mathcal{I}_{k-1}} - \beta^*_{\eta_{k-1}}|_2^2}(\beta^*_{\mathcal{I}_{k-1}} - \beta^*_{\eta_{k-1}})^{\top}\Big\{X_t(X_t - X_{t,\{0\}})^{\top} + (X_t - X_{t,\{0\}})X_{t,\{0\}}^{\top}\Big\}(\beta^*_{\mathcal{I}_{k-1}} - \beta^*_{\eta_{k-1}}) \bigg\|_2\\
    \leq& 2\sup_{|v|_2 = 1}\big\|v^{\top}X_1\big\|_{4}\delta^X_{t,4}.
\end{align*}
By \Cref{assume:X} $$\sup_{m \geq 0}\exp(cm^{\gamma_1(X)})\sum_{s = m}^{\infty}\big\|z_s - z_{s,\{0\}}\big\|_2 \leq 2D_X\sup_{|v|_2 = 1}\big\|v^{\top}X_1\big\|_{4},$$
and by \Cref{lemma:subweibull},  there exists $C_1$  depending on $C_X$ such that
 $$\p ( |z_t| \ge \tau) \le 2 \exp(- (\tau/C_1)^ {\gamma_2/2}).$$
Then, by \Cref{thm:bernstein_exp_subExp_nonlinear}, we have with probability at least $1 - n^{-5}$ that
\begin{align*}
  \bigg | \frac{1}{n}\sum_{t=1}^n z_t  \bigg |  \le  C_2 \bigg\{  \sqrt { \frac{\log(n) }{n}  }  +   \frac{ \log^{1/\gamma }( n) }{n} \bigg \}.
\end{align*}
Consequently,
\begin{align*}
    |I_2| \leq& C_3\kappa_k^{-2}\big( |\beta^*_{\mathcal{I}_{k}} - \beta^*_{\eta_{k}}|_2^2 + |\beta^*_{\mathcal{I}_{k-1}} - \beta^*_{\eta_{k-1}}|_2^2\big)\bigg\{  \sqrt { \frac{\log(n) }{n}  }  +   \frac{ \log^{1/\gamma }( n) }{n} \bigg \}\\
    \leq& C_4\alpha_n^{-2}\bigg\{  \sqrt { \frac{\log(n) }{n}  }  +   \frac{ \log^{1/\gamma }( n) }{n} \bigg \} \leq C_4\alpha_n^{-2},
\end{align*}
where the second inequality follows from \eqref{eq:bias_1} and \eqref{eq:bias_2}.
\
\\
\textbf{Term $I_3$.} It holds that 
\begin{align*}
    |I_3| \leq& \frac{1}{n\kappa_k^2}\big(|\widehat{\beta}_{k} - \beta_{\eta_k}^*|_1 +  |\widehat{\beta}_{k-1} - \beta_{\eta_{k-1}}^*|_1\big) \bigg|\sum_{t = 1}^n(X_tX_t^{\top} - \Sigma)(\beta_{\mathcal{I}_k}^* - \beta_{\mathcal{I}_{k-1}}^*)\bigg|_{\infty}.
\end{align*}
For any $j \in \{1, \dots, p\}$, let $$z_{tj} = \frac{1}{|\beta^*_{\mathcal{I}_{k}} - \beta^*_{\mathcal{I}_{k-1}}|_2}\big(X_{tj}X_t^{\top} - \Sigma_{j\cdot}\big)(\beta^*_{\mathcal{I}_{k-1}} - \beta^*_{\mathcal{I}_{k-1}}).$$
Similarly, we have for any $j \in \{1, \dots, p\}$ that
\begin{align*}
    \big\|z_{tj} - z_{tj,\{0\}}\big\|_2 \leq 2\sup_{|v|_2 = 1}\big\|v^{\top}X_1\big\|_{4}\delta^X_{t,4}.
\end{align*}
Under \cref{assume:X}, we can verify that
$$\sup_{m \geq 0}\exp(cm^{\gamma_1(X)})\sum_{s = m}^{\infty}\big\|z_{sj} - z_{sj,\{0\}}\big\|_2 \leq 2D_X\sup_{|v|_2 = 1}\big\|v^{\top}X_1\big\|_{4},$$
and 
 $$\p ( |z_{tj}| \ge \tau) \le 2 \exp(- (\tau/C_5)^ {\gamma_2/2}).$$
Then, by \Cref{thm:bernstein_exp_subExp_nonlinear} and the union bound, we have with probability at least $1 - n^{-5}$ that
\begin{align*}
  \max_{1 \leq j \leq p}\bigg | \frac{1}{n}\sum_{t=1}^n z_{tj}  \bigg |  \le  C_6 \bigg\{  \sqrt { \frac{\log(pn) }{n}  }  +   \frac{ \log^{1/\gamma }( pn) }{n} \bigg \}.
\end{align*}
Consequently,
\begin{align*}
    |I_3| \leq& C_7\kappa_k^{-2}\big(|\widehat{\beta}_{k} - \beta_{\eta_k}^*|_1 +  |\widehat{\beta}_{k-1} - \beta_{\eta_{k-1}}^*|_1\big)|\beta^*_{\mathcal{I}_{k}} - \beta^*_{\mathcal{I}_{k-1}}|_2\bigg\{  \sqrt { \frac{\log(pn) }{n}  }  +   \frac{ \log^{1/\gamma }( pn) }{n} \bigg \}\\
    \leq& C_8\alpha_n^{-1/2}\bigg\{  \sqrt { \frac{\s\log(pn) }{n}  }  +   \frac{\s^{1/2} \log^{1/\gamma }( pn) }{n} \bigg \}\\
    \leq& C_8\alpha_n^{-1/2}\bigg\{  \sqrt { \frac{\s\log(pn) }{\Delta}  }  +   \frac{\s^{1/2} \log^{1/\gamma }( pn) }{\Delta} \bigg \}\\
    \leq& C_9\alpha_n^{-1},
\end{align*}
where the second inequality follows from \Cref{lemmma:refinement_lasso}, \eqref{eq:bias_2} and \eqref{eq:bias_3}, and the last inequality follows from \eqref{eq:SNR_cond}.
\
\\
\textbf{Terms $I_4$, $I_5$ and $I_6$.} By \eqref{eq:interval_lasso_k_1}, it holds that
\begin{align*}
    |I_4| \leq 2\kappa_k^{-2}\Lambda_{\max}(\Sigma)\big(| \widehat \beta_k -\beta^*_{\I_k}  | _2^2 + | \widehat \beta_{k-1} -\beta^*_{\I_{k-1}}  | _2^2\big) \leq C_{10}\frac{\s\log(pn)}{\Delta\kappa_k^2} \leq C_{11}\alpha_n^{-1}.
\end{align*}
By \eqref{eq:bias_1} and \eqref{eq:bias_2}
\begin{align*}
    |I_5| \leq 2\kappa_k^{-2}\Lambda_{\max}(\Sigma)\big(| \beta^*_{\eta_k} -\beta^*_{\I_k}  | _2^2 + | \beta^*_{\eta_{k-1}} -\beta^*_{\I_{k-1}}  | _2^2\big) \leq C_{12}\alpha_n^{-2}.
\end{align*}
By \Cref{lemmma:refinement_lasso}, \eqref{eq:bias_2} and \eqref{eq:bias_3}
\begin{align*}
    |I_6| \leq \kappa_k^{-2}\Lambda_{\max}(\Sigma)\big(| \widehat \beta_k -\beta^*_{\eta_k}  | _2 + | \widehat \beta_{k-1} -\beta^*_{\eta_{k-1}}  | _2\big)|\beta^*_{\I_{k}} - \beta^*_{\I_{k-1}}|_2 \leq C_{13}\alpha_n^{-1/2}.
\end{align*}
Combining the upper bounds of $|I_i|$ for $i = 1, \dots, 6$, we have that
\begin{align*}
    \frac{1}{n\kappa_k^2}\sum_{t = 1}^n\big\{\widehat{\Psi}_k^{\top}X_tX_t^{\top}\widehat{\Psi}_k - \Psi_k^{\top}X_tX_t^{\top}\Psi_k\big\} = o_p(1).
\end{align*}
Since $\frac{\kappa_k}{\widehat \kappa_k} - 1 = o_p(1)$ due to \Cref{lemma:consistency_jump_size_est}, by Slutsky's theorem, we have that
$$I = o_p(1).$$
\\
\\
\textbf{Step 2.} We consider the terms $II$ and $III$. Using the similar arguments as for the term $I_2$, we have that with probability at least $1 - n^{-5}$ that
\begin{align*}
  \bigg | \frac{1}{n}\sum_{t=1}^n v_k^{\top}(X_tX_t^{\top} - \Sigma)v_k  \bigg |  \le  C_2 \bigg\{  \sqrt { \frac{\log(n) }{n}  }  +   \frac{ \log^{1/\gamma }( n) }{n} \bigg \}.
\end{align*}
It follows that
\begin{align*}
    II = o_p(1),
\end{align*}
since $\frac{\kappa_k}{\widehat \kappa_k} - 1 = o_p(1)$ due to \Cref{lemma:consistency_jump_size_est}. Moreover,
\begin{align*}
    III = o_p(1).
\end{align*}
Therefore, we have that as $n \to \infty$
$$\widehat{\varpi}_k \overset{P.}{\longrightarrow} \varpi_k.$$
\end{proof}

\begin{proof}[Proof of \Cref{thm:sampling_dist}]
Note that when $M = \infty$, \eqref{eq:simu_est} can be equivalently written as   
        \begin{align*}
            \widehat{u}^{(b)} = \argmin_{r \in \mathbb{R}} \big\{\widehat{\varpi}_k|r| + \widehat{\sigma}_{\infty}(k)\mathbb{W}^{(b)}(r)\big\},
        \end{align*}
Denote $\widehat u =  \widehat u^{(b)} $ and $z_i = z_i^{(b)}$ for simplicity.  
\\
\\
{\bf Step 1.}  We show $\widehat u$ is uniformly tight.  A similar argument was  used in  the proof of \Cref{thm:asymptotics}. 
Without loss of generality, assume that 
$\widehat u >0$. If $\widehat u  \le 1 $, there is nothing to show. 
Suppose $\widehat u\ge 1$. 
 By  \Cref{lemma:iterated log under dependence} with $\nu = 1$, uniformly for all $r\ge 1 $ and $n \in \mathbb Z_+ $,
$$    \frac{1}{\sqrt { nr }  \log ( r )    }\sum_{i=1}^{nr } z_i      =O_p(1) .  $$
Since 
$\widehat \sigma_{\infty}(k) - \sigma_{\infty}(k) =o_p(1   ) $ and 
 $\widehat \varpi_k - \varpi_k =o_p(1 ) $, it follows with probability approaching to 1, $$ 0< \varpi_k/2 \le \widehat \varpi_k \;\; \text{and} \;\;  \widehat \sigma_{\infty}(k) \in [0, \infty).  $$
Since $\widehat u$ is the minimizer, 
$$ \widehat u  \widehat \varpi_k  + \widehat \sigma_{\infty}(k) \frac{1}{\sqrt { n   } } \sum_{i=1}^{n \widehat u   } z_i \le 0  . $$ 
Therefore,
$$ \widehat u   \varpi_k /2  \le \widehat u  \widehat \varpi_k  \le  - \widehat \sigma_{\infty}(k) \frac{1}{\sqrt { n   } } \sum_{i=1}^{n \widehat u   } z_i    = O_p( \sqrt { \widehat u}    \log(\widehat u ) ).$$
This implies that $\widehat u =O_p(1). $
\\
\\
{\bf Step 2.} We show that for any fixed $M>0$,   
$$\widehat{\varpi}_k|r| + \widehat{\sigma}_{\infty}(k)\mathbb{W}^{(b)}(r) \overset{\mathcal{D}} {\longrightarrow}  \varpi_k|r| + \sigma_{\infty}(k)\mathbb{W}(r) \;\; \text{for} \;\; |r| \le M  .$$  
Uniformly for all $r\le M$, the functional CLT for i.i.d. random variables leads to
$$  \frac{1}{\sqrt { n   } } \sum_{i=1}^{nr  } z_i \overset{\mathcal{D}}{\longrightarrow } \mathbb B_1(r) .$$
Since 
$\widehat \sigma_{\infty}(k) - \sigma_{\infty}(k) =o_p(1   ) $ and 
 $\widehat \varpi_k - \varpi_k =o_p(1 ) $, the Argmax (or Argmin) continuous mapping theorem and Slutsky's theorem lead to
    \begin{align*}
        \widehat{u} \overset{\mathcal{D}}{\longrightarrow} \argmin_r \big\{\varpi_k|r| + \sigma_{\infty}(k)\mathbb{W}(r)\big\}.
    \end{align*}
Finally, \Cref{lemma:consistency_jump_size_est} and Slutsky's theorem directly give the desired result.
\end{proof}

\section{Additional Technical Lemmas}\label{sec-app-add} 
     \begin{lemma}[Burkholder's inequality]\label{lemma:Burkholder}
     Let $q>1$, $q^{\prime} = \min\{q,2\}$ and $K_q = \max\{(q-1)^{-1}, \sqrt{q-1}\}$. Let $\{ X_t\}_{t=1}^n $ be a martingale difference sequence with $\| X_i \|_{q} < \infty$. That is, there exists a martingale
     $ \{S_k\}_{k=0}^n$ such that $X_k =S_k-S_{k-1} $.
     Then 
     $$\| S_n\|_q^{q^{\prime}} \le K_q^{q^{\prime}}\sum_{t=1}^n \|X_t\|_q^{q^{\prime}}. $$

     \end{lemma}
     \begin{proof}
      See \cite{rio2009moment} for $q > 2$ and \cite{burkholder1988sharp} for $1 < q \leq 2$.
 
 \end{proof}	
 
\begin{lemma}
Suppose $\{Z_t\}_{t\in \mathbb Z}$ is a (possibly) nonstationary process in the form of \eqref{eq:nonstationary}. Then
$$\| \mathcal P_{t-i} Z_t\|_q = \| \E(Z_t | \mathcal F_{t-i} ) - \E(Z_t | \mathcal F_{t-i-1} )\|_ q  \le \| Z_t-Z_{t, \{t-i\}} \|_q  . $$

\end{lemma} 
 \begin{proof}
 Note  that 
 $$\E(Z_t | \mathcal F_{t-i-1} ) =  \E(Z_{ t,\{ t-i\}}| \mathcal F_{t-i-1} ) =\E(Z_{ t,\{ t-i\}}| \mathcal F_{t-i} ) .$$
 Therefore
 \begin{align*}
 \| \mathcal P_{t-i} Z_t\|_q  = &\| \E(Z_t | \mathcal F_{t-i} ) - \E(Z_t | \mathcal F_{t-i-1} )\|_q  
 \\= &\| \E (Z_t | \mathcal F_{t-i}) -\E(Z_{ t,\{ t-i\}}| \mathcal F_{t-i-1} )\|_ q  
 \\ 
 = & \| \E (Z_t-Z_{ t,\{ t-i\}}  | \mathcal F_{t-i})  \|_ q  
 \\
    \le &  \| Z_t-Z_{t, \{t-i\}} \| _q 
 \end{align*}
 where the last inequality follows from Jensen's inequality. 
 \end{proof}
 
 \begin{lemma}\label{lemma:subweibull}
 Suppose $Z_1$ and $Z_2$ are two random variables such that 
 \begin{align} \label{eq:Z subweibull} \p ( |Z_i| \ge \tau) \le 2 \exp(- (\tau/C_i)^ {\gamma_2} ).
 \end{align}
 Then there exists a constant $C$ depending on only $C_1$ and $C_2$ such that 
 $$ \p ( |Z_1Z_2 -\E(Z_1Z_2) | \ge \tau) \le 2 \exp(- (\tau/C)^ {\gamma_2 /2}) .$$
 \end{lemma}
 \begin{proof}
This  is a well-known property for  sub-Weibull random variables. See \cite{wong2020lasso}  for a detailed introduction of sub-Weibull random variables. 
 \end{proof}

\subsection{A maximal inequality for weighted partial sums under dependence}
Consider a stationary process $\{Z_t\}_{t \in \mathbb{Z}} \subset \mathbb{R}$ with $\mathbb{E}[Z_t] = 0$ and the functional dependence measure $\delta_{s,q}$ defined similarly in \eqref{eq-functional-dependence-eps-def}. Let $S_n = \sum_{k = 1}^nZ_k$ and $S^*_n = \max_{1 \leq i \leq n}|S_i|$. The following lemma is a   maximal inequality  for weighted partial sums.

\begin{lemma} \label{lemma:iterated log under dependence}
Assume $\sum_{s=1}^\infty   \delta _{s,q  } < \infty$ for some $q > 2$. Let  $\nu>0$ be given.  For any $ 0<a<1$,  it holds that
$$\p\bigg( \bigg| \sum_{i=1}^{ r  } Z _i  \bigg| \le  \frac{C }{  a } \sqrt r  \big \{   \log(  r\nu) +1   \big \}      \text{ for all } r\ge 1/\nu \bigg)  \ge 1- a^2 , $$
where $C > 0$ is an absolute constant.
\end{lemma}   
   \begin{proof}
   Let $ s\in \mathbb Z_+$ and $\mathcal T_s=[ 2^s /\nu, 2^{s+1}/\nu] $. By \Cref{lemma:sup is stopping time}, for all $x>0 $, 
   $$ \p \bigg( \sup_{r \in  \mathcal T_s } \frac{ \big | \sum_{i=1}^r Z _i\big | }{ \sqrt {r }  } \ge  x    \bigg)  \le 
   C_1x^{-2}  . $$
   Therefore  by a  union bound,  for any $0<a<1$, 
  \begin{align} \label{eq:sub-gaussian dependence sup of random walk}
    \p \bigg(\exists s \in \mathbb Z_+ :   \sup_{r \in  \mathcal T_s } \frac{ \big| \sum_{i=1}^r Z _i \big| }{ \sqrt {r }} \ge    \frac{\sqrt {  C_1 }\pi }{\sqrt{6}a   }  ( s +1)  \bigg)    
  \le     \sum_{s=0}^\infty  \frac{6a^2 }{ \pi^2( s+1)^2 }  = a^2.
  \end{align}
   For any $r \in \mathcal{T}_s$, $s \le \log(r \nu)/\log(2), $ and therefore 
   \begin{align} \label{eq:sub-gaussian dependence sup of random walk 2}
    \p \bigg(\exists s \in \mathbb Z_+ :   \sup_{r \in  \mathcal T_s } \frac{ \big| \sum_{i=1}^r Z _i \big| }{ \sqrt {r }} \ge    \frac{\sqrt C_1 \pi}{\sqrt{6}a }   \bigg \{  \frac{\log(r\nu) }{ \log(2) } +1 \bigg \}   \bigg)    
  \le     a ^2 .
  \end{align} 
  \Cref{eq:sub-gaussian dependence sup of random walk 2} directly gives 
   \begin{align*} 
  & \p \bigg(  \sup_{r \in  \mathcal T_s } \frac{ \big| \sum_{i=1}^r Z _i \big| }{ \sqrt {r }}   \le   \frac{C}{  a }  \big \{   \log(  r\nu) +1   \big \}   \text{ for all } s \le \log(r \nu)/\log(2)  \bigg)  
  \ge  1-    a ^2   .
  \end{align*}
   \end{proof} 

The proof of \Cref{lemma:iterated log under dependence} relies on the following two lemmas. The next lemma is a slight modification of the Rosenthal-type inequality given in Theorem 1 of \cite{liu2013probability}.
 
\begin{lemma} \label{lemma:sup of functional dependence}
It holds for $q > 2$ that
  \begin{align*}
  \| S^*_N\|_p  \le  
    & N^{1/2 }\bigg\{\frac{87p }{\log(p)} \sum_{j=1}^N \delta_{j,2} + 3(p-1)^{1/2}\sum_{j=1}^\infty \delta_{j,p} +\frac{29p }{ \log(p)} \| Z_1\|_2  \bigg\}
  \\ 
  &+ N^{1/2} \bigg\{ \frac{87p (p-1)^{1/2}}{\log(p)}  \sum_{j=1}^N \delta_{j, p} + \frac{29p}{\log(p)}\|Z_1\|_p \bigg\}  .
  \end{align*}

   \end{lemma}
   \begin{proof} By Theorem 1 of  \cite{liu2013probability},
   \begin{align*}
  \| S^*_N\|_p  \le & N^{1/2 }\bigg\{\frac{87p }{\log(p)} \sum_{j=1}^N \delta_{j,2} + 3(p-1)^{1/2}\sum_{j=1}^\infty \delta_{j,p} +\frac{29p }{ \log(p)} \| Z_1\|_2  \bigg\}
  \\ 
  &+ N^{1/p} \bigg\{ \frac{87p (p-1)^{1/2}}{\log(p)} \sum_{j=1}^N j^{1/2-1/p}\delta_{j, p} + \frac{29p}{\log(p)}\|Z_1\|_p \bigg\}
  \\
  \le & N^{1/2 }\bigg\{\frac{87p }{\log(p)} \sum_{j=1}^N \delta_{j,2} + 3(p-1)^{1/2}\sum_{j=1}^\infty \delta_{j,p} +\frac{29p }{ \log(p)} \| Z_1\|_2  \bigg\}
  \\ 
  &+ N^{1/p} \bigg\{ \frac{87p (p-1)^{1/2}}{\log(p)} N^{1/2-1/p}\sum_{j=1}^N \delta_{j, p} + \frac{29p}{\log(p)}\|Z_1\|_p \bigg\} 
  \\
  \le   & N^{1/2 }\bigg\{\frac{87p }{\log(p)} \sum_{j=1}^N \delta_{j,2} + 3(p-1)^{1/2}\sum_{j=1}^\infty \delta_{j,p} +\frac{29p }{ \log(p)} \| Z_1\|_2  \bigg\}
  \\ 
  &+ N^{1/2} \bigg\{ \frac{87p (p-1)^{1/2}}{\log(p)}  \sum_{j=1}^N \delta_{j, p} + \frac{29p}{\log(p)}\|Z_1\|_p \bigg\}  .
  \end{align*} 
   
   \end{proof}

The key to proving \Cref{lemma:iterated log under dependence} is applying the uniform bound on a series of maximal inequalities for the weighted partial sums of dependent random variables within a progressive block. The following lemma gives such maximal inequality.
\begin{lemma} \label{lemma:sup is stopping time}
Assume $\sum_{s=1}^\infty   \delta _{s,q  } < \infty$ for some $q > 2$. Then, it follows that for any $d \in \mathbb{Z}_+$, $\alpha>0$ and $ x>0$, 
$$  \p \bigg (  \max_{ k \in [d, (\alpha+1) d ]}    \frac{| \sum_{i=1}^k    Z_i  | }{ \sqrt { k  }  }  \ge   x \bigg )\le Cx^{-2} ,    $$
where $C > 0$ is an absolute constant .
\end{lemma}   
\begin{proof} 
Let $N = (\alpha+1 )d $ and recall that
$$ S_N^*=  \max_{k=1,\ldots, N} \bigg|  \sum_{i=1}^k    Z_i \bigg|  .   $$
\
\\
Since $\sum_{s=1}^\infty   \delta _{s,2  } < \infty$, it follows from \Cref{lemma:sup of functional dependence} that
 \begin{align*}  
    \quad  \| S^*_N\|_2  \le   C _1 N^{1/2}  .
   \end{align*} 
 Therefore   it holds that 
$$ \p \bigg (  \bigg | \frac{ S^*_N}{\sqrt N } \bigg |\ge x \bigg ) =
\p \bigg (  \bigg | \frac{ S^*_N}{\sqrt N } \bigg |^2 \ge x ^2  \bigg ) \le C_1^2x^{-2}  . $$
Observe  that 
$$  \frac{ |S_N^* | }{\sqrt N }=  \max_{k=1,\ldots, N}  \frac{ | \sum_{i=1}^k    Z_i  | }{ \sqrt N }    \ge  \max_{ k \in [d, (\alpha+1) d ]}    \frac{| \sum_{i=1}^k    Z_i  | }{ \sqrt N  }  \ge  \max_{ k \in [d, (\alpha+1) d ]}    \frac{| \sum_{i=1}^k    Z_i  | }{ \sqrt { (\alpha + 1 ) k  }  }  .  $$ 
Therefore 
$$ \p \bigg (  \max_{ k \in [d, (\alpha+1) d ]}    \frac{| \sum_{i=1}^k    Z_i  | }{ \sqrt { k  }  }  \ge  \sqrt { (\alpha + 1 ) }x \bigg ) \le  \p \bigg (  \bigg | \frac{ S^*_N}{\sqrt N } \bigg |\ge x \bigg ) \le  C_1^2x^{-2}   , $$
which gives 
$$ \p \bigg (  \max_{ k \in [d, (\alpha+1) d ]}    \frac{| \sum_{i=1}^k    Z_i  | }{ \sqrt { k  }  }  \ge   x \bigg )\le  C_2x^{-2}  . $$ 
\end{proof}

 \subsection{Pointwise deviation bounds under dependence}
 
\begin{lemma}\label{lemma:lasso deviation bound 1}Let 
 $1/\gamma  = 1/\gamma_1  +  2/\gamma_2>1$.
Under \Cref{assume:X} and \Cref{assume:epsilon},  it holds for any integer $n \geq 3$ that 
\begin{align*}
  \p\bigg(   \bigg | \frac{1}{n}\sum_{t=1}^n X_{t } \epsilon_t  \bigg |_\infty  \ge  C  \bigg\{  \sqrt { \frac{\log(pn) }{n}  }  +   \frac{ \log^{1/\gamma }( pn ) }{n} \bigg \}  \bigg)   \leq   n^{-5}  
\end{align*}  
for some sufficiently large constant $C$ depending on $C_\epsilon$, $D_{\epsilon}$, $C_X$ and $D_{X}$. Thus when $ n \ge C_\zeta( \s \log(pn) )^{2/\gamma-1}$,
\begin{align*}
  \p\bigg(   \bigg | \frac{1}{n}\sum_{t=1}^n X_{t } \epsilon_t  \bigg |_\infty  \ge  2 C     \sqrt { \frac{\log(pn) }{n}  }     \bigg)   \leq   n^{-5}  
\end{align*}  
\end{lemma}  
 \begin{proof} 
  For any $j \in \{1,\ldots,p\} $,
 let $ Z_t  = X_{t j}\epsilon_t $. Note that $\{Z_t\}_{t \in \mathbb{Z}}$ is a stationary process.     
 Observe that by \Cref{lemma:subweibull},  there exists $C_1$  depending on $C_X$ and $ C_\epsilon$ such that
 $$\p ( |Z_t| \ge \tau) \le 2 \exp(- (\tau/C_1)^ {\gamma_2/2}).$$
For any $ t\in \mathbb Z$, we have that
 \begin{align*} 
 \delta_{s, 2}^Z  =& \|Z_t - Z_{t,\{ t-s\}} \|_2 = \|X_{t j}\epsilon_t - X_{t j,\{ t-s\} }\epsilon_{t,\{ t-s\}}    \|_2
 \\
 \le & \| X_{t j} (\epsilon_t -\epsilon_{t,\{ t-s\}}  )\|_2 +  \|  (X_{t j} -X_{t j, \{ t-s\}}  ) \epsilon_{t, \{ t-s\} }  \|_2
 \\
 \le & \|  X_{t j}\|_4  \| \epsilon_t -\epsilon_{t,\{ t-s\}}   \|_4 +  \|   X_{t j} -X_{t j, \{ t-s\}}   \|_4 \| \epsilon_{t, \{ t-s\} }  \|_4
 \\
 \le & \delta^{\epsilon}_{s ,4}\max_{1 \leq j \leq p}\|X_{0j}\|_4   + \delta^X_{s ,4}  \|\epsilon_0\|_{4}  = \delta^{\epsilon}_{s ,4}\|X_{0\cdot}\|_4   + \delta^X_{s ,4}  \|\epsilon_0\|_{4}.
 \end{align*}
 As a result, it holds that
\begin{align*}
    \Delta_{m, 2} ^Z   =    \sum_{s = m}^{\infty}\delta ^Z   _{s, 2}  \le \|X_{0\cdot}\|_4  \Delta ^\epsilon_{m ,4}  
    +\|\epsilon_0\|_{4}   \Delta ^X_{m ,4}   .
\end{align*}
Consequently, 
 \begin{align*} 
    \sup_{m \geq 0} \exp(cm^{\gamma_1}) \Delta_{m,2} ^Z    \leq&    \|X_{0\cdot}\|_4     \sup_{m \geq 0}\exp(cm^{\gamma_1}) \Delta ^\epsilon_{m ,4}  
    +\|\epsilon_0\|_{4}  \sup_{m \geq 0}\exp(cm^{\gamma_1})  \Delta ^X_{m ,4} \\
    \le &    2\max\{\|\epsilon_0\|_{4}D_X,  \|X_{0\cdot}\|_4D_\epsilon\} < \infty.
\end{align*}  

By \Cref{thm:bernstein_exp_subExp_nonlinear}, it  holds that for any $\tau \ge 1$ and interger $n \geq 3$, 
\begin{align*}
  \p\bigg(   \bigg| \sum_{t=1}^n X_{tj} \epsilon_t  \bigg|  \ge \tau   \bigg) =   \mathbb{P}\Big( \Big|\sum_{ t  = 1}^nZ _t \Big| \geq \tau  \Big) \leq& n\exp\Big(-c_1  \tau ^{\gamma} \Big) + 2 \exp\Big(-\frac{c_2 \tau ^2}{ n   }\Big).
\end{align*}  
By the union bound argument, it follows that for any $\tau \ge 1$,
\begin{align*}
  \p\bigg(   \bigg  | \sum_{t=1}^n X_{t } \epsilon_t  \bigg  |_\infty  \ge \tau   \bigg)  \leq  n p\exp\Big(-c_1  \tau ^{\gamma} \Big) + 2 p\exp\Big(-\frac{c_2 \tau ^2}{ n   }\Big).
\end{align*}  
Therefore, with sufficiently large constant $C_2 $,  
\begin{align*}
  \p\bigg(   \bigg | \frac{1}{n}\sum_{t=1}^n X_{t } \epsilon_t  \bigg |_\infty  \ge  C_2 \bigg\{  \sqrt { \frac{\log(pn) }{n}  }  +   \frac{ \log^{1/\gamma }( pn) }{n} \bigg \}  \bigg)   \leq   n^{-5} .
\end{align*} 
 \end{proof}

\begin{lemma} \label{lemma: mis-specified deviation}
Suppose  \Cref{assume:X}  holds and that $ \{ \alpha_t \}_{t=1}^n \in \mathbb R^p$ is a collection of deterministic vector such that 
$$ \max_{1\le t \le n }  | \alpha_t |_2 \le 1  \quad \text{and that} \quad 
\sum_{t=1}^n \alpha_t = 0. $$ Then it holds for any integer $n \geq 3$ that 
\begin{align*}
  \p\bigg(   \bigg | \frac{1}{n}\sum_{t=1}^n X_{t } X^\top _{t}\alpha_t  \bigg |_\infty  \ge  C  \bigg\{  \sqrt { \frac{\log(pn) }{n}  }  +   \frac{ \log^{1/\gamma }(np) }{n} \bigg \}  \bigg)   \leq   n^{-5},
\end{align*}
for some sufficiently large constant $C$ depending on $C_X$ and $D_{X}$.
Thus when $ n \ge C_\zeta( \s \log(pn) )^{2/\gamma-1}$,
\begin{align*}
  \p\bigg(   \bigg | \frac{1}{n}\sum_{t=1}^n X_{t } X^\top _{t}\alpha_t  \bigg |_\infty  \ge  2C     \sqrt { \frac{\log(pn) }{n}  }   \bigg)   \leq   n^{-5}  .
\end{align*}   
\end{lemma}

\begin{proof} Since $\sum_{t=1}^n \alpha_t = 0$, we have
$$ \E \bigg( \sum_{t=1}^n X_{t } X^\top _{t}\alpha_t  \bigg) =\sum_{t=1}^n  \Sigma \alpha_t  = 0. $$
 For any $j \in \{1,\ldots,p\} $,
 let $ Z_t  = X_{t j} X_{t}^\top \alpha_t  -\E( X_{t j} X_{t}^\top \alpha_t)  $.    
By \Cref{lemma:subweibull},  there exists $C_1$  depending on $ C_X$   such that
 $$\p ( |Z_t | \ge \tau) \le 2 \exp(- (\tau/C_1)^ {\gamma_2/2}).$$
For any $ t\in \mathbb Z$, we have that
 \begin{align*} 
 \delta_{s, 2}^Z = & \|Z_t - Z_{t,\{ t-s\}} \|_2 =  \|X_{t j} X_{t}^\top \alpha _t - X_{t j,\{ t-s\} }X^\top  _{t,\{ t-s\}}  \alpha_t    \|_2
 \\
 \le & \| ( X_{t j} - X_{t j , \{ t-s\}} ) X_t^\top \alpha_t \|_2 
 +  \|  X_{t j , \{ t-s\}}  (X_{t  } ^\top \alpha_t  - X_{t j, \{ t-s\}}^\top \alpha_t   )    \|_2
 \\
 \le & \|   X_{t j} - X_{t j , \{ t-s\}}  \|_4 \|  X_t^\top \alpha_t \|_4 
 +  \|  X_{t j , \{ t-s\}} \|_4 \|    X_{t  } ^\top \alpha_t  - X_{t j, \{ t-s\}}^\top \alpha_t      \|_4
 \\
 \le & 2\delta  ^X_{s ,4}   \sup_{|v|_2 = 1}\|v^{\top}X_{0}\|_4.
 \end{align*}
 As a result, it holds that
\begin{align*}
    \Delta_{m, 2} ^Z   = \sum_{s = m}^\infty  \delta_{s, 2}^Z \leq 2\Delta  ^X_{s ,4}   \sup_{|v|_2 = 1}\|v^{\top}X_{0}\|_4.
\end{align*}
Consequently, 
 \begin{align*} 
    \sup_{m \geq 0}\exp(cm^{\gamma_1}) \Delta_{m,2} ^Z    \leq     2D_X\sup_{|v|_2 = 1}\|v^{\top}X_{0}\|_4 < \infty.
\end{align*}  

By \Cref{thm:bernstein_exp_subExp_nonlinear}, it  holds that for any $\tau \ge 1$ and integer $n \geq 3$, 
\begin{align*}
  \p\bigg(   \bigg|  \sum_{t=1}^n  X_{tj } X^\top _{t}\alpha_t   \bigg|  \ge \tau   \bigg) =   \mathbb{P}\Big( \Big|\sum_{t  = 1}^n Z _i    \Big| \geq \tau  \Big) \leq& n\exp\Big(-c_1  \tau ^{\gamma} \Big) + 2 \exp\Big(-\frac{c_2 \tau ^2}{ n   }\Big),
\end{align*}
where  $ \E (\sum_{t  = 1}^nZ_t)=0$ is used in the  equality.   
By the union bound argument, it follows that for any $\tau \ge 1$ and integer $n \geq 3$,
\begin{align*}
  \p\bigg(   \bigg | \sum_{t=1}^n X_{t } X^\top _{t}\alpha_t  \bigg |_\infty  \ge \tau   \bigg)  \leq  n p\exp\Big(-c_1  \tau ^{\gamma} \Big) + 2 p\exp\Big(-\frac{c_2 \tau ^2}{ n   }\Big).
\end{align*}  
Therefore  with a sufficiently large constant $C_2 $,  
\begin{align*}
  \p\bigg(   \bigg | \frac{1}{n}\sum_{t=1}^n  X_{t } X^\top _{t}\alpha_t   \bigg |_\infty  \ge  C_2 \bigg\{  \sqrt { \frac{\log(pn) }{n}  }  +   \frac{ \log^{1/\gamma }(np) }{n} \bigg \}  \bigg)   \leq   n^{-5} .
\end{align*}  
\end{proof} 
 
% \subsubsection{(Pointwise) restricted eigenvalue conditions (REC) under dependence}
Throughout the rest of this subsection, denote
\begin{align*} 
\mathbb B_0( a )  = \{ v \in \mathbb R^p:  |v |_0\le a  \} ,
\quad 
\mathbb B_1 ( a )  = \{ v \in \mathbb R^p:  |v |_1\le a  \}  
\quad \text{and} \quad 
\mathbb B_2(a)   = \{ v \in \mathbb R^p :  |v |_2\le a   \}.
\end{align*} 
\begin{theorem}[REC Version I]\label{theorem:RES}
Denote $\widehat \Sigma = \frac{1}{n} \sum_{t=1}^n X_t X_t^\top $. Under   \Cref{assume:X},
  it holds that for $n\ge 3$,
$$  \p \bigg(    v^\top   \widehat \Sigma    v  \ge \frac{\Lambda_{\min } (\Sigma )}{2   }     |v  |_2^2  - C    \frac{\log(pn)}{n^\gamma  }       |v  |_1^2   \quad \text{for all } v \in \mathbb R^p \bigg) \ge 1- 2n^{-5},$$  
 where $C > 0$ is an absolute constant.

 \end{theorem}

\begin{proof}  
 
By \Cref{lemma:RES step 2}
   for any integer $s\ge 1$, with $\mathcal K(s) =\mathbb  B_0(s) \cap \mathbb B_2(1) $, there exist  sufficiently large constants
    $C_1, C_2$ such that  
$$ \p\bigg\{  \sup_{v \in \mathcal K(2s) } \bigg| v^\top (\widehat \Sigma - \Sigma ) v\bigg| \ge   \frac{ \tau}{n }  \bigg\} 
\le    \exp\Big(-c_1  \tau ^{\gamma} +\log(n)+C _1s\log(p) \Big) + 2 \exp\Big(-\frac{c_2 \tau ^2}{ n   }  +C_2 s\log(p)  \Big) .
$$ 
  By Lemma 12 in the  supplementary material  of   \cite{loh2012high}, it follows that  
  for any $\tau\ge 1$,
\begin{align*}
  &\p \bigg (   \big | v^\top  (\widehat \Sigma -\Sigma )v\big | \le   \frac{ C_3\tau}{n }   \big  \{   |v  |_2^2 + \frac{1}{s}  |v  |_1^2 \big \}  \quad \text{for all } v \in \mathbb R^p  \bigg)\\
  \ge& 1 - \exp\Big(-c_1  \tau ^{\gamma} +\log(n)+C _1s\log(p) \Big) - 2 \exp\Big(-\frac{c_2 \tau ^2}{ n   }  +C_2 s\log(p)  \Big)   .
 \end{align*}
For $s $  to be chosen later, set
$$\tau = \tau_s  = C_4\bigg\{  \sqrt { n s   \log(pn)    }  +     ( s \log(pn ) )^{1/\gamma }  \bigg \}     ,$$ for some sufficiently large constant  $C_4$. 
Since $ \tau_s\ge 1$, 
with probability at least $ 1-2n^{-5}$, it follows for all $v \in \mathbb R^p$ that
$$     \big | v^\top  (\widehat \Sigma -\Sigma )v\big | \le   C_5  \bigg  \{  \sqrt { \frac{ s\log(pn)   }{n } }    +   \frac{  ( s \log(pn ) )^{1/\gamma } }{ n }\bigg \}  |v  |_2^2   +  C_5 \bigg\{ \sqrt { \frac{  \log(pn)   }{s n } }   + \frac{  ( s \log(pn ) )^{1/\gamma } }{ sn } \bigg \}  |v  |_1^2.$$
Set   
$$s =  \min \bigg\{ \bigg\lfloor \frac {  \Lambda_{\min}^2 (\Sigma) n}{ 16 C_5^2\log(pn)}  \bigg \rfloor  , \bigg\lfloor \frac {  \Lambda_{\min}^\gamma  (\Sigma) n^\gamma }{ 4^\gamma C_5^\gamma   \log(pn)}  \bigg \rfloor   \bigg \}  . $$  
Since $ \gamma<1$, $ s\asymp ( \frac{ n^\gamma}{ \log(pn) })$.
With probability at least 
 $ 1-2  n^{-5}  $, it holds for all $v \in \mathbb R^p$ that 
$$  \bigg| v^\top  (\widehat \Sigma -\Sigma )v\bigg| \le \frac{\Lambda_{\min } (\Sigma )}{2   }     |v |_2^2 + C_6  \bigg( \frac{\log(pn)}{n^{(1+\gamma)/2} } + \frac{\log(pn) }{n^\gamma }\bigg)   |v  |_1^2  \le \frac{\Lambda_{\min } (\Sigma )}{2   }     |v |_2^2 + C_7    \frac{\log(pn) }{n^\gamma }    |v  |_1^2  . $$ 
The desired result follows from the assumption that $\gamma<1  $ and the observation that 
$ v^\top \Sigma v \ge \Lambda_{\min } |v|_2^2  $ for all $v\in \mathbb R^p$.
\end{proof}

 \begin{theorem}[REC Version II] \label{theorem:RES Version II} 
Let $\alpha >1$. 
Denote 
$$ \mathcal C = \{ v \in \mathbb R^p :  | v_{U^c}   |_1 \le  \alpha    | v_{U  } |_1 \quad \text{ for any }  U\subset [1,\ldots,p] \text{ such that } |U|  =  \s   \} . $$
Suppose    \Cref{assume:X} holds.
\
\\
{\bf a.}
There exists an absolute constant $C > 0$ such that for any integer $n\ge 3$, 
\begin{align*}
 \p\bigg\{  \sup_{ v\in \mathcal C } \bigg| v^\top  (\widehat \Sigma -\Sigma )v\bigg|        \ge   C |v|_2^2  \bigg( \sqrt { \frac{\s\log(np)}{ n}}  + \frac{\{ \s \log(np) \} ^{1/\gamma }}{n } \bigg)  \bigg\}    \le n^{-5}.
\end{align*}  
{\bf b.} Suppose in addition that 
$$ n \ge \zeta = C_\zeta\{ \s \log(pn)\}^{2/ \gamma  - 1},   $$ for sufficiently large $C_\zeta > 0$. Then 
there exists an absolute constant $C > 0$ such that 
\begin{align*}
 \p\bigg\{  \sup_{ v\in \mathcal C } \bigg| v^\top  (\widehat \Sigma -\Sigma )v\bigg|        \ge  C |v|_2^2 \sqrt { \frac{\s\log(pn)}{ n}}    \bigg\}    \le n^{-5}.
\end{align*} 
Consequently,
it holds that 
  \begin{align} \label{eq:RES lower bound}
 \p\bigg\{     v^\top  \widehat \Sigma   v     \ge   \frac{ \Lambda_{\min }(\Sigma) }{2}| v|_2^2 \quad \text{for all }  v\in \mathcal C   \bigg\}    \ge 1- n^{-5}.
 \end{align}  
\end{theorem} 
\begin{proof}
For any set $A \subset \mathbb R^p$, let $\CL(A) $ denote the closure of $A$ and $\conv(A) $ denote the convex hull of $A$. Let  
 $$\mathcal K(\s) = \mathbb B_0(\s) \cap \mathbb B_2(1)  .$$ 
By \Cref{lemma:RES step 2}
   for any $\tau\ge  1$, there exist  sufficiently large constants
    $C_1, C_2$ such that  
$$ \p\bigg\{  \sup_{v \in \mathcal K(2\s) } \bigg| v^\top (\widehat \Sigma - \Sigma ) v\bigg| \ge     \frac{\tau}{n }  \bigg\} 
\le    \exp\Big(-c_1  \tau ^{\gamma} +\log(n)+C _1\s\log(p) \Big) + 2 \exp\Big(-\frac{c_2 \tau ^2}{ n   }  +C_2 \s\log(p)  \Big) .
$$ 
 
  By Lemma F.1  in the supplementary of   \cite{basu2015regularized}
 $$  \mathcal C \cap   \mathbb  B_2(1)   \subset (\alpha +2) \CL \{ \conv ( \mathcal K (\s) )\} ,$$
 it follows that
 $$ \sup_{ v\in \mathcal C \cap \mathbb B_2(1) } \bigg| v^\top  (\widehat \Sigma -\Sigma )v\bigg|   
 \le 
\sup_{ v\in (\alpha +2) \CL \{ \conv ( \mathcal K (\s) )\}  }   \bigg| v^\top  (\widehat \Sigma -\Sigma )v\bigg|   .$$
By Lemma F.3  in the supplementary of   \cite{basu2015regularized}
$$ \sup_{ v \in \CL \{ \conv (\mathcal K(\s) )\} } | v^\top D v|  \le 3 \sup_{ v \in \mathcal K(2\s)} |v^\top D v| . $$
Therefore 
$$ 
\sup_{ v\in (\alpha +2) \CL \{ \conv ( \mathcal K (\s) )\}  }   \bigg| v^\top  (\widehat \Sigma -\Sigma )v\bigg|   
\le 3 (\alpha +2)^2 \sup_{ v \in \mathcal K(2\s)} |v^\top D v|  .$$
So
 $$ 
 \frac{1}{ 3(\alpha +2 )^2 }\sup_{  v\in \mathcal C \cap \mathbb B_2(1)   }   \bigg| v^\top  (\widehat \Sigma -\Sigma )v\bigg|   
\le  \sup_{ v \in \mathcal K(2\s)} |v^\top D v| .$$

  It follows that   
\begin{align*}
 \p\bigg\{  \sup_{ v\in \mathcal C \cap \mathbb B_2(1) } \bigg| v^\top  (\widehat \Sigma -\Sigma )v\bigg|    \ge C _1   \frac{\tau}{n}\bigg\} 
\le   \exp\Big(-c_1  \tau ^{\gamma} +\log(n)+C _1\s\log(p) \Big) + 2 \exp\Big(-\frac{c_2 \tau ^2}{ n   }  +C_2 \s\log(p)  \Big) .
\end{align*}
where 
$C_1  = 3  (\alpha +2)^2.   $
For some sufficiently large $C_2$, let
$$ \frac{\tau}{n} =   C_2 \bigg\{  \sqrt {  \frac{\s \log(pn) }{n}  }  +   \frac{ \{ \s \log(np)\} ^{1/\gamma } }{n} \bigg \},  $$
Therefore
\begin{align*}
 \p\bigg\{  \sup_{ v\in \mathcal C \cap \mathbb B_2(1) } \bigg| v^\top  (\widehat \Sigma -\Sigma )v\bigg|    \ge 2 C_2 \bigg(  \sqrt {  \frac{\s \log(pn) }{n}  }  +   \frac{ \{ \s \log(np)\} ^{1/\gamma } }{n} \bigg )  \bigg\} 
\le 2 \exp \bigg ( -5 \s\log(pn)  \bigg) \le n^{-5}.
\end{align*} 
\\
\\
In addition, if $ n\ge C_{\zeta} \{ \s \log(pn)\}^{2/ \gamma  - 1}$, then
$$ \frac{\tau}{n} =   C_2 \bigg\{  \sqrt {  \frac{\s \log(pn) }{n}  }  +   \frac{ \{ \s \log(np)\} ^{1/\gamma } }{n} \bigg \}  \le 2 C_2   \sqrt {  \frac{\s \log(pn) }{n}  } ,  $$
and therefore
\begin{align}\label{eq:RES version 2 step 1}
 \p\bigg\{  \sup_{ v\in \mathcal C \cap \mathbb B_2(1) } \bigg| v^\top  (\widehat \Sigma -\Sigma )v\bigg|    \ge 2 C _2\sqrt { \frac{\s\log(pn)}{ n}}  \bigg\} 
\le 2 \exp \bigg ( -5 \s\log(pn)  \bigg) \le n^{-5}.
\end{align} 
Finally, \Cref{eq:RES lower bound} is a direct consequence of \Cref{eq:RES version 2 step 1} and the fact that
 $$n\ge C_{\zeta} \{ \s \log(pn)\}^{2/ \gamma  - 1}   \ge  C_{\zeta} \s \log(pn) .$$ 
\end{proof}

\begin{definition}
Let $V $ be a subset of $\mathbb R^p$. The  $\epsilon$-net  $\mathcal N_\epsilon $ of  $V $ is a subset of $V$ such that for any $v \in V$, there exists a vector $w\in \mathcal N_\epsilon $ such that 
$$  | v-w |_2 \le \epsilon. $$

\end{definition}

\begin{lemma} \label{lemma:RES step 1}
Let 
 $\frac{ 1}{\gamma}  = \frac{1}{\gamma_1}  +  \frac{ 2} { \gamma_2}$.
Under   \Cref{assume:X},
   for any $\tau\ge  1$  and any non-random $v\in \mathbb R ^p $, it holds that 
$$ \p\bigg\{   \big | v^\top (\widehat \Sigma - \Sigma ) v\big | \ge      \frac{\tau}{n}  \bigg\} 
\le   n  \exp\Big(-c_1  \tau ^{\gamma} \Big) + 2 \exp\Big(-\frac{c_2 \tau ^2}{ n   }\Big) .
$$ 
\end{lemma}

\begin{proof} 
Note   that 
$$v^\top (\widehat \Sigma -\Sigma )v  = \frac{1}{n} \sum_{t =1}^n ( v^\top X_t  )^2 -v^\top \Sigma v .  $$
Let $Z_t =  ( v^\top X_t  )^2 -v^\top \Sigma v$.
Then 
$\E(Z_t) =0$.  
By \Cref{lemma:subweibull},  there exists $C_1$  depending on $ C_X$   such that
 $$\p ( |Z_t  | \ge \tau) \le 2 \exp(- (\tau/C_1)^ {\gamma_2/2}).$$
For any $ t\in \mathbb Z$, we have that
 \begin{align*} 
 \|Z_t - Z_{t,\{ t-s\}} \|_2 =  &  \| ( v^\top X_t  )^2  - ( v^\top X_{t, \{t-s\}}   )^2    \|_2
 \\
 \le &  \big  ( \|   v^\top X_t   \|_4  +  \| v^\top X_{t, \{t-s\}}       \|_4  \big ) \|   v^\top X_t  - v^\top X_{t, \{t-s\}}       \|_4
 \\
 \le & 2 \delta  ^X_{s ,4}   \sup_{|v|_2 = 1}\|v^{\top}X_{0}\|_4.
 \end{align*}
The rest of the proof is similar to the proof of  \Cref{lemma: mis-specified deviation} and is omitted for brevity. 
\end{proof}

\begin{lemma}\label{lemma:RES step 2}For any integer $s\in \mathbb Z_+$,  denote  
  $   \mathcal K(s) = B_0(s) \cap B_2(1) .   $
  Under   \Cref{assume:X},
   for any $\tau\ge  1$ and  any   $s\in \mathbb Z_+ $    , there exist  sufficiently large constants
    $C_1, C_2$ such that  
$$ \p\bigg\{  \sup_{v \in \mathcal K(s) } \bigg| v^\top (\widehat \Sigma - \Sigma ) v\bigg| \ge     \frac{\tau}{n }  \bigg\} 
\le    \exp\Big(-c_1  \tau ^{\gamma} +\log(n)+C _1s\log(p) \Big) + 2 \exp\Big(-\frac{c_2 \tau ^2}{ n   }  +C_2 s\log(p)  \Big)  ,
$$ 

\end{lemma}
\begin{proof} Denote $ D= \widehat \Sigma - \Sigma .$

{\bf Step 1.}   
  For any vector $v$, let  $ \supp(v)$ be the support of $v$.  Let $U \subset[1,\ldots, p] $ be such that $|U| =  s$.  
Denote 
$$ S_U = \{ v \in \mathbb R^p:  |v |_2 \le 1, \supp(v) \subset U  \}.$$ 
Let $\mathcal N_{1/10}$ be the $1/10$-net of $S_U$. Then by standard covering lemma, $|\mathcal N_{1/10} | \le 21^s $. For every 
$v\in S_U$, there exists some $u_i\in \mathcal N_{1/10}$ such that $ |  v- u_i  |_2 \le 1/10 $.
Denote 
$$ \varphi = \sup_{v \in S_U} | v^\top D v|.$$
Then for  any
$v\in S_U$, we have 
\begin{align*}
 v^\top Dv = (v -u_i+u_i)^\top D(v -u_i+u_i)  =    u_i^\top D u_i  +2     u_i^\top D (v - u_i)    +  (v - u_i) ^\top D  (v - u_i)  .
\end{align*}
and therefore 
\begin{align*}
| v^\top Dv| \le& | u_i^\top D u_i |+2  |  u_i^\top D (v - u_i) |  + |(v - u_i) ^\top D  (v - u_i) |.
\end{align*}
Note that 
$$
  | u_i^\top D u_i | \le \sup_{ u \in  \mathcal N_{1/10 }} |u^\top D u|  .$$
Since 
$ 10 (v-u_i ) \in S_U$, it follows that 
$$|(v - u_i) ^\top D  (v - u_i) |\le \varphi/100. $$
In addition, denote $\Delta v  = 10( v - u_i) $. Then   $  \Delta v \in S_U$. 
\begin{align*}
   2\bigg| u_i^\top D \Delta v   \bigg|  =  &  \bigg|    ( \Delta v +u_i)   ^\top D   (\Delta v+u_i)      -  u_i ^\top Du_i  -  \Delta v ^\top D \Delta v \bigg| 
  \\= &  \bigg|  4  (  \frac{ \Delta v +u_i}{2 })   ^\top D   ( \frac{ \Delta v +u_i}{2 } )  \bigg|     + \bigg|   u_i ^\top Du_i   \bigg|  +  \bigg|   \Delta v ^\top D \Delta v   \bigg| 
  \\
  \le & 4 \varphi + \sup_{ u \in  \mathcal N_{1/10 }} |u^\top D u|   + \varphi, 
\end{align*}
where $ \frac{ \Delta v + u_i}{2 }\in S_U  $ because $ | \Delta v+ u_i   |_2\le 2 $.
So 
\begin{align*}
 2|u_i^\top D (v-u_i) | = \frac{1}{10}   \bigg| u_i^\top D \Delta v   \bigg|  
  \le &  \frac{4}{10} \varphi +  \frac{\sup_{ u \in  \mathcal N_{1/10 }} |u^\top D u| }{10 }  + \frac{\varphi}{10}. 
\end{align*}
Putting these calculations together,
it holds that for any $ v \in S_U$,
\begin{align*}
  | v^\top Dv| \le& | u_i^\top D u_i |+2  |  u_i^\top D (v - u_i) |  + |(v - u_i) ^\top D  (v - u_i) | 
\\
\le &  \sup_{ u \in  \mathcal N_{1/10 }} |u^\top D u|  + \frac{4}{10} \varphi +  \frac{\sup_{ u \in  \mathcal N_{1/10 }} |u^\top D u| }{10 }  + \frac{\varphi}{10} + \varphi/100.
\end{align*}
Taking the sup over  all $v \in S_U$, 
$$\varphi = \sup_{v \in S_U }  | v^\top Dv|  \le  \sup_{ u \in  \mathcal N_{1/10 }} |u^\top D u|  + \frac{4}{10} \varphi +  \frac{\sup_{ u \in  \mathcal N_{1/10 }} |u^\top D u| }{10 }  + \frac{\varphi}{10} + \varphi/100 . $$
Rearranging this gives 
$$\varphi  (1- \frac{4}{10} -\frac{1}{10} -\frac{1}{100}) \le  \frac{11}{10} \sup_{ u \in  \mathcal N_{1/10 }} |u^\top D u|     , $$
which leads to 
$$ \sup_{v \in S_U} | v^\top D v| = \varphi \le 3   \sup_{ u \in  \mathcal N_{1/10 }} |u^\top D u|  .$$

{\bf Step 2.} It suffices to consider 
$$ \sup_{ u \in  \mathcal N_{1/10 }} |u^\top D u|.$$
Note that by union bound, 
\begin{align*}
\p\bigg( \sup_{ u \in  \mathcal N_{1/10 }} |u^\top D u| \ge \tau  \bigg)
\le& |\mathcal N_{1/10}| \sup_{ u  \in \mathcal N_{1/10 }  } \p (|u^\top D u| \ge \tau  )
\\
\le&   n 21^s  \exp\Big(-c_1  \tau ^{\gamma} \Big) +  2 \cdot  21^s   \exp\Big(-\frac{c_2 \tau ^2}{ n   }\Big) ,
 \end{align*}
 where the second inequality follows from \Cref{lemma:RES step 1}.  It follows from  {\bf Step 1} that 
 \begin{align*}
\p\bigg( \sup_{v \in S_U} | v^\top D v|
  \ge 3\tau  \bigg)
\le  n 21^s  \exp\Big(-c_1  \tau ^{\gamma} \Big) +  2 \cdot  21^s   \exp\Big(-\frac{c_2 \tau ^2}{ n   }\Big) ,
 \end{align*}

{\bf Step 3.} Note that 
$$ \mathcal K(s) = \bigcup_{|U| =  s } S_U . $$
  There are at most 
$ \binom{p}{s} \le p^s $ choices of $U$.
So 
taking another union bound, 
we have  
\begin{align*}
\p\bigg( \sup_{ v  \in   \mathcal K(s) } |u^\top D u| \ge 3\tau  \bigg)
\le    n p^s 21^s  \exp\Big(-c_1  \tau ^{\gamma} \Big) +  2  p^s 21^s   \exp\Big(-\frac{c_2 \tau ^2}{ n   }\Big) .
 \end{align*}
 This immediately leads to the desired result.
\end{proof}

\subsection{Uniform deviation bounds under dependence}
Without loss of generality, assume  that 
$ p\ge n^\alpha $ for some $\alpha>0$.
\begin{lemma} \label{lemma:lasso l infty deviations}
Suppose  \Cref{assume:X} and \Cref{assume:epsilon} hold. 
Then  
\begin{align*} 
  \p\bigg(    \bigg  |   \sum_{t =1}^r X_t \epsilon_t    \bigg |_\infty \ge C_2  \bigg( \sqrt {  r \log(p )  }  +   \log ^{1/\gamma }(p )  \bigg)   \text{ for all } 3\le r\le n  \bigg)\le n^{-4} 
 \end{align*}
Suppose in addition $u\in \mathbb R^p$ is a deterministic vector such that $ |u|_2 = 1 $.
Then it holds that 
 \begin{align*} 
 &\p\bigg(  \bigg| \sum_{t=1}^r  u^\top X_t X_{t j} - u^\top \Sigma_{ \cdot j }  \bigg|    \ge C_1  \bigg( \sqrt { r  \log(np )  }  +  \log ^{1/\gamma }(np )  \bigg)    \text{ for all } 3\le r\le n  \bigg)\le n^{-4} .
 \end{align*}  
\end{lemma}
\begin{proof} For the first probability bound, by \Cref{lemma:lasso deviation bound 1}, for any $r\in\{3, 4, \ldots, n\}$,
\begin{align*} 
  \p\bigg(    \bigg  |   \sum_{i=1}^r X_i\epsilon_i   \bigg |_\infty \ge C_2  \bigg( \sqrt {  r \log(p )  }  +   \log ^{1/\gamma }(p )  \bigg)     \bigg)\le n^{-5} 
 \end{align*}    
So by a union bound, 
\begin{align*} 
  \p\bigg(    \bigg  |   \sum_{i=1}^r X_i\epsilon_i   \bigg |_\infty \ge C_2  \bigg( \sqrt {  r \log(p )  }  +   \log ^{1/\gamma }(p )  \bigg)   \text{ for all } 3\le r\le n  \bigg)\le n^{-4} 
 \end{align*} 
\\
\\
For the second probability bound,    for fixed $ j \in [ 1,\ldots, p]$, let
 $$ W_t  = u^\top X_t X_{t j} - u^\top \Sigma_{ \cdot j },$$
 where $\Sigma_{ \cdot j }$ denote the $j$-th column of $\Sigma$.
For any $ t\in \mathbb Z$, we have that
 \begin{align*} 
  |W_t - W _{t,\{ t-s\}}  |_2 =  &   |u^\top X_t X_{t j}  - u^\top \Sigma_{\dot j } - \big\{  u^\top X_{t, \{ t-s\}}  X_{t j, \{ t-s\} }  - u^\top \Sigma_{\dot j }     \big\}   |_2
 \\
 \le &  |u^\top X_t X_{t j}  -  u^\top X_{t, \{ t-s\}}  X_{t j  }  |_2 -   |  u^\top X_{t, \{ t-s\}}  X_{t j  }  -  u^\top X_{t, \{ t-s\}}  X_{t j, \{ t-s\} }   |_2
 \\
 \le &  |  X_{t j} |_4  | u^\top X_t   -u^\top X_{t, \{ t-s\}}    |_4 +   |   X_{t j} -X_{t j, \{ t-s\}}   |_4  | u^\top X_{t  , \{ t-s\} }    |_4
 \\
 \le & \Delta ^X_{0 ,4}  \delta^{X}_{s ,4}   + \delta^X_{s ,4}  \Delta ^{X }_{0 ,4}  .
 \end{align*}
 As a result, it holds that
\begin{align*}
  \delta_{s, 2}^W   : = \sup_{t \in \mathbb{Z}} |W_t - W _{t,\{ t-s\}}  |_2   \le \Delta ^X_{0 ,4}  \delta^{X}_{s ,4}   + \delta^X_{s ,4} \Delta ^{X}_{0 ,4}  
\end{align*}
and  that 
\begin{align*}
    \Delta_{m, 2} ^W    =    \sum_{s = m}^{\infty}\delta ^W   _{s, 2}  \le \Delta ^X_{0 ,4}   \sum_{s=m}^\infty\delta^{X}_{s ,4}  +   \Delta ^X_{0 ,4} \sum_{s=m}^\infty\delta^{X}_{s ,4}  
      =    2 \Delta ^X_{0 ,4}  \Delta ^X _{m ,4}  .
\end{align*}
Consequently, 
 \begin{align*} 
    \sup_{m \geq 0} \exp(cm^{\gamma_1}) \Delta_{m,2} ^Z    \leq   2  \Delta ^X_{0 ,4}     \sup_{m \geq 0}\exp(cm^{\gamma_1}) \Delta ^X_{m ,4}       
    \le     2D_X ^2 .
\end{align*}  
\
\\ 
Let 
 $\frac{ 1}{\gamma}  = \frac{1}{\gamma_1}  +  \frac{ 2} { \gamma_2}$. By \Cref{thm:bernstein_exp_subExp_nonlinear}, it  holds that for any $\tau \ge 1$ and $r\in\{ 3,4,\ldots, n\}$, 
\begin{align*}
  \p\bigg(   \bigg| \sum_{t=1}^r  u^\top X_t X_{t j} - u^\top \Sigma_{ \cdot j }  \bigg|  \ge \tau   \bigg) =   \mathbb{P}\Big( \Big|\sum_{ t  = 1}^r  Z _t \Big| \geq \tau  \Big) \leq& r \exp\Big(-c_1  \tau ^{\gamma} \Big) + 2 \exp\Big(-\frac{c_2 \tau ^2}{ r    }\Big).
\end{align*}  
By   union bound, it follows that for any $\tau \ge 1$,
\begin{align*}
   \p\bigg(   \bigg| \sum_{t=1}^r  u^\top X_t X_{t j} - u^\top \Sigma_{ \cdot j }  \bigg|  \ge \tau   \bigg)  \leq  r p\exp\Big(-c_1  \tau ^{\gamma} \Big) + 2 p\exp\Big(-\frac{c_2 \tau ^2}{ r    }\Big).
\end{align*}  
Therefore, with sufficiently large constant $C_2 $,  
\begin{align*}
  \p\bigg(  \bigg| \sum_{t=1}^r  u^\top X_t X_{t j} - u^\top \Sigma_{ \cdot j }  \bigg|  _\infty  \ge  C_2 \bigg\{  \sqrt {  r  \log(pn)    }  +    \log^{1/\gamma }( pn)   \bigg \}  \bigg)   \leq   n^{-5} .
\end{align*}   
 By  another  union bound,
 \begin{align*} 
 &\p\bigg(  \bigg| \sum_{t=1}^r  u^\top X_t X_{t j} - u^\top \Sigma_{ \cdot j }  \bigg|    \ge C_1  \bigg( \sqrt { r  \log(np )  }  +  \log ^{1/\gamma }(np )  \bigg)    \text{ for all } 3\le r\le n  \bigg)\le n^{-4} .
 \end{align*} 
\end{proof}

 \begin{lemma}  \label{lemma:RES Version II} 
Suppose    \Cref{assume:X} holds. Let $\alpha >1$ and 
$$ \mathcal C = \{ v \in \mathbb R^p :  | v_{U^c}   |_1 \le  \alpha    | v_{U  } |_1 \quad \text{ for any }  U\subset [1,\ldots,p] \text{ such that } |U|  =  \s   \} . $$ 
Then there exists an absolute constant $C > 0$ such that 
\begin{align*}
 \p\bigg\{    \bigg| v^\top  \bigg\{  \sum_{i=1}^r X_iX_i^\top  -\Sigma\bigg \} v\bigg|        \ge   C  |v|_2^2   \bigg( \sqrt {  r \s\log(np) }  +  \{  \s \log(np) \} ^{1/\gamma }  \bigg)   \ 
 \forall  r \in\{3, \ldots, n\}  \ \forall v\in\mathcal C  \bigg\}    \le n^{-4}.
\end{align*}  
 
\end{lemma} 
 \begin{proof}
 By \Cref{theorem:RES Version II}, 
 \begin{align*}
 \p\bigg\{  \sup_{ v\in \mathcal C  } \bigg| v^\top  \bigg\{  \sum_{i=1}^r X_iX_i^\top  -\Sigma\bigg \} v\bigg|        \ge   C  |v|_2^2   \bigg( \sqrt {  r \s\log(np) }  +  \{  \s \log(np) \} ^{1/\gamma }    \bigg)  \ \forall v\in\mathcal C     \bigg\}    \le n^{-5}.
\end{align*}    
The desired result follows from a union bound. 
 \end{proof}

\section{Proof of a Bernstein inequality under temporal dependence (Theorem 4)}\label{sec-proof-th4}

\subsection{Introduction of nonstationary processes}
Let $\{Z_t\}_{t = \mathbb{Z}}$ be an $\mathbb{R}$-valued (possibly) nonstationary process, such that
    \begin{equation}\label{eq:nonstationary-2}
        Z_t = g_t(\mathcal{F}_t),
    \end{equation}
where $\{g_t(\cdot)\}_{t \in \mathbb{Z}}$ are time-dependent $\mathbb{R}$-valued measurable functions and $\mathcal{F}_{t} = \{\epsilon_s\}_{s \leq t}$ for any $t \in \mathbb{Z}$ and $\mathcal{F}_{-\infty} = \{\emptyset, \Omega\}$. Let $s_1, s_2 \in \mathbb{Z}$, such that $s_2 \leq s_1 \leq t$. Let $\epsilon_t^*$ be an independent copy of $\epsilon_t$. For notational convenience, let $\{\epsilon_{s_1+1}, \dots, \epsilon_t\} = \emptyset$ when $s_1+1 > t$. Define $Z_{t, \{s_1, s_2\}} = g_t(\mathcal{F}_{t,\{s_1, s_2\}})$ with $\mathcal{F}_{t,\{s_1, s_2\}} = \{\dots, \epsilon_{s_2-1}, \epsilon_{s_2}^*, \dots, \epsilon_{s_1}^*, \epsilon_{s_1+1}, \dots, \epsilon_t\}$. Write $Z_{t, \{s,s\}} = Z_{t,\{s\}}$. Under the general nonstationary representation \eqref{eq:nonstationary-2}, the functional dependence measure and its tail cumulative version are defined respectively as
\begin{equation*}
  \delta_{s, q} = \sup_{t \in \mathbb{Z}}\Vert Z_{t} - Z_{t,\{t-s\}} \Vert_{q} \;\;\text{and}\;\; \Delta_{m, q} = \sum_{s = m}^{\infty}\delta_{s, q}.
\end{equation*}
Note that in the special case when $g_t(\cdot) = g(\cdot)$, for all $t = 1, \dots, n$, the representation \eqref{eq:nonstationary-2} reduces to that of a stationary time series
\begin{equation*}
    Z_t = g(\mathcal{F}_t),
\end{equation*}
and the functional dependence measure and related conditions can be defined accordingly.

\subsection{Auxiliary lemmas}
The first lemma can be used to relate the Laplace transform of a partial sum to the products of the Laplace transforms of each individual summand.
\begin{lemma}\label{lemma:diff_laplace_joint_marginal}
Let $\{Z_i\}_{i \in \mathbb{Z}}$ be an $\mathbb{R}$-valued possibly nonstationary process of form \eqref{eq:nonstationary-2}. Assume there exists a positive $M$ such that $|Z_i| \leq M$ for any $i \in \mathbb{Z}$. Then for any $a > 0$, we have
\begin{equation*}
\begin{aligned}
    \bigg| \mathbb{E}\Big[\exp\Big(a\sum_{i = 1}^n Z_i \Big) \Big] - \prod_{i = 1}^n \mathbb{E}[\exp(aZ_i)] \bigg| \leq& a\exp(anM)\sum_{i = 2}^n\Vert Z_i - Z_{i, \{i-1, -\infty\}} \Vert_2\\
    \leq& a\exp(anM)(n-1) \Delta_{1,2}.
\end{aligned}
\end{equation*}
\end{lemma}
\begin{proof}
For the product of the Laplace transforms of each individual random variable, we have the telescoping decomposition as
\begin{align*}
    \prod_{i = 1}^n \mathbb{E}[\exp(aZ_i)] & = \mathbb{E}\Big[\exp\Big(a\sum_{i = 1}^{n-1} Z_i \Big) \Big]\mathbb{E}[\exp(aZ_n)]\\
    &+\mathbb{E}\Big[\exp\Big(a\sum_{i = 1}^{n-2} Z_i \Big) \Big]\prod_{j = n-1}^n\mathbb{E}[\exp(aZ_j)] - \mathbb{E}\Big[\exp\Big(a\sum_{i = 1}^{n-1} Z_i \Big) \Big]\mathbb{E}[\exp(aZ_n)]\\
    &+\mathbb{E}\Big[\exp\Big(a\sum_{i = 1}^{n-3} Z_i \Big) \Big]\prod_{j = n-2}^n\mathbb{E}[\exp(aZ_j)] - \mathbb{E}\Big[\exp\Big(a\sum_{i = 1}^{n-2} Z_i \Big) \Big]\prod_{j = n-1}^n\mathbb{E}[\exp(aZ_n)]\\
    &+ \dots\\
    &+\prod_{j = 1}^n\mathbb{E}[\exp(aZ_j)] - \mathbb{E}\Big[\exp\Big(a\sum_{i = 1}^{2} Z_i \Big) \Big]\prod_{j = 3}^n\mathbb{E}[\exp(aZ_n)].
\end{align*}
For notational simplicity, given a real value sequence $\{b_i\}_{i = 1}^n$, we write $\prod_{i = n+1}^n b_i = 1$. 
Then, it satisfies that
\begin{align}\label{eq:laplace_extract_all}
    &\mathbb{E}\Big[\exp\Big(a\sum_{i = 1}^n Z_i \Big) \Big] - \prod_{i = 1}^n \mathbb{E}[\exp(aZ_i)] \nonumber \\
    =& \sum_{s = 2}^{n}\bigg\{\bigg(\mathbb{E}\Big[\exp\Big(a\sum_{i = 1}^{s} Z_i \Big) \Big] - \mathbb{E}\Big[\exp\Big(a\sum_{i = 1}^{s-1} Z_i \Big) \Big]\mathbb{E}[\exp(aZ_s)] \bigg)\prod_{j = s+1}^n\mathbb{E}[\exp(aZ_j)] \bigg\}.
\end{align} 
Using coupling, we have that $Z_{s, \{s-1, -\infty\}}$ and $Z_{s^{\prime}}$ are independent for any $s^{\prime} \leq s-1$, and $Z_{s, \{s-1, -\infty\}}$ and $Z_s$ have the same distribution. We have that 
\begin{align}\label{eq:laplace_extract_one}
    &\bigg|\mathbb{E}\Big[\exp\Big(a\sum_{i = 1}^{s} Z_i \Big) \Big] - \mathbb{E}\Big[\exp\Big(a\sum_{i = 1}^{s-1} Z_i \Big) \Big]\mathbb{E}[\exp(aZ_s)]\bigg|\prod_{j = s+1}^n\mathbb{E}[\exp(aZ_j)]\nonumber \\
    =& \bigg|\mathbb{E}\Big[\exp\Big(a\sum_{i = 1}^{s-1} Z_i \Big) \Big(\exp(aZ_s) - \exp(aZ_{s,\{s-1, -\infty\}})\Big)\Big]\bigg| \prod_{j = s+1}^n\mathbb{E}[\exp(aZ_j)]\nonumber \\
    \leq& a\exp(anM)\mathbb{E}\big|Z_s - Z_{s,\{s-1, -\infty\}}\big|\leq  a\exp(anM)\big\Vert Z_s - Z_{s,\{s-1, -\infty\}}\big\Vert_2,    
\end{align} 
where the first inequality is due to the mean value theorem and the fact that $Z_i$ are bounded, and the second inequality follows from H\"older's inequality.
Combining \eqref{eq:laplace_extract_all} and \eqref{eq:laplace_extract_one}, we have
\begin{equation*}
\begin{aligned}
    \bigg|\mathbb{E}\Big[\exp\Big(a\sum_{i = 1}^n Z_i \Big) \Big] - \prod_{i = 1}^n \mathbb{E}[\exp(aZ_i)]\bigg| \leq& a\exp(anM)\sum_{i = 2}^n\big\Vert Z_i - Z_{i,\{i-1, -\infty\}}\big\Vert_2\\
    \leq& a(n-1)\exp(anM)\sum_{j = 1}^{\infty}\delta_{j,2},
\end{aligned}
\end{equation*}
where the second inequality follows from the definition of $\delta_{s,q}$.
\end{proof}

The following lemma can be used to relate the bound of the log-Laplace transform of sum of random variables to that of each individual random variable. This lemma is the Lemma 3 in \cite{merlevede2011bernstein}, we reproduce it for completeness.
\begin{lemma}\label{lemma:aggregate_log-laplace}
    Let $Z_1, Z_2, \dots$ be a sequence of $\mathbb{R}$-valued random variables. Assume that there exist positive constants $\sigma_1, \sigma_2, \dots$ and $c_1, c_2, \dots$ such that, for any positive integer $i$ and any $t \in [0, 1/c_i)$,
    \begin{equation*}
        \log\mathbb{E}[\exp(tZ_i)] \leq (\sigma_it)^2/(1-c_it).
    \end{equation*}
    Then, for any positive $n$ and any $t$ in $[0,1/(c_1+c_2+ \cdots+c_n))$,
    \begin{equation}\label{eq:tensorization}
        \log\mathbb{E}[\exp(t(Z_1+Z_2+\cdots+Z_n))] \leq (\sigma t)^2/(1-Ct),
    \end{equation}
    where $\sigma = \sigma_1+\sigma_2+\cdots+\sigma_n$ and $C = c_1+c_2+\cdots+c_n$.
\end{lemma}
\begin{proof}
    For $i \geq 1$, denote the partial sums $S_i = \sum_{j = 1}^iZ_j$. The proof is by induction.
    For $n = 1$, we have $\sigma = \sigma_1$ and $C = c_1$, and \eqref{eq:tensorization} holds obviously.
    
    Assuming \eqref{eq:tensorization} holds for $n = k$, i.e.~for any $t \in [0,1/(c_1+c_2+\cdots+c_k))$, it satisfies that
    \begin{equation*}
        \log\mathbb{E}[\exp(tS_k)] \leq \Big(t\sum_{j=1}^k\sigma_j\Big)^2\Big/\Big(1-t\sum_{j=1}^kc_j\Big).
    \end{equation*}
    For $n = k+1$, by H\"older's inequality we have for any $u \in (0,1)$ that
    \begin{equation} \label{eq:tensorization_stepk}
        \log\mathbb{E}[\exp(t(S_k+Z_{k+1}))] \leq u\log\mathbb{E}[\exp(u^{-1}tS_k)] + (1-u)\log\mathbb{E}[\exp((1-u)^{-1}tZ_{k+1})].
    \end{equation}
    Choose
    \begin{equation*}
        u = \Big(\sum_{j = 1}^{k}\sigma_j\Big/\sum_{j = 1}^{k+1}\sigma_j\Big)\Big(1 - t\sum_{j = 1}^{k+1}c_j\Big) + t\sum_{j = 1}^{k}c_j,
    \end{equation*}
    and thus 
    \begin{equation*}
        1-u = \Big(\sigma_{k+1}\Big/\sum_{j = 1}^{k+1}\sigma_j\Big)\Big(1 - t\sum_{j = 1}^{k+1}c_j\Big) + tc_{k+1}.        
    \end{equation*}
    Since $1/(c_1+c_2+\cdots+c_{k+1})$ is less than $1/(c_1+c_2+\cdots+c_{k})$ and $1/c_{k+1}$,
    we have for any $t \in [0, 1/(c_1+c_2+\cdots+c_{k+1}))$ that
    \begin{equation*}
        \eqref{eq:tensorization_stepk} \leq \frac{t^2\big(\sum_{j=1}^{k}\sigma_j\big)^2}{u - t\sum_{j=1}^{k}c_j} + \frac{t^2\sigma_{k+1}^2}{(1-u)-tc_{k+1}} = \frac{\big(t\sum_{j=1}^{k+1}\sigma_j\big)^2}{1 - t\sum_{j = 1}^{k+1}c_j},
    \end{equation*}
    which completes the proof.
\end{proof}

The following lemma is the Burkholder's moment inequality for martingale.
\begin{lemma}[Theorem 2.1 in \cite{rio2009moment}]\label{lemma:burkholder}
Let $M_n = \sum_{i = 1}^n\xi_i$, where $\{\xi_i\}_{i \in \mathbb{Z}}$ are martingale differences. Then for any $q \geq 2$,
\begin{equation*}
    \Vert M_n \Vert_q \leq \sqrt{q-1}\Big(\sum_{i = 1}^{n}\Vert \xi_i \Vert_q^2 \Big)^{1/2}.
\end{equation*}
\end{lemma}

For a possibly nonstationary process with dependence measure decays exponentially, the next lemma gives an upper bound on its long-run variance.
\begin{lemma}\label{lemma:lrv_univar}
    Let $\{X_i\}_{i \in \mathbb{Z}}$ be a sequence of $\mathbb{R}$-valued random variables in the form of \eqref{eq:nonstationary-2}. Assume $\sup_{i \in \mathbb{Z}}\Vert X_i \Vert_2 < \infty$ and there exist some absolute constants $c, C_{\mathrm{FDM}}, \gamma_1 > 0$ such that
    \begin{equation*}
        \sup_{m \geq 0}\exp(cm^{\gamma_1})\Delta_{m,2} \leq C_{\mathrm{FDM}}.
    \end{equation*}
    Then, the long-run variance of $\{X_t\}_{t \in \mathbb{Z}}$ satisfies that
    \begin{equation}\label{eq:def_LRV_UB}
        \lim_{n \to \infty}\var\Big(\frac{1}{\sqrt{n}}\sum_{i = 1}^n X_i\Big) \leq \sum_{l = -\infty}^{\infty}\sup_{t \in \mathbb{Z}}\big|\cov(X_{t}, X_{t+l})\big| = C_{\mathrm{LRV}} < \infty.
    \end{equation}
\end{lemma}
\begin{proof}
    Define the projection operator $\mathcal{P}_j\cdot = \mathbb{E}[\cdot|\mathcal{F}_j] - \mathbb{E}[\cdot| \mathcal{F}_{j-1}]$ with $j \in \mathbb{Z}$.  A random variable $X_i$ is decomposed as
    \begin{equation*}
        X_i - \mathbb{E}[X_i] = \sum_{k = 0}^{\infty}\big(\mathbb{E}[X_i|\mathcal{F}_{i-k}] - \mathbb{E}[X_i| \mathcal{F}_{i-k-1}]\big) = \sum_{k = 0}^{\infty}\mathcal{P}_{i-k}X_i.
    \end{equation*}
    It holds for any $l \geq 0$ that 
    \begin{equation*}
    \begin{aligned}
        &\sup_{t \in \mathbb{Z}}\big|\cov(X_{t}, X_{t+l})\big| = \sup_{t \in \mathbb{Z}}\bigg|\mathbb{E}\Big[\Big(\sum_{k = 0}^{\infty}\mathcal{P}_{-k}X_t\Big)\Big(\sum_{k = 0}^{\infty}\mathcal{P}_{l-k}X_{t+l}\Big)\Big]\bigg|\\
        =& \sup_{t \in \mathbb{Z}}\bigg|\sum_{k =     0}^{\infty}\mathbb{E}\big[(\mathcal{P}_{-k}X_t)(\mathcal{P}_{-k}X_{t+l})\big]\bigg| \leq \sum_{k =     0}^{\infty}\sup_{t \in \mathbb{Z}}\Big|\mathbb{E}\big[(\mathcal{P}_{-k}X_t)(\mathcal{P}_{-k}X_{t+l})\big]\Big|\\
        \leq& \sum_{k = 0}^{\infty}\sup_{t \in \mathbb{Z}}\Vert \mathcal{P}_{-k}X_{t} \Vert_2 \Vert \mathcal{P}_{-k}X_{t+l}\Vert_2 \leq \sum_{k = 0}^{\infty} \delta_{k,2}\delta_{k+l,2} \leq \sqrt{\sum_{k = 0}^{\infty} \delta_{k,2}^2}\sqrt{\sum_{k = 0}^{\infty} \delta_{k+l,2}^2}\\
        \leq& \Big(\sum_{k = 0}^{\infty} \delta_{k,2}\Big) \Big(\sum_{k = l}^{\infty} \delta_{k,2}\Big) = \Delta_{0,2}\Delta_{l,2},
    \end{aligned}
    \end{equation*}
    where the first inequality follows the triangle inequality and the second and fourth inequalities follow H\"older's inequality. The second equality also follows the orthogonality of $\mathcal{P}_j\cdot$, i.e.~for $i < j$
    \begin{equation*}
        \mathbb{E}[(\mathcal{P}_iX_r)(\mathcal{P}_jX_s)] = \mathbb{E}[\mathbb{E}[(\mathcal{P}_iX_r)(\mathcal{P}_jX_s)|\mathcal{F}_{i}]] = \mathbb{E}[(\mathcal{P}_iX_r)\mathbb{E}[X_s-X_s|\mathcal{F}_i]] = 0,
    \end{equation*}
    and the orthogonality also holds for $i > j$ by symmetry. The third inequality is due to the fact that
    \begin{align*}
        \Vert \mathcal{P}_{j}X_i \Vert_2 =& \Vert \mathbb{E}[X_i| \mathcal{F}_{j}] - \mathbb{E}[X_i| \mathcal{F}_{j-1}]\Vert_2 = \Vert\mathbb{E}[X_i| \mathcal{F}_{j}] - \mathbb{E}[X_{i,\{j\}}| \mathcal{F}_{j-1}]\Vert_2\\
       =& \Vert\mathbb{E}[X_i-X_{i,\{j\}}| \mathcal{F}_{j}]\Vert_2 \leq \Vert X_{i}-X_{i,\{j\}}\Vert_2 \leq \delta_{i-j,2},    
    \end{align*}
    where the second and the third equality follows the definition of the coupled random variables $X_{i,\{j\}}$, the first inequality follows Jensen's inequality, and the second inequality follows the definition of the functional dependence measure. By the same arguments, we have for any $l < 0$ that
    \begin{equation*}
    \begin{aligned}
        &\sup_{t \in \mathbb{Z}}\big|\cov(X_t, X_{t+l})\big| \leq \Delta_{0,2}\Delta_{-l,2}.
    \end{aligned}
    \end{equation*}
Therefore, we have that
\begin{equation}\label{eq:LRV_univariate_ub}
\begin{aligned}
    &C_{\mathrm{LRV}} = \sum_{l = -\infty}^{\infty}\sup_{t \in \mathbb{Z}}\big|\cov(X_t, X_{t+l})\big| \leq 2\Delta_{0,2}\sum_{l = 0}^{\infty}\Delta_{l,2} \leq 2C_{\mathrm{FDM}}^2\sum_{l = 0}^{\infty}\exp(-cl^{\gamma_1}).
\end{aligned}
\end{equation}
To bound \eqref{eq:LRV_univariate_ub}, we compare the series $\{a_m = \exp(-cm^{\gamma_1})\}_{m = 1}^{\infty}$ and $\{b_m = (m+1)^{-\nu}\}_{m=1}^{\infty}$ with $\nu > 1$. Since by the properties of the Riemann zeta function, it holds that $\sum_{m = 1}^{\infty}b_m < \infty$.

We embed the series $\{a_m\}_{m = 1}^{\infty}$ and $\{b_m\}_{m=1}^{\infty}$ into continuous time processes by defining $a_x = a_{\lceil x \rceil}$ and $b_x = b_{\lceil x \rceil}$ for $x \in [1,\infty)$. Since for any $x \in [1,\infty)$, $a_x, b_x > 0$ and we define a function $g(x)$ on $x \in [1, \infty)$ as
\begin{equation*}
    g(x) = \log\bigg(\frac{a_x}{b_x}\bigg) = -cx^{\gamma_1} + \nu\log(x+1).
\end{equation*}
We have the derivative as
\begin{equation*}
    g^{\prime}(x) = -c\gamma_1x^{\gamma_1-1} + \frac{\nu}{x+1}.
\end{equation*}
Since for any absolute constants $c, \gamma_1 > 0$ and $\nu > 1$, there exists a finite $x^* \geq 1$ such that for any $x \geq x^*$, it holds that
\begin{equation*}
    \exp(-cx^{\gamma_1}) \leq (x+1)^{-\nu} \;\; \text{and} \;\; g^{\prime}(x) < 0.
\end{equation*}
Letting $m^* = \lceil x^* \rceil$, we have that
\begin{align*}
    C_{\mathrm{LRV}} \leq& 2C_{\mathrm{FDM}}^2\sum_{m = 0}^{\infty}\exp(-cm^{\gamma_1}) = 2C_{\mathrm{FDM}}^2\sum_{m = 0}^{m^*}\exp(-cm^{\gamma_1}) + 2C_{\mathrm{FDM}}^2\sum_{m = m^*+1}^{\infty}\exp(-cm^{\gamma_1})\\
    \leq& 2C_{\mathrm{FDM}}^2\sum_{m = 0}^{m^*}\exp(-cm^{\gamma_1}) + 2C_{\mathrm{FDM}}^2\sum_{m = m^*+1}^{\infty}(m + 1)^{-\nu} < \infty.    
\end{align*}
\end{proof}

We compare the following two conditions in the next lemma. (i) There exist some absolute constants $c^{\prime}, \nu > 0$ such that
    \begin{equation}\label{ass:polynomial_decay_order_2}
        c^{\prime}\sup_{m \geq 0}(m+1)^{\nu}\Delta_{m,2} \leq 1.
    \end{equation}
    (ii) There exist some absolute constants $c,  C_{\mathrm{FDM}}, \gamma_1 > 0$ such that
    \begin{equation}\label{ass:exponential_decay_order_2}
        \sup_{m \geq 0}\exp(cm^{\gamma_1})\Delta_{m,2} \leq C_{\mathrm{FDM}}.
    \end{equation}
    
\begin{lemma}
    Let $\{X_t\}_{t \in \mathbb{Z}}$ be a sequence of $\mathbb{R}$-valued random variables in the form of \eqref{eq:nonstationary-2}. If there exist some absolute constants $c, \gamma_1 > 0$ such that \eqref{ass:exponential_decay_order_2} holds, then \eqref{ass:polynomial_decay_order_2} holds for any $c^{\prime}, \nu > 0$.
\end{lemma}
\begin{proof}
    If \eqref{ass:exponential_decay_order_2} holds, then we have for any $m \geq 0$ that 
    \begin{equation*}
        \Delta_{m,2} = \sum_{k = m}^{\infty}\delta_{k,2} \leq C_{\mathrm{FDM}}\exp(-cm^{\gamma_1}),
    \end{equation*}
    and
    \begin{equation*}
    \begin{aligned}
        &\delta_{m,2} = \sum_{k = m}^{\infty}\delta_{k,2} - \sum_{k = m+1}^{\infty}\delta_{k,2} \leq C_{\mathrm{FDM}}\exp(-cm^{\gamma_1}).
    \end{aligned}
    \end{equation*}
    This further leads to
    \begin{equation}\label{eq:DAN_poly_nu_2}
        \sup_{m \geq 0}(m+1)^{\nu}\sum_{k = m}^{\infty}\delta_{k,2} \leq C_{\mathrm{FDM}}\sup_{m \geq 0}(m+1)^{\nu}\sum_{k = m}^{\infty}\exp(-ck^{\gamma_1}) \leq C_{\mathrm{FDM}}\sum_{k = 0}^{\infty}(k+1)^{\nu}\exp(-ck^{\gamma_1}).
    \end{equation}

Define the following function on $x \in [1, \infty)$ as
\begin{equation*}
    g(x) = \log\bigg(\frac{\exp(-cx^{\gamma_1})}{(x+1)^{-\nu^{\prime}}}\bigg) = -cx^{\gamma_1} + \nu^{\prime}\log(x+1),
\end{equation*}
where $\nu^{\prime}$ is some absolute constant such that $\nu^{\prime} - \nu > 1$. 
We have the derivative as
\begin{equation*}
    g^{\prime}(x) = -c\gamma_1x^{\gamma_1-1} + \frac{\nu^{\prime}}{x+1}.
\end{equation*}
Since for any absolute constants $c, \gamma_1 > 0$ and $\nu^{\prime} > \nu + 1$, there exists a finite $x^* \geq 1$ such that for any $x \geq x^*$, it holds that
\begin{equation*}
    g^{\prime}(x) < 0 \;\; \text{and} \;\; \exp(-cx^{\gamma_1}) \leq (x+1)^{-\nu^{\prime}}.
\end{equation*}
Letting $k^* = \lceil x^* \rceil$, we have that
\begin{equation*}
\begin{aligned}
    \sup_{m \geq 0}(m+1)^{\nu}\sum_{k = m}^{\infty}\delta_{k,2} \leq C_{\mathrm{FDM}}\sum_{k = 0}^{k^*}(k+1)^{\nu}\exp(-ck^{\gamma_1}) + C_{\mathrm{FDM}}\sum_{k = k^*+1}^{\infty}(k+1)^{-(\nu^{\prime}-\nu)}< \infty.
\end{aligned}
\end{equation*}
Thus, there exists such $c^{\prime}$.
\end{proof}

\subsection{Bernstein-type inequality for nonstationary processes}
\begin{theorem}\label{thm:bernstein_exp_subExp_nonlinear-2}
Let $\{X_i\}_{i \in \mathbb{Z}}$ be an $\mathbb{R}$-valued possibly nonstationary process of the form \eqref{eq:nonstationary-2}. Assume that $\mathbb{E}[X_i] = 0$ for any $i \in \mathbb{Z}$ and that there exist absolute constants $\gamma_1, \gamma_2, c, C_{\mathrm{FDM}} > 0$ such that
\begin{equation}\label{eq:dep_exponential_gamma1_2}
    \sup_{m \geq 0}\exp(cm^{\gamma_1})\Delta_{m,2} \leq C_{\mathrm{FDM}},
\end{equation}
for any $x > 0$ \begin{equation}\label{eq:moment_psi_gamma2_2}
    \sup_{i \in \mathbb{Z}}\mathbb{P}(|X_i| > x) \leq \exp(1 - x^{\gamma_2}),
\end{equation}
and $1/\gamma = 1/\gamma_1 + 1/\gamma_2 > 1$.
Then, we have for any $n \geq 3$ and $x > 0$,
\begin{align*}
    \mathbb{P}\Big( \Big|\sum_{i = 1}^nX_i\Big| \geq x \Big) \leq& n\exp\Big(-\frac{x^{\gamma}}{C_1}\Big) + \exp\Big(-\frac{x^2}{C_2(1+nC_{\mathrm{LRV}})}\Big)\\
    &+ \exp\Big(-\frac{x^2}{C_3n}\exp\Big(\frac{x^{\gamma(1-\gamma)}}{C_4(\log x)^{\gamma}}\Big)\Big),
\end{align*}
where $C_1, C_2, C_3, C_4 > 0$ are absolute constants and $C_{\mathrm{LRV}}$ is defined in \eqref{eq:def_LRV_UB}. 
\end{theorem}

\begin{remark}
For completeness, we give the proof of Theorem \ref{thm:bernstein_exp_subExp_nonlinear-2}. The proof closely follows that of Theorem 1 in \cite{merlevede2011bernstein}. There are only two differences. (1) We replace Lemma 2 therein with Lemma \ref{lemma:diff_laplace_joint_marginal}, which is based on the functional dependence measure. (2) Lemma 5 in \cite{dedecker2004coupling} gives a coupling property of the $\tau$-mixing coefficient, i.e.~letting $\mathcal{M}$ be a sigma algebra and $X$ be an integrable
real-valued random variable, there exists a random variable $\tilde{X}$ independent of $\mathcal{M}$ and distributed as $X$ such that $\mathbb{E}|X - \tilde{X}|$ equals the $\tau$-mixing coefficient associated with $\mathcal{M}$. Instead, for random variables $\{X_t\}_{t \in \mathbb{Z}}$ of the form \eqref{eq:nonstationary-2}, we use coupling method \cite[see e.g.][]{berkes2009asymptotic} to directly construct $X_{s, \{k, -\infty\}}^{(l)}$ with $s > k$, which is independent of $\mathcal{F}_{k}$ and distributed as $X_s$.
\end{remark}

Before proving Theorem \ref{thm:bernstein_exp_subExp_nonlinear-2}, we introduce the truncation argument.

\textbf{Truncation.} For $\mathbb{R}$-valued random variables $\{X_i\}_{i \in \mathbb{Z}}$, define the truncated and centered version with parameter $M > 0$ as
\begin{equation*}
    \overline{X}_M(i) = \psi_M(X_i) - \mathbb{E}[\psi_M(X_i)],
\end{equation*}
where $\psi_M(x) = \sign(x)(|x| \wedge M)$ for any $x \in \mathbb{R}$. The choice of $M$ will be specified later. We list below some facts about $\overline{X}_M(i)$.
\begin{itemize}
    \item Due to the triangle inequality and the boundness, we have that 
    \begin{equation*}
        |\overline{X}_M(i)| \leq 2M.
    \end{equation*}
    \item Due to the triangle inequality, we have that 
    \begin{equation*}
        \mathbb{E}[|X_i - \overline{X}_M(i)|] \leq 2\mathbb{E}[|X_i - \psi_M(X_i)|].
    \end{equation*}
    \item $|X_i - \psi_M(X_i)| = (|X_i|-M)\mathbbm{1}\{|X_i| > M\} \leq |X_i|\mathbbm{1}\{|X_i| > M\}$.
\end{itemize}

\begin{proof}[Proof of Theorem \ref{thm:bernstein_exp_subExp_nonlinear-2}]
Without loss of generality, we let $C_{\mathrm{FDM}} = 1$ and the assumption \eqref{eq:dep_exponential_gamma1_2} becomes 
\begin{equation}\label{eq:dep_exponential_gamma1_2_CLRV1}
    \sup_{m \geq 0}\exp(cm^{\gamma_1})\Delta_{m,2} \leq 1,
\end{equation}
To see this, if $C_{\mathrm{FDM}} < 1$, \eqref{eq:dep_exponential_gamma1_2} implies \eqref{eq:dep_exponential_gamma1_2_CLRV1}.
If $C_{\mathrm{FDM}} > 1$, we can consider the normalized process $\{C_{\mathrm{FDM}}^{-1}X_i\}_{i \in \mathbb{Z}}$. In this case, by the assumption \eqref{eq:moment_psi_gamma2_2} we have for any $x > 0$ that \begin{equation*}
    \sup_{i \in \mathbb{Z}}\mathbb{P}(C_{\mathrm{FDM}}^{-1}|X_i| > x) \leq \exp(1 - C_{\mathrm{FDM}}x^{\gamma_2}) \leq \exp(1 - x^{\gamma_2}),
\end{equation*}

We define function $H(x) = \exp(1 - x^{\gamma_2})$ on $x > 0$.
Our goal is to obtain a Bernstein-type inequality for the partial sum $\sum_{i = 1}^nX_i$ for any $t > 0$ and any $n \geq 3$.
Using the truncation argument, we decompose the partial sum as
\begin{equation}\label{eq:decompose_partial_sum}
    \sum_{i = 1}^nX_i = \sum_{i = 1}^n(X_i - \overline{X}_M(i)) + \sum_{i = 1}^n\overline{X}_M(i) = \sum_{i = 1}^n(X_i - \overline{X}_M(i)) + \overline{S}_M.
\end{equation}

The first term on the right-hand side of \eqref{eq:decompose_partial_sum} can be handled using tail probability of the marginal distributions. To be specific, for any $x > 0$, it satisfies that
\begin{align*}
    &\mathbb{P}\Big( \Big|\sum_{i = 1}^n(X_i - \overline{X}_M(i))\Big| > x \Big) \leq \frac{1}{x}\sum_{i = 1}^n\mathbb{E}\big[\big|X_i - \overline{X}_M(i)\big|\big] \leq \frac{2}{x}\sum_{i = 1}^n\mathbb{E}[|X_i - \psi_M(X_i)|] \\
    \leq & \frac{2}{x}\sum_{i = 1}^n\mathbb{E}[|X_i|\mathbbm{1}\{|X_i| \geq M\}]= \frac{2}{x}\sum_{i = 1}^n\int_M^{\infty}\mathbb{P}(|X_i| \geq t)dt \leq \frac{2n}{x}\int_M^{\infty}H(t)dt,
\end{align*}
where the first inequality follows from the triangle and Markov's inequalities, the second, the third and the fourth inequalities follow from the definitions of $\overline{X}_M(i)$, $\psi_M(X_i)$ and $H(t)$ respectively.
Note that $\log(H(t)) = 1-t^{\gamma_2}$. It follows from that the function $t \mapsto t^2H(t)$ is nonincreasing for $t \geq (2/\gamma_2)^{1/\gamma_2}$. Hence, assuming
\begin{equation}\label{eq:M_requirement_1}
    M \geq (2/\gamma_2)^{1/\gamma_2},
\end{equation}
we have that
\begin{equation*}
    \int_M^{\infty}H(t)dt \leq M^2H(M)\int_M^{\infty}\frac{dt}{t^2} = MH(M).
\end{equation*}
Therefore, with any truncation parameter $M \geq (2/\gamma_2)^{1/\gamma_2}$, we have for any $x > 0$ that
\begin{equation}\label{eq:tail_truncation}
\begin{aligned}
    &\mathbb{P}\Big( \Big|\sum_{i = 1}^n(X_i - \overline{X}_M(i))\Big| > x \Big) \leq 2nx^{-1}MH(M).
\end{aligned}
\end{equation}

Bounding the tail probability of $\overline{S}_M$ is the key for obtaining the Bernstein-type inequality of $\sum_{i = 1}^nX_i$. 
Letting $u > 0$, we consider the following three cases.

\textbf{Case 1.} Suppose that $u \geq n^{\gamma_1/\gamma}$. Letting $M = u/n$, we have that
\begin{equation*}
    \big|\overline{S}_M\big| \leq 2nM = 2u.
\end{equation*}
We require $M = u/n \geq (2/\gamma_2)^{1/\gamma_2}$. Since $u/n \geq u^{1-\gamma/\gamma_1} = u^{\gamma/\gamma_2}$, this requirement holds if $u \geq (2/\gamma_2)^{1/\gamma}$.
We have by \eqref{eq:decompose_partial_sum} and \eqref{eq:tail_truncation} that for any $u \geq (n^{\gamma_1} \vee 2/\gamma_2)^{1/\gamma}$ that
\[    
    \mathbb{P}\Big( \Big|\sum_{i = 1}^nX_i\Big| > 3u \Big) \leq \mathbb{P}\Big( \Big|\sum_{i = 1}^n(X_i - \overline{X}_M(i))\Big| > u \Big) \leq 2H(M) =  2\exp(1-(u/n)^{\gamma_2}) \leq 2e\exp(-u^{\gamma}),
\]
where the last inequality follows from that $(u/n)^{\gamma_2} \geq u^{\gamma}$. While, if $n^{\gamma_1} < 2/\gamma_2$, for $n^{\gamma_1/\gamma} \leq u < (2/\gamma_2)^{1/\gamma}$, it holds trivially that
\begin{equation*}
    \mathbb{P}\Big( \Big|\sum_{i = 1}^nX_i\Big| > 3u \Big) \leq 1 \leq e\exp(-u^{\gamma}/C),
\end{equation*}
where $C$ is some positive constant such that $C \geq 2/\gamma_2$. Therefore, we have for any $u \geq n^{\gamma_1/\gamma}$ that
\begin{equation*}
    \mathbb{P}\Big( \Big|\sum_{i = 1}^nX_i\Big| > 3u \Big) \leq \exp(-u^{\gamma}/C_1),
\end{equation*}
for some sufficient large $C_1 > 0$.

\textbf{Case 2.} Suppose that $2^{(1-\gamma_1)/\gamma_2}(4\zeta)^{\gamma_1/\gamma} \leq u < n^{\gamma_1/\gamma}$, where $\zeta = \mu \vee (2/\gamma_2)^{1/\gamma_1}$ and $\mu = (2(2 \vee 4c_0^{-1})/(1-\gamma))^{2/(1-\gamma)}$. Let $r \in [1, n/2]$ be some real value to be chosen later. Let
\begin{equation*}
    A = \Big\lfloor \frac{n}{2r} \Big\rfloor, \;\; k = \Big\lfloor \frac{n}{2A} \Big\rfloor \;\; \text{and} \;\; M = H^{-1}(\exp(-A^{\gamma_1})).
\end{equation*}
Denoted $I_j = \{1+(j-1)A, \dots, jA\}$ for $j = 1, \dots, 2k$, and also $I_{2k+1} = \{1+2kA, \dots, n\}$. Denote $\overline{S}_M(I_j) = \sum_{i \in I_{j}}\overline{X}_M(i)$ for any $j = 1, \dots, 2k+1$. We have that
\begin{equation}\label{eq:decomp_odd_even}
\begin{aligned}
    \sum_{i = 1}^nX_i =& \sum_{i = 1}^n(X_i - \overline{X}_M(i)) + \sum_{j = 1}^k\overline{S}_M(I_{2j-1}) + \sum_{j = 1}^k\overline{S}_M(I_{2j}) + \overline{S}_M( I_{2k+1}).
\end{aligned}
\end{equation}
We now use the coupling method introduced by \cite{berkes2009asymptotic} to construct independent block sums for odd blocks and even blocks respectively. To be specific, for any $j = 3, 4, \dots, 2k$, let $\{\epsilon_t^{(j)}\}_{t \in \mathbb{Z}}$ be an i.i.d. sequence such that $\epsilon_t^{(j)}$ is an independent copy of $\epsilon_t$ and $\epsilon_t, \epsilon_t^{(3)}, \epsilon_t^{(4)}, \dots, \epsilon_t^{(2k)}$ are mutually independent. This is always possible by enlarging the original probability space. For $s > h$, let 
\begin{equation}\label{eq:coupling_block}
    X_{s, \{h, -\infty\}}^{(j)} = g_s(\mathcal{F}_{s, \{h, -\infty\}}^{(j)}),
\end{equation}
where $\mathcal{F}_{s, \{h, -\infty\}}^{(l)} = \sigma(\dots, \epsilon_{h-1}^{(l)}, \epsilon_h^{(l)}, \epsilon_{h+1}, \dots, \epsilon_s)$. 

We only consider the odd blocks, since the same argument applies for the even blocks. Using the notation in \eqref{eq:coupling_block}, for any $j = 2, \dots, k$ and $i \in I_{2j-1}$, let
\begin{equation*}
    \overline{X}_{M}^{(2j-1)}(i,\{(2j-3)A, -\infty\}) = \psi_{M}(X^{(2j-1)}_{i,\{(2j-3)A, -\infty\}}) - \mathbb{E}[\psi_M(X^{(2j-1)}_{i,\{(2j-3)A, -\infty\}})],
\end{equation*}
and denote
\begin{equation*}
    \overline{S}_M^*(I_{2j-1}) = \sum_{i = I_{2j-1}}\overline{X}_{M}^{(2j-1)}(i,\{(2j-3)A, -\infty\})
\end{equation*}
By the above construction, we have $\big\{\overline{S}_M(I_1), \overline{S}_M^*(I_{3}), \dots, \overline{S}_M^*(I_{2k-1})\big\}$ are mutually independent, and for any $j = 2,\dots, k$, $\overline{S}_M^*(I_{2j-1})$ and $\overline{S}_M(I_{2j-1})$ have the same distribution. Moreover, we have for any $j = 2,\dots, k$ that
\begin{align*}
    &\mathbb{E}\big|\overline{S}_M(I_{2j-1}) - \overline{S}^*_M(I_{2j-1})\big|\leq \sum_{i \in I_{2j-1}}\big| \overline{X}_M^{(2j-1)}(i) - \overline{X}_{M}^{(2j-1)}(i,\{(2j-3)A, -\infty\}) \big|\\
    \leq& \sum_{i \in I_{2j-1}}\big\Vert \overline{X}_M^{(2j-1)}(i) - \overline{X}_{M}^{(2j-1)}(i,\{(2j-3)A, -\infty\}) \big\Vert_2 \leq 2\sum_{i \in I_{2j-1}}\big\Vert X_i - X^{(2j-1)}_{i,\{(2j-3)A, -\infty\}} \big\Vert_2 \\
    \leq& 2\sum_{i \in I_{2j-1}}\sum_{s = i-(2j-3)A}^{\infty}\delta_{s,2} \leq 2A\sum_{s = A+1}^{\infty}\delta_{s,2} \leq 2A\exp(-cA^{\gamma_1}).
\end{align*}
where the first inequality follows from the triangle inequality, the second inequality follows from H\"older's inequality, the third inequality follows from the second fact about the truncaton, the fourth inequality follows from the definition of $\delta_{s,2}$, and by writing the difference as telescoping sums, and the last two inequalities follow from \eqref{eq:dep_exponential_gamma1_2}. The same argument can be applied to the even blocks, with block sums denoted as
\begin{equation*}
    \overline{S}_M^*(I_{2j}) = \sum_{i = I_{2j}}\overline{X}_{M}^{(2j)}(i,\{(2j-2)A, -\infty\}) \;\; \text{for} \;\; j = 2, \dots, k.
\end{equation*}
Denote also $\overline{S}^*_M(I_{1}) = \overline{S}_M(I_{1})$ and $\overline{S}^*_M(I_{2}) = \overline{S}_M(I_{2})$. Then, \eqref{eq:decomp_odd_even} can be further decomposed as
\begin{equation}\label{eq:decomp_odd_even_coup}
\begin{aligned}
    \sum_{i = 1}^nX_i =& \sum_{i = 1}^n \{X_i - \overline{X}_M(i)\} + \sum_{j = 1}^k\overline{S}^*_M(I_{2j-1}) + \sum_{j = 1}^k\overline{S}^*_M(I_{2j}) + \overline{S}_M( I_{2k+1})\\
    &+ \sum_{j = 2}^k\big\{\overline{S}_M(I_{2j-1}) - \overline{S}^*_M(I_{2j-1})\big\} + \sum_{j = 2}^k\big\{\overline{S}_M(I_{2j}) - \overline{S}^*_M(I_{2j})\big\}.
\end{aligned}
\end{equation}
Since $\big| \overline{S}_M( I_{2k+1}) \big| \leq 2AM$, for any $u \geq 2AM$, it holds that
\begin{equation}
\begin{aligned}
    \mathbb{P}\Big( \Big|\sum_{i = 1}^nX_i\Big| > 6u \Big)
    \leq& \mathbb{P}\Big( \Big|\sum_{i = 1}^n(X_i - \overline{X}_M(i))\Big| > u \Big) + \mathbb{P}\Big( \Big|\sum_{j = 1}^k\overline{S}^*_M(I_{2j-1})\Big| > u \Big)\\
    &+ \mathbb{P}\Big( \Big|\sum_{j = 1}^k\overline{S}^*_M(I_{2j})\Big| > u \Big) + 4(k-1)Au^{-1}\exp(-cA^{\gamma_1})\\
    =& I + II + III + 4(k-1)Au^{-1}\exp(-cA^{\gamma_1}),
\end{aligned}
\end{equation}
where the first inequality follows from Markov's inequality
We bound $I$ using \eqref{eq:tail_truncation}. Since $H^{-1}(y) = (\log(e/y))^{1/\gamma_2}$, we have $M = (1+A^{\gamma_1})^{1/\gamma_2} \geq (A^{\gamma_1})^{1/\gamma_2}$. For any $A \geq (2/\gamma_2)^{1/\gamma_1}$, which satisfies the requirement \eqref{eq:M_requirement_1}, we have for any $u > 0$ that
\begin{equation*}
    I \leq 2nu^{-1}M\exp(-A^{\gamma_1}).
\end{equation*}
Since $\big\{\overline{S}^*_M(I_{2j-1})\big\}_{j = 1}^k$ are mutually independent, by Chebyshev's inequality, we have for any $t > 0$ that
\begin{equation*}
    II \leq 2e^{-ut}\prod_{i = 1}^k\mathbb{E}\Big[\exp\big(t\overline{S}_M(I_{2j-1})\big)\Big] .
\end{equation*}
Applying Proposition \ref{prop:bernstein_bound_cantor_2} for each Laplace transform of block sum, if $A \geq \mu$, we have for any $t < \nu A^{-\gamma_1(1-\gamma)/\gamma}$ that
\begin{equation*}
    \sum_{i = 1}^k\log\mathbb{E}\Big[\exp\big(t\overline{S}_M(I_{2j-1})\big)\Big] \leq \frac{kAV(A)t^2}{1-\nu^{-1}A^{\gamma_1(1-\gamma)/\gamma}t}.
\end{equation*}
Then, if $A \geq \zeta$, we have for any $u \geq 0$ that
\begin{equation*}
\begin{aligned}
    II \leq& 2\exp\Big( -ut + \frac{kAV(A)t^2}{1-\nu^{-1}A^{\gamma_1(1-\gamma)/\gamma}t} \Big) = 2\exp\Big( -\frac{u^2}{4kAV(A) + 2\nu^{-1}A^{\gamma_1(1-\gamma)/\gamma}u} \Big),
\end{aligned}
\end{equation*}
where the equality follows by letting $t = u\big(2kAV(A) + \nu^{-1}A^{\gamma_1(1-\gamma)/\gamma}u\big)^{-1}$. Using the same argument, we obtain the same bound for $III$.

Therefore, if $A \geq \zeta$, we have for any $u \geq 2AM$ that
\begin{equation}\label{eq:tail_case2}
\begin{aligned}
    \mathbb{P}\Big( \Big|\sum_{i = 1}^nX_i\Big| > 6u \Big) =& 4(k-1)Au^{-1}\exp(-cA^{\gamma_1}) + 2nu^{-1}M\exp(-A^{\gamma_1})\\
    &+ 4\exp\Big( -\frac{u^2}{4kAV(A) + 2\nu^{-1}A^{\gamma_1(1-\gamma)/\gamma}u} \Big).
\end{aligned}
\end{equation}
We now choose $r = 2^{(1-\gamma_1)/(\gamma_1+\gamma_2)}nu^{-\gamma/\gamma_1}$. By the definition of $A$, we have that 
\begin{equation*}
    2^{(\gamma_1-1)/(\gamma_1+\gamma_2)}u^{\gamma/\gamma_1} - 2 \leq 2A \leq 2^{(\gamma_1-1)/(\gamma_1+\gamma_2)}u^{\gamma/\gamma_1}.
\end{equation*}
Since $M = (1+A^{\gamma_1})^{1/\gamma_2} \leq 2^{1/\gamma_2}A^{\gamma_1/\gamma_2}$, the choice of $r$ implies that 
\begin{equation*}
    2AM \leq 2^{1+1/\gamma_2}A^{\gamma_1/\gamma} \leq u.
\end{equation*}
Moreover, by the definition of $\zeta$ and $\mu$, we have that $\zeta \geq \mu \geq 16$ and $u^{\gamma/\gamma_1} \geq 4\zeta \geq 64$. Since $u \geq 2^{(1-\gamma_1)/\gamma_2}(4\zeta)^{\gamma_1/\gamma}$, it holds that 
\begin{equation*}
    A \geq 2^{(\gamma_1-1)/(\gamma_1+\gamma_2)}u^{\gamma/\gamma_1}/2 - 1 \geq (4\zeta)/2-1 \geq (4\zeta)/4 \geq \zeta.
\end{equation*}
Therefore, we have that \eqref{eq:tail_case2} holds for any $2^{(1-\gamma_1)/\gamma_2}(4\zeta)^{\gamma_1/\gamma} \leq u < n^{\gamma_1/\gamma}$. We list the following facts.
\begin{itemize}
    \item By the definition of $k$ and the fact that $4\zeta \geq 2^6$, we have that 
    \begin{equation*}
        4(k-1)Au^{-1} \leq 2nu^{-1} \leq 2^{1+(\gamma_1-1)/\gamma_2}n(4\zeta)^{-\gamma_1/\gamma} \leq 2^{-(5\gamma_1+1)/\gamma_2-5}n.
    \end{equation*}
    \item Since $A \geq 2\zeta - 1 \geq 31$, we have that $2nu^{-1}M \leq n/A \leq n/31$.
    \item Since $2A \leq 2^{(\gamma_1-1)/(\gamma_1+\gamma_2)}u^{\gamma/\gamma_1}$, $0 < \gamma < 1$ and $\gamma_2 > 0$, we have that 
    \begin{equation*}
        A^{\gamma_1(1-\gamma)/\gamma}u \leq 2^{-(1+\gamma_2)(1-\gamma)\gamma_2}u^{2-\gamma} \leq u^{2-\gamma}.
    \end{equation*}
\end{itemize}
We have that
\begin{equation*}
\begin{aligned}
    \eqref{eq:tail_case2} \leq& \frac{n}{2^{5+(1+5\gamma_1)/\gamma_2}}\exp\Big(-c\frac{u^{\gamma}}{2^{\gamma_1}}\Big) + \frac{n}{31}\exp\Big(-\frac{u^{\gamma}}{2^{\gamma_1}}\Big)+ 4\exp\Big( -\frac{u^2}{2nV(A) + 2\nu^{-1}u^{2-\gamma}} \Big).
\end{aligned}
\end{equation*}
Recall that $V(A) = (\nu_1C_{\mathrm{LRV}} + \nu_2\exp\big(-\nu_3A^{\gamma_1(1-\gamma)}(\log A)^{-\gamma}\big)$, where $\nu_1, \nu_2, \nu_3 > 0$ are constants depending only on $c$, $\gamma_1$ and $\gamma_2$. Therefore, we have for any $2^{(1-\gamma_1)/\gamma_2}(4\zeta)^{\gamma_1/\gamma} \leq u < n^{\gamma_1/\gamma}$ that
\begin{equation*}
\begin{aligned}
    \mathbb{P}\Big( \Big|\sum_{i = 1}^nX_i\Big| > 6u \Big) =& n\exp\Big(-\frac{u^{\gamma}}{C_1}\Big) + \exp\Big( -\frac{u^2}{C_2(1+nC_{\mathrm{LRV}})} \Big)\\
    &+ \exp\Big( -\frac{u^2}{C_3n}\exp\Big(\frac{x^{\gamma(1-\gamma)}}{C_4(\log A)^{\gamma}} \Big) \Big),
\end{aligned}
\end{equation*}
where $C_1, C_2, C_3, C_4 > 0$ are absolute constants depending only on $\gamma_1$, $\gamma_2$ and $c$.

\textbf{Case 3.} Suppose that $u < 6\times2^{(1-\gamma_1)/\gamma_2}(4\zeta)^{\gamma_1/\gamma}$. It holds trivially that
\begin{equation*}
    \mathbb{P}\Big( \Big|\sum_{i = 1}^nX_i\Big| > u \Big) \leq 1 \leq e\exp\Big(-\frac{u^{\gamma}}{6^{\gamma}2^{(1-\gamma_1)\gamma/\gamma_2}(4\zeta)^{\gamma_1}}\Big),
\end{equation*}
which is smaller than $n\exp(-u^{\gamma}/C_1)$, if $n \geq 3$ and $C_1 \geq 6^{\gamma}2^{(1-\gamma_1)\gamma/\gamma_2}(4\zeta)^{\gamma_1}$. Combining the three cases completes the proof.
\end{proof}

Theorem \ref{thm:bernstein_exp_subExp_nonlinear-2} relies on the upper bound on log-Laplace transform of a block sum of bounded random variables, where the block size $A$ is relatively smaller than the sample size $n$. In the following, we describe the idea to obtain such bound. The results are summarized in Propositions \ref{prop:bernstein_bound_cantor} and \ref{prop:bernstein_bound_cantor_2}.

Let $A$ be a positive integer and $\{X_i\}_{i = 1}^A$ be a sequence of random variables of the form \eqref{eq:nonstationary-2} and with mean zero.
Our goal is to upper bound the log-Laplace transform of partial sums i.e.~$\log\mathbb{E}\big[\exp\big(t\sum_{i = 1}^AX_i\big)\big]$, for any small $t > 0$. To this end, we introduce the construction of the Cantor-like set $K_A^{(l)} = \bigcup_{k=1}^{2^l} I_{l,k} \subset \{1, \dots, A\}$, where $\{I_{l,k}\}_{k = 1}^{2^l}$ are integer intervals with same cardinality $n_l$, and all neighboring integer intervals are separated by some integer intervals. Heuristically, this construction reduces the dependence by creating gaps.

\textbf{Construction of $K_A^{(l)}$}.
The construction of $K_A^{(l)}$ of $\{1, \dots, A\}$ involves $l$ recursive steps, where $l$ satisfies $A2^{-l} \geq (2 \vee 4c_0^{-1})$ is a positive integer. Let $c_0 \in (0, (2^{1/\gamma-1}-1)(2(2^{1/\gamma}-1))^{-1}]$ be some constant. We will see that such requirements of $l$ and $c_0$ justify the following construction of $K_A^{(l)}$. We do not need the specific values of $l$ and $c_0$ until we give them respectively in \eqref{eq:select_l} and \eqref{eq:requirement_c0}.
\begin{enumerate}
    \item[Step C$1$.] Let $n_0 = A$. Define integer $d_0$ as
    \begin{equation*}
        d_0 = \begin{cases}
            \max\{m \in 2\mathbb{N}, m \leq c_0n_0\} \;\; &\text{if $n_0$ is even},\\
            \max\{m \in 2\mathbb{N}+1, m \leq c_0n_0\} \;\; &\text{if $n_0$ is odd}.
        \end{cases}
    \end{equation*}
    Note that $n_0 - d_0$ is even. Define $n_1 = (n_0 - d_0)/2$ and two integer intervals of cardinality $n_1$ separated by a gap of $d_0$ integers as
    \begin{equation*}
        I_{1,1} = \{1, \dots, n_1\} \;\; \text{and} \;\; I_{1,2} = \{n_1+d_0+1, \dots, n_0\}.
    \end{equation*}
    We discard the gap of $d_0$ integers for the next step.
    \item[Step C$2$.] Define integer $d_1$ as
    \begin{equation*}
        d_1 = \begin{cases}
            \max\{m \in 2\mathbb{N}, m \leq c_0n_02^{-(l \wedge 1/\gamma)}\} \;\; &\text{if $n_1$ is even},\\
            \max\{m \in 2\mathbb{N}+1, m \leq c_0n_02^{-(l \wedge 1/\gamma)}\} \;\; &\text{if $n_1$ is odd}.
        \end{cases}
    \end{equation*}
    Note that $n_1 - d_1$ is even, and define $n_2 = (n_1 - d_1)/2$. Based on $I_{1,1}$, define two integer intervals of cardinality $n_2$ separated by a gap of $d_1$ integers as
    \begin{equation*}
        I_{2,1} = \{1, \dots, n_2\} \;\; \text{and} \;\; I_{2,2} = \{n_2+d_1+1, \dots, n_1\}.
    \end{equation*}
    Based on $I_{1,2}$, define similarly two integer intervals of cardinality $n_2$ separated by a gap of $d_1$ integers as
    \begin{equation*}
        I_{2,3} = (n_1 + d_0) + I_{2,1} \;\; \text{and} \;\; I_{2,4} = (n_1 + d_0) + I_{2,2}.
    \end{equation*}
    We discard the two gaps of $d_1$ integers for the next step.
    \item[Step C$k$.] After Step C$(k-1)$ with $2 \leq k \leq l$, we obtain integer intervals $\{I_{k-1, j}\}_{j = 1}^{2^{k-1}}$. Denote for $1 \leq j \leq 2^{k-1}$, $I_{k-1, j} = \{a_{k-1,j}, \dots, b_{k-1,j}\}$.
    Define integer $d_{k-1}$ as
    \begin{equation}\label{eq:gap_size_j-1}
        d_{k-1} = \begin{cases}
            \max\{m \in 2\mathbb{N}, m \leq c_0n_02^{-(l \wedge (k-1)/\gamma)}\} \;\; &\text{if $n_{k-1}$ is even},\\
            \max\{m \in 2\mathbb{N}+1, m \leq c_0n_02^{-(l \wedge (k-1)/\gamma)}\} \;\; &\text{if $n_{k-1}$ is odd}.
        \end{cases}
    \end{equation}
    Note that $n_{k-1} - d_{k-1}$ is even, and define $n_k = (n_{k-1} - d_{k-1})/2$. For each $j = 1, \dots, 2^{k-1}$, based on $I_{k-1,j}$, we define two integer intervals of cardinality $n_k$ separated by a gap of $d_{k-1}$ integers as
    \begin{equation*}
        I_{k,2j-1} = \{a_{k,2j-1}, \dots, b_{k,2j-1}\} \;\; \text{and} \;\; I_{k,2j} = \{a_{k,2j}, \dots, b_{k,2j}\},
    \end{equation*}
    where $a_{k,2j-1} = a_{k-1,j}$, $b_{k,2j-1} = a_{k,2j-1} + n_k - 1$, $a_{k,2j} = b_{k,2j}-n_k+1$ and $b_{k,2j} = b_{k-1,j}$. Moreover, we have $a_{k,2j} - b_{k,2j-1} - 1 = d_{k-1}$.
    We discard the $2^{k-1}$ gaps of $d_{k-1}$ integers.
\end{enumerate}
Finally, after Step C$l$, we obtain $\{I_{l,1}, \dots, I_{l,2^l}\}$ and $K_A^{(l)} = \bigcup_{i = 1}^{2^l}I_{l,i}$. Moreover, for any $k = 0, 1, \dots, l$ and $j = 1, \dots, 2^k$, we also define
\begin{equation}\label{eq:cantor_decomp_1}
    K_{A,k,j}^{(l)} = \bigcup_{i = (j-1)2^{l-k}+1}^{j2^{l-k}}I_{l,i},
\end{equation}
and we have
\begin{equation}\label{eq:cantor_decomp_2}
    K_A^{(l)} = \bigcup_{j=1}^{2^k}K_{A,k,j}^{(l)}.
\end{equation}

Moreover, it follows that
\begin{equation*}
\begin{aligned}
    &A - \card(K_A^{(l)}) = \sum_{j = 0}^{l-1}2^jd_j \leq Ac_0\Big(\sum_{j = 0}^{\lfloor \gamma l \rfloor}2^{j(1-1/\gamma)} + \sum_{j = \lfloor \gamma l \rfloor + 1}^{l-1}2^{j-l}\Big)\\
    \leq& Ac_0\Big(\frac{2 - 2^{1-1/\gamma} - 2^{\lfloor \gamma l\rfloor - l}}{1-2^{1-1/\gamma}}\Big) \leq Ac_0\Big(\frac{2 - 2^{1-1/\gamma}}{1-2^{1-1/\gamma}}\Big).
\end{aligned}
\end{equation*}
To ensure $\card(K_A^{(l)}) \geq A/2$, we choose some $c_0$ such that
\begin{equation}\label{eq:c0}
    0 < c_0 \leq \frac{2^{1/\gamma-1}-1}{2(2^{1/\gamma}-1)}. 
\end{equation}

Also, by definition, $n_k = 2^{-k}\big(n_0 - \lfloor c_0n_0 \rfloor - \sum_{i = 1}^{k-1}2^i\lfloor c_0n_02^{-(l \wedge i/\gamma)} \rfloor\big)$. For any $k$ such that $(k - 1)/\gamma \leq l$, we need $n_k \geq 1$ to ensure the construction of $K_A^{(l)}$. Since 
    \begin{equation}\label{eq:requirement_c0_nk}
    \begin{aligned}
        n_k \geq& 2^{-k}\Big(n_0 - c_0n_0 - c_0n_0\sum_{i = 1}^{k-1}2^{i(1-1/\gamma)}\Big)\\
        =& 2^{-k}\Big(n_0 - c_0n_0 - c_0n_0\frac{2^{1-1/\gamma}(1-2^{(k-1)(1-1/\gamma)})}{1-2^{1-1/\gamma}}\Big)\\
        \geq& 2^{-k}\Big(n_0 - c_0n_0 - c_0n_0\frac{2^{1-1/\gamma}}{1-2^{1-1/\gamma}}\Big) = 2^{-k}A\Big(1 - c_0\frac{1}{1-2^{1-1/\gamma}}\Big).
    \end{aligned}
    \end{equation}
    We choose 
    \begin{equation}\label{eq:requirement_c0}
    c_0 = \frac{2^{1/\gamma-1}-1}{2(2^{1/\gamma}-1)}.
    \end{equation}
    Plugging \eqref{eq:requirement_c0} into \eqref{eq:requirement_c0_nk}, we have for any $\gamma \in (0, 1)$ and any $k$ such that $(k - 1)/\gamma \leq l$ that
    \begin{equation*}
    \begin{aligned}
        n_k \geq& 4c_0^{-1}\Big(1 - c_0\frac{1}{1-2^{1-1/\gamma}}\Big) = 4\Big(c_0^{-1} - \frac{1}{1-2^{1-1/\gamma}}\Big)
        = 12 + \frac{8}{2^{1/\gamma}-2} \geq 12,
    \end{aligned}
    \end{equation*}
    where the first inequality follows from the definition of $l$, i.e.~$A2^{-l} \geq (2 \vee 4c_0^{-1})$ and $k \leq l$. The second inequality follows from the fact $2^{1/\gamma} > 2$ for any $\gamma \in (0,1)$.
    Moreover, since we require $l$ such that $A2^{-l} \geq (2 \vee 4c_0^{-1})$, we have that $n_l \geq A2^{-l-1} \geq 1$.

We list the following facts regarding $K_A^{(l)}$.
\begin{itemize}
    \item $\card(I_{l,j}) = n_l$ for any $j = 1, \dots, 2^l$, and $\card(K_A^{(l)}) = 2^ln_l$.
    \item $A/2 \leq \card(K_A^{(l)}) \leq A$.
\end{itemize}

Recall $c_0$ and $c$ given respectively in \eqref{eq:c0} and \eqref{eq:dep_exponential_gamma1_2}, we define the following positive constants depending only on $c$, $\gamma_1$ and $\gamma_2$.
\begin{equation}\label{eq:c1c2c3}
\begin{aligned}
    c_1 = \min\{c^{1/\gamma_1}c_0/4, 2^{-1/\gamma}\},\; c_2 = 2^{-(2+\gamma_1/\gamma+1/\gamma_2)}c_1^{\gamma_1},\; c_3 = 2^{-(1+1/\gamma_2)},
\end{aligned}
\end{equation}
and
\begin{equation}\label{eq:kappa}
    \kappa = \min\{c_2, c_3\}.
\end{equation}
Note that by definition, we have $c_1, c_2, c_3, \kappa \in (0,1)$.

The next proposition gives upper bounds on the log-Laplace transform for sum of random variables on a Cantor-like set.
\begin{proposition}\label{prop:bernstein_bound_cantor}
    Let $A$ be a positive integer and $\{X_i\}_{i = 1}^A$ be a sequence of random variables of the form \eqref{eq:nonstationary-2}, with mean zero. Assume that \eqref{eq:dep_exponential_gamma1_2} and \eqref{eq:moment_psi_gamma2_2} hold.
    Let $M = H^{-1}(\exp(-A^{\gamma_1}))$ and $l$ be a positive integer such that $A2^{-l} \geq (1 \vee 2c_0^{-1})$. Then, we have the following results.
    \begin{enumerate}
        \item[(i)] For any $t \leq 2^{-(1+1/\gamma_2)}A^{-\gamma_1/\gamma}$, it holds that
        \begin{equation*}
            \log\mathbb{E}\Big[\exp\Big(t\sum_{j \in K_{A}^{(l)}}\overline{X}_{M}(j)\Big)\Big] \leq C_{\mathrm{LRV}}t^2A,
        \end{equation*}
        where $C_{\mathrm{LRV}}$ is defined in \eqref{eq:def_LRV_UB}.
        \item[(ii)] For any $2^{-(1+1/\gamma_2)}A^{-\gamma_1/\gamma} \leq t \leq \kappa\big(A^{\gamma-1} \wedge (2^l/A)\big)^{\gamma_1/\gamma}$, it holds that
        \begin{equation*}
        \log\mathbb{E}\Big[\exp\Big(t\sum_{j \in K_{A}^{(l)}}\overline{X}_{M}(j)\Big)\Big] \leq (tAl + 6A^{\gamma})\exp(-(2^{-l}c_1A)^{\gamma_1}/2) + C_{\mathrm{LRV}}t^2A.
    \end{equation*}
    \end{enumerate}
\end{proposition}

\textbf{Decomposition of partial sums}. 
After obtaining a Cantor-like set from an integer interval $\{1, \dots, A\}$, there are $A - \card(K_A^{l})$ indices left over. These indices can be reassembled and reused to construct Canto-like set. We perform such reassembling and construction recursively until the number of remaining indices is small enough, and the final remainder can be dealt easily. Recall $l$ is a positive integer, to be chosen later, which satisfies $A2^{-l} \geq (2 \vee 4c_0^{-1})$. For any $j = 0, 1, \dots$, we will define $A_j$ as the number of indices to start with at Step D$j$ of the decomposition. Since $A_j$ decreases as $j$ increases, the decomposition stops at Step D$m(A)$, where 
\begin{equation}\label{eq:def_m(A)}
    m(A) = \min\{m \in \mathbb{N}, A_m \leq A2^{-l}\}.
\end{equation}
Let the initial truncation parameter $T_0 = M = H^{-1}(\exp(-A_0^{\gamma_1}))$ with $A_0 = A$. We need to refine the truncation parameter at each step, define for $1 \leq j \leq m(A)-1$ that
\begin{equation}\label{eq:Tj}
    T_j = H^{-1}(\exp(-A_j^{\gamma_1})).
\end{equation}
Also, for $1 \leq j \leq m(A)-1$, $A_j$ may varies. We choose $l_j$ as
\begin{equation*}
    l_j = \min\{k \in \mathbb{N}, A_j2^{-k} \leq A_02^{-l}\}.
\end{equation*}
\begin{enumerate}
    \item[Step D$0$] Let $A_0 = A$, $T_0 = M$ and $\overline{X}_{T_0}^{(0)}(k) = \overline{X}_M(k)$ for any $k = 1, \dots, A_0$. Let $l_0 = l$. Let $K_{A_0}^{(l_0)}$ be the Cantor-like set defined from $\{1, \dots, A_0\}$.
    \item[Step D$1$] Let $A_1 = A_0 - \card(K_{A_0}^{(l_0)})$ and $\{i_1, \dots, i_{A_1}\} = \{1, \dots, A\} \setminus K_{A_0}^{(l_0)}$. Define for any $k = 1, \dots, A_1$, $\overline{X}_{T_1}^{(1)}(k) = \overline{X}_{T_1}(i_k)$. Let $K_{A_1}^{(l_1)}$ be the Cantor-like set defined from $\{i_1, \dots, i_{A_1}\}$.
    \item[Step D$j$] For $1 \leq j \leq m(A)-1$, let $A_j = A_{j-1} - \card(K_{A_{j-1}}^{(l_{j-1})})$ and $\{i_1, \dots, i_{A_j}\} = \{1, \dots, A\} \setminus \big(\bigcup_{k = 0}^{j-1} K_{A_{k}}^{(l_{k})}\big)$. Define for any $k = 1, \dots, A_j$, $\overline{X}_{T_j}^{(j)}(k) = \overline{X}_{T_j}(i_k)$. Let $K_{A_j}^{(l_j)}$ be the Cantor-like set defined from $\{i_1, \dots, i_{A_j}\}$.
    \item[Step D$m(A)$] Let $A_{m(A)} = A_{m(A)-1} - \card(K_{A_{m(A)-1}}^{(l_{m(A)-1})})$ and $\{i_1, \dots, i_{A_{m(A)}}\} = \{1, \dots, A\} \setminus \big(\bigcup_{k = 0}^{m(A)-1} K_{A_{k}}^{(l_{k})}\big)$. Define for any $k = 1, \dots, A_{m(A)}$, $\overline{X}_{T_{m(A)-1}}^{(m(A))}(k) = \overline{X}_{T_{m(A)-1}}(i_k)$.
\end{enumerate}

Regarding $m(A)$, we have the following facts.
\begin{itemize}
    \item For $0 \leq j \leq m(A)$, $A_j \leq A2^{-j}$, which follows from the fact of $K_A^{(l)}$.
    \item $m(A) \geq 1$, since $l \geq 1$ and $A_0 > A2^{-l}$. Moreover, $m(A) \leq l$, which follows from the definition of $m(A)$.
\end{itemize}
Therefore, we have the following decomposition.
\begin{align}\label{eq:decompose_partial_sum_2}
    &\sum_{i = 1}^A\overline{X}_M(i) = \sum_{i = 1}^{A_0}\overline{X}_{T_0}^{(0)}(i) = \sum_{i = K_{A_0}^{(l_0)}}\overline{X}_{T_0}^{(0)}(i) + \sum_{i = 1}^{A_1}\overline{X}_{T_0}^{(1)}(i) \nonumber \\
    =& \sum_{i = K_{A_0}^{(l_0)}}\overline{X}_{T_0}^{(0)}(i) + \sum_{i = 1}^{A_1}\big\{\overline{X}_{T_0}^{(1)}(i) - \overline{X}_{T_1}^{(1)}(i) \big\} + \sum_{i = 1}^{A_1}\overline{X}_{T_1}^{(1)}(i) \nonumber \\
    =& \sum_{i = K_{A_0}^{(l_0)}}\overline{X}_{T_0}^{(0)}(i) + \sum_{i = 1}^{A_1}\big\{\overline{X}_{T_0}^{(1)}(i) - \overline{X}_{T_1}^{(1)}(i) \big\} + \sum_{i = K_{A_1}^{(l_1)}}\overline{X}_{T_1}^{(1)}(i) + \sum_{i = 1}^{A_2}\overline{X}_{T_1}^{(2)}(i)\nonumber \\
    =& \sum_{i = K_{A_0}^{(l_0)}}\overline{X}_{T_0}^{(0)}(i) + \sum_{i = 1}^{A_1}\big\{\overline{X}_{T_0}^{(1)}(i) - \overline{X}_{T_1}^{(1)}(i) \big\} + \sum_{i = K_{A_1}^{(l_1)}}\overline{X}_{T_1}^{(1)}(i)+ \sum_{i = 1}^{A_2}\big\{\overline{X}_{T_1}^{(2)}(i) - \overline{X}_{T_2}^{(2)}(i)\big\} + \sum_{i = 1}^{A_2}\overline{X}_{T_2}^{(2)}(i)\nonumber \\
    =& \sum_{j = 0}^1\sum_{i = K_{A_j}^{(l_j)}}\overline{X}_{T_j}^{(j)}(i) + \sum_{j=1}^2\sum_{i = 1}^{A_j}\big\{\overline{X}_{T_{j-1}}^{(j)}(i) - \overline{X}_{T_j}^{(j)}(i) \big\} + \sum_{i = 1}^{A_2}\overline{X}_{T_2}^{(2)}(i)\nonumber \\
    =& \dots\nonumber \\
    =& \sum_{j = 0}^{m(A)-2}\sum_{i \in K_{A_j}^{(l_j)}}\overline{X}_{T_j}^{(j)}(i) + \sum_{j = 1}^{m(A)-1}\sum_{i = 1}^{A_j}\big\{\overline{X}_{T_{j-1}}^{(j)}(i) - \overline{X}_{T_j}^{(j)}(i)\big\} + \sum_{i = 1}^{A_{m(A)-1}}\overline{X}_{T_{m(A)-1}}^{(m(A)-1)}(i)\nonumber \\
    =& \sum_{j = 0}^{m(A)-1}\sum_{i \in K_{A_j}^{(l_j)}}\overline{X}_{T_j}^{(j)}(i) + \sum_{j = 1}^{m(A)-1}\sum_{i = 1}^{A_j}\big\{\overline{X}_{T_{j-1}}^{(j)}(i) - \overline{X}_{T_j}^{(j)}(i)\big\} + \sum_{i = 1}^{A_{m(A)}}\overline{X}_{T_{m(A)-1}}^{(m(A))}(i)\nonumber \\
    =& I + II + III.
\end{align} 
    
The next proposition gives an upper bound on the log-Laplace transform of \eqref{eq:decompose_partial_sum_2}.
Recall $\kappa$ given in \eqref{eq:kappa}, we define the following positive constants depending only on $c$, $\gamma_1$ and $\gamma_2$.
\begin{equation}\label{eq:def_mu_c4}
    \mu = (2(2 \vee 4c_0^{-1})/(1-\gamma))^{2/(1-\gamma)}, \;\; c_4 = 2^{1+1/\gamma_2}3^{\gamma_1/\gamma_2}c_0^{-\gamma_1/\gamma_2}
\end{equation}
and
\begin{equation}\label{eq:def_nu}
    \nu = (3c_4 + \kappa^{-1})^{-1}(1 - 2^{-\gamma_1(1-\gamma)/\gamma}) .
\end{equation}
Note that by definition, we have $\nu \in (0,1)$ and $c_4 > 1$.

\begin{proposition}\label{prop:bernstein_bound_cantor_2}
    Let $A$ be a positive integer and $\{X_i\}_{i = 1}^A$ be a sequence of random variables of the form \eqref{eq:nonstationary-2}, with mean zero. Assume \eqref{eq:dep_exponential_gamma1_2} and \eqref{eq:moment_psi_gamma2_2} hold.
    Let $M = H^{-1}(\exp(-A^{\gamma_1}))$. Then, if $A > \mu$, we have for any $t \in [0,\nu A^{-\gamma_1(1-\gamma)/\gamma})$ that
\begin{equation*}
    \log\mathbb{E}\Big[\exp\Big(t\sum_{i = 1}^A\overline{X}_M(i)\Big)\Big] \leq \frac{AV(A)t^2}{1-\nu^{-1}A^{\gamma_1(1-\gamma)/\gamma}t},
\end{equation*}
where $V(A) = (\nu_1 C_{\mathrm{LRV}} + \nu_2\exp\big(-\nu_3A^{\gamma_1(1-\gamma)}(\log A)^{-\gamma}\big)$ with constants $\nu_1, \nu_2, \nu_3 > 0$ depending only on $c$, $\gamma_1$ and $\gamma_2$, and $C_{\mathrm{LRV}}$ is defined in \eqref{eq:def_LRV_UB}.
\end{proposition}

\subsection{Proofs of propositions}
\begin{proof}[Proof of Proposition \ref{prop:bernstein_bound_cantor}]
    Recall the definitions \eqref{eq:cantor_decomp_1} and \eqref{eq:cantor_decomp_2}, we have for $k = 2, \dots, l$ and $j = 1, \dots, 2^k$ that
    \begin{equation*}
    K_A^{(l)} = \bigcup_{i=1}^{2}K_{A,1,i}^{(l)} \;\; \text{and} \;\; K_{A,k-1,j}^{(l)} = \bigcup_{i = 2j-1}^{2j}K_{A,k,i}^{(l)}.
    \end{equation*}
    Each block at previous step is split into two sub-blocks at the current step. The key of this proof is to upper bound the telescoping distance between the log-Laplace transform of the block and the sum of the log-Laplace transforms of two sub-blocks at each step. Our proof contains $l$ steps. Let $k_l = \max\{k \in \mathbb{N}, k/\gamma < l\}$. Since $\gamma < 1$, $k_l < \gamma l \leq l-1$, we need at least $k_l + 2$ steps. At each step $k = 1, \dots, k_l+1$, we apply Lemma \ref{lemma:diff_laplace_joint_marginal} and adapt the truncation parameter to a smaller value. While at steps $k_l + 2, \dots, l$, the truncation parameter no longer needs to be adapted.
    Let $M_0 = M$. Define the truncation parameter at Step $k$ for any $k = 1, \dots, k_l+1$ as
    \begin{equation*}
        M_k = H^{-1}(\exp(-A^{\gamma_1}2^{-\gamma_1(l \wedge k/\gamma)})),
    \end{equation*}
    and the truncation parameter maintains $M_{k_l+1}$ after Step $k_l+1$.
    Since $H^{-1}(y) = (\log(e/y))^{1/\gamma_2}$ for any $y \leq e$, we have for any $x \geq 1$,
    \begin{equation}\label{eq:property_H^-1}
        H^{-1}(\exp(-x^{\gamma_1})) = (1 + x^{\gamma_1})^{1/\gamma_2} \leq 2^{1/\gamma_2}x^{\gamma_1/\gamma_2}.
    \end{equation}
    Note that for any $k = 0, \dots, k_l$, $A2^{-k/\gamma} \geq A2^{-l} \geq 1$, we have that
    \begin{equation}\label{eq:mj_ub}
        M_k \leq 2^{1/\gamma_2}(A2^{-(l \wedge k/\gamma)})^{\gamma_1/\gamma_2} \;\; \text{for any} \;\; k = 0, \dots, k_l+1.
    \end{equation}
    \textbf{Step $1$.}
    Recall that for $i = 1, 2$, $\card(K_{A,1,i}^{(l)}) \leq A/2$ and for any $j = 1, \dots, A$, $|\overline{X}_{M_0}(j)| \leq 2M_0$. Applying Lemma \ref{lemma:diff_laplace_joint_marginal}, we have for any $t > 0$ that
    \begin{equation*}
    \begin{aligned}
        &\Big|\mathbb{E}\Big[\exp\Big(t\sum_{j \in K_{A}^{(l)}}\overline{X}_{M_0}(j)\Big)\Big] - \prod_{i=1}^2\mathbb{E}\Big[\exp\Big( t\sum_{j \in K_{A,1,i}^{(l)}}\overline{X}_{M_0}(j) \Big)\Big]\Big|\\
        \leq& t\exp(2tAM_0)\Big\Vert \sum_{j \in K_{A,1,2}}\{\overline{X}_{M_0}(j) - \overline{X}_{M_0, \{n_1, -\infty\}}(j)\} \Big\Vert_2\\
        \leq& t\exp(2tAM_0)\sum_{j \in K_{A,1,2}^{(l)}} \sum_{s = j-n_1}^{\infty}\delta_{s,2} \leq \frac{tA}{2}\exp(2tAM_0)\exp(-cd_0^{\gamma_1}),
    \end{aligned}
    \end{equation*}
    where $n_1$ is the maximal integer in $K_{A,1,1}^{(l)}$, $d_0$ is the gap between $K_{A,1,1}^{(l)}$ and $K_{A,1,2}^{(l)}$, the second inequality follows from the triangle inequality and the definition of $\delta_{j,2}$, and the third inequality follows from the facts that $\card(K_{A,1,2}^{(l)}) = A/2$ and $j - n_1 \geq d_0$ for any $j \in K_{A,1,2}^{(l)}$ and the assumption \eqref{eq:dep_exponential_gamma1_2}. Since $\overline{X}_{M_0}(j)$ are centered and by Jensen's inequality, we have that the Laplace transforms are all greater than $1$. By the inequality 
    \begin{equation}\label{eq:elementary_inequ_1}
        |\log(x) - \log(y)| \leq |x -y| \;\; \text{for} \;\; x,y \geq 1,
    \end{equation}
    we have that
    \begin{equation}\label{eq:decomp_step1}
    \begin{aligned}
        &\Big|\log\mathbb{E}\Big[\exp\Big(t\sum_{j \in K_{A}^{(l)}}\overline{X}_{M_0}(j)\Big)\Big] - \sum_{i=1}^2\log\mathbb{E}\Big[\exp\Big( t\sum_{j \in K_{A,1,i}^{(l)}}\overline{X}_{M_0}(j) \Big)\Big]\Big|\\
        \leq& \frac{tA}{2}\exp(2tAM_0)\exp(-cd_0^{\gamma_1}).
    \end{aligned}
    \end{equation}
    We adapt the truncation parameter to $M_1$. By the inequality 
    \begin{equation}\label{eq:elementary_inequ_2}
        |e^{tx} - e^{ty}| \leq |t||x-y|(e^{|tx|} \vee e^{|ty|}),
    \end{equation}
    we have for $i = 1, 2$ and any $t >0$ that
    \begin{equation*}
    \begin{aligned}
        &\Big|\mathbb{E}\Big[\exp\Big( t\sum_{j \in K_{A,1,i}^{(l)}}\overline{X}_{M_0}(j) \Big)\Big] - \mathbb{E}\Big[\exp\Big( t\sum_{j \in K_{A,1,i}^{(l)}}\overline{X}_{M_1}(j) \Big)\Big]\Big|\\
        \leq& te^{tAM_0}\sum_{j \in K_{A,1,i}^{(l)}} \mathbb{E}[|\overline{X}_{M_0}(j)-\overline{X}_{M_1}(j)|] \leq 2te^{tAM_0}\sum_{j \in K_{A,1,i}^{(l)}} \mathbb{E}[|\psi_{M_0}(X_j) -\psi_{M_1}(X_j)|]\\
        \leq& 2tM_0e^{tAM_0}\sum_{j \in K_{A,1,i}^{(l)}} \mathbb{P}(|X_j| \geq M_1) \leq 2tAM_0e^{tAM_0}H(M_1)\\
        =& 2tAM_0e^{tAM_0}\exp(-A^{\gamma_1}2^{-\gamma_1/\gamma}) \leq 2e^{2tAM_0}\exp(-A^{\gamma_1}2^{-\gamma_1/\gamma}),
    \end{aligned}
    \end{equation*}
    where the first inequality follows from that 
    \begin{equation*}
        \Big|t\sum_{j \in K_{A,1,i}^{(l)}}\overline{X}_{M_1}(j)\Big| \leq \Big|t\sum_{j \in K_{A,1,i}^{(l)}}\overline{X}_{M_0}(j)\Big| \leq tAM_0,
    \end{equation*}
    and the triangle inequality, the second and the third inequalities follow from the facts about $\overline{X}_M(j)$, the fourth inequality follows from the assumption \eqref{eq:moment_psi_gamma2_2}, the equality follows from the definition of $M_1$, and the last inequality follows from that $1 + x \leq e^x$.
    By the inequality \eqref{eq:elementary_inequ_1}, we have for $i = 1,2$ and any $t > 0$ that
    \begin{equation}\label{eq:adapt_trunc_step1}
    \begin{aligned}
        &\Big|\log\mathbb{E}\Big[\exp\Big( t\sum_{j \in K_{A,1,i}^{(l)}}\overline{X}_{M_0}(j) \Big)\Big] - \log\mathbb{E}\Big[\exp\Big( t\sum_{j \in K_{A,1,i}^{(l)}}\overline{X}_{M_1}(j) \Big)\Big]\Big|\\
        \leq& 2e^{2tAM_0}\exp(-A^{\gamma_1}2^{-\gamma_1/\gamma}),
    \end{aligned}
    \end{equation}
    Combining \eqref{eq:decomp_step1} and \eqref{eq:adapt_trunc_step1}, we have that
    \begin{equation}\label{eq:together_step1}
    \begin{aligned}
        &\Big|\log\mathbb{E}\Big[\exp\Big(t\sum_{j \in K_{A}^{(l)}}\overline{X}_{M_0}(j)\Big)\Big] - \sum_{i=1}^2\log\mathbb{E}\Big[\exp\Big( t\sum_{j \in K_{A,1,i}^{(l)}}\overline{X}_{M_1}(j) \Big)\Big]\Big|\\
        \leq& \frac{tA}{2}\exp(2tAM_0)\exp(-cd_0^{\gamma_1}) + 4\exp(2tAM_0)\exp(-A^{\gamma_1}2^{-\gamma_1/\gamma}).
    \end{aligned}
    \end{equation}
    
    \textbf{Step k.} For any $k = 2, \dots, k_l+1$, we use the same arguments as in Step 1. Note that $\card(K_{A,k,i}^{(l)}) \leq 2^{-k}A$, for any $i = 1, \dots, 2^k$, we have that
    \begin{equation*}
    \begin{aligned}
        &\Big|\sum_{i = 1}^{2^{k-1}}\log\mathbb{E}\Big[\exp\Big(t\sum_{j \in K_{A, k-1,i}^{(l)}}\overline{X}_{M_{k-1}}(j)\Big)\Big] - \sum_{i=1}^{2^k}\log\mathbb{E}\Big[\exp\Big( t\sum_{j \in K_{A,k,i}^{(l)}}\overline{X}_{M_k}(j) \Big)\Big]\Big|\\
        \leq& 2^{k-1}\frac{tA}{2^k}\exp(t2M_{k-1}A/2^{k-1})\exp(-cd_{k-1}^{\gamma_1}) + 2^{k+1}\exp(t2M_{k-1}A/2^{k-1})\exp(-A^{\gamma_1}2^{-\gamma_1(l \wedge k/\gamma)})\\
        =& \Big[\frac{tA}{2}\exp(-cd_{k-1}^{\gamma_1}) + 2^{k+1}\exp(-A^{\gamma_1}2^{-\gamma_1(l \wedge k/\gamma)})\Big]\exp(2^{2-k}tM_{k-1}A).
    \end{aligned}
    \end{equation*}

    After $k_l+1$ steps, we upper bound all telescoping distances to the sum of log-Laplace transforms for $K_{A,k_l+1,i}$, $i = 1, \dots, 2^{k_l+1}$. Combining them together using the triangle inequality, we have that
    \begin{align}\label{eq:together_stepkl}
        &\Big|\log\mathbb{E}\Big[\exp\Big(t\sum_{j \in K_{A}^{(l)}}\overline{X}_{M_0}(j)\Big)\Big] - \sum_{i=1}^{2^{k_l+1}}\log\mathbb{E}\Big[\exp\Big( t\sum_{j \in K_{A,k_l+1,i}^{(l)}}\overline{X}_{M_{k_l+1}}(j) \Big)\Big]\Big| \nonumber \\
        \leq& \frac{tA}{2}\sum_{k = 1}^{k_l+1}\big\{\exp(-cd_{k-1}^{\gamma_1})\exp(2^{2-k}tM_{k-1}A)\big\}+ \sum_{k = 1}^{k_l}\big\{2^{k+1}\exp(-A^{\gamma_1}2^{-k\gamma_1/\gamma})\exp(2^{2-k}tM_{k-1}A)\big\}\nonumber \\
        &+ 2^{k_l+1}\exp(-A^{\gamma_1}2^{-\gamma_1l})\exp(2^{1-k_l}tM_{k_l}A).
    \end{align}

    Now we consider the steps $k = k_l+2, \dots, l$. These steps can be simplified as one step, since at these steps, by the definitions of $d_{k-1}$ and $M_{k-1}$, both the size of gaps and the truncation parameter remain unchanged. We consider this extra step. For any $i = 1, \dots, 2^{k_l+1}$, $\sum_{j \in K_{A, k_l+1, i}^{(l)}}\overline{X}_{M_{k_l}+1}$ is a sum of $2^{l-k_l-1}$ sub-blocks, each of size $n_l$ and bounded by $2M_{k_l+1}n_l$. In additional, the sub-blocks are equidistant with size $d_{k_l+1}$ between two neighbouring sub-blocks. By Lemma \ref{lemma:diff_laplace_joint_marginal}, we have that
    \begin{equation}\label{eq:together_step_extra}
    \begin{aligned}
        &\Big|\log\mathbb{E}\Big[\exp\Big(t\sum_{j \in K_{A, k_l+1,i}^{(l)}}\overline{X}_{M_{k_l+1}}(j)\Big)\Big] - \sum_{s=(i-1)2^{l-k_l-1}+1}^{i2^{l-k_l-1}}\log\mathbb{E}\Big[\exp\Big( t\sum_{j \in I_{l,s}}\overline{X}_{M_{k_l+1}}(j) \Big)\Big]\Big|\\
        \leq& tn_l2^{l-k_l-1}\exp(2^{l-k_l}tM_{k_l+1}n_l).
    \end{aligned}
    \end{equation}
    
    Combining \eqref{eq:together_stepkl} and \eqref{eq:together_step_extra}, we have that
    \begin{align} \label{eq:together}
        &\Big|\log\mathbb{E}\Big[\exp\Big(t\sum_{j \in K_{A}^{(l)}}\overline{X}_{M_0}(j)\Big)\Big] - \sum_{i=1}^{2^{l}}\log\mathbb{E}\Big[\exp\Big( t\sum_{j \in I_{l,i}}\overline{X}_{M_{k_l+1}}(j) \Big)\Big]\Big| \nonumber \\
        \leq& \frac{tA}{2}\sum_{k = 1}^{k_l+1}\big\{\exp(-cd_{k-1}^{\gamma_1})\exp(2^{2-k}tM_{k-1}A)\big\} + \sum_{k = 1}^{k_l}\big\{2^{k+1}\exp(-A^{\gamma_1}2^{-k\gamma_1/\gamma})\exp(2^{2-k}tM_{k-1}A)\big\}\nonumber \\
        &+ 2^{k_l+1}\exp(-A^{\gamma_1}2^{-\gamma_1l})\exp(2^{1-k_l}tM_{k_l}A) + tn_l2^{l}\exp(-cd_{k_l+1}^{\gamma_1})\exp(2^{l-k_l}tM_{k_l+1}n_l).
    \end{align} 
    Due to the definition of $d_j$ given in \eqref{eq:gap_size_j-1}, we have for any $j = 0, \dots, l-1$ that
    \begin{equation*}
        d_j+1 \geq \lfloor c_0A2^{-(l \wedge j/\gamma)} \rfloor,
    \end{equation*}
    where for $x \in \mathbb{R}$, $\lfloor x \rfloor$ is the largest integer no greater than $x$. 
    Moreover, since $c_0A2^{-(l\wedge j/\gamma)} \geq 2$ by assumption, we have
    \begin{equation*}
        d_j \geq (d_j+1)/2 \geq c_0A2^{-(l \wedge j/\gamma)-2}.
    \end{equation*}
    Recall also that $\gamma l < k_l + 1 \leq l-1$ and $2^ln_l \leq A$.
    Thus, we have that
    \begin{align*}
        \eqref{eq:together} \leq& \frac{tA}{2}\sum_{k = 1}^{k_l+1}\big\{\exp(-c(c_02^{-2}A2^{- (k-1)/\gamma})^{\gamma_1} + 2^{2-k}tM_{k-1}A)\big\}\\
        &+ \sum_{k = 1}^{k_l}\big\{2^{k+1}\exp(-A^{\gamma_1}2^{-k\gamma_1/\gamma} + 2^{2-k}tM_{k-1}A)\big\} + 2^{k_l+1}\exp(-A^{\gamma_1}2^{-\gamma_1l}+2^{1-k_l}tM_{k_l}A)\\
        &+ tA\exp(-c(c_02^{-2}A2^{-l})^{\gamma_1})\exp(2^{-k_l}tM_{k_l+1}A)\\
        =& \frac{tA}{2}\sum_{k = 0}^{k_l}\big\{\exp(-(c_1A2^{-k/\gamma})^{\gamma_1} + 2^{1-k}tM_{k}A)\big\}\\
        &+ \sum_{k = 0}^{k_l-1}\big\{2^{k+2}\exp(-(c_1A2^{-k/\gamma})^{\gamma_1} + 2^{1-k}tM_{k}A)\big\}\\
        &+ 2^{k_l+1}\exp(-(A2^{-l})^{\gamma_1}+2^{1-k_l}tM_{k_l}A)+ tA\exp(-(c_1A2^{-l})^{\gamma_1}+2^{-k_l}tM_{k_l+1}A),
    \end{align*}
    where the equality is due to the definition $c_1 = \min\{c^{1/\gamma_1}c_0/4, 2^{-1/\gamma}\}$.
    
    By \eqref{eq:mj_ub}, $1/\gamma = 1/\gamma_1 + 1/\gamma_2$ with $\gamma < 1$ and $\gamma l < k_l + 1 \leq l-1$, we have the following facts
    \begin{itemize}
        \item $2^{1-k}M_{k}A \leq 2^{1+1/\gamma_2}(2^{-k}A)^{\gamma_1/\gamma}$ for any $k = 0, \dots, k_l$,
        \item $M_{k_l+1} \leq 2^{1/\gamma_2}(A2^{-l})^{\gamma_1/\gamma_2} \leq 2^{1/\gamma_2}(A2^{-\gamma l})^{\gamma_1/\gamma_2}$,
        \item $2^{-k_l}M_{k_l+1}A = 2AM_{k_l+1}2^{-(k_l+1)} \leq 2^{1+1/\gamma_2}A^{\gamma_1/\gamma}2^{-\gamma l\gamma_1/\gamma_2-\gamma l} \leq 2^{1+1/\gamma_2}A^{\gamma_1/\gamma}2^{-\gamma_1l}$,
        \item and \begin{align*}
            & 2^{1-k_l}M_{k_l}A \leq 2^{1+1/\gamma_2}A^{\gamma_1/\gamma}2^{-k_l}(2^{-k_l/\gamma})^{\gamma_1/\gamma_2} \leq 2^{1+1/\gamma_2+\gamma_1/\gamma}(A2^{-k_l-1})^{\gamma_1/\gamma} \\
            \leq & 2^{2+(1+\gamma_1)/\gamma_2}(A2^{-\gamma l})^{\gamma_1/\gamma} \leq 2^{2+(1+\gamma_1)/\gamma_2}A^{\gamma_1/\gamma}2^{-\gamma_1l}.
        \end{align*}
    \end{itemize}
    Recall that $c_2 = 2^{-(2+\gamma_1/\gamma+1/\gamma_2)}c_1^{\gamma_1}$, and for $t \leq c_2A^{\gamma_1(\gamma-1)/\gamma}$, we have that
    \begin{align*}
        \eqref{eq:together}
        \leq& \frac{tA}{2}\sum_{k = 0}^{k_l}\big\{\exp(-(c_1A2^{-k/\gamma})^{\gamma_1} + 2^{-(1+\gamma_1/\gamma)}(c_12^{-k/\gamma}A)^{\gamma_1})\big\}\\
        &+ \sum_{k = 0}^{k_l-1}\big\{2^{k+2}\exp(-(c_1A2^{-k/\gamma})^{\gamma_1} + 2^{-(1+\gamma_1/\gamma)}(c_12^{-k/\gamma}A)^{\gamma_1})\big\}\\
        &+ 2^{k_l+1}\exp(-(c_1A2^{-l})^{\gamma_1}+2^{-1}(c_12^{-l}A)^{\gamma_1}) + tA\exp(-(c_1A2^{-l})^{\gamma_1}+2^{-(1+\gamma_1/\gamma)}(c_12^{-l}A)^{\gamma_1})\\
        \leq& \frac{tA}{2}\sum_{k = 0}^{k_l}\exp(-(c_1A2^{-k/\gamma})^{\gamma_1}/2) + \sum_{k = 0}^{k_l-1}\big\{2^{k+2}\exp(-(c_1A2^{-k/\gamma})^{\gamma_1}/2)\big\}\\
        &+ (2^{k_l+1} + tA)\exp(-(c_1A2^{-l})^{\gamma_1}/2)\\
        \leq& \frac{tA}{2}(k_l+1)\exp(-(c_1A2^{-k_l/\gamma})^{\gamma_1}/2) + (2^{k_l+2}-2)\exp(-(c_1A2^{-(k_l-1)/\gamma})^{\gamma_1}/2)\\
        &+ (2^{k_l+1} + tA)\exp(-(c_1A2^{-l})^{\gamma_1}/2)\\
        \leq& (tAl + 6A^{\gamma})\exp(-2^{-1}(2^{-l}c_1A)^{\gamma_1}),
    \end{align*}
    where the last inequality follows from $2^{k_l} \leq 2^{l\gamma} \leq A^{\gamma}$, $k_l+1 \leq l-1$ and $l > 2$.
    
    We upper bound the log-Laplace transform of $\sum_{j \in I_{l,i}}\overline{X}_{M_{k_l+1}}(j)$ in \eqref{eq:together}. The sum is bounded by $2M_{k_l+1}n_l \leq n_l2^{1+1/\gamma_2}(A2^{-l})^{\gamma_1/\gamma_2} \leq 2^{1+1/\gamma_2}(A2^{-l})^{\gamma_1/\gamma}$, which is due to \eqref{eq:mj_ub} and $n_l \leq A2^{-l}$. Note that the function $x \mapsto g(x) = x^{-2}(e^x-x-1)$ is increasing on $\mathbb{R}$ and $g(1) < 1$. For any centered random variable $U \in \mathbb{R}$ such that $|U| \leq M$, and any $t > 0$, we have that
    \begin{equation}\label{eq:laplace_small_quantity}
        \mathbb{E}[\exp(tU)] \leq 1+t^2g(tM)\mathbb{E}[U^2] \leq \exp\big(t^2g(tM)\mathbb{E}[U^2]\big).
    \end{equation}
    Let $t \leq 2^{-(1+1/\gamma_2)}(A2^{-l})^{-\gamma_1/\gamma} = c_3(A2^{-l})^{-\gamma_1/\gamma}$, we have $t\sum_{j \in I_{l,i}}\overline{X}_{M_{k_l+1}}(j)$ is bounded by $1$. Thus, by \eqref{eq:laplace_small_quantity} it holds that
    \begin{equation}\label{eq:individual_logLaplace}
    \begin{aligned}
        &\log\mathbb{E}\Big[\exp\Big(t\sum_{j \in I_{l,i}}\overline{X}_{M_{k_l+1}}(j)\Big)\Big] \leq t^2g(1)\mathbb{E}\Big[\Big(\sum_{j \in I_{l,i}}\overline{X}_{M_{k_l+1}}(j)\Big)^2\Big]\leq C_{\mathrm{LRV}}t^2n_l,
    \end{aligned}
    \end{equation}
    where the second inequality follows from Lemma \ref{lemma:lrv_univar} and the fact that the functional dependence measure of the truncated process is bounded by that of the original process, i.e. for any $M > 0$
    \begin{equation*}
        \big\Vert \overline{X}_{M}(s) - \overline{X}_{M}^*(s) \big\Vert_2 = \Vert \psi_{M}(X_s) - \psi_{M}(X_s^*) \Vert_2 \leq \delta_{s,2}, 
    \end{equation*}
    where the inequality follows from that $\psi_M(x)$ is Lipschitz continuous.
    
    Recall \eqref{eq:mj_ub} that $\big|\sum_{j \in K_{A}^{(l)}}\overline{X}_{M_0}(j)\big| \leq 2M_0A \leq 2^{1+1/\gamma_2}A^{\gamma_1/\gamma}$.
    If $t \leq 2^{-(1+1/\gamma_2)}A^{-\gamma_1/\gamma}$, then $t\sum_{j \in K_{A}^{(l)}}\overline{X}_{M_0}(j)$ is bounded by $1$. By \eqref{eq:laplace_small_quantity}, we have for any $t \leq 2^{-(1+1/\gamma_2)}A^{-\gamma_1/\gamma}$
    \begin{equation*}
        \begin{aligned}
        \log\mathbb{E}\Big[\exp\Big(t\sum_{j \in K_{A}^{(l)}}\overline{X}_{M_0}(j)\Big)\Big] \leq C_{\mathrm{LRV}}t^2A.
    \end{aligned}
    \end{equation*}
    Recall the definition of $\kappa$ given in \eqref{eq:kappa}. For any $2^{-(1+1/\gamma_2)}A^{-\gamma_1/\gamma} \leq t \leq \kappa\big(A^{\gamma_1(\gamma-1)/\gamma} \vee (A2^{-l})^{-\gamma_1/\gamma}\big)$, combining \eqref{eq:together} and \eqref{eq:individual_logLaplace} leads to
    \begin{equation*}
    \begin{aligned}
        &\log\mathbb{E}\Big[\exp\Big(t\sum_{j \in K_{A}^{(l)}}\overline{X}_{M_0}(j)\Big)\Big]\leq (tAl + 6A^{\gamma})\exp(-(2^{-l}c_1A)^{\gamma_1}/2) + C_{\mathrm{LRV}}t^2n_l2^l\\
        \leq& (tAl + 6A^{\gamma})\exp(-(2^{-l}c_1A)^{\gamma_1}/2) + C_{\mathrm{LRV}}t^2A.
    \end{aligned}
    \end{equation*}
    where the first inequality follows from the triangle inequality and that the log-Laplace transform of a centered random is nonnegative.
\end{proof}

\begin{proof}[Proof of Proposition \ref{prop:bernstein_bound_cantor_2}]
The proof is essentially bounding all the three terms in \eqref{eq:decompose_partial_sum_2}.

\textbf{Log-Laplace transform of the term $I$ in \eqref{eq:decompose_partial_sum_2}.} For any $j = 0, \dots, m(A)-1$, by definition we have that
\begin{equation}\label{eq:decomp_property_Aj}
    2^{-l_j}A_j = 2^{1-l_j}A_j/2 > 2^{-l}A/2 \geq (1 \vee 2c_0^{-1}).
\end{equation}
By Proposition \ref{prop:bernstein_bound_cantor}, for any $j = 0, \dots, m(A)-1$, we have for any $t \leq 2^{-(1+1/\gamma_2)}A_j^{-\gamma_1/\gamma}$ that
\begin{equation}\label{eq:laplace_whole_set_II_case1}
\begin{aligned}
    &\log\mathbb{E}\Big[\exp\Big(t\sum_{i \in K_{A_j}^{(l_j)}}\overline{X}_{T_j}^{(j)}(i)\Big)\Big] \leq C_{\mathrm{LRV}}t^2A_j.
\end{aligned}
\end{equation}
Moreover, for any 
$2^{-(1+1/\gamma_2)}A_j^{-\gamma_1/\gamma} \leq t \leq \kappa(A_j^{\gamma-1} \wedge (2^{l_j}/A_j))^{\gamma_1/\gamma}$, it holds that
\begin{align}\label{eq:laplace_whole_set_II_case2}
    &\log\mathbb{E}\Big[\exp\Big(t\sum_{i \in K_{A_j}^{(l_j)}}\overline{X}_{T_j}^{(j)}(i)\Big)\Big]\leq (tA_jl_j + 4A_j^{\gamma})\exp(-(2^{-l_j}c_1A_j)^{\gamma_1}/2) + C_{\mathrm{LRV}}t^2A_j \nonumber \\
    \leq& t^2\Big((C_{\mathrm{LRV}}A_j)^{1/2} +  \big(l_i^{1/2}A_j^{1/2}2^{1/2+1/(2\gamma_2)}A_j^{\gamma_1/(2\gamma)} + 2A_j^{\gamma/2}2^{1+1/\gamma_2}A_j^{\gamma_1/\gamma}\big) \exp(-(2^{-l_j}c_1A_j)^{\gamma_1}/4)\Big)^2.
\end{align}
Note that, by construction, we have the following facts.
\begin{itemize}
    \item For each $j = 0, \dots, m(A)-1$, it holds that $l_j \leq l \leq A$.
    \item For any $j = 0, \dots, m(A)$, it holds that $A_j \leq A2^{-j}$.
    \item For any $j = 0, \dots, m(A)-1$, it holds that $2^{-l-1}A < 2^{-l_j}A_j \leq 2^{-l}A$.
\end{itemize}
We have that
\begin{equation*}
\begin{aligned}
    \eqref{eq:laplace_whole_set_II_case2} 
    \leq& t^2\bigg[\bigg(\frac{C_{\mathrm{LRV}}A}{2^{j}}\bigg)^{1/2} +  \bigg(\frac{A^{1+\frac{\gamma_1}{2\gamma}}}{(2^j)^{\frac{1}{2}+\frac{\gamma_1}{2\gamma}}}2^{\frac{1}{2}+\frac{1}{2\gamma_2}} + 2^{2+\frac{1}{\gamma_2}}\frac{A^{\frac{\gamma}{2}+\frac{\gamma_1}{\gamma}}}{(2^j)^{\frac{\gamma}{2}+\frac{\gamma_1}{\gamma}}}\bigg) \exp\bigg(-\frac{c_1^{\gamma_1}}{2^{2+\gamma_1}}\bigg(\frac{A}{2^l}\bigg)^{\gamma_1}\bigg)\bigg]^2\\
    \leq& t^2\bigg[\bigg(\frac{C_{\mathrm{LRV}}A}{2^{j}}\bigg)^{1/2} +  \bigg(\frac{A^{1+\frac{\gamma_1}{\gamma}}}{(2^j)^{\frac{\gamma}{2}+\frac{\gamma_1}{2\gamma}}}2^{3+\frac{1}{\gamma_2}}\bigg) \exp\bigg(-\frac{c_1^{\gamma_1}}{2^{2+\gamma_1}}\bigg(\frac{A}{2^l}\bigg)^{\gamma_1}\bigg)\bigg]^2
    = t^2\sigma_{I,j}^2.
\end{aligned}
\end{equation*}
Combining \eqref{eq:laplace_whole_set_II_case1} and \eqref{eq:laplace_whole_set_II_case2}, we have for any $t \leq \kappa(A_j^{\gamma-1} \wedge (2^{l_j}/A_j))^{\gamma_1/\gamma}$, it holds that
\begin{equation}\label{eq:log-laplace_termI}
\begin{aligned}
    \log\mathbb{E}\Big[\exp\Big(t\sum_{i \in K_{A_j}^{(l_j)}}\overline{X}_{T_j}^{(j)}(i)\Big)\Big] \leq t^2\sigma_{I,i}^2.
\end{aligned}
\end{equation}

\textbf{Log-Laplace transform of the term $II$ in \eqref{eq:decompose_partial_sum_2}.} Before proceeding, we state the following fact. For any $j = 0, \dots, m(A)-1$, it holds that
\begin{equation}\label{eq:relation_Ajs}
    A_{j+1} \geq c_0A_j/3.
\end{equation}
Recall in the construction of $K_A^{(l)}$ that $A - \card(K_A^{(l)}) = \sum_{i = 0}^{l-1}2^id_j \geq \lfloor c_0A \rfloor - 1$ by the definition of $d_i$. We thus have that $A_{j+1} \geq \lfloor c_0A_j \rfloor - 1$. Since by \eqref{eq:decomp_property_Aj} $c_0A_j \geq 2$, we have that $\lfloor c_0A_j \rfloor - 1 \geq (\lfloor c_0A_j \rfloor + 1)/3 \geq c_0A_j/3$, which leads to \eqref{eq:relation_Ajs}.

By the property of the truncation function, for any $j = 1, \dots, m(A)-1$, we have  that
\begin{equation}\label{eq:bound_Aj_diff_trun}
\begin{aligned}
    &\Big|\sum_{i = 1}^{A_j}\big(\overline{X}_{T_{j-1}}^{(j)}(i) - \overline{X}_{T_j}^{(j)}(i)\big)\Big|\leq \sum_{i = 1}^{A_j}\big\{|\psi_{T_{j-1}}(X_i) - \psi_{T_{j}}(X_i)| + \mathbb{E}[|\psi_{T_{j-1}}(X_i) - \psi_{T_{j}}(X_i)|]\big\}\\
    \leq& T_{j-1}\sum_{i = 1}^{A_j}\big(\mathbbm{1}\{|X_i| > T_j\} + \mathbb{P}(|X_i| > T_j)\big).
\end{aligned}
\end{equation}
Then, \eqref{eq:bound_Aj_diff_trun} leads to
\begin{equation}\label{eq:2mom_Aj_diff_trun}
\begin{aligned}
    &\mathbb{E}\Big[\Big|\sum_{i = 1}^{A_j}\big(\overline{X}_{T_{j-1}}^{(j)}(i) - \overline{X}_{T_j}^{(j)}(i)\big)\Big|^2\Big] \leq (2T_{j-1})^2\mathbb{E}\bigg[\Big(\sum_{i = 1}^{A_j}\frac{\mathbbm{1}\{|X_i| > T_j\} + \mathbb{P}(|X_i| > T_j)}{2}\Big)^2\bigg]\\
    \leq& (2T_{j-1})^2A_j\sum_{i = 1}^{A_j}\mathbb{E}\bigg[\Big(\frac{\mathbbm{1}\{|X_i| > T_j\} + \mathbb{P}(|X_i| > T_j)}{2}\Big)^2\bigg] \leq (2A_jT_{j-1})^2H(T_j).
\end{aligned}
\end{equation}
By \eqref{eq:Tj} and \eqref{eq:relation_Ajs}, \eqref{eq:bound_Aj_diff_trun} and \eqref{eq:2mom_Aj_diff_trun} can be further bounded respectively as
\begin{equation}\label{eq:bounded_2AjTj-1}
\begin{aligned}
    &\Big|\sum_{i = 1}^{A_j}\big(\overline{X}_{T_{j-1}}^{(j)}(i) - \overline{X}_{T_j}^{(j)}(i)\big)\Big| \leq 2A_jT_{j-1} \leq 2A_j2^{1/\gamma_2}A_{j-1}^{\gamma_1/\gamma_2} \leq (3/c_0)^{\gamma_1/\gamma_2}2^{1+1/\gamma_2}A_j^{\gamma_1/\gamma} = c_4A_j^{\gamma_1/\gamma},
\end{aligned}
\end{equation}
and
\begin{equation*}
\begin{aligned}
    &\mathbb{E}\Big[\Big|\sum_{i = 1}^{A_j}\big(\overline{X}_{T_{j-1}}^{(j)}(i) - \overline{X}_{T_j}^{(j)}(i)\big)\Big|^2\Big] \leq (c_4A_j^{\gamma_1/\gamma})^2\exp(-A_j^{\gamma_1}),
\end{aligned}
\end{equation*}
where $c_4 = 2^{1+1/\gamma_2}3^{\gamma_1/\gamma_2}c_0^{-\gamma_1/\gamma_2}$ defined in \eqref{eq:def_mu_c4}. Since $A_j \leq A2^{-j}$ and for any 
    \[
        t \leq (2c_4)^{-1}(2^j/A)^{\gamma_1(1-\gamma)/\gamma},
    \]
    we have
\begin{equation*}
    c_4tA_j^{\gamma_1/\gamma} \leq 2^{-1}A_j^{\gamma_1 - \gamma_1/\gamma}A_j^{\gamma_1/\gamma} = A_j^{\gamma_1}/2.
\end{equation*}
Then by \eqref{eq:laplace_small_quantity}, we have for any $t \leq (2c_4)^{-1}(2^j/A)^{\gamma_1(1-\gamma)/\gamma}$ that
\begin{equation}\label{eq:log-laplace_termII}
\begin{aligned}
    &\log\mathbb{E}\Big[ \exp\Big(t \sum_{i = 1}^{A_j}\big(\overline{X}_{T_{j-1}}^{(j)}(i) - \overline{X}_{T_j}^{(j)}(i)\big)\Big) \Big] \leq t^2c_4^2\Big(\frac{A}{2^j}\Big)^{\frac{2\gamma_1}{\gamma}}\exp\Big(-\frac{A_j^{\gamma_1}}{2}\Big)\\
    \leq& t^2\bigg(c_4\Big(\frac{A}{2^j}\Big)^{\frac{\gamma_1}{\gamma}}\exp\Big(-\frac{1}{4}\Big(\frac{A}{2^l}\Big)^{\gamma_1}\Big)\bigg)^2 = t^2\sigma_{II,i}^2,
\end{aligned}
\end{equation}
where the last inequality follows from that $A_j \geq A_{m(A)-1} > A2^{-l}$ by the definition of $m(A)$ given in \eqref{eq:def_m(A)}.

\textbf{Log-Laplace transform of the term $III$ in \eqref{eq:decompose_partial_sum_2}.} We have that
\begin{equation*}
    \Big|\sum_{i = 1}^{A_{m(A)}}\overline{X}_{T_{m(A)-1}}^{(m(A))}(i)\Big| \leq 2T_{m(A)-1}A_{m(A)} \leq c_4A_{m(A)}^{\gamma_1/\gamma} \leq c_4(A2^{-l})^{\gamma_1/\gamma},
\end{equation*}
where the second inequality follows from the same arguments in \eqref{eq:bounded_2AjTj-1}, and the third inequality follows from that $A_{m(A)} \leq A2^{-l}$ by the definition of $m(A)$. Then by \eqref{eq:laplace_small_quantity}, for any $t \leq c_4^{-1}(A2^{-l})^{-\gamma_1/\gamma}$, we have that
\begin{equation}\label{eq:log-laplace_termIII}
\begin{aligned}
    &\log\mathbb{E}\Big[\exp\Big(t\sum_{i = 1}^{A_{m(A)}}\overline{X}_{T_{m(A)-1}}^{(m(A))}(i)\Big)\Big] \leq t^2C_{\mathrm{LRV}}A2^{-l} \leq t^2\Big[\Big(\frac{C_{\mathrm{LRV}}A}{2^l}\Big)^{1/2}\Big]^2 =  t^2\sigma_{III}^2.
\end{aligned}
\end{equation}

Let 
\begin{equation*}
    C = \frac{1}{\kappa}\sum_{j = 0}^{m(A)-1}\Big(\Big(\frac{A}{2^j}\Big)^{1-\gamma} \vee \frac{A}{2^l}\Big)^{\gamma_1/\gamma} + 2c_4\sum_{j = 1}^{m(A)-1}\Big(\frac{A}{2^j}\Big)^{\gamma_1(1-\gamma)/\gamma} + c_4\Big( \frac{A}{2^l}\Big)^{\gamma_1/\gamma},
\end{equation*}
and
\begin{equation*}
\begin{aligned}
    &\sigma = \sum_{j = 0}^{m(A)-1}\sigma_{I,j} + \sum_{j = 1}^{m(A)-1}\sigma_{II,j} + \sigma_{III}\\
    =& \sum_{j = 0}^{m(A)-1}\bigg\{\bigg(\frac{C_{\mathrm{LRV}}A}{2^{j}}\bigg)^{1/2} +  \bigg(\frac{A^{1+\frac{\gamma_1}{\gamma}}}{(2^j)^{\frac{\gamma}{2}+\frac{\gamma_1}{2\gamma}}}2^{3+\frac{1}{\gamma_2}}\bigg) \exp\bigg(-\frac{c_1^{\gamma_1}}{2^{2+\gamma_1}}\bigg(\frac{A}{2^l}\bigg)^{\gamma_1}\bigg)\bigg\}\\
    &+ \sum_{j = 1}^{m(A)-1}\bigg\{c_4\bigg(\frac{A}{2^j}\bigg)^{\gamma_1/\gamma}\exp\bigg(-\frac{1}{4}\bigg(\frac{A}{2^l}\bigg)^{\gamma_1}\bigg)\bigg\} + \bigg(\frac{C_{\mathrm{LRV}}A}{2^l}\bigg)^{1/2}.
\end{aligned}
\end{equation*}
Note that $m(A) \leq l$ and $l \leq \log A/\log 2$. We select $l$ as
\begin{equation}\label{eq:select_l}
    l = \min\{k \in \mathbb{N}, 2^k \geq A^{\gamma}(\log A)^{\gamma/\gamma_1}\}.
\end{equation}
This selection satisfy our initial requirement for $l$, i.e.~$A2^{-l} \geq (2 \vee 4c_0^{-1})$, if
\begin{equation}\label{eq:condition_select_l}
    (2 \vee 4c_0^{-1})(\log A)^{\gamma/\gamma_1} \leq A^{1-\gamma}.
\end{equation}
We now show that \eqref{eq:condition_select_l} holds. For $A \geq 3$, since $\gamma_1/\gamma > 1$, we have that
\begin{equation*}
    (2 \vee 4c_0^{-1})(\log A)^{\gamma/\gamma_1} \leq (2 \vee 4c_0^{-1})\log A \leq 2(1-\gamma)^{-1}(2 \vee 4c_0^{-1})A^{(1-\gamma)/2},
\end{equation*}
where the last inequality follows from the inequality $\delta\log u \leq u^{\delta}$ for any $\delta, u > 0$. Recall the definition  
$\mu = \big(2(2 \vee 4c_0^{-1})/(1-\gamma)\big)^{\frac{2}{1-\gamma}}$ given in \eqref{eq:def_mu_c4}.
This implies that if $A > \mu$, then \eqref{eq:condition_select_l} holds and the selection of $l$ is valid. Moreover, recall
$\nu = \big(3c_4 + \kappa^{-1} \big)^{-1}(1 - 2^{-\gamma_1(1-\gamma)/\gamma})$ given in \eqref{eq:def_nu}.
Note that
\begin{equation*}
    \sum_{j=0}^{m(A)-1}\Big(\frac{A}{2^j}\Big)^{\gamma_1(1-\gamma)/\gamma} \leq \frac{A^{\gamma_1(1-\gamma)/\gamma}}{1 - 2^{-\gamma_1(1-\gamma)/\gamma}},
\end{equation*}
\begin{equation*}
    \sum_{j=0}^{m(A)-1}\Big(\frac{A}{2^l}\Big)^{\gamma_1/\gamma} \leq l\Big(\frac{A}{2^l}\Big)^{\gamma_1/\gamma} \leq \Big(\frac{A}{2^l}(\log A)^{\gamma/\gamma_1}\Big)^{\gamma_1/\gamma} \leq (A^{\gamma_1(1-\gamma)/\gamma}),
\end{equation*}
\begin{equation*}
    \sum_{j = 1}^{m(A)-1}\Big(\frac{A}{2^j}\Big)^{\gamma_1(1-\gamma)/\gamma} \leq \sum_{j=0}^{m(A)-1}\Big(\frac{A}{2^j}\Big)^{\gamma_1(1-\gamma)/\gamma},
\end{equation*}
and
\begin{equation*}
    \Big(\frac{A}{2^l}\Big)^{\gamma_1/\gamma} \leq l\Big(\frac{A}{2^l}\Big)^{\gamma_1/\gamma} \leq (A^{\gamma_1(1-\gamma)/\gamma}).
\end{equation*}
Therefore, it holds that
\begin{equation*}
    C \leq \Big(3c_4 + \frac{1}{\kappa}\Big)\frac{A^{\gamma_1(1-\gamma)/\gamma}}{1 - 2^{-\gamma_1(1-\gamma)/\gamma}} = \nu^{-1}A^{\gamma_1(1-\gamma)/\gamma}.
\end{equation*}

Since $\gamma_1/\gamma > 1$, it holds that
\begin{equation*}
    \sum_{j = 0}^{m(A)-1}2^{-j/2} \leq \frac{1}{1-2^{-1/2}}, \;\; \sum_{j = 0}^{m(A)-1}2^{-j\big(\frac{\gamma}{2}+\frac{\gamma_1}{2\gamma}\big)} \leq \frac{1}{1-2^{-\big(\frac{\gamma}{2}+\frac{\gamma_1}{2\gamma}\big)}} \leq \frac{1}{1-2^{-1/2}},
\end{equation*}
and
\begin{equation*}
    \sum_{j = 1}^{m(A)-1}2^{-j\frac{\gamma_1}{\gamma}} \leq \frac{2^{-\frac{\gamma_1}{\gamma}}}{1-2^{-\frac{\gamma_1}{\gamma}}} \leq 1.
\end{equation*}
Therefore, it holds that
\begin{equation*}
\begin{aligned}
    \sigma \leq& 4(C_{\mathrm{LRV}}A)^{1/2} + 28\times 2^{\frac{1}{\gamma_2}}A^{1+\frac{\gamma_1}{\gamma}}\exp\bigg(-\frac{c_1^{\gamma_1}}{2^{2+\gamma_1}}\bigg(\frac{A}{2^l}\bigg)^{\gamma_1}\bigg) + c_4\exp\bigg(-\frac{1}{4}\bigg(\frac{A}{2^l}\bigg)^{\gamma_1}\bigg).
\end{aligned}
\end{equation*}
Since by our choice of $l$, $A2^{-l} \geq 2^{-1}A^{1-\gamma}(\log A)^{-\gamma/\gamma_1}$, there exist constants $\nu_1, \nu_2, \nu_3 > 0$ depending only on $c$, $\gamma_1$ and $\gamma_2$ such that
\begin{equation*}
\begin{aligned}
    \sigma^2 \leq& A(\nu_1 C_{\mathrm{LRV}} + \nu_2\exp\big(-\nu_3A^{\gamma_1(1-\gamma)}(\log A)^{-\gamma}\big) = AV(A).
\end{aligned}
\end{equation*}
Combining \eqref{eq:decompose_partial_sum_2}, \eqref{eq:log-laplace_termI}, \eqref{eq:log-laplace_termII} and \eqref{eq:log-laplace_termIII}, and by Lemma \ref{lemma:aggregate_log-laplace}, we have for any $t \in [0,\nu A^{-\gamma_1(1-\gamma)/\gamma})$ that
\begin{equation*}
    \log\mathbb{E}\Big[\exp\Big(t\sum_{i = 1}^A\overline{X}_M(i)\Big)\Big] \leq \frac{AV(A)t^2}{1-\nu^{-1}A^{\gamma_1(1-\gamma)/\gamma}t}.
\end{equation*}
\end{proof}

\end{document}